\documentclass[10pt]{amsart}
\usepackage[pagebackref]{hyperref}
\usepackage{pinlabel}
\usepackage{overpic}
\usepackage{svg}
\usepackage{sseq}
\hypersetup{
    colorlinks = true,
    citecolor = [rgb]{0, 0.5, 0.00},
}

\renewcommand*{\backref}[1]{}
\renewcommand*{\backrefalt}[4]{%
    \ifcase #1 (Not cited.)%
    \or        (Cited on page~#2.)%
    \else      (Cited on pages~#2.)%
    \fi}

\usepackage{mathtools}
\usepackage{amssymb}

\usepackage[margin=1.175in]{geometry}
\usepackage{todonotes}
\usepackage{multirow}
\usepackage{longtable}
\usepackage{enumerate}
\usepackage{mathabx}
\usepackage{amsmath}
\usepackage{amsfonts}
\usepackage{amssymb}
\usepackage{float}
\usepackage{amsthm}
\usepackage{comment}
\usepackage{amscd}
\usepackage{amsxtra}
\usepackage{epsfig}
\usepackage{slashed}

\usepackage{listings}
\usepackage{color}

\definecolor{dkgreen}{rgb}{0,0.6,0}
\definecolor{gray}{rgb}{0.5,0.5,0.5}
\definecolor{mauve}{rgb}{0.58,0,0.82}

\usepackage{color}
\usepackage[all]{xy}
\usepackage{verbatim}

\usepackage{eucal}
\usepackage{mathrsfs}

\usepackage{graphicx}
\usepackage{tikz-cd}

\usepackage{url}
\usepackage{verbatim}
\usepackage{xy}
\xyoption{all}

\usepackage{amsthm}
\usepackage{tikz}
\usetikzlibrary{decorations.pathreplacing}

\usepackage{caption} 
\makeatletter
\def\@tocline#1#2#3#4#5#6#7{\relax
  \ifnum #1>\c@tocdepth % then omit
  \else
    \par \addpenalty\@secpenalty\addvspace{#2}%
    \begingroup \hyphenpenalty\@M
    \@ifempty{#4}{%
      \@tempdima\csname r@tocindent\number#1\endcsname\relax
    }{%
      \@tempdima#4\relax
    }%
    \parindent\z@ \leftskip#3\relax \advance\leftskip\@tempdima\relax
    \rightskip\@pnumwidth plus4em \parfillskip-\@pnumwidth
    #5\leavevmode\hskip-\@tempdima
      \ifcase #1
       \or\or \hskip 1em \or \hskip 2em \else \hskip 3em \fi%
      #6\nobreak\relax
    \hfill\hbox to\@pnumwidth{\@tocpagenum{#7}}\par% <---- \dotfill -> \hfill
    \nobreak
    \endgroup
  \fi}
\makeatother

\DeclareMathOperator{\Aut}{Aut}

\DeclareMathOperator{\coker}{coker}
\DeclareMathOperator{\ind}{ind}

\DeclareMathOperator{\End}{End}

\DeclareMathOperator{\im}{im}
\DeclareMathOperator{\imaginary}{\mathfrak{Im}}

\newtheorem{lemma}{Lemma}[section]

\newtheorem{corollary}[lemma]{Corollary}

\newtheorem{theorem}[lemma]{Theorem}

\newtheorem{prop}[lemma]{Proposition}

\theoremstyle{definition}

\newtheorem{definition}[lemma]{Definition}

\newtheorem{example}[lemma]{Example}

\newtheorem{remark}[lemma]{Remark}

\numberwithin{equation}{section} % 在每个小节开始时重新编号方程
\newcommand{\Z}{\mathbb{Z}}

\newcommand{\M}{\mathfrak{M}}
\newcommand{\FM}{\mathcal{F}\mathfrak{M}}

\renewcommand{\L}{\mathcal{L}}

\newcommand{\CP}{\mathbf{C}P}
\newcommand{\RP}{\mathbf{R}P}

\newcommand{\Map}{\mathrm{Map}}

\renewcommand{\L}{\mathcal{L}}

\newcommand{\C}{\mathbb{C}}

\newcommand{\e}{\mathrm{e}}

\newcommand{\R}{\mathbb{R}}

\newcommand{\D}{\slashed{\mathfrak D}}

\renewcommand{\S}{\mathbb{S}}
\renewcommand{\ss}{\mathfrak{s}}

\newcommand{\GL}{\mathrm{GL}}

\newcommand{\SO}{\mathrm{SO}}

\newcommand{\SU}{\mathrm{SU}}

\renewcommand{\H}{\mathcal{H}}

\title{Surgery formulas for Seiberg-Witten invariants and family Seiberg-Witten invariants}
\author{Haochen Qiu}

\begin{document}

\maketitle

\begin{abstract}
    We prove a surgery formula for the ordinary Seiberg-Witten invariants, and surgery formulas for the families Seiberg-Witten invariants of families of $4$-manifolds obtained through  fibrewise surgery. Our formula expresses the Seiberg-Witten invariants of the manifold after the surgery, in terms of the original Seiberg-Witten moduli space cut down by a cohomology class in the configuration space. We use these surgery formulas to study how a surgery can preserve or produce exotic phenomena.
\end{abstract}

\tableofcontents
%\charpter{Introduction}
\section{Introduction}
Let $\gamma$ be a loop in a closed smooth $4$-manifold $X$. A surgery along $\gamma$ is removing a neighborhood of $\gamma$ with a trivialization of the normal bundle, and gluing back a copy of $D^2\times S^2$. For example, a surgery along $ S^1\times\{pt\} \subset S^1\times S^3 $ would produce $S^4$, while a surgery along a trivial loop on $S^4$ may produce $S^2\times S^2$ or $\CP^2 \# \overline{\CP^2}$.  So such surgery establishes relations between lots of $4$-manifolds. The four projects in this paper describe how a surgery can preserve or produce exotic phenomena.

The tool we use comes from the Seiberg-Witten equations, which depends on a metric and a self-dual $2$-form. The input of the equation for $X$ includes a $\text{Spin}^c$-structure (they are related to elements in $H^2(X;\Z)$), a  $U(1)$-connection, and a ``spinor''. The set of equivalence classes of $U(1)$-connections and spinors under the ``gauge group'' $\text{Map}(X,S^1)$ is called the \textbf{configuration space} (denoted by $\mathcal{B}$), which is a fiber bundle with fiber $\CP^\infty$ and base a torus $T^{b_1(X)}$. A tuple consisting of a metric and a perturbing $2$-form is called a \textbf{parameter}. The solution of this equation with a suitable parameter is a smooth compact manifold in the configuration space. This manifold is called the SW \textbf{moduli space} (denoted by $\M$). Its dimension is computed by the Atiyah-Singer index theorem, and if it is even, we can integrate a poduct of $c_1(\CP^\infty)$ on the moduli space and get the so-called \textbf{SW invariant} (when the dimension is $0$, the integral just counts the points with signs). This is an invariant under diffeomorphism. Many examples of exotic $4$-manifolds were found by computing this invariant for two homeomorphic manifolds.

The family SW invariant ($FSW$), on the other hand, can detect higher dimensional exotic phenomena. Given a smooth family of $X$ over a base $B$ and a corresponding family of parameters, the union of the solutions is called the parameterized moduli space, and if its dimension is $0$ then $FSW$ is the signed counts of points with orientation. For each $k\ge 0$, Ruberman-Auckly construct a $(k+1)$-family of $X$ such that the $FSW$ for this family is an invariant of $\pi_k(\text{Diff}(X))$.

In the following projects, we generalize $SW$ and $FSW$ to $1$-dimensional moduli space, such that new invariants (we call them $SW^\Theta$ and $FSW^\Theta$) can detect exotic phenomena. Then we prove several surgery formulas that show how a surgery changes $SW$, $FSW$, $SW^\Theta$ and $FSW^\Theta$.

\subsection{Surgery formula for homologically nontrivial loop}
For a $4$-manifold $X$ with 
\[
H^1(X;\Z)=\Z,
\]
suppose $\ss$ is a $\text{Spin}^c$-structure such that $\dim \M(X,\ss)=1$. The configuration space is homotopy equivalent to a bundle over $S^1$ with fiber $ \CP^\infty$. Let $\Theta$ be the pullback of a generator of $H^1(S^1;\Z)$. Define the cut-down Seiberg-Witten invariant $SW^\Theta(X,\ss)$ be the integral of $\Theta$ on $\M(X,\ss) $. We prove that this invariant detects exotic smooth structures.

Let $\gamma \subset X$ be a loop that represents a generator of $H_1(X;\Z)/\text{torsion} = \Z$. Suppose a surgery along $\gamma$ produces $X'$. We show that any $\text{Spin}^c$-structure $\ss$ on $X$ can be extended to a unique $\text{Spin}^c$-structure $\ss'$ on $X'$. Since the surgery kills the first cohomology group, $H^1(X';\Z)=0$ and therefore $\dim \M(X',\ss')=0$. Hence $SW(X',\ss')$ is defined by counting points in $ \M(X',\ss)$. The main theorem of this project is 
\begin{theorem}
$SW^\Theta(X,\ss) = SW(X',\ss')$.
\end{theorem}
This is proved by applying the classical gluing result in Nicolaescu's book \cite{Nicolaescu2000NotesOS} twice. Let $S^1 \times D^3$ be a neighborhood of $\gamma$, and let $X_0 = X - S^1 \times D^3$. Then gluing $X_0$ with $S^1 \times D^3$ produces $X$, while gluing $X_0$ with $D^2 \times S^2$ produces $X'$. The classical gluing result says, if a certain ``obstruction space'' is trivial on $X_0$, then $\M(X)$ is the fiber product $\M(X_0)\times_{\M(S^1 \times S^2)} \M(S^1 \times D^3)$ while $\M(X')$ is the fiber product $\M(X_0)\times_{\M(S^1 \times S^2)} \M(D^2 \times S^2)$. We prove that since $\gamma$ is homologically nontrivial, for generic parameters such obstruction space is trivial. Furthermore, we can choose suitable metrics such that $\M(S^1 \times D^3) \to \M(S^1 \times S^2)$ is the identity map of a circle, and $\M(D^2 \times S^2) \to \M(S^1 \times S^2)$ is the inclusion of one point into a circle. Hence if we cut $\M(X')$, we get $\M(X)$, and the theorem follows.

As lots of exotic smooth structures are detected by $SW$, we can now generalize those results to nonsimply connected manifolds, for example:
\begin{corollary}
$E(n) \# S^1 \times S^3$ admits infinitely many exotic smooth structures.
\end{corollary}
%We can also generalize the adjunction formula to nonsimply connected manifolds by the surgery formula:

The method developed in this project also works for the homologically trivial case. Let $\gamma \subset X$ be a loop that represents $0\in H_1(X;\Z)$. Suppose a surgery along $\gamma$ produces $X'$. We show that for any extension $\ss'$ of any $\text{Spin}^c$-structure $\ss$ on $X$ with $\dim \M(X,\ss)=0$, we have $\dim \M(X',\ss')=0$. Since $\gamma$ is homologically trivial, we will have
\begin{theorem}\label{thm:vanishing}
$ SW(X',\ss')=0$.
\end{theorem}
This generalizes the vanishing result of the connected sum with $S^2\times S^2$. Theorem \ref{thm:vanishing} can also be obtained by the generalized adjunction formula (\cite{KM94surfaces}), but the method in this project fits in the proof of family surgery formula below, where a homologically trivial loop has nontrivial higher exotic phenomena.
%This recovers a vanishing theorem which says if there is a self-intersection $0$ sphere in $X'$, all SW invariant vanishes.

\subsection{Family surgery formula for homologically nontrivial loop}
The motivation for this work is the following question: 

\textbf{Q:} If $X$ is a smooth manifold with an exotic diffeomorphism, can we find an exotic diffeomorphism on $X \# (\S^1\times \S^3)$? 

Here, an exotic diffeomorphism $f$ is a self diffeomorphism of $X$ such that $f$ is continously isotopic to the identity, but $f$ is not smoothly isotopic to the identity. Ruberman \cite{Ruberman1998AnOT} proves that $\C P^2 \# 2\overline{\C P}^2 \# E(2)$ admits an exotic diffeomorphism by the family Seiberg-Witten invariant (which would be explained later). Baraglia and Konno show that 
\[ n (\S^2 \times \S^2) \# (nK3)\] for $n \leq 2$ and \[2n\C P^2 \#(m \overline{\C P}^2)\] for $n\leq 2$ and $m\leq 10n +1$ admit exotic diffeomorphisms, by a gluing formula of the family Seiberg-Witten invariant. All these examples are simply connected. 

In this project we consider a nonsimply connected manifold $X$ with $H_1(X;\Z) = \Z$ and a smooth family $E_X$ of $X$ indexed by the parameter space $B$. Let $E_{S^1}$ be a subbundle such that each fiber of $E_{S^1}$ is a loop that represents a generator of $H_1(X;\Z) = \Z$. Suppose a family of surgeries along $E_{S^1}$ produces $E_{X'}$. Suppose $\ss$ is a $\text{Spin}^c$-structure such that $\dim \M(X,\ss)=\dim B +1$. As before any $\text{Spin}^c$-structure $\ss$ on $X$ can be extended to a unique $\text{Spin}^c$-structure $\ss'$ on $X'$, and we are able to define $\Theta$ similarly. Since the surgery kills the first cohomology group, $H^1(X';\Z)=0$ and therefore the parameterized moduli space on $X'$ has dimension $\dim \FM(X',\ss')=0$. Hence $FSW(X',\ss')$ is defined by counting points in $ \FM(X',\ss)$. The main theorem of this project is 
\begin{theorem}
$FSW^\Theta(E_X,\ss) = FSW(E_{X'},\ss')$.
\end{theorem}
%To obtain a useful

%We hope the parameterized moduli space on $X$ is $1$-dimensional. Then locally there would be two cases: For an isolated parameter the solution is $1$-dimensional, and there is no other nearby parameter such that the equation has solutions; Or there exists a $1$-dimesional family of parameters such that the solutions are $0$-dimensional for each of them.

The main issue here is that the parameterized moduli space on $X$ is $1$-dimensional. Then locally there would be two cases:
\begin{enumerate}
\item[1)]  For an isolated parameter the solution is $1$-dimensional, and there is no other nearby parameter such that the equation has solutions; 
\item[2)]  There exists a $1$-dimensional family of parameters such that the solutions are $0$-dimensional for each of them.
\end{enumerate}
By analysing Hodge star operator and an exact sequence, it turns out that these cases depend purely on topological properties of $X_0$. When $\gamma$ is homologically nontrivial, we prove that for a generic parameter, the parameterized moduli space on $X_0$ is of case $1$, and the dimension of the obstruction space on $X_0$ is equal to $\dim B$, and therefore we can apply a method developed by Baraglia-Konno\cite{BK2020}.

%Suppose $\mathcal{Z}_{reg}$ is the space of parameter families for which the parameterized moduli space on $X$ is $1$-dimensional smooth manifold. Then $\mathcal{Z}_{reg}$ is open dense.

This cut-down family invariant generalizes exotic diffeomorphisms found by Ruberman\cite{Ruberman1998AnOT} and Baraglia-Konno\cite{BK2020}. For example:
\begin{corollary}
Let $X$ be one of the following manifolds:
\begin{enumerate}
\item[\textbullet] $\mathbb{C}\mathbb{P}^2\# (\#^2\overline{\mathbb{C}\mathbb{P}}^2) \# Y$ and $b^+_2(Y) >2$.
\item[\textbullet]  $\#^n(\mathbb{S}^2\times \mathbb{S}^2) \# (\#^nK3)$ for $n \geq 2$.
\item[\textbullet] $\#^{2n}\mathbb{C}\mathbb{P}^2  \# (\#^m\overline{\mathbb{C}\mathbb{P}}^2)$ for $n \geq 2$ and $m \geq 10n+1$.
\end{enumerate}
Then $X\#( \S^1 \times \S^3)$ admits an exotic diffeomorphism.
\end{corollary}
Ruberman\cite{ruberman2002positive} gives examples of simply connected manifolds for which the space of positive scalar curvature (psc) metrics is disconnected. This is demonstrated using family Seiberg-Witten invariant. We can generalize these results by the family surgery formula:
\begin{corollary}
Let $X$ be one of the following manifolds:
\begin{enumerate}
\item[\textbullet] $\mathbb{C}\mathbb{P}^2\# (\#^2\overline{\mathbb{C}\mathbb{P}}^2) \# Y$ for $b^+_2(Y)\ge 3$ .
\item[\textbullet] $\#^{2n}\mathbb{C}\mathbb{P}^2  \# (\#^m\overline{\mathbb{C}\mathbb{P}}^2)$ for $n \geq 2$ and $m \geq 10n+1$.
\end{enumerate}
Then the space of psc metrics on $X\# (\S^1 \times \S^3)$ has infinite many path components.
\end{corollary}

Konno proves that $\pi_0(\text{Diff}(X))$ is not finitely generated for some simply connected $4$-manifold. We can generalize his result to nonsimply connected $4$-manifolds:
\begin{corollary}\label{coro:inf-gen}
There exists a simply connected $4$-manifold $X$ that is not a sphere, such that 
\[
\pi_0(\text{Diff}(X\# (\S^1 \times \S^3)))
\]
is not finitely generated.
\end{corollary}

\subsection{Family surgery formula for homologically trivial loops}
In this project, we suppose each fiber of $E_{S^1}$ is a homologically trivial loop. Then we have 
\begin{theorem}\label{thm:generalize-stabilization}
Use the notation as before and assume the following:
\begin{itemize}
\item $\dim B > 0$;
\item $E_{\S^1} $ is an orientable $\S^1$-subbundle of $E_X$.
\end{itemize}
Then
\[
FSW(E_{X'}, s')=0.
\]
\end{theorem}
As we remark above, a surgery along a homologically trivial loop can preoduce nontrivial exotic phenomena:
\begin{theorem}\label{thm:nullhomology-nonorientable}
Use the notation as before and assume the following:
\begin{itemize}
\item $B$ is a circle;
\item $E_{\S^1} $ is an $\S^1$-subbundle of $E_X$, and it is a Klein bottle;
\end{itemize}
Then
\[
FSW^{\Z/2}(E_{X'}, s')\equiv SW(X, s) \mod 2.
\]
(Here the family invariant is defined by counting the points mod $2$.)
\end{theorem}
When $\gamma$ is homologically trivial, we prove that for a generic parameter, the parameterized moduli space on $X_0$ is of case $2$: there exists a $1$-dimesional family of parameters such that the solutions are $0$-dimensional for each of them. The dimension of the obstruction space on $X_0$ is one higher than $\dim B$, and therefore we have to generalize the method developed by Baraglia-Konno and estimate the errors by some inequalities.

A special example of these theorems is that each fiber of $E_{S^1}$ is a homotopically trivial loop. In this case $X'$ is $X\#( S^2\times S^2)$ or $X\#  \CP^2 \# \overline{\CP^2}$, and the results for $X\# (S^2\times S^2)$ were previously obtained by Baraglia-Konno\cite{BK2020}. But Theorem \ref{thm:nullhomology-nonorientable} works also for a homotopically nontrivial loop, so it has the potential to produce exotic diffeomorphisms on a irreducible manifold.

%\section{$\text{Spin}^\C$ structure }
\section{Setup for the $1$-surgery formula}
\subsection{$\text{Spin}^\C$ structure }\label{subsection:Spin}
The definition of the Seiberg-Witten moduli space depends on a choice of the $\text{Spin}^\C$ structure, so we first review the theory of the $\text{Spin}^\C$ structure. 
Definitions in this subsection can be found in section 1.4.2 and 2.4.1 of \cite{gompf19994}. We also provide some auxilary examples (Example \ref{example:spin} and Remark \ref{remark:spincline}). The main theorem in this subsection is Theorem \ref{thm:changeOfSpinc}. It deals with the change of $\text{Spin}^\C$ structures by a $1$-surgery.

To understand the $\text{Spin}^\C$ structure, we first review the theory of the $\text{spin}$ structure. % (for details, see section 1.4.2 of \cite{gompf19994}).
\begin{definition}
    \[
        \text{Spin}(4) =  SU(2) \times SU(2) 
    \]
    is called the $\text{spin}$ group of dimension $4$.
\end{definition}
Note that, $\text{Spin}(4)$ is the connective double cover of $SO(4) = SU(2) \times SU(2) / \{\pm (I,I)\}$. 

\begin{remark}\label{remark:spin}
    Double covers of $X$ correspond to $H^1(X;\Z_2) =[X, \RP^\infty]$. The correspondence is given by the sphere bundle of pull back of the universal line 
    bundle (tautological line bundle over $\RP^\infty$). For $SO(4)$, $H^1(SO(4);\Z_2) =[SO(4), \RP^\infty] = \Z_2$. So the double covers of $SO(4)$ are charecterized by the homotopy class of the image of the nontrivial loop of $SO(4)$ in $\RP^\infty$. 
    If that loop is homotopic to a constant loop in $\RP^\infty$, then the corresponding double cover is $SO(4)\sqcup SO(4)$. 
    If that loop is homotopic to the $1$-cell of $\RP^\infty$, then the double cover is $\text{Spin}(4)$.
\end{remark}

\begin{definition}
    A $\text{spin}$ structure $\ss$ on a $4$-manifold $M$ is a principal $\text{Spin}(4)$-bundle 
    $P_{\text{Spin}(4)} \to M$, with a bundle map from $P_{\text{Spin}(4)}$ to the frame bundle $P_{SO(4)}$ of $M$, which restricts
    to the double cover $\rho: \text{Spin}(4) \to SO(4)$ on each fiber.
    \end{definition}

Note that $P_{\text{Spin}(4)}$ is a double cover of $P_{SO(4)}$, which restricts to the double cover $\rho: \text{Spin}(4) \to SO(4)$ on each fiber. By Remark \ref{remark:spin}, 
this corresponds to an element in $H^1(P_{SO(4)};\Z_2)=[P_{SO(4)}, \RP^\infty]$ which restricts to the nontrivial element in $H^1(SO(4);\Z_2)=[SO(4), \RP^\infty]$ on each fiber. 
From the Leray-Serre spectral sequence, we have the following exact sequence:
\[
    0\to H^1(M, \Z_2) \to H^1(P_{SO(4)}, \Z_2) \stackrel{i^*}{\to} H^1(SO(4), \Z_2) \stackrel{\delta}{\to} H^2(M, \Z_2).
\]
Here $\delta(1) = w_2(P_{SO(4)})$, and $i^*$ is the restriction map. By the discussion above, the set of $\text{spin}$ structures on $M$ 
is in one-to-one correspondence with $(i^*)^{-1}(1)$. When $\delta(1) = w_2(P_{SO(4)}) = 0$, $(i^*)^{-1}(1)$ is nonempty, 
and 
\begin{align*}
\#(i^*)^{-1}(1) &=\#(i^*)^{-1}(0)\\
& = \# \im (H^1(M, \Z_2) \to H^1(P_{SO(4)}, \Z_2)) \\
&= \# H^1(M, \Z_2) .
\end{align*}
So the set of $\text{spin}$ structures on $M$ 
is in noncanonical one-to-one correspondence with $H^1(M, \Z_2)$. 
When $\delta(1) = w_2(P_{SO(4)}) \neq 0$, $(i^*)^{-1}(1)$ is empty.

\begin{example}\label{example:spin}
    Let $M = \S^1 \times \R^3$. Then $w_2(TM) = 0$ and $H^1(M, \Z_2) = \Z_2$. Hence there are two $\text{spin}$ structures on $M$. 
    They are principal $\text{Spin}(4)$-bundles that cover the trivial bundle $P_{SO(4)} = M \times SO(4)$, and the covering maps are nontrivial on each fiber. 
    Namely, the preimage of the nontrivial loop of $SO(4)$ is $\S^1$, and the covering maps restrict to this preimage are both 
    \begin{align*}
    \S^1 &\stackrel{2}{\to} \S^1\\
     z &\mapsto z^2.
     \end{align*}
    These two $\text{spin}$ structures are distinguished by the covering maps on the $\S^1$ factor of $M$. They are nontrivial double cover 
    $\S^1 \stackrel{2}{\to} \S^1 \subset SO(4)$ and trivial double cover $\S^1\sqcup \S^1 {\to} \S^1 \subset SO(4)$, respectively. 
    
    We can construct these principal $\text{Spin}(4)$-bundles explicitly. Let $\{U_\alpha, U_\beta\}$ 
    be a good cover of $M$ such that $U_\alpha$ and $U_\beta$ are diffeomorphic to $\R \times \R^3$. Let $U_0 \sqcup U_1 = U_\alpha \cap U_\beta$. 
    Let $P_{SO(4)}$ be the frame bundle of $M$ with local trivialization on $\{U_\alpha, U_\beta\}$ and transition functions $g_i: U_i \to SO(4)$ for $i = 0,1$. 
    Fix $m_i \in U_i$ and a lift 
    \[\widetilde{g_i(m_i)} \in \text{Spin}(4)\] 
    for ${g_i}(m_i)$ respectively. Since $P_{\text{Spin}(4)} \to P_{SO(4)}$ is a fibration and $U_i$ is contractible, 
    we can lift $g_i$ to a map $\tilde{g_i}: U_i \to \text{Spin}(4)$ such that \[\tilde{g_i}(m_i) = \widetilde{g_i(m_i)}.\] 
    This gives the transition functions for a principal $\text{Spin}(4)$-bundle $P_{\text{Spin}(4)}$ over $M$ which is locally trivial on $\{U_\alpha, U_\beta\}$. 
    To construct another principal $\text{Spin}(4)$-bundle, we choose the same lift of $g_0(m_0)$ but a different lift of $g_1(m_1)$.

    For example, if 
    \[g_i(m) = [I,I] \in SO(4) = SU(2) \times SU(2) / \{\pm (I,I)\}\] 
    and 
    \[\tilde{g_i}(m) = (I,I)\in \text{Spin}(4) =  SU(2) \times SU(2)\]
    for any $i$ and $m\in U_i$, then the principal $\text{Spin}(4)$-bundle is trivial. For the loop $l = \S^1 \times \{0\} \times \{I\} \subset M \times  SO(4) =  P_{SO(4)}$, 
    the preimage of $l$ under the double cover $P_{\text{Spin}(4)} \to P_{SO(4)}$ is $\S^1\sqcup \S^1 \subset P_{\text{Spin}(4)}$. On the other hand, if
    \begin{align*}
        \tilde{g_0}(m) &= (I,I), m\in U_0\\
        \tilde{g_1}(m) &= (-I,-I), m\in U_1,
    \end{align*}
    then geometrically, when a particle runs along $l$, it's preimage under the double cover $P_{\text{Spin}(4)} \to P_{SO(4)}$ changes to another orbit when this particle passes $U_1$. Thus the preimage of $l$ is a single $\S^1 \subset P_{\text{Spin}(4)}$. 
    This example shows that the set of $\text{spin}$ structures on $M$ 
is in one-to-one correspondence with $H^1(M, \Z_2)$. Moreover, such correspondence is noncanonical: There is not a priori choice of the lift of $g_i(m_i)$.

\end{example}

Now we introduce the $\text{spin}^\C$ structure.

\begin{definition}
    \[
        \text{Spin}^\C(4) = \{(A,B)\in U(2) \times U(2) ; \text{det}(A) = \text{det}(B) \}
    \]
    is called the $\text{spin}^\C$ group of dimension $4$. 
\end{definition}

Note that, $\text{Spin}^\C(4)$ is isomorphic to $S^1 \times SU(2) \times SU(2) / \{\pm (1,I,I)\}$, while $SO(4)$ is isomorphic to 
$ SU(2) \times SU(2) / \{\pm (I,I)\}$. Hence we have an $S^1$-fiberation
\begin{align}\label{equ:rhoc}
    \rho^c:\text{Spin}^\C(4) &\to SO(4)\\
    [(z,A,B)] &\mapsto [(A,B)].
\end{align}

\begin{definition}\label{def:spinc}
A $\text{spin}^\C$ structure $\ss$ on a manifold $M$ is a principal $\text{Spin}^\C(4)$-bundle 
$P_{\text{Spin}^\C(4)} \to M$, with a bundle map from $P_{\text{Spin}^\C(4)}$ to the frame bundle $P_{SO(4)}$ of $M$, which restricts
to $\rho^c$ on each fiber.
\end{definition}

Looking at the definition of $\rho^c$, we find that a $\text{spin}^\C$ structure contains one more infomation than the frame bundle:
\begin{definition}
Let 
\begin{align}
    \det: \text{Spin}^\C(4) &\to S^1 \\
    \label{equ:determinant}[(z,A,B)] &\mapsto z^2.
\end{align}
The line bundle $\L = P_{\text{Spin}^\C(4)} \times_{\det} \C $ is called the determinant line bundle associated to the $\text{spin}^\C$ structure $\ss$. 
\end{definition}

A $\text{spin}^\C$ structure is actually a double cover of the frame bundle tensor the determinent line bundle. We have 
an exact sequence
\begin{align}\label{equ:doublespinc}
    1\to \Z_2 \to \text{Spin}^\C(4) &\stackrel{\rho'}{\to} S^1 \times SO(4) = SO(2) \times SO(4) \to 1\\
    [(z,A,B)] &\mapsto (z^2,[(A,B)]).
\end{align}
The double cover $\rho'$ can be extended to a double cover of $SO(6)$ (see page 56 of \cite{gompf19994}). Hence the $\text{spin}^\C$ structure exists if and only if
the second Stiefel-Whitney class $w_2(P_{\S^1\times SO(4)})$ vanishes, by the theory of the existence of spin structures metioned above. Namely, 
\begin{align}
    w_2(P_{\S^1\times SO(4)}) &= w_2(P_{\S^1}) + w_2(P_{SO(4)}) \\
    &= w_2(\L) + w_2(TM)\\
    &= 0 \in \Z/2.
\end{align}
Namely, $ w_2(TM) \equiv c_1(\L) \mod 2$. An integral cohomology class congruent to $w_2(TM)$ is called characteristic element. The set of characteristic elements is nonempty for any $4$-manifold (see Proposition 5.7.4 of \cite{gompf19994}). Thus the $\text{spin}^\C$ structure always exists.

\begin{remark}\label{remark:spincline}
Different choices of the double covers of $P_{\S^1\times SO(4)}$ (with the covering map $\rho'$ fiberwise) do not always give different $\text{spin}^\C$ structures. 
Indeed, the set of $\text{spin}^\C$ structures over $M$ is in (non-canonical) one-to-one correspondence with the isomorphism classes of complex line bundles over $M$. 
Recall that,
\begin{align}
    SO(4)  &= SU(2) \times SU(2) / \{\pm (I,I)\}\\
    \text{Spin}(4) &=  SU(2) \times SU(2) \\
    \text{Spin}^\C(4) &=  S^1 \times SU(2) \times SU(2) / \{\pm (1,I,I)\} = S^1 \times  \text{Spin}(4) / \{\pm (1,I)\}. \label{equ:spinc}
\end{align}
Thus the transition functions of a principal $\text{Spin}^\C(4)$-bundle over $M$ are given by $[z_{\alpha\beta},g_{\alpha\beta}]$ 
where $z_{\alpha\beta}: U_{\alpha\beta}\to S^1$ and $g_{\alpha\beta}: U_{\alpha\beta}\to \text{Spin}(4)$ for a good cover $\{U_{\alpha}\}$. 
Suppose we have two $\text{spin}^\C$ structures 
\begin{align*}
&P^{(1)}_{\text{Spin}^\C(4)} \to M\\ 
&P^{(2)}_{\text{Spin}^\C(4)} \to M
\end{align*}
with transition functions 
$[z_{\alpha\beta}^{(1)},g_{\alpha\beta}^{(1)}]$ and $[z_{\alpha\beta}^{(2)},g_{\alpha\beta}^{(2)}]$ respectively. 
Note that by the definition of the $\text{spin}^\C$ structure, 
\[
\rho^c([z_{\alpha\beta}^{(i)},g_{\alpha\beta}^{(i)}])=[g_{\alpha\beta}^{(i)}] \in SO(4)
\] 
would be the transition functions of the frame bundle $P_{SO(4)}$. 
Hence we have either 
\[g_{\alpha\beta}^{(1)} = g_{\alpha\beta}^{(2)}\in  \text{Spin}(4)\]
or 
\[g_{\alpha\beta}^{(1)} = -g_{\alpha\beta}^{(2)} \in  \text{Spin}(4).\] 
If it's the latter case, we can always choose a different representative of $[z_{\alpha\beta}^{(2)},g_{\alpha\beta}^{(2)}]$. Thus we can assume that 
$g_{\alpha\beta}^{(1)} = g_{\alpha\beta}^{(2)}$. Then 
\[
\theta_{\alpha\beta} =  z_{\alpha\beta}^{(2)}/z_{\alpha\beta}^{(1)}
\]
would give 
the transition functions of a complex line bundle $\L$ over $M$, such that 
\[
P^{(1)}_{\text{Spin}^\C(4)}\otimes \L \cong P^{(2)}_{\text{Spin}^\C(4)}.
\]
(This shows that the action of $H^2(M;\Z) = [M,\C P^\infty]$ on the set of $\text{spin}^c$-structures is transitive. Actually this action is also free.)

By the definition of the determinant line bundle, 
\[
\det (P^{(1)}_{\text{Spin}^\C(4)}\otimes \L) = \det(P^{(1)}_{\text{Spin}^\C(4)}) \otimes \L^2.
\]
Hence 
\[
c_1(\det (P^{(1)}_{\text{Spin}^\C(4)}\otimes \L)) = c_1( \det(P^{(1)}_{\text{Spin}^\C(4)})) + 2c_1(\L).
\]
When $H^2(M;\Z)$ has no $2$-torsion, $2c_1(\L) = 0$ iff $c_1(\L) = 0$, iff $\L$ is trivial. Hence $c_1\circ \det$ is injective. If $P^{(1)}_{\text{Spin}^\C(4)}$ and $P^{(2)}_{\text{Spin}^\C(4)}$ are two different choices of the double covers of $P_{\S^1\times SO(4)}$ (with the covering map $\rho'$ fiberwise), then the difference line bundle $\L$ has transition functions $\theta_{\alpha\beta} = \pm 1$ such that $\L^2$ is trivial. Hence 
\[
\det(P^{(2)}_{\text{Spin}^\C(4)}) = \det (P^{(1)}_{\text{Spin}^\C(4)}\otimes \L) = \det(P^{(1)}_{\text{Spin}^\C(4)}) \otimes \L^2 =  \det (P^{(1)}_{\text{Spin}^\C(4)})
\]
and therefore $c_1\circ \det$ sends them to the same element. Hence they are isomorphic $\text{spin}^\C$ structures.

In conclusion, although it seems that by (\ref{equ:doublespinc}) and (\ref{equ:spinc}) a $\text{spin}^\C$ structure encodes some infomation of the spin structure, and by Example \ref{example:spin}, each element of $H^1$ would produce a different spin structure, but that difference comes from the different choice of the lift of $\text{Spin}(4) \to SO(4)$, which can be passed to the difference of the complex line bundle in $P_{\text{Spin}^\C(4)}$.
\end{remark}

%\begin{remark}[Remark of Remark \ref{remark:spincline}]
%With the exitence of a 
%\end{remark}
For a $1$-surgery along a nontrivial loop, all $\text{spin}^\C$ structures can be extended to the new manifold. The extension is not unique. However, it would not change the index of Dirac operator.
\begin{theorem}\label{thm:changeOfSpinc}%\label{thm:index}
Let $X$ be any $4$-manifold with $H^1(X;\Z) = \Z$. Let $\alpha$ be a generator of $H^1(X;\Z)$. Let $\gamma$ be the loop we choose to do the surgery, with $ \langle \alpha, \gamma \rangle = 1$. Let $N = \S^1 \times D^3$ be a small enough tubular neighborhood of $\gamma$.
    Let $X_0$ be the complement of $N$. Let $X'= X_0 \cup_{\S^1 \times \S^2} (D^2 \times \S^2)
    $ be the manifold obtained by doing the surgery on $X$ along $\gamma$. Let $\ss$ be any $\text{Spin}^\C$ structure over $X$ and $\L$ be the corresponding determinant line bundle. Let $\mathcal{S}(X')$ be the set of $\text{spin}^\C$ structures on $X'$, and
    \[
    \mathcal{S}(X', \ss) := \{\Gamma \in \mathcal{S}(X') ; \left.\Gamma\right|_{X_0} =\left. \ss\right|_{X_0}\}.
    \]
    Then $ \mathcal{S}(X', \ss) $ contains a unique (up to an isomorphism) $\text{Spin}^\C$ structure $\ss'$ over $X'$,
    and the determinant line bundles $\L'$ associated to $\ss'$ satisfies
    \[
    \langle c_1(\L')^2,X' \rangle = \langle c_1(\L)^2 ,X \rangle.
    \]
In particular, above results do not depend on the framing of the $1$-surgery.
\end{theorem}

\begin{proof}
We first show that $\mathcal{S}(X', \ss)$ is nonempty. Let $\ss'$ be any $\text{Spin}^\C$ structure over $X'$. By Remark \ref{remark:spincline}, 
the difference between $\ss'|_{X_0}$ and $\ss|_{X_0}$ is a complex line bundle $L_0$ over $X_0$, namely, $\ss'|_{X_0} \otimes L_0 = \ss|_{X_0}$. 

We claim that $L_0$ can be extended to a complex line bundle $L'$ over $X'$. Indeed, for the inclusions 
\begin{align*}
    i_\partial: \partial X_0 = \S^1 \times \S^2 &\to D^2 \times \S^2\\
    i: X_0 &\to X',
\end{align*} 
the induced homomorphisms
\begin{align*}
    i_\partial^*: H^2(D^2 \times \S^2) &\to H^2(\S^1 \times \S^2)\\  
    i^*: H^2(X') &\to H^2(X_0)
\end{align*}
are all isomorphisms. This follows from the following Mayer-Vietoris sequence (the last three terms form a split short exact sequence):

\begin{center}
\begin{tikzpicture}[commutative diagrams/every diagram]
\node (P0) at (0cm, 0cm) {$H^1(X_0)\oplus H^1(D^2 \times \S^2)$};
\node (P1) at (-1cm, -0.1cm) {};
\node (P2) at (3.3cm, 0cm) {$H^1(\S^1 \times \S^2)$} ;
\node (P6) at (3.3cm, -1cm) {$\Z$} ;
\node (P3) at (5.5cm, 0cm) {$H^2(X')$};
%\node (P2) at (90+2*72:2cm) {\makebox[5ex][r]{$(X\otimes (Y\otimes Z))\otimes T$}};
%\node (P3) at (90+3*72:2cm) {\makebox[5ex][l]{$((X\otimes Y)\otimes Z)\otimes T$}};
\node (P4) at (-1cm, -1cm) {$\Z$};
\node (P5) at (8.5cm, 0cm) {$H^2(X_0)\oplus H^2(D^2 \times \S^2)$};
\node (P7) at (8.7cm, -0.1cm) {};
\node (P8) at (8.7cm, -1cm) {$\Z$};
\node (P9) at (12cm, 0cm) {$H^2(\S^1 \times \S^2)$};
\node (P10) at (12cm, -1cm) {$\Z$};
\path[commutative diagrams/.cd, every arrow, every label]
(P0) edge node {} (P2)
%(P1) edge node[swap] {$\phi$} (P2)
(P2) edge node {$0$} (P3)
%(P4) edge node {$\phi$} (P3)
(P1) edge node {$\cong$} (P4)
(P3) edge node {$i^*$} (P5)
(P2) edge node {$\cong$} (P6)
(P4) edge node {$\cong$} (P6)
(P7) edge node {$\cong$} (P8)
(P5) edge node {} (P9)
(P8) edge node {$ i_\partial^*$} (P10)
(P9) edge node {$\cong$} (P10);
\end{tikzpicture}
\end{center}
Topologically, the dual of $c_1(L_0)|_{\partial X_0}$ is some copies of $\S^1 \times \{pt\} \subset \S^1 \times \S^2 = \partial X_0$, and they can be extended to $D^2 \times \{pt\} \subset D^2 \times \S^2$. Anyway, there exists a cohomology class in $H^2(X',\Z)= [X' , \CP^\infty]$ such that it restricts to $c_1(L_0) \in H^2(X_0,\Z)$, and by the property of the universal complex line bundle over $\CP^\infty$, the pullback $L'$ is a complex line bundle over $X'$ that restricts to $L_0$. Therefore, we have
\[
\left.(\ss'\otimes L')\right|_{X_0}  =\left.\ss'\right|_{X_0} \otimes L_0 =  \left.\ss\right|_{X_0}.
\]
So $\ss'\otimes L' \in \mathcal{S}(X', \ss)$.

Next, we prove that all elements in $\mathcal{S}(X', \ss)$ are isomorphic. Let $\ss'_{(1)}, \ss'_{(2)} \in \mathcal{S}(X', \ss)$. Let $L'$ be a complex line bundle on $X'$ such that 
\[
 \ss'_{(1)} \otimes L' = \ss'_{(2)}.
 \]
Then
\begin{align*}
\left.\ss'_{(2)}\right|_{X_0}  &=  \left.(\ss'_{(1)} \otimes L')\right|_{X_0}\\
 &= \left.\ss'_{(1)}\right|_{X_0} \otimes \left.L'\right|_{X_0} \\
 &= \left.\ss\right|_{X_0} \otimes \left.L'\right|_{X_0}\\
 &= \left.\ss'_{(2)}\right|_{X_0}\otimes \left.L'\right|_{X_0}.
\end{align*}
Remark \ref{remark:spincline} shows that the action of $H^2(X_0)= [X_0,\C P^\infty]$ on $\mathcal{S}(X_0)$ is transitive. Actually this action is also free. Hence $c_1( L'|_{X_0})= 0\in H^2(X_0,\Z)$. Note that $i^*(c_1( L')) = c_1( L'|_{X_0})$ and $i^*$ is an isomorphism. Therefore $c_1( L')=0\in H^2(X',\Z)$. So $L'$ is trivial and $\ss'_{(1)}= \ss'_{(2)}$.

Lastly, we show that 
\[
    \langle c_1(\L')^2,X' \rangle = \langle c_1(\L)^2 ,X \rangle.
    \]
The intersection between a generic section of $\L$ and the zero section is a $2$-manifold $\Sigma \subset X$. For dimension reason we can assume $\gamma \cap \Sigma  = \emptyset$. By choosing a small enough neighborhood of $\gamma$ we can further assume $\Sigma \subset X_0$. $ \langle c_1(\L)^2 ,X \rangle$ is the self-intersection $[\Sigma]^2$ of $\Sigma$.

Since $\ss'|_{X_0} = \ss|_{X_0}$, $\L'|_{X_0} = \det(\ss')|_{X_0} =\det(\ss)|_{X_0} =\L|_{X_0}$. As a complex line bundle, $\L|_{\S^1\times D^3}$ must be trivial. Hence it's a trivial line bundle over $\partial X_0$. Since $ i^*: H^2(X') \to H^2(X_0)$ is an isomorphism, $\L'$ is the unique extension of $\L'|_{X_0}=\L|_{X_0}$, and therefore it must extend $\L|_{\partial X_0}$ trivially. Hence the generic section of $\L|_{X_0}$ mentioned above can be extended to $X'$ without additional zeros. Hence $\langle c_1(\L')^2,X' \rangle = [\Sigma]^2 = \langle c_1(\L)^2 ,X \rangle$.
\end{proof}

In the gluing theory of Seiberg-Witten monopoles, the Seiberg-Witten equations and thus the $\text{spin}^c$ structure of the boundary $\partial X_0 = \partial N = \S^1\times\S^2$ would be considered. Hence one has to consider how to restrict the $\text{spin}^c$ structure of the $4$-manifold $X_0$ to the $3$-manifold $\S^1\times\S^2$.

Let $X$ be any $4$-manifold with boundary $\partial X$. Identify $TX|_{\partial X}$ with $T\partial X \oplus \nu$ where $\nu$ is the normal bundle of $\partial X\subset X$. 
Let $P_{SO(4)}$, $P_{SO(3)}$ be the frame bundles of $T\partial X \oplus \nu$ and $T\partial X$, Let $g^{(4)}\in SO(4)$, $g^{(3)}\in SO(3)$ be corresponding transition functions on a point $x\in \partial X$. The following diagram commutes:
\begin{center}
\begin{tikzpicture}[commutative diagrams/every diagram]
\node (P0) at (0cm, 0cm) {Fr(3)};
\node (P1) at (0cm,  -1.6cm) {Fr(3)};
\node (P2) at (1.7cm, 0cm) {Fr(4)};
\node (P3) at (1.7cm, -1.6cm) {Fr(4)};
\node (P4) at (0.4cm, -0.8cm) {$g^{(3)}$};
\node (P5) at (1.4cm, -0.8cm) {$g^{(4)}$};
\node (P5) at (0.9cm, -0.9cm) {$\mapsto$};
\node (P6) at (0.9cm, -0.7cm) {$i$};
%\draw[->] (P4) edge node (P5);
\path[commutative diagrams/.cd, every arrow, every label]
(P0) edge node {} (P1)
(P0) edge node {} (P2)
(P1) edge node {} (P3)
(P2) edge node {} (P3);
%(P4) edge node {} (P5) {\(\mapsto\)};
\end{tikzpicture}
\end{center}
where the top and bottom horizontal arrows are given by adding an inner vector. Then the map $i$ between transition functions is given by the natural embedding of $SO(3)\to SO(4)$.

Let $\mathbb H$ be quaternions and $\SU(2)=\S^3$ be the group of unit quaternions. $q\in \SU(2)$ acts on $\imaginary\mathbb H$ by
\[
x \mapsto qxq^{-1},
\]
which gives the double cover $\rho_3:\SU(2) = \text{Spin}(3) \to \SO(3)$. $(p,q)\in \SU(2)  \times \SU(2) =\text{Spin}(4) $ acts on $\mathbb H$ by
\[
x \mapsto pxq^{-1},
\]
which gives the double cover $\rho :\text{Spin}(4) \to \SO(4)$. Regard the real axis of $\mathbb{H}$ as the normal space of $x\in \partial X$, then 
\begin{align*}
i: \text{Spin}(3)  &\to \text{Spin}(4) \\
q&\mapsto (q,q)
\end{align*}
covers the embedding $i:SO(3)\to SO(4)$. Similarly we have a map
\begin{align*}
i^c: \text{Spin}^c(3):=S^1 \times  \text{Spin}(3) / \{\pm (1,I)\} &\to \text{Spin}^c(4)\\
[z,q]&\mapsto [z,q,q]
\end{align*}
that covers $i:SO(3)\to SO(4)$. Hence a $\text{spin}$($\text{spin}^c$) structure of $X$ induces a $\text{spin}$($\text{spin}^c$) structure of $\partial X$. Moreover, from the definition of $i^c$, the restriction of a $\text{spin}^c$ structure is compatible with the restriction of its determinant line bundle.

\begin{prop}\label{prop:spincRestriction}
Use the notations in Theorem \ref{thm:changeOfSpinc}. Then $\ss|_{\partial X_0}$ is the only $\text{spin}^c$ structure of $ \S^1 \times \S^2 $ such that the first Chern class of the determinant line bundle is zero, and $\ss'|_{D^2 \times \S^2}$ is the only $\text{spin}^c$ structure of $ D^2 \times \S^2 $ such that the first Chern class of the determinant line bundle is zero.
\end{prop}
\begin{proof}
$\det(\ss|_{\partial X_0}) = \det(\ss)|_{\partial X_0}$ is the restriction of the trivial line bundle $ \det(\ss)|_{\S^1 \times D^3}$.  So $\det(\ss|_{\partial X_0})$ is trivial. $H^2(\partial X_0;\Z) = 0$ so by Remark \ref{remark:spincline} $\ss|_{\partial X_0}$ is the only $\text{spin}^c$ structure of $ \S^1 \times \S^2 $.

$\det(\ss'|_{D^2 \times \S^2})|_{\partial (D^2 \times \S^2)}=\det(\ss)|_{\S^1 \times \S^2} $ is trivial. Since the restriction $H^2(D^2 \times \S^2) \cong H^2( \S^1 \times\S^2) $ is an isomorphism, $c_1(\det(\ss'|_{D^2 \times \S^2})) = 0$. 
$H^2(D^2 \times \S^2;\Z)$ has no torsion so by Remark \ref{remark:spincline} $c_1\circ \det$ is injective. Hence $\ss'|_{D^2 \times \S^2}$ is the only $\text{spin}^c$ structure of $D^2 \times \S^2$ such that the first Chern class of the determinant line bundle is zero.
\end{proof}
%So it can be extended to a trivial line bundle over $D^2\times  \S^2$. So one has a determinant line bundle $\L'$ over $X'$ 
%that coincides with $\L$ on $X_0$. The framing of the surgery does not affect this construction, 
%as we are simply extending trivial bundle over $\S^1$ to $D^2$. 

%One has to check that $c_1(\L')$ is a characteristic element. This is because $c_1(\L'|D^2\times  \S^2) = 0$ and 
%$w_2(T(D^2\times  \S^2)) = 0 $. Hence $c_1(\L') = w_2(TX') \mod 2$. So $\L'$ is the determinant line bundle associated to some
%$\text{spin}^\C$ structure $\ss'$ over $X'$.
%\end{proof}

%Let $\ss$ be a $\text{Spin}^\C$ structure over $X$. Suppose $\L$ is the corresponding determinant line bundle. As a complex line bundle, $\L|_{\S^1\times \S^2}$ must be trivial on each slice of $\S^1$. Thus, it can be extended to a line bundle over $D^2\times  \S^2$. So one has a determinant line bundle $\L_1$ over $X'$ that coincides with $\L$ on $X_0$. The framing of the surgery does not affect this construction, as we are simply extending trivial bundle over $\S^1$ to $D^2$.

%Since $H_2(X')$ doesn't have torsion, $\text{Spin}^\C$ structure of $X'$ is uniquely determined by its determinant line bundle. Suppose $\ss_1$ is the $\text{Spin}^\C$ structure corresponds to $\L_1$.

\subsection{Seiberg-Witten equation, transversality results, and ASD operator}

%\section{Positive scalar curvature}
\subsection{Positive scalar curvature}\label{subsection:PSCmetric}

A positive scalar curvature will give two desired properties: First, by the Weitzenb{\"o}ck formula, a non-negative scalar curvature on $3$- or $4$-manifolds leads solely to reducible solutions of the Seiberg-Witten equation (see \cite{kronheimer_mrowka_2007} (4.22)). Second, by the Weitzenb{\"o}ck formula and integration by parts, we have (see page 105 of \cite{Nicolaescu2000NotesOS})
\[
\int_M|\D_A\psi|^2dv_g = \int_M (|\nabla^A\psi|^2  + \frac{s}{4}|\psi|^2 +\frac{1}{2}\langle {\mathbf{c}}(F_A^+)\psi,\psi\rangle ) dv_g 
\]
where $A$ is a connection, $\D_A$ is the twisted Dirac operator, $s$ is the scalar curvature, and $\mathbf{c}$ is Clifford multiplication. So if $s$ is everywhere positive and $A$ is flat, the twisted Dirac operator would have trivial kernel.

It turns out that we can construct bullet metrics on $\S^1 \times D^3$ and $D^2\times \S^2$ such that the corresponding Levi-Civita connections have positive scalar curvature everywhere.

To construct the bullet metric on $D^2 \times S^2$, embed it in $\R^3\times \R^3$ such that the component $\S^2$ is standard sphere, and $D^2$ is the union of a standard semi-sphere $\S^2_+$ and a cylinder $\partial D^2 \times I$, which is the collar neighborhood of $\partial D^2$. One can perturb this  embedding to make it smooth, and the metric $g$ of $D^2\times S^2$ induced by the standard metric of $\R^3\times \R^3$ is so-called bullet metric.

One can compute the scalar curvature of this metric using the following formula:
\[
s= \sum_{i\neq j}\mathsf{sec}(e_i,e_j)
\]
where $\mathsf{sec}$ is the sectional curvature and $\{e_i\}$ is a set of orthonormal basis. The sectional curvature of $\S^2$ and $\S^2_+$ is positive. If two vectors lie in different copies of $\R^3$ in $\R^3\times \R^3$, the sectional curvature of the plane identified by these vectors is zero. This means that
\[
s({D^2\times S^2})= s({D^2}) + s({S^2}).
\]
Therefore, the scalar curvature is everywhere positive.

For $\S^1 \times D^3$, embed it in  $\R^2\times \R^4$ such that $\S^1$ is standard circle and $D^3$ is the union of a standard semi-sphere $\S^3_+$ and a cylinder $\partial D^3 \times I$. By the same reasoning and the fact that $\partial D^3=\S^3$ also has positive scalar curvature, the scalar curvature of $\S^1 \times D^3$ is everywhere positive.

\section{Apply ordinary gluing theory to $1$-surgery}
%\subsection{Gluing description}
In ordinary gluing theory, one obtain the union ${N}_r$ of two manifolds ${N}_1$ and ${N}_2$ by gluing along their boundaries $N$, and consider the relation between monopoles over ${N}_1$ and ${N}_2$ and monopoles over the union ${N}_r$.

Given a pair of monopoles on ${N}_1$ and ${N}_2$, respectively, if they are compatible over boundaries, one can glue them to obtain a point of configuration space over the union ${N}_r$. It turns out that there exists a genuine monopole of ${N}_r$ near this point. Moreover, the space of genuine monopoles over the union ${N}_r$ is actually isotopic to the manifold of configurations obtained by gluing in this way.
 
 The proof of the global gluing theorem is divided to four steps: The \textbf{linear gluing theorem} will give an approximation of the kernel of boundary difference map. The \textbf{local gluing theorem} will describe the set of genuine monopoles in a neighborhood of each glued configuration point. The \textbf{local surjectivity theorem} will prove that, the set of such neighborhoods is a cover of the manifold of genuine monopoles. The \textbf{global gluing theorem} will prove that, the moduli space of genuine monopoles is homeomorphic to the moduli space of glued configuration points, if the obstruction space is trivial.
 
 In this section, we will follow the strategy in Nicolaescu's book \cite{Nicolaescu2000NotesOS}.  In our case, i.e, $N = \S^1 \times \S^2$, $N_2 = \S^1 \times D^3$ or $D^2\times \S^2$, one can just apply the linear gluing theorem and the local surjectivity theorem in charpter 4 of \cite{Nicolaescu2000NotesOS}, and prove the condition of the local gluing theorem is satisfied. However, the global gluing theorem in this situation is slightly different from what Nicolaescu presented.
% TODO

\subsection{Abstract linear gluing results}
In this subsection, we review the abstract linear gluing results in section 4.1 of \cite{Nicolaescu2000NotesOS}.

It's natural to expect that, a longer neck of ${N}_r$ will narrow the difference between genuine monopoles and configurations obtained by gluing, since there should be no difference when the length of the neck $r=\infty$. So we first consider manifolds with necks of infinite length, say, $\hat{N}_1= {N}_1\cup_N N\times [0,\infty)$ and  $\hat{N}_2= {N}_2\cup_N N\times [0,\infty)$. Such manifolds are called cylindrical manifolds.

Suppose $\beta(t)$ is a smooth cutoff function such that $\beta(t)=0$ on $(-\infty,1/2]$ and $\beta(t) = 1$ on $[1,\infty)$. Set $\alpha(t)=1-\beta(t)$. These functions will be used to glue a pair of sections.

Denote by $\hat{E}$ a cylindrical bundle over a cylindrical manifold $\hat{N}$, that is, a vector bundle $\hat{E} \to \hat{N}$ together with a vector bundle ${E} \to {N}$ and a bundle isomorphism
\[
\hat{E}|_{N\times [0,\infty)} \to \pi^*E,
\]
where $ \pi: N\times [0,\infty) \to N$ is the projection map. Let $L^p(\hat{E})$ be the space of $L^p$-sections of $\hat{E}$. 
Let $L^p_{loc}(\hat{E})$ be the space of measurable sections $u$ such that $u\varphi \in L^p(\hat{E})$ for any smooth, compactly supported function $\varphi$ on $\hat{N}$. 
Denote by $\hat{u}$ an $L^2_{loc}$-section of $\hat{E}$. If there exists an $L^2_{loc}$-cylindrical section $\hat{u}_0$ such that
\[
\hat{u}-\hat{u}_0 \in L^2(\hat{E}),
\]
then $\hat{u}$ is called \textbf{asymptotically cylindrical} (or \textbf{a-cylindrical}). Define the asymptotic value of $\hat{u}$ to be
\[
\partial_\infty \hat{u} := \partial_\infty \hat{u}_0.
\]
Let $ L^2_\mu(\hat{E}) = \{ u\in  L^2(\hat{E}); \| u|_{\hat{N} \setminus N\times [0,\infty)}\|_{L^2} + \| u|_{N\times [0,\infty)} \cdot \e^{\mu t} \|_{L^2} < \infty\}$. The supremum of all $\mu \geq 0$ such that 
\[
\hat{u} -\hat{u}_0 \in  L^2_\mu(\hat{E})
\] 
is called the \textbf{decay rate} of the a-cylindrical section $\hat{u}$.

The norm on the space of a-cylindrical sections is defined by
\[
\|\hat{u}\|_{ex} = \|\hat{u}-\hat{u}_0\|_{L^2} + \|\partial_\infty \hat{u}\|_{L^2}
\]
The resulting Hilbert space is called $L_{ex}^2$.

%\begin{remark}
%In the following text, the subscript \text{\mu} will indicate the space with norm $\

Given a pair of compatible cylindrical sections $\hat{u}_i$ of $\hat{E}_i$, i.e they share the same constant value over the neck, they can be glued to form a section $ \hat{u}_1  \#_r  \hat{u}_2$ of $\hat{E}_1  \#_r  \hat{E}_2$. If $\hat{u}_i$ are just compatible $L_{ex}^2$-sections, i.e they are a-cylindrical sections with identical asymptotic values $\partial_\infty \hat{u}_1 = \partial_\infty \hat{u}_2$, they should be modified by cutoff functions first. Let $\hat{u}_i(r)$ be the same section as $\hat{u}_i$ outside the neck, and on the neck
\begin{equation}\label{equ:cut}
\hat{u}_i(r)(t) = \alpha(t-r)\hat{u}_i + \beta(t-r)\partial_\infty \hat{u}_i.
\end{equation}
When $t<r$, $\hat{u}_i(r) = \hat{u}_i$, and when $t>r+1$, $\hat{u}_i(r)$ is just the asymptotic value of $ \hat{u}_i$. Thus $ \hat{u}_i(r)$ is an approximation of $ \hat{u}_i$ as $r\to \infty$. Now these genuine cylindrical sections can be glued along the neck, so we define
\begin{equation}\label{equ:paste}
\hat{u}_1  \#_r  \hat{u}_2 := \hat{u}_1(r)  \#_r  \hat{u}_2(r)
\end{equation}

In the following description, all verifications of smoothness, Fredholmness and exactness are obmitted. See Section 4.3 of Nicolaescu's book for details.

%def 1.2.13 of Nic
Let $L^{m,p}$ be the space of sections with finite Sobolev norm $\|\cdot\|_{m,p}$.
%page 101 of Nic
Let $\hat{\sigma}$ be a $\text{spin}^c$ structure of $\hat{N}$ such that it induces a $\text{spin}^c$ structure $\sigma$ of $N$. Denote by $\mathcal{C}_\sigma$ the space of configurations in $L^{2,2}$ over the $3$-manifold $N$, by
\[
\mathcal{Z}_\sigma \subset \mathcal{C}_\sigma
\]
the set of monopoles (solutions of Seiberg-Witten equations) on $N$, and by
\[
\mathfrak{M}_\sigma = \mathcal{Z}_\sigma/ \mathcal{G}_\sigma
\]
the moduli space of monopoles on $N$.

Define
\begin{equation}\label{equ:defCsw}
\hat{\mathcal{C}}_{\mu ,sw} := \partial_\infty^{-1}(\mathcal{Z}_\sigma)
\end{equation}
and 
\[
\hat{\mathcal{Y}}_\mu := L_\mu^{1,2}(\hat{S}_{\hat{\sigma}}^- \oplus \mathbf{i}\Lambda_+^2T^*\hat{N}).
\]
The Seiberg-Witten equations give the Seiberg-Witten map
\begin{align*}
\widehat{SW}: \hat{\mathcal{C}}_{\mu ,sw} &\to \hat{\mathcal{Y}}_\mu,\\
(\hat{\psi},\hat{A}) &\mapsto \D_{\hat{A}}\hat{\psi}\oplus (\sqrt{2}(F^+_{\hat{A}} - \frac{1}{2}\hat{\mathbf{c}}^{-1}(q(\hat{\psi}))),
\end{align*}
where $\D_{\hat{A}}$ is the Dirac operator twisted by the connection $\hat{A}$, and $ \hat{\mathbf{c}}$ is the Clifford multiplication on $\hat{N}$.

We will use the following notation:
\[
\widehat{\mathcal{G}}_{\mu,ex} := \{ \hat{u} \in L^{ 3,2}_{\mu,ex}(\hat{N},\C); | \hat{u}(p)|=1\text{  } \forall p\in \hat{N}\}
\]
\[
\widehat{\M}_\mu := \widehat{SW}^{-1}(0) / \widehat{\mathcal{G}}_{\mu,ex}.
\]

$\hat{\mathsf{C}}_0=(\hat{\psi}_0 , \hat{A}_0)$: A fixed smooth finite energy monopole on $\hat{N}$. $\hat{\mathsf{C}}_0$ modulo a gauge transformation is in $\hat{\mathcal{C}}_{\mu ,sw}$  (see section 4.2.4 of Nicolaescu's book \cite{Nicolaescu2000NotesOS}). So in this paper we always assume that $\hat{\mathsf{C}}_0 \in \hat{\mathcal{C}}_{\mu ,sw}$.

${\mathsf{C}}_\infty$: A fixed smooth finite energy monopole on ${N}$.

$\widehat{\underline{SW}}_{\hat{\mathsf{C}}_0}$: The linearization of $\widehat{SW}$ at $\hat{\mathsf{C}}_0$.

As a Lie group, the component of $\mathbf{1}$ of $\hat{\mathcal{G}}_{\mu,ex}$ consists of elements that can be written as $\e^{\mathbf{i}f}$ where $f\in  L_{\mu,ex}^{3,2}(\hat{N},\mathbf{i}\R)$. Recall that we have fixed $\hat{\mathsf{C}}_0$, so the gauge action gives a map%The Seiberg-Witten map is $\hat{\mathcal{G}}_{\mu,ex}$-equivariant (See Proposition 2.1.9 of Nicolaescu's book \cite{Nicolaescu2000NotesOS}). Thus the gauge action gives a map 
\begin{align*}
\hat{\mathcal{G}}_{\mu,ex} &\to \hat{\mathcal{C}}_{\mu,sw} \\
\hat{u} &\mapsto \hat{u}  \cdot \hat{\mathsf{C}}_0.
\end{align*}
Denote the stabilizer of $\hat{\mathsf{C}}_0$ under the gauge action by $\hat{G}_0$. The differential of the above map is 
\begin{align*}
\mathfrak{L}_{\hat{\mathsf{C}}_0} :T_{\mathbf{1}}\hat{\mathcal{G}}_{\mu,ex} &\to T_{\hat{\mathsf{C}}_0}\hat{\mathcal{C}}_{\mu,sw}\\
\mathbf{i}f &\mapsto (\mathbf{i}f\hat{\psi}_0,-2\mathbf{i} df)
\end{align*}

We have three differential complexes:
\[\label{equ:complexF}
0\to T_1 \hat{\mathcal{G}}_{\mu} 
\xrightarrow{\mathfrak{L}_{\hat{\mathsf{C}}_0}} T_{\hat{\mathsf{C}}_0} \partial_\infty^{-1}(\mathsf{C}_\infty) 
\xrightarrow{\widehat{\underline{SW}}_{\hat{\mathsf{C}}_0}} T_0\mathcal{Y}_\mu
\to 0
\tag{$F_{\hat{\mathsf{C}}_0}$}
\]
\[\label{equ:complexK}
0
\to T_1 \hat{\mathcal{G}}_{\mu,ex} 
\xrightarrow{\frac{1}{2}\mathfrak{L}_{\hat{\mathsf{C}}_0}} T_{\hat{\mathsf{C}}_0} \hat{\mathcal{C}}_{\mu,sw} 
\xrightarrow{\widehat{\underline{SW}}_{\hat{\mathsf{C}}_0}} T_0\mathcal{Y}_\mu
\to 0
\tag{$\widehat{\mathcal{K}}_{\hat{\mathsf{C}}_0}$}
\]
\[\label{equ:complexB}
0\to T_1\mathcal{G}_\sigma 
\xrightarrow{\frac{1}{2}\mathfrak{L}_{{\mathsf{C}}_\infty}} T_{\mathsf{C}_\infty}\mathcal{Z}_\sigma
\to 0
\to 0
\tag{$B_{\hat{\mathsf{C}}_0}$}
\]
In the category of differential complexes, it's easy to verify that
\begin{equation}\label{equ:shortExact}
0\to \text{\ref{equ:complexF}}%F_{\hat{\mathsf{C}}_0}
\stackrel{i}{\hookrightarrow} \text{\ref{equ:complexK}} %\widehat{\mathcal{K}}_{\hat{\mathsf{C}}_0}
\stackrel{\partial_\infty}{\twoheadrightarrow} \text{\ref{equ:complexB}} %B_{\hat{\mathsf{C}}_0}
\to 0
\tag{\textbf E}
\end{equation}
is an exact sequence. Namely, each column of the diagram 
\begin{equation}\label{equ:diagram} 
\xymatrix{
&
0\ar[d] &
0\ar[d] & 
0 \ar[d]& 
\\
0\ar[r] &
T_1 \hat{\mathcal{G}}_{\mu}  \ar[d] \ar[r]^(0.4){\frac{1}{2}\mathfrak{L}_{\hat{\mathsf{C}}_0}} &
T_{\hat{\mathsf{C}}_0} \partial_\infty^{-1}(\mathsf{C}_\infty) \ar[d] \ar[r]^(0.6){\widehat{\underline{SW}}_{\hat{\mathsf{C}}_0}} &
T_0\mathcal{Y}_\mu \ar[d]^{=}\ar[r] & 
0\\
0\ar[r] &
T_1 \hat{\mathcal{G}}_{\mu,ex}  \ar[d] \ar[r]^(0.5){\frac{1}{2}\mathfrak{L}_{\hat{\mathsf{C}}_0}} &
T_{\hat{\mathsf{C}}_0} \hat{\mathcal{C}}_{\mu,sw}  \ar[d]^{\partial_\infty^0} \ar[r]^(0.5){\widehat{\underline{SW}}_{\hat{\mathsf{C}}_0}} &
T_0\mathcal{Y}_\mu \ar[d]\ar[r] &
 0 \\
0\ar[r] &
T_1\mathcal{G}_\sigma  \ar[d] \ar[r]^(0.4){\frac{1}{2}\mathfrak{L}_{{\mathsf{C}}_\infty}} &
T_{\mathsf{C}_\infty}\mathcal{Z}_\sigma \ar[d] \ar[r]&
0 \ar[d]\ar[r] & 
0 \\
&
0 &
0& 
0& 
\\
}
\tag{\textbf{D}}
\end{equation}
is exact.
Set 
\[
H^i_{\hat{\mathsf{C}}_0} := H^i( \widehat{\mathcal{K}}_{\hat{\mathsf{C}}_0}).
\]
For $i=0$, observe that 
\[
H^0_{\hat{\mathsf{C}}_0}  \cong T_1\hat{G}_0
\]
is the tangent space of the stabilizer of $\hat{\mathsf{C}}_0$ under gauge action. It is one dimensional if $\hat{\mathsf{C}}_0$ is reducible and trivial otherwise. For $i=1$, observe that $\dim_\R(H^1_{\hat{\mathsf{C}}_0})$ is the dimension of the formal tangent space of $\widehat{\M}_\mu$ at $[\hat{\mathsf{C}}_0]$. For $i=2$, $H^2_{\hat{\mathsf{C}}_0}$ is called the \textbf{obstruction space} at $\hat{\mathsf{C}}_0$.

From the diagram \ref{equ:diagram} %short exact sequence \ref{equ:shortExact} 
we obtain a long exact sequece
\begin{equation}\label{equ:longExact}
\begin{tikzcd}
& H^0(F_{\hat{\mathsf{C}}_0}) \arrow[d]
& H^1(F_{\hat{\mathsf{C}}_0}) \arrow[d]
& H^2(F_{\hat{\mathsf{C}}_0}) \arrow[d] & \\
\arrow[r, phantom, ""{coordinate, name=Y}] & H^0_{\hat{\mathsf{C}}_0}  \arrow[d]\arrow[r, phantom, ""{coordinate, name=Z}]
&  H^1_{\hat{\mathsf{C}}_0}  \arrow[d]\arrow[r, phantom, ""{coordinate, name=T}]
& H^2_{\hat{\mathsf{C}}_0} \arrow[d] &\\
0   \arrow[ruu,
"",
rounded corners,
to path={
-| (Y) [near end]\tikztonodes
|-  (\tikztotarget)}]  
& H^0(B_{\hat{\mathsf{C}}_0})  \arrow[ruu,
"",
rounded corners,
to path={
-| (Z) [near end]\tikztonodes
|-  (\tikztotarget)}]  
& H^1(B_{\hat{\mathsf{C}}_0})\arrow[ruu,
"",
rounded corners,
to path={
-| (T) [near end]\tikztonodes
|-  (\tikztotarget)}]  
& 0 & 
\end{tikzcd}
\tag{\textbf L}
\end{equation}

%\begin{equation}\label{equ:longExact}
%\cdots
%\to H^2(F_{\hat{\mathsf{C}}_0}) 
%\to H^2_{\hat{\mathsf{C}}_0}
%\to H^2(B_{\hat{\mathsf{C}}_0}) = 0
%\to 0
%\tag{\textbf L}
%\end{equation}
$\hat{\mathsf{C}}_0$ is called \textbf{regular} if $H^2_{\hat{\mathsf{C}}_0}= 0$, and \textbf{strongly regular} if $H^2(F_{\hat{\mathsf{C}}_0}) = 0$. Note that by the long exact sequance, strong regularity implies regularity.

The integer
\[
d(\hat{\mathsf{C}}_0) := -\chi(\widehat{\mathcal{K}}_{\hat{\mathsf{C}}_0}) =- \dim_\R H^0_{\hat{\mathsf{C}}_0}+  \dim_\R H^1_{\hat{\mathsf{C}}_0} - \dim_\R H^2_{\hat{\mathsf{C}}_0}
\]
is called the \textbf{virtual dimension} at $[\hat{\mathsf{C}}_0]$ of the moduli space $\widehat{\M}_\mu$. If $\hat{\mathsf{C}}_0$ is regular irreducible, $\widehat{\M}_\mu$ is smooth at $\hat{\mathsf{C}}_0$, and
\[
d(\hat{\mathsf{C}}_0)=-0+ \dim_\R H^1_{\hat{\mathsf{C}}_0}-0
\]
is indeed the dimension of the tangent space of $\widehat{\M}_\mu$ at $[\hat{\mathsf{C}}_0]$. On the other hand, if $\hat{\mathsf{C}}_0$ is regular reducible, we have 
\[
d(\hat{\mathsf{C}}_0)=-1+ \dim_\R H^1_{\hat{\mathsf{C}}_0}-0
\]
So $\dim_\R H^1_{\hat{\mathsf{C}}_0} = d(\hat{\mathsf{C}}_0)+1$. The difference between irreducibles and reducibles, comes from the fact that the orbit of irreducible $\hat{\mathsf{C}}_0$ is $1$-dimensional in $\hat{\mathcal{C}}_{\mu,sw}$, given by the action of constant gauge, while the constant gauge acs on reducibles trivially.

The $L^2_\mu$-adjoint of $\mathfrak{L}_{\hat{\mathsf{C}}_0}$ is
\begin{equation}\label{equ:L*}
\mathfrak{L}^{*_\mu}_{\hat{\mathsf{C}}_0} : (\dot\psi ,\mathbf{i}\dot a) \mapsto -2
\mathbf{i} d^{*_\mu} \dot a - \mathbf{i} \imaginary \langle\psi,\dot \psi\rangle_\mu.
\end{equation}
Now define 
\[
\hat{\mathcal{T}}_{\hat{\mathsf{C}}_0,\mu} := \widehat{\underline{SW}}_{\hat{\mathsf{C}}_0} \oplus \frac{1}{2}\mathfrak{L}^{*_\mu}_{\hat{\mathsf{C}}_0} :  L_\mu^{2,2}(\hat{\S}_{\hat{\sigma}}^+ \oplus \mathbf{i}T^*\hat{N}) \to \hat{\mathcal{Y}}_\mu \oplus L_\mu^{1,2}(N,\mathbf{i}\R).
\]
We can deduce that (see the proof of Lemma 4.3.19 of Nicolaescu's book)
\begin{equation}\label{equ:Tinf}
\vec{\partial}_\infty\hat{\mathcal{T}}_{\hat{\mathsf{C}}_0,\mu} =
\mathcal{T}_{{\mathsf{C}}_\infty,\mu } =
\begin{bmatrix}
\underline{SW}_{{\mathsf{C}}_\infty}  & -\frac{1}{2}\mathfrak{L}_{{\mathsf{C}}_\infty} \\[10pt]
\frac{1}{2}\mathfrak{L}^{*}_{{\mathsf{C}}_\infty} & -2\mu
\end{bmatrix}
\end{equation}
It turns out that we can remove the dependence on the choice of $\mu$, such that everything is independant of $\mu$ (Page 387 of \cite{Nicolaescu2000NotesOS}). Set $\mu = 0$ formally:
\begin{equation}\label{defOfTWithoutMu}
\mathcal{T}_{\hat{\mathsf{C}}_0} := \widehat{\underline{SW}}_{\hat{\mathsf{C}}_0} \oplus \frac{1}{2}\mathfrak{L}^*_{\hat{\mathsf{C}}_0}
\end{equation}
From the description \ref{equ:Tinf} above of $\mathcal{T}_{{\mathsf{C}}_\infty,\mu }$ ($\mu=0$), we have a decomposition
\[
\ker\mathcal{T}_{{\mathsf{C}}_\infty} = T_{{\mathsf{C}}_\infty}\M_\sigma \oplus T_1G_\infty,
\]
where $G_\infty$ is the stabilizer of $\hat{\mathsf{C}}_\infty$ under gauge action.
Denote the two components of the boundary map
\[
\partial_\infty: \ker_{ex}\hat{\mathcal{T}}_{\hat{\mathsf{C}}_0} \to \ker\mathcal{T}_{{\mathsf{C}}_\infty} = T_{{\mathsf{C}}_\infty}\M_\sigma \oplus T_1G_\infty
\]
by
\begin{align*}
\partial_\infty^0:  \ker_{ex}\hat{\mathcal{T}}_{\hat{\mathsf{C}}_0} &\to T_1G_\infty \\
\partial_\infty^c:  \ker_{ex}\hat{\mathcal{T}}_{\hat{\mathsf{C}}_0} &\to  T_{{\mathsf{C}}_\infty}\M_\sigma.
\end{align*}
Explictly, for $(\hat{\psi},\hat{\alpha}) \in L^{2,2}_{ex}(\hat{\S}_{\hat{\sigma}}^+ \oplus \mathbf{i}T^*\hat{N})$, if $\hat \alpha =\mathbf i \alpha + \mathbf i fdt$ on the neck $\R \times {N}$, where $\alpha(t)$ is a $1$-form on $N$ for each $t$, then
\begin{align}
\partial_\infty^0(\hat{\psi},\hat{\alpha}) &= \mathbf i\partial_\infty f \in T_1G_\infty \label{equ:partial0}\\
\partial_\infty^c(\hat{\psi},\hat{\alpha}) &=  (\partial_\infty \hat{\psi},{\partial_\infty \alpha}) \in T_{{\mathsf{C}}_\infty}\M_\sigma.
\end{align}

\subsection{Local gluing theorem}
%\section{$3\times 3$-asymptotic exact sequences}
Now we discuss how to apply the results in section 4.5 of Nicolaescu's book \cite{Nicolaescu2000NotesOS} to our cases.

Let's define
\[
\mathfrak{X}^k_+:= L^{k,2}(\hat{\S}^+_{\hat \sigma} \oplus \mathbf{i} T^*\hat{N}(r)), 
\mathfrak{X}^k_-:= L^{k,2}(\hat{\S}^-_{\hat \sigma} \oplus \mathbf{i} \Lambda^2_+T^*\hat{N}(r)), 
\]
\[
\mathfrak{X}^k := \mathfrak{X}^k_+ \oplus \mathfrak{X}^k_-.
\]
Define
\[
\hat{L}_r:= 
\begin{bmatrix}
0 &\hat{\mathcal T}^*_r \\
\hat{\mathcal T}_r & 0
\end{bmatrix}
:\mathfrak{X}^0 \to  \mathfrak{X}^0.
\]
We want to use the eigenspace corresponds to very small eigenvalues to approximate the kernel of this operator. Let $\H_r$ be the subspace of $\mathfrak{X}^0$ spanned by
\[
\{ v; \hat{L}_r v = \lambda v, |\lambda| < r^{-2}\}.
\]
Let $\mathcal{Y}_r$ be the orthogonal complement of $\H_r$ in $\mathfrak{X}^0$. Let $\H_r^\pm$ be the orthogonal projection of $\H_r$ to $\mathfrak{X}^0_\pm$.  Let $\mathcal{Y}_r^\pm$ be the orthogonal projection of $\mathcal{Y}_r$ to $\mathfrak{X}^0_\pm$.

Each row and column of the following diagrams is asymptotically exact (see page 434 of Nicolaescu's book \cite{Nicolaescu2000NotesOS}).

Virtual tangent space diagram:
\begin{equation}\label{equ:3T} 
\xymatrix{
&
0\ar[d] &
0\ar[d] & 
0 \ar[d]& 
\\
0\ar[r] &
\ker\Delta_+^c \ar[d] \ar[r]^{S_r} &
H^1_{\hat{\mathsf{C}}_1}\oplus H^1_{\hat{\mathsf{C}}_2}\ar[d] \ar[r]^{\Delta_+^c} &
L_1^+ + L_2^+ \ar[d]\ar[r] & 
0\\
0\ar[r] &
{\H}_r^+ \ar[d] \ar[r]^(0.3){S_r} &
\ker_{ex}\hat{\mathcal{T}}_{\hat{\mathsf{C}}_1}\oplus \ker_{ex}\hat{\mathcal{T}}_{\hat{\mathsf{C}}_2} \ar[d]^{\partial_\infty^0} \ar[r]^(0.64){\Delta_+^c} &
\hat{L}_1^+ + \hat{L}_2^+ \ar[d]\ar[r] &
 0 \\
0\ar[r] &
\ker\Delta_+^0 \ar[d] \ar[r]^{S_r} &
\mathfrak{C}_1^+\oplus \mathfrak{C}_2^+ \ar[d] \ar[r]^{\Delta_+^0} &
\mathfrak{C}_1^+ + \mathfrak{C}_2^+ \ar[d]\ar[r] & 
0 \\
&
0 &
0& 
0& 
\\
}
\tag{\textbf{T}}
\end{equation}

Obstruction space diagram:
\begin{equation}\label{equ:3O} 
\xymatrix{
&
0\ar[d] &
0\ar[d] & 
0 \ar[d]& 
\\
0\ar[r] &
\ker\Delta_-^c \ar[d] \ar[r]^(0.35){S_r} &
H^2(F_{\hat{\mathsf{C}}_1}) \oplus H^2(F_{\hat{\mathsf{C}}_2}) \ar[d] \ar[r]^(0.62){\Delta_-^c} &
L_1^- + L_2^- \ar[d]\ar[r] & 
0\\
0\ar[r] &
{\H}_r^- \ar[d] \ar[r]^(0.3){S_r} &
\ker_{ex}\hat{\mathcal{T}}^*_{\hat{\mathsf{C}}_1}\oplus \ker_{ex}\hat{\mathcal{T}}^*_{\hat{\mathsf{C}}_2} \ar[d] \ar[r]^(0.64){\Delta_+^c} &
\hat{L}_1^- + \hat{L}_2^- \ar[d]\ar[r] &
 0 \\
0\ar[r] &
\ker\Delta_-^0 \ar[d] \ar[r]^{S_r} &
\mathfrak{C}_1^-\oplus \mathfrak{C}_2^- \ar[d] \ar[r]^{\Delta_-^0} &
\mathfrak{C}_1^- + \mathfrak{C}_2^- \ar[d]\ar[r] & 
0 \\
&
0 &
0& 
0& 
\\
}
\tag{\textbf{O}}
\end{equation}
where 
\[
L_i^+ := \partial^c_\infty \ker_{ex}\hat{\mathcal{T}}_{\hat{C}_i}\subset T_{C_\infty}\M_\sigma
\]
\[
 \mathfrak{C}_i^+ := \partial^0_\infty \ker_{ex}\hat{\mathcal{T}}_{\hat{C}_i}\subset T_1G_\infty
\]
\[
L_i^- := \partial^c_\infty \ker_{ex}\hat{\mathcal{T}}^*_{\hat{C}_i}\subset T_{C_\infty}\M_\sigma
\]
\[
 \mathfrak{C}_i^- := \partial^0_\infty \ker_{ex}\hat{\mathcal{T}}^*_{\hat{C}_i}\subset T_1G_\infty
\]

Here is a short explanation of the middle column of the diagram \ref{equ:3T}: We can first look at the beginning of the long exact sequence \ref{equ:longExact}:
\[
\cdots
\to H^0_{\hat{\mathsf{C}}_i} = T_1G_i
\stackrel{\partial_\infty}{\to} H^0(B_{\hat{\mathsf{C}}_0}) = T_1G_\infty
\stackrel{\delta}{\to} H^1(F_{\hat{\mathsf{C}}_0}) = \ker_{\mu}\hat{\mathcal{T}}_{\hat{\mathsf{C}}_i}
\stackrel{\phi}{\to} H^1_{\hat{\mathsf{C}}_i}
\to H^1(B_{\hat{\mathsf{C}}_0})
\to \cdots
\]
Consider $\ker_{ex}\hat{\mathcal{T}}_{\hat{\mathsf{C}}_i} \supset \ker_{\mu}\hat{\mathcal{T}}_{\hat{\mathsf{C}}_i}$. 
Intuitively, $\ker_{\mu}\hat{\mathcal{T}}_{\hat{\mathsf{C}}_i}$ is the tangent space of ``monopoles in $L_{\mu}$ modulo the action of the gauge group in $L_\mu$'', $\ker_{ex}\hat{\mathcal{T}}_{\hat{\mathsf{C}}_i} $ is the tangent space of ``monopoles in $L_{ex}$ modulo the action of the gauge group in $L_\mu$'', and $H^1_{\hat{\mathsf{C}}_i}$ is the tangent space of ``monopoles in $L_{ex}$ modulo the action of the gauge group in $L_{ex}$''. 
Thus the map from $\ker_{ex}\hat{\mathcal{T}}_{\hat{\mathsf{C}}_i}$ to $H^1_{\hat{\mathsf{C}}_i}$ is surjective with the same kernel as $\ker \phi = T_1(G_\infty/\partial_\infty \hat{G}_i)$ (see Lemma 4.3.25 of Nicolaescu's book \cite{Nicolaescu2000NotesOS} for details), and this kernel is $ \mathfrak{C}_i^+$ (see the proof of Propsition \ref{prop:trivialObs}).

\begin{remark}\label{rem:imL(kerLinfty)}
$\delta$ is nontrivial if and only if $\hat{\mathsf{C}}_i$ is irreducible and ${\mathsf{C}}_\infty$ is reducible. We assume this is the case. Then $\ker \phi =T_1(G_\infty/\partial_\infty \hat{G}_0) = \R$ is generated by constant function $\mathbf{i}f \in  T_1G_\infty$.

Now consider the definition of the connecting homomorphism $\delta$. We can choose the preimage of $\mathbf{i}f$ in $T_1 \hat{\mathcal{G}}_{\mu,ex} $ to be the constant function $\mathbf{i}\hat{f}$, or we can choose the preimage to be $ \mathbf{i} \beta(t-r)\hat f$. 
In first case, it's sent to $(\mathbf{i}\hat{f}\hat{\psi},0) \in T_{\hat{\mathsf{C}}_0} \partial_\infty^{-1}(\mathsf{C}_\infty) $, while in the second case, it's sent to $(\mathbf{i}\beta(t-r)\hat{f}\hat{\psi},2\mathbf{i} g dt)$, where $gdt =d (\beta(t-r)\hat{f})$ is a bump function aroud $t=r$. 
These two certainly represent the same class in $H^1(F)$, but only the first one is harmonic and hence in $\ker_{\mu}\hat{\mathcal{T}}_{\hat{\mathsf{C}}_i}$ (By (4.2.2) and Example 4.1.24 of Nicolaescu's book \cite{Nicolaescu2000NotesOS}, elements in $\ker_{\mu}\hat{\mathcal{T}}_{\hat{\mathsf{C}}_i}$ must be harmonic without any $dt$-terms). 
However, the second one, $(\mathbf{i}\beta(t-r)\hat{f}\hat{\psi},2\mathbf{i} g dt)$, shows explicitly that the map $\partial^0_\infty$ in \ref{equ:partial0} is the inverse of $\delta$.
\end{remark}

Here is a short explanation of the middle column of the diagram
\ref{equ:3O}:
$H^2(F_{\hat{\mathsf{C}}_i}) = \ker_{\mu}\hat{\mathcal{T}}^*_{\hat{\mathsf{C}}_i}$ since every self dual $2$-form on $\hat{N}_i$ is in $L_\mu$. On the other hand, the kernel of $\mathfrak{L}_{\hat{\mathsf{C}}_i} $ is exactly $T_1G_i$ which is not in $L_\mu$ (they are constant functions). Hence 
\[
\ker_{ex}\hat{\mathcal{T}}^*_{\hat{C}_i}= \ker_{ex} (\widehat{\underline{SW}}^*_{\hat{\mathsf{C}}_i} \oplus \frac{1}{2}\mathfrak{L}_{\hat{\mathsf{C}}_i})\]
decomposes to the direct sum of $H^2(F_{\hat{\mathsf{C}}_i}) $ and $\mathfrak{C}_i^- = T_1G_i$.

The virtual tangent space and obstruction space will give all monopoles of $\hat{N}(r)$ in a small neighborhood of $\hat{\mathsf{C}}_r$ in its slice:
\begin{theorem}[\cite{Nicolaescu2000NotesOS} Theorem 4.5.7]\label{thm:lgt}
For large enough $r$, the set
\[
\{\hat{\mathsf{C}}; \hat{\mathsf{C}} \text{ are monopoles on } \hat{N}(r), \mathfrak{L}^{*}_{\hat{\mathsf{C}}_r} (\hat{\mathsf{C}} - \hat{\mathsf{C}}_r) = 0, \|\hat{\mathsf{C}} - \hat{\mathsf{C}}_r\|_{2,2} \leq r^{-3}\}
\]
is in one-to-one correspondence with the set 
\[
\{
\hat{\mathsf{C}}_r + \underline{\hat{\mathsf{C}}}_0\oplus \underline{\hat{\mathsf{C}}}^\bot; \|\underline{\hat{\mathsf{C}}}_0\|_{2,2}\leq r^{-3}, \kappa_r (\underline{\hat{\mathsf{C}}}_0) = 0, \underline{\hat{\mathsf{C}}}^\bot = \Phi(\underline{\hat{\mathsf{C}}}_0)
\}
\]
where
\begin{align*}
\hat{\mathsf{C}}_r  &= \hat{\mathsf{C}}_1 \#_r \hat{\mathsf{C}}_2\\
\underline{\hat{\mathsf{C}}}_0 &\in \H_r^+ \\
\underline{\hat{\mathsf{C}}}^\bot &\in \mathcal{Y}_r^-\\
\kappa_r : B_0(r^{-3})\subset \H_r^+ &\to \H_r^- \\
\Phi :B_0(r^{-3})\subset \H_r^+ &\to \mathcal{Y}_r^-
\end{align*}
\end{theorem}

We can also prove that, in the slice of $\hat{\mathsf{C}}_r$, any pair of configurations in small enough neighborhood of $\hat{\mathsf{C}}_r$, are gauge inequivalent (see Lemma 4.5.9 of Nicolaescu's book \cite{Nicolaescu2000NotesOS}). Thus we have
\begin{theorem}[\cite{Nicolaescu2000NotesOS} Corollary 4.5.10]\label{thm:lgc}
For large enough $r$,
\[
\{
\hat{\mathsf{C}}_r + \underline{\hat{\mathsf{C}}}_0\oplus \underline{\hat{\mathsf{C}}}^\bot; \|\underline{\hat{\mathsf{C}}}_0\|_{2,2}\leq r^{-3}, \kappa_r (\underline{\hat{\mathsf{C}}}_0) = 0, \underline{\hat{\mathsf{C}}}^\bot = \Phi(\underline{\hat{\mathsf{C}}}_0), \mathfrak{L}^{*}_{\hat{\mathsf{C}}_r} (\underline{\hat{\mathsf{C}}}_0\oplus \underline{\hat{\mathsf{C}}}^\bot) = 0
\}
\]
is an open set of moduli space $\M(\hat{N}_r, \hat{\sigma}_1\#\hat{\sigma}_2)$.
\end{theorem}

Moreover, this collection of open sets is an open cover of moduli space $\M(\hat{N}_r, \hat{\sigma}_1\#\hat{\sigma}_2)$:
\begin{theorem}[\cite{Nicolaescu2000NotesOS} Theorem 4.5.15]\label{thm:ls}
Let 
\[
\hat{\mathcal{Z}}_{\Delta} := \{ (\hat{\mathsf{C}}_1,\hat{\mathsf{C}}_2)\in \hat{\mathcal{Z}}_1\times\hat{\mathcal{Z}}_2; \partial_\infty \hat{\mathsf{C}}_1 =  \partial_\infty \hat{\mathsf{C}}_2\}
\]
be the space of compatible monopoles. Then
\[
\bigcup_{\mathsf{C}_r = \hat{\mathsf{C}}_1 \#_r \hat{\mathsf{C}}_2,(\hat{\mathsf{C}}_1,\hat{\mathsf{C}}_2)\in\hat{\mathcal{Z}}_{\Delta}} \{
\hat{\mathsf{C}}_r + \underline{\hat{\mathsf{C}}}_0\oplus \underline{\hat{\mathsf{C}}}^\bot; \|\underline{\hat{\mathsf{C}}}_0\|_{2,2}\leq r^{-3}, \kappa_r (\underline{\hat{\mathsf{C}}}_0) = 0, \underline{\hat{\mathsf{C}}}^\bot = \Phi(\underline{\hat{\mathsf{C}}}_0), \mathfrak{L}^{*}_{\hat{\mathsf{C}}_r} (\underline{\hat{\mathsf{C}}}_0\oplus \underline{\hat{\mathsf{C}}}^\bot) = 0
\}
\]
is $ \M(\hat{N}_r, \hat{\sigma}_1\#\hat{\sigma}_2)$.
\end{theorem}

\subsection{Computation of virtual tangent space and obstruction space}
Now we have stated all results we need. Next we compute the dimension of the moduli space $\dim H^1_{\hat{\mathsf{C}}_0}$ and the dimension of the obstruction space $\dim H^2(F_{\hat{\mathsf{C}}_0})$ for any monopole $\hat{\mathsf{C}}_0$ on $X_0$, $D^3\times \S^1$, and $\S^2\times D^2$.

\begin{prop}\label{prop:dimOfReducible}
%Let $\ss$ be any $\text{spin}^c$ structure of $X$ and $\ss'$ be its unique extension to $X'$ as in Theorem \ref{thm:changeOfSpinc}.  
Let matrics $g_{bullet}$ be the ones chosen in subsection \ref{subsection:PSCmetric}. Let $\ss(\S^1 \times D^3)$ be the unique $\text{spin}^c$ structure of $\S^1 \times D^3$, and $\ss(D^2 \times \S^2)$ be the unique $\text{spin}^c$ structure of $ D^2 \times \S^2 $ such that the first Chern class of the determinant line bundle is zero. 
Then the moduli space of SW equations without perturbation $\M(\S^1 \times D^3,g_{bullet} ,\ss(\S^1 \times D^3))$ is a circle and $\M(D^2 \times \S^2,g_{bullet} ,\ss(D^2 \times \S^2))$ is a point. 
\end{prop}
\begin{proof}
By the Weitzenb{\"o}ck formula, a non-negative scalar curvature on $3$- or $4$-manifolds leads solely to reducible solutions of the Seiberg-Witten equations (see \cite{kronheimer_mrowka_2007} (4.22)). Hence all monopoles are of the form $(A,0)$, and the Seiberg-Witten equations degenerate to one equation
\[
F_{A}^+ = 0.
\]
Since $F_{A}^+ = \frac{1}{2}(dA + *dA)$ and $\im d\cap \im d^* = \im d\cap \im *d = 0$, $F_{A}^+ = 0$ is equivalent to $dA = 0$.

Fix any $U(1)$-connection $A_0$ of the determinant line bundle of the chosen $\text{spin}^c$ structure. 
In Proposition \ref{prop:spincRestriction} we showed that the first Chern class of the determinant line bundle is zero. Hence $F_{A_0}$ is exact. Let $da_0 = -F_{A_0}$. Then $(A,0)$ is a monopole iff 
\[
A =  A_0 + a_0 + a
\]
for some closed imaginary $1$-form $a$. Hence the space of monopoles is the coset of the space of closed forms.

Now consider the action by the gauge group $\mathscr{G}=Map(M,\S^1)$. Elements in the identity component $I$ of $\mathscr{G}$ can be written as $\e^{\mathbf{i}f}$ where $f$ can be any smooth function ($0$-form), and it changes $A$ by the addition of $\mathbf{i}df$. Also $\mathscr{G}/I = H^1(M;\Z)$. Hence for $M =D^3\times \S^1$ or $\S^2\times D^2$, the moduli space of monopoles can be identified with the torus $H^1(M;\R)/H^1(M;\Z)$.
\end{proof}

By Proposition \ref{prop:spincRestriction} and Proposition \ref{prop:dimOfReducible}, we have 
\begin{corollary}\label{cor:dimOfReducible}
Let $\ss$ be any $\text{spin}^c$ structure of $X$ and $\ss'$ be its unique extension to $X'$ as in Theorem \ref{thm:changeOfSpinc}.  Let matrics $g_{bullet}$ be the ones chosen in subsection \ref{subsection:PSCmetric}. Then the moduli space of SW equations without perturbation $\M(D^3\times \S^1,g_{bullet} ,\ss|_{D^3\times \S^1})$ is a circle and $\M(\S^2\times D^2,g_{bullet} ,\ss'|_{\S^2\times D^2})$ is a point. All monopoles are reducible.
\end{corollary}

\begin{prop}\label{prop:virtualDimX0}
Let $g(X)$ be a metric of $X$ such that $g|_{\partial X_0}$ is the product of canonical metrics on $\S^1$ and $\S^2$. Let $\ss$ be any $\text{spin}^c$ structure of $X$ satisfying the dimension assumption (\ref{equ:dimAssumption}). Let $\hat{\ss}$ be the restriciton of $\ss$ on $X_0$. Then the virtual dimension 
\[
d(\hat{\mathsf{C}}_0)= 1
\]
for any monopole $\hat{\mathsf{C}}_0$ on $X_0$.
\end{prop}
\begin{proof}
Let $\hat{N}$ be a cylindrical manifold with boundary $N = \partial_\infty \hat{N}$. Let $\hat{g}$ be a metric on $\hat{N}$ and $\hat{A}_0$ be a connection on $\hat{N}$. Let $A_0 = \partial_\infty\hat{A}_0$ and $g =\partial_\infty \hat{g}$. Define
\[
\mathbf{F}(g,A_0) := 4\eta_{Dir}({A_0})+ \eta_{sign}(g),
\]
where $\eta_{Dir}({A_0})$ is the eta invariant of the Dirac operator $\mathfrak{D}_{A_0}$, and $\eta_{sign}(g)$ is the eta invariant of the metric $g = \partial_\infty \hat{g}$.

Let ${\mathsf{C}}_\infty = \partial_\infty \hat{\mathsf{C}}_0$. Recall that we always assume that $\hat{\mathsf{C}}_0 \in \hat{\mathcal{C}}_{\mu ,sw}$. Hence ${\mathsf{C}}_\infty$ is a monopole on $\hat{N}$. By Corollary \ref{cor:dimOfReducible}, ${\mathsf{C}}_\infty$ is reducible. %If $N$ is connected and ${\mathsf{C}}_\infty$ is reducible,
Then the formula of virtual dimension for the cylindrical manifold $\hat{N}$ is (see page 393 of Nicolaescu's book \cite{Nicolaescu2000NotesOS})
\[
d(\hat{\mathsf{C}}_0)= \frac{1}{4}\left(\int_{\hat{N}} c_1(\hat{A}_0)^2 - 2(\chi_{\hat{N}}+3\sigma_{\hat{N}})\right) + \beta({\mathsf{C}}_\infty),
\]
where
\[
\beta({\mathsf{C}}_\infty) := \frac{1}{2}(b_1(N) - 1) - \frac{1}{4}\mathbf{F}({\mathsf{C}}_\infty).
\]
The integral term is the same as the compact case, and the second term $ \beta({\mathsf{C}}_\infty)$ is called boundary correction term. In our case $N = \partial_\infty\hat{N} =\S^1\times \S^2$, and the metric $\hat{g} = g(X)|_{X_0}$ ensures that $g = \partial_\infty \hat{g}$ is the product of canonical metrics on $\S^1$ and $\S^2$. 
In this situation $\eta_{sign}(g) = 0$ (\cite{Komuro84}) and $\eta_{Dir}(\partial_\infty\hat{C}_0) = 0$ (\cite{Nicolaescu98} Appendix C). Hence $\mathbf{F}(\partial_\infty\hat{C}_0) = 0$. Moreover $b_1(\S^1\times \S^2) = 1$, so $\beta({\mathsf{C}}_\infty) = 0$.

Let $\L$ be the determinant line bundle of $\ss$ and $\hat{\L}$ be the determinant line bundle of $\hat{\ss}$. In the proof of Theorem \ref{thm:changeOfSpinc}, we see that 
\[
c_1(\hat{A}_0)^2 = \langle c_1(\hat{\L})^2, X_0 \rangle = \langle c_1({\L})^2, X \rangle = c_1({\L})^2.
\]
From the triangulation of the boundary sum one can compute that 
\begin{align*}
\chi(X) &=\chi(X_0)+ \chi(\S^1\times D^3) - \chi(\S^1\times \S^2)\\
&=  \chi(X_0)+ (1-1)- (1-1+1-1)\\
&=  \chi(X_0).
\end{align*}
To compute $\sigma(X_0)$ consider the following Mayer-Vietoris sequence
\begin{center}
\begin{tikzpicture}[commutative diagrams/every diagram]
\node (P0) at (0cm, 0cm) {$H^1(X_0)\oplus H^1(\S^1\times D^3)$};
\node (P1) at (-1cm, -0.1cm) {};
\node (P2) at (3.3cm, 0cm) {$H^1(\S^1 \times \S^2)$} ;
\node (P6) at (3.3cm, -1cm) {$\Z$} ;
\node (P3) at (5.5cm, 0cm) {$H^2(X)$};
%\node (P2) at (90+2*72:2cm) {\makebox[5ex][r]{$(X\otimes (Y\otimes Z))\otimes T$}};
%\node (P3) at (90+3*72:2cm) {\makebox[5ex][l]{$((X\otimes Y)\otimes Z)\otimes T$}};
\node (P4) at (-1cm, -1cm) {$\Z$};
\node (P5) at (8.5cm, 0cm) {$H^2(X_0)\oplus H^2(\S^1\times D^3)$};
\node (P7) at (7.2cm, -0.1cm) {};
\node (P8) at (7.2cm, -1cm) {?};
\node (P9) at (12cm, 0cm) {$H^2(\S^1 \times \S^2)$};
\node (P10) at (12cm, -1cm) {$\Z$};
\path[commutative diagrams/.cd, every arrow, every label]
(P0) edge node {} (P2)
%(P1) edge node[swap] {$\phi$} (P2)
(P2) edge node {$0$} (P3)
%(P4) edge node {$\phi$} (P3)
(P1) edge node {$\cong$} (P4)
(P3) edge node {$i^*$} (P5)
(P2) edge node {$\cong$} (P6)
(P4) edge node {$\cong$} (P6)
(P7) edge node {$\cong$} (P8)
(P5) edge node {} (P9)
(P8) edge node {$ i_\partial^*$} (P10)
(P9) edge node {$\cong$} (P10);
\end{tikzpicture}
\end{center}
From the assumption of the loop $\gamma$ we choose to do the surgery (the pairing of $\gamma$ and the generator of $H^1(X)=\Z$ is $1$), the dual of $\gamma$ is a $3$-manifold $M\subset X$ and $M\setminus (\S^1\times D^3) \subset X_0$ has the boundary $\{*\}\times \S^2 \subset \S^1\times \S^2 = \partial X_0$. Hence $ i_\partial^* = 0$ and therefore $i^*:H^2(X)\to H^2(X_0)$ is an isomorphism. 
For $2$-manifolds $\Sigma_1,\Sigma_2 \subset X$, we can assume $\gamma \cap \Sigma_i  = \emptyset$ for dimension reason. By choosing a small enough neighborhood of $\gamma$ we can further assume $\Sigma_i \subset X_0$. Hence the pairing of $\Sigma_1$ and $\Sigma_2$ is the same in $X$ and $X_0$. Therefore
\[
\sigma(X_0)= \sigma(X).
\]
Hence $d(\hat{\mathsf{C}}_0)= 1$.
%If the boundary of the cylindrical manifold is $\S^1\times \S^2$, the boundary correction term is zero (for an explanation, see example 4.3.35 of Nicolaescu's book \cite{Nicolaescu2000NotesOS}).
\end{proof}

It turns out that our cases are simple: the obstruction space is trivial.
\begin{prop}\label{prop:noncompactTransversality}
Let $\hat{N} = X_0$ and $N= \partial X_0 = \S^1\times \S^2$. Let $\ss$ be any $\text{spin}^c$ structure of $X$. Let $\hat{\ss}$ be the restriciton of $\ss$ on $X_0$. We can choose a generic perturbation $\eta$ on $X_0$ such that if $\hat{\mathsf{C}}_0$ is an $\eta$-monopole, it is irreducible and $H^2(F(\hat{\mathsf{C}}_0))=0$.
\end{prop}
\begin{proof}
To mimic the definition of the wall in the compact case, define
\begin{equation*}\label{equ:noncampactWall}
\mathcal{W}^{k-1}_{\mu} := \{\eta\in L_\mu^{k-1,2}(i\Lambda^+(X_0));  \exists A\in \mathscr{A}(s),  F^{+_g}_A + i\eta = 0\}.
\end{equation*}
By the computation of the \textbf{ASD} operator $d^+ \oplus d^*
$, one can show that $\mathcal{W}^{k-1}_{\mu}$ is an affine space of codimension $b^+$ (see \cite{Nicolaescu2000NotesOS} Page 404) just as in the compact case. For each $\eta$ outside $\mathcal{W}^{k-1}_{\mu}$, all $\eta$-monopoles are irreducible. Consider the configuration space
\[
\hat{\mathcal{C}}^*_{\mu,sw}/\hat{\mathcal{G}}_{\mu,ex}.
\]
Here $\hat{\mathcal{C}}_{\mu,sw}$ is the space of configurations on $X_0$ that restrict to monopoles on $\partial X_0 = \S^1\times \S^2$, as defined in (\ref{equ:defCsw}).
Let $s=\hat{\ss}|_{\partial X_0}$ and 
\[
\M_s = \M(\S^1\times \S^2,s, g_{round}).
\]
Exactly as in the proof of Proposition \ref{prop:dimOfReducible}, one can show that $\M_s = \S^1$. Let
\[
\mathcal{Z} := \mathcal{Z}^{k-1}_{\mu} := L_\mu^{k-1,2}(i\Lambda^+(X_0)) \setminus \mathcal{W}^{k-1}_{\mu}
\]
be the space of nice perturbations. Consider
\begin{align*}
\mathcal{F} : \hat{\mathcal{C}}^*_{\mu,sw}/\hat{\mathcal{G}}_{\mu,ex}\times \M_s\times \mathcal{Z} &\to \hat{\mathcal{Y}}_\mu \times \M_s\times \M_s\\
( \hat{\mathsf{C}},\mathsf{C},\eta) &\mapsto ( \widehat{{SW}}_{\eta}(\hat{\mathsf{C}}),\partial_\infty \hat{\mathsf{C}}, \mathsf{C}).
\end{align*}
Let $\Delta$ be the diagonal of $\M_s\times \M_s$. One can show that $\mathcal{F}$ is transversal to $0\times \Delta \subset  \hat{\mathcal{Y}}_\mu \times \M_s\times \M_s$ by the diffenrential
\begin{align*}
D_{( \hat{\mathsf{C}}_0, \mathsf{C}_\infty, \eta)}\mathcal{F}: 
 T_{\hat{\mathsf{C}}_0} \mathcal{B}^*_{\mu,sw}\oplus T_{\mathsf{C}_\infty}\M_s\oplus T_\eta \mathcal{Z} &\to T_0\hat{\mathcal{Y}}_{g(b),\mu}\oplus T_{\mathsf{C}_\infty}\M_s \oplus T_{\mathsf{C}_\infty}\M_s\\
(\underline{\hat{\mathsf{C}}}_0,\underline{{\mathsf{C}}}_\infty,\zeta) &\mapsto  (\widehat{\underline{SW}}_{\eta}(\underline{\hat{\mathsf{C}}}_0) +\zeta, \partial_\infty \underline{\hat{\mathsf{C}}}_0,\underline{{\mathsf{C}}}_\infty).
\end{align*} 
Then apply Sard-Smale to the projection
\[
\pi: \mathcal{F}^{-1}(0\times \Delta) \to \mathcal{Z}
\]
to show that $\mathcal{Z}^0_{reg}$, the set of regular values of $\pi$, is of the second category in the sense of Baire (a countable intersection of open dense sets). 

For each $\eta \in \mathcal{Z}^0_{reg}$, the map 
\begin{align*}
\mathcal{F}_\eta : \hat{\mathcal{C}}^*_{\mu,sw}/\hat{\mathcal{G}}_{\mu,ex}\times \M_s&\to \hat{\mathcal{Y}}_\mu \times \M_s\times \M_s\\
( \hat{\mathsf{C}},\mathsf{C}) &\mapsto ( \widehat{{SW}}_{\eta}(\hat{\mathsf{C}}),\partial_\infty \hat{\mathsf{C}}, \mathsf{C}).
\end{align*}
is transversal to $0\times \Delta \subset  \hat{\mathcal{Y}}_\mu \times \M_s\times \M_s$. Let $pr_1$ be the projection to the first summand:
\[
pr_1: \hat{\mathcal{Y}}_\mu \times \M_s\times \M_s \to  \hat{\mathcal{Y}}_\mu.
\]
Then $Dpr_1\circ D\mathcal{F}_\eta$ must be surjective since  $Dpr_1(0\times \Delta)$ is zero. Hence 
\begin{align*}
D_{( \hat{\mathsf{C}}_0, \mathsf{C}_\infty)}(pr_1 \circ \mathcal{F}_\eta): 
 T_{\hat{\mathsf{C}}_0} \mathcal{B}^*_{\mu,sw}\oplus T_{\mathsf{C}_\infty}\M_s &\to T_0\hat{\mathcal{Y}}_{g(b),\mu}\\
(\underline{\hat{\mathsf{C}}}_0,\underline{{\mathsf{C}}}_\infty) &\mapsto  (\widehat{\underline{SW}}_{\eta}\underline{\hat{\mathsf{C}}}_0) .
\end{align*} 
is surjective. This means that $ \widehat{\underline{SW}}_{\eta}$ is surjective, i.e. $H^2_{\hat{\mathsf{C}}_0}= 0$. By the last several terms of the long exact sequence \ref{equ:longExact}
\begin{equation}
\cdots
\to H^1_{\hat{\mathsf{C}}_0}
\stackrel{\partial_\infty}{\to} H^1(B_{\hat{\mathsf{C}}_0})
\to H^2(F_{\hat{\mathsf{C}}_0}) 
\to H^2_{\hat{\mathsf{C}}_0}=0
\to H^2(B_{\hat{\mathsf{C}}_0}) = 0
\to 0,
\tag{\textbf L}
\end{equation}
$H^2(F_{\hat{\mathsf{C}}_0}) =0$ if and only if $\partial_\infty$ is surjective. This is equivalent to say that $\partial_\infty: \widehat{\M}(X_0,\eta) \to \M_s$ is a submersion at $\hat{\mathsf{C}}_0$.

Recall that
\begin{align}
\mathcal{F}_0\begin{pmatrix}
        A\\
          \Phi\\
    \end{pmatrix}
    &=
    \begin{pmatrix}
         d^* A\\
          \D_A \Phi \\
    \end{pmatrix}, \\
\label{equ:0defF1}\mathcal{F}_{1,\eta}\begin{pmatrix}
        A\\
          \Phi\\
    \end{pmatrix}
    &=     F^{+}_A + i\eta - \rho^{-1}(\sigma(\Phi, \Phi)).
\end{align}
Fix a $\mathsf{C}_\infty \in \M_s$, then
\begin{align*}
\mathcal{F}_{\mathsf{C}_\infty }:  \partial_\infty^{-1}(\mathsf{C}_\infty)/\hat{\mathcal{G}}_{\mu}  \times \mathcal{Z} &\to \hat{\mathcal{Y}}_\mu\\
( \hat{\mathsf{C}},\eta) &\mapsto  \mathcal{F}_{1,\eta}(\hat{\mathsf{C}})
\end{align*}
is transversal to $0 \in  \hat{\mathcal{Y}}_\mu $. As above, we can find a set $\mathcal{Z}_{reg}^1$ of the second category in the sense of Baire, such that for each $\eta \in \mathcal{Z}_{reg}$, $\mathcal{F}_{\mathsf{C}_\infty,\eta } =\mathcal{F}_{1,\eta}$ is transversal to $0 \in  \hat{\mathcal{Y}}_\mu $. This means that
\begin{equation}\label{equ:surjective-for-solution}
 H^2(F_{\hat{\mathsf{C}}_0}) =0
\end{equation} 
for any $\hat{\mathsf{C}}_0 \in ( \partial_\infty^{-1}(\mathsf{C}_\infty)/\hat{\mathcal{G}}_{\mu}  )\cap \mathcal{F}_{1,\eta}^{-1}(0)$. 

%Hence $\mathsf{C}_\infty$ is a regular value of $\partial_\infty: \widehat{\M}(X_0,\eta) \to \M_s$. Being submersion is an open condition, so we can find an open neighborhood $U_{\mathsf{C}_\infty} \subset \M_s$ of $\mathsf{C}_\infty$ such that the argument is true for any point in $U_{\mathsf{C}_\infty}$. Find a finite open cover $\{U_{\mathsf{C}_\infty^i}\}_i$ of such open sets on $\M_s = \S^1$, and the corresponding $\mathcal{Z}^i_{reg}$. Let

Let $(0, A)$ be a representative of $ {\mathsf{C}}_\infty$. Choose any $(\hat\Phi, \hat A)\in \hat{\mathcal{C}}^*_{\mu,sw}$, then $\partial_\infty (\hat\Phi, \hat A)$ is an $(\eta|_N)$-monopole on ${N}$. We want to show that even if $\partial_\infty(\hat\Phi, \hat A)$ does not represent $ {\mathsf{C}}_\infty$, $d_{(\hat\Phi, \hat A)}\mathcal{F}_1|_{\partial_\infty^{-1}(\partial_\infty(\hat A))/\hat{\mathcal{G}}_{\mu}} $ is still surjective.

Since $\eta$ is zero on the neck, $\partial_\infty \hat A$ is closed and $\partial_\infty \hat\Phi =0$  (see the proof of Proposition \ref{prop:dimOfReducible}). 
Hence $ \partial_\infty(\hat A)-A$ is closed. Since $H^1(\hat{N};\R)\to H^1(N;\R)$ is surjective, 
one can find a closed form $\hat a$ on $\hat N$ such that $ \partial_\infty(\hat A)-A =\partial_\infty(\hat a) +df $ for some function $f$ on $N$. Hence $ \partial_\infty(\hat A + \hat a) = A +df$, which belongs to the gauge equivalence class of $A$. This means $\partial_\infty(\hat A + \hat a) = {\mathsf{C}}_\infty$. 
Because $\hat a$ is closed, if $(\hat\Phi, \hat A)$ is a solution of $\mathcal{F}_{1,\eta}$, $(\hat\Phi, \hat A + \hat a)$ is also a solution of $\mathcal{F}_{1,\eta}$. By (\ref{equ:surjective-for-solution}),  
\begin{align}
 d_{(\hat\Phi, \hat A + \hat a)}(\mathcal{F}_1|_{\partial_\infty^{-1}(\mathsf{C}_\infty)/\hat{\mathcal{G}}_{\mu}}): T_{(\hat\Phi, \hat A + \hat a)} \partial_\infty^{-1}(\mathsf{C}_\infty)/\hat{\mathcal{G}}_{\mu} &\to \hat{\mathcal{Y}}_\mu\\
 (\alpha,\phi) &\mapsto d^+ \alpha  -\rho^{-1}(\sigma(\hat\Phi, \phi)+\sigma(\phi, \hat\Phi)) \label{equ:differentialOfF1}
\end{align}
is surjective. Note that $d_{(\hat\Phi, \hat A)}\mathcal{F}_1$ does not depend on $\hat A$. Also an element of either $T_{(\hat\Phi, \hat A )} \partial_\infty^{-1}(\partial_\infty(\hat A))/\hat{\mathcal{G}}_{\mu}$ or $T_{(\hat\Phi, \hat A + \hat a)} \partial_\infty^{-1}(\mathsf{C}_\infty)/\hat{\mathcal{G}}_{\mu}$ can be written as $(\alpha,\phi)$ such that $\partial_\infty \alpha $ represents $0\in H^1(N;\R)$.
%$\langle \partial_\infty \alpha , v \rangle_{L^2} = 0$ for any nonzero $v\in H^1(N;\R)$. 
Hence
\[
d_{(\hat\Phi, \hat A)}\mathcal{F}_1|_{\partial_\infty^{-1}(\partial_\infty(\hat A))/\hat{\mathcal{G}}_{\mu}} = d_{(\hat\Phi, \hat A + \hat a)}\mathcal{F}_1|_{\partial_\infty^{-1}(\mathsf{C}_\infty)/\hat{\mathcal{G}}_{\mu}}
\]
is surjective.

Let
\[
\mathcal{Z}_{reg} = \mathcal{Z}^0_{reg} \cap \mathcal{Z}^1_{reg}.
\]
For any $\eta \in \mathcal{Z}_{reg}$, if $\hat{\mathsf{C}}_0$ is an $\eta$-monopole, it is irreducible and $H^2(F(\hat{\mathsf{C}}_0))=0$. Moreover, $\mathcal{Z}_{reg}$ is still a countable intersection of open and dense sets, so it is of the second category in the sense of Baire.
\end{proof}

\begin{remark}
The statement of Proposition \ref{prop:noncompactTransversality} is not true in general. If the boundary $N = \S^1\times \S^2$ and $L^1_{top} =0$, we must have 
\[
\dim H^2(F(\hat{\mathsf{C}}_0)) = \dim H^2( \widehat{\mathcal{K}}_{\hat{\mathsf{C}}_0}) + 1.
\] 
To prove this, it suffices to find an element in $T_{\hat{\mathsf{C}}_0} \hat{\mathcal{C}}_{\mu,sw} $, such that its image is not in the image of $T_{\hat{\mathsf{C}}_0} \partial_\infty^{-1}(\mathsf{C}_\infty) $. Indeed, there exists a $1$-form $\alpha \in \Omega^1(\hat N)$ (constructed explicitly in  (\ref{equ:baseDirectionVector})), such that $\partial_\infty \alpha $ generates $H^1(N)$ (namely $(\alpha,0)\notin T_{\hat{\mathsf{C}}_0} \partial_\infty^{-1}(\mathsf{C}_\infty) $), and $d^+ \alpha$ is a nonzero element in $H^2(\hat{N})$. 
Conversely, if $d^+ \alpha'$ is nonzero in $H^2(\hat{N})$, then it's not compactly supported, otherwise it would be orthogonal to any self dual harmornic $2$-forms. Hence $\partial_\infty \alpha' $ is nonzero in $H^1(N)$ (c.f. Figure \ref{fig:BoundaryExactStructure2}). Therefore 
\begin{equation}\label{extraVector}
\left. d\mathcal{F}_1 (\alpha,0) \neq d\mathcal{F}_1\right|_{T ( \partial_\infty^{-1}(\mathsf{C}_\infty) )} (\beta,0)
\end{equation}
for any $(\beta,0) \in T ( \partial_\infty^{-1}(\mathsf{C}_\infty) )$. When the virtual dimension of the moduli space is less then $1$, for a generic perturbation such that for any solution $(\hat{A},\hat{\Phi})$,
\begin{equation}\label{extraVector2}
d_{(\hat{A},\hat{\Phi})}\mathcal{F}_1 (\alpha,0) \notin \text{im}\left. d_{(\hat{A},\hat{\Phi})}\mathcal{F}_1 \right|_{T\partial_\infty^{-1}(\mathsf{C}_\infty) },
\end{equation}
even though $\hat{b}_+ > 0$. This is because in this case the connection part is not able to kill $d_{(\hat{A},\hat{\Phi})}\mathcal{F}_1 (\alpha,0) $ by (\ref{extraVector}), and the spinor part is responsible to kill the other complement, instead of $d_{(\hat{A},\hat{\Phi})}\mathcal{F}_1 (\alpha,0) $, otherwise it will produce one more dimension of the cokernel and one more dimension of the moduli space, which would not happen by the classical transversality argument. Hence $\dim H^2(F(\hat{\mathsf{C}}_0)) = \dim H^2( \widehat{\mathcal{K}}_{\hat{\mathsf{C}}_0}) + 1$ for any solution $\hat{\mathsf{C}}_0$.

In fact, the condition on the virtual dimension can be omitted. $d^+ \alpha$ is not compactly supported , and the harmonic projection $\H(d^+ \alpha)$ satisfies
\[
\partial_\infty^0 \H(d^+ \alpha) \neq 0
\]
where $\partial_\infty^0$ is defined in (\ref{L2exsequence}). 
On the other hand, the second term 
\[
-\rho^{-1}(\sigma(\hat\Phi, \phi)+\sigma(\phi, \hat\Phi))
\]
of (\ref{equ:differentialOfF1}) is in $L_\mu$ since $\partial_\infty \hat{\Phi} = 0$. Hence (\ref{extraVector2}) is true as long as all solutions on the boundary are reducible.

This example is a counter example of \cite{Nicolaescu2000NotesOS} Proposition 4.4.1. The equation
\[
\dim H^2( \widehat{\mathcal{K}}_{\hat{\mathsf{C}}_0})  = \hat{b}_+
\]
for $\hat \Phi=0$ computed in \cite{Nicolaescu2000NotesOS} Page 404, combined with the equation
\[
\left. \dim \ker_{ex} (\textbf{ASD}^*) \right|_{\Omega^2(\hat{N})} =  \hat{b}_+ + \dim L^2_{top}
\]
computed in \cite{Nicolaescu2000NotesOS} Page 312, also shows the the existence of $\alpha$ satisfying (\ref{extraVector}) without any explicit construction.
\end{remark}

\begin{prop}\label{prop:triObs}
For $\hat{N}= X_0$, $\widehat{\S^1\times \S^3}$ or $D^3\times \S^1$ with positive scalar curvature metric $\hat{g}$ chosen in subsection \ref{subsection:PSCmetric}, We can choose suitable perturbations $\eta = \eta(\hat{N})$ such that if $\hat{\mathsf{C}}_0$ is an $\eta$-monopole, $H^2(F(\hat{\mathsf{C}}_0))=0$.
\end{prop}

\begin{proof}
As in the usual argument of transversality, we just need to take care of the boundary term to prove that, if $b_+(\hat{N})>0$, we can choose a pertubation $\eta \in H^2_+(\hat{N})$ such that all $\eta$-monopoles are strongly regular (and irreducible) (Proposition \ref{prop:noncompactTransversality}). Since $H^2_+(X_0)$ is assumed to be nontrivial, the statement is true for $X_0$.

For $\hat{N} \simeq \S^1\times \S^3$, $D^3\times \S^1$ or $\S^2\times D^2$, all monopoles are reducible. Let $\hat{\mathsf{C}}_0 = (\hat{A}_0,0)$ be a reducible monopole for the SW equations without perturbation. The connection $\hat{A}_0$ on the cylindrical manifold $\hat{N}$ gives an asymptotically cylindrical Dirac operator $\D^*_{\hat{A}_0}$ with
\[
\partial_\infty \D^*_{\hat{A}_0} = \mathfrak{D}^*_{{A}_0}.
\]
The middle column of the Obstruction space diagram \ref{equ:3O} comes from the exact sequence (\cite{Nicolaescu2000NotesOS} Proposition 4.3.30)
\[
0\to H^2(F(\hat{\mathsf{C}}_0)) \to \text{ker}_{\text{ex}}\hat{\mathcal{T}}^*_{\hat{C}_0} \stackrel{\partial_\infty^0}{\to} \im(T_1\hat{G}_0 \stackrel{\partial_\infty}{\to} T_1G_\infty) \to 0.
\]
Recall that in (\ref{defOfTWithoutMu}) we define
\[
\hat{\mathcal{T}}_{\hat{\mathsf{C}}_0} := \widehat{\underline{SW}}_{\hat{\mathsf{C}}_0} \oplus \frac{1}{2}\mathfrak{L}^*_{\hat{\mathsf{C}}_0}.
\]
If $\mathbf{i}f\in T_1\hat{G}_0$, then it's in the kernel of $\mathfrak{L}_{\hat{\mathsf{C}}_0}$, and therefore in $\ker_{ex}\hat{\mathcal{T}}^*_{\hat{C}_0}= \ker_{ex} (\widehat{\underline{SW}}^*_{\hat{\mathsf{C}}_0} \oplus \frac{1}{2}\mathfrak{L}_{\hat{\mathsf{C}}_0})$. On the other hand, if
\[
(\Psi, \mathbf{i}f)\in  L_{ex}^{1,2}(\hat{\S}_{\hat{\sigma}}^- \oplus \mathbf{i}\Lambda_+^2T^*\hat{N}) \oplus L_{ex}^{1,2}( \mathbf{i}\Lambda^0 T^*\hat{N})
\]
is in $\ker_{ex}\hat{\mathcal{T}}^*_{\hat{C}_i}$, then $\mathbf{i}f\in T_1G_0$. Thus
\[
  \partial^0_\infty \ker_{ex}\hat{\mathcal{T}}^*_{\hat{C}_0}
 \cong \partial_\infty T_1G_0.
\]
Namely, $H^2(F(\hat{\mathsf{C}}_0))$ doesn't contain constant functions. Hence 
\[H^2(F(\hat{\mathsf{C}}_0)) = \ker_{\text{ex}}\D^*_{\hat{A}_0} \oplus \ker_{\text{ex}}(d^+ \oplus d^*)^*|_{\Lambda_+^2(T^*\hat{N}) \oplus \Lambda^0_0 (T^*\hat{N})}.
\]
Then by the computation of the \textbf{ASD} operator $d^+ \oplus d^*
$ (\cite{Nicolaescu2000NotesOS} Example 4.1.24), %((4.3.21) in Nicolaescu's book page 398)
\begin{equation}\label{equ:obstructionF=kerexD*}
H^2(F(\hat{\mathsf{C}}_0)) = \text{ker}_{\text{ex}}\D^*_{\hat{A}_0}\oplus H^2_+(\hat{N})\oplus L^2_{top}, 
\end{equation}
where $L^2_{top}=\im( i^*:H^2(\hat{N})\to H^2(\partial\hat{N}))$ for inclusion map $i:N\to \hat{N}$. Thus the second and the third components are trivial for $\hat{N}=\widehat{\S^1\times \S^3} $ or ${D^3\times \S^1}$. Now compute the dimension of $\text{ker}_{\text{ex}}\D^*_{\hat{A}_0} $. Since each of them has a positive scalar curvature metric, by the Weitzenb{\"o}ck formula, the twisted Dirac operater is invertible since $A_0$ is flat. This means that $\ker \mathfrak{D}^*_{A_0} = 0$ and therefore
\begin{equation}\label{equ:kerex=ker}
\ker_{ex}\D^*_{\hat{A}_0} = \ker_{L^2}\D^*_{\hat{A}_0}.
\end{equation}
Hence
\[
I_{APS}(\D_{\hat{A}_0} ) = \dim_{\C} \ker_{L^2} \D_{\hat{A}_0} -  \dim_{\C} \ker_{L^2} \D^*_{\hat{A}_0},
\]
where $I_{APS}(\hat{L})$ is the Atiyah-Patodi-Singer index of the APS operator $\hat{L}$. One can also prove that $ \ker_{L^2} \D_{\hat{A}_0}$ is trivial by the Weitzenb{\"o}ck formula (see \cite{Nicolaescu2000NotesOS} Page 323). Hence 
\[
-\dim \text{ker}_{\text{ex}}\D^*_{\hat{A}_0}= I_{APS}(\D_{\hat{A}_0} ).
\]
By the Atiyah-Patodi-Singer index theorem (\cite{atiyah_patodi_singer_1975}) we have % (4.1.3)
\[
I_{APS}(\D_{\hat{A}_0} ) = \frac{1}{8} \int_{\hat{N}} (p_1(\hat{\nabla}^{\hat{g}}) + c_1(\hat{A}_0)^2) -\frac{1}{2}(\dim\ker\mathfrak{D}_{A_0}  +\eta_{Dir}({A_0})),
\]
where $\hat{\nabla}^{\hat{g}}$ is the Levi-Civita connection of $\hat{g}$, $p_1(\hat{\nabla}^{\hat{g}})$ and $c_1(\hat{A}_0)$ are the first Pontryagin class and the first Chern class determined by the Chern-Weil construction, and $\eta_{Dir}({A_0})$ is the eta invariant of the Dirac operator $\mathfrak{D}_{A_0}$. 
For any $4$-manifold with boundary, one has ``signature defect'' (see \cite{Nicolaescu2000NotesOS} (4.1.34), see also \cite{atiyah_patodi_singer_1975}, \cite{atiyah_patodi_singer_1975II} and \cite{atiyah_patodi_singer_1976} for the motivation)
\[
\eta_{sign}(g) = \frac{1}{3}\int_{\hat{N}} p_1(\hat{\nabla}^{\hat{g}}) - \sigma(\hat{N}) 
\]
where $\eta_{sign}(g)$ is the eta invariant of the metric $g = \partial_\infty \hat{g}$. Also recall that
\[
\mathbf{F}(g,A_0) := 4\eta_{Dir}({A_0})+ \eta_{sign}(g).
\]
Combine all of these, one has %(Page 324 (4.1.36)) 
\[
8\dim \text{ker}_{\text{ex}}\D^*_{\hat{A}_0} =\mathbf{F}(\partial_\infty\hat{C}_0)  + \sigma(\hat{N})  - \int_{\hat{N}} c_1(\hat{A}_0)^2.
\]
For $\hat{N} \simeq \S^1\times \S^3$, $D^3\times \S^1$ or $\S^2\times D^2$, $\sigma(\hat{N})  = 0$. 
For $\hat{N} \simeq D^3\times \S^1$ or $\S^2\times D^2$, $N = \partial_\infty\hat{N} =\S^1\times \S^2$, and the metric $\hat{g}$ chosen in subsection \ref{subsection:PSCmetric} ensures that $g = \partial_\infty \hat{g}$ is the product of canonical metric on $\S^1$ and $\S^2$. 
In this situation $\eta_{sign}(g) = 0$ (\cite{Komuro84}) and $\eta_{Dir}(\partial_\infty\hat{C}_0) = 0$ (\cite{Nicolaescu98} Appendix C). Hence $\mathbf{F}(\partial_\infty\hat{C}_0) = 0$. %(Example 4.3.35). 
For $\hat{N} \simeq D^3\times \S^1$ or $\S^2\times D^2$, as shown in Proposition \ref{prop:dimOfReducible}, $\hat{A}_0$ is a flat connection. Hence for $\hat{N} \simeq D^3\times \S^1$ or $\S^2\times D^2$, $\dim \text{ker}_{\text{ex}}\D^*_{\hat{A}_0}=0$. 
So the first component of $H^2(F(\hat{\mathsf{C}}_0)) $ is also trivial. Thus $\hat{C}_0$ is strongly regular for $\hat{N}=\widehat{\S^1\times \S^3} $ or ${D^3\times \S^1}$ without perturbations.
\end{proof}

For $\hat{N}=\S^2\times D^2$, unfortunately, $L^2_{top}$ is $1$-dimensional ($ i^*:H^2({\S^2\times D^2})\to H^2(\S^2\times \S^1)$ is an isomorphism between two copies of $\Z$), so $H^2(F(\hat{\mathsf{C}}_2))$ is $1$-dimensional in the obstruction diagram for $\hat{\mathsf{C}}_1$ on $X_0$ and $\hat{\mathsf{C}}_2$ on $\S^2\times D^2$. However, we have
\begin{prop}\label{prop:trivialObs}
When $r$ is large enough, the obstruction space $\mathcal{H}^-_r$ for $X'=X_0 \cup_{\S^1\times \S^2}D^2\times  \S^2$ is still $0$.
\end{prop}
\begin{proof}
Let $\hat{N}_1 = X_0$, $\hat{N}_2 = \S^2\times D^2$. Then $N=\partial_\infty \hat{N}_i= \S^2\times \S^1$. The method is to trace the Obstruction diagram.

First, by Propsition \ref{prop:triObs}, $H^2(F(\hat{\mathsf{C}}_1))=0$, and $H^2(F(\hat{\mathsf{C}}_2))\cong \R$.

Next, we identify $\mathfrak{C}_i^-$. Recall that
\[
\hat{\mathcal{T}}_{\hat{\mathsf{C}}_i} := \widehat{\underline{SW}}_{\hat{\mathsf{C}}_i} \oplus \frac{1}{2}\mathfrak{L}^*_{\hat{\mathsf{C}}_i}.
\]
If $\mathbf{i}f\in T_1G_i$, then it's in the kernel of $\mathfrak{L}_{\hat{\mathsf{C}}_i}$, and therefore in $\ker_{ex}\hat{\mathcal{T}}^*_{\hat{C}_i}= \ker_{ex} (\widehat{\underline{SW}}^*_{\hat{\mathsf{C}}_i} \oplus \frac{1}{2}\mathfrak{L}_{\hat{\mathsf{C}}_i})$. On the other hand, if
\[
(\Psi, \mathbf{i}f)\in  L_{ex}^{1,2}(\hat{\S}_{\hat{\sigma}}^- \oplus \mathbf{i}\Lambda_+^2T^*\hat{N}_i) \oplus L_{ex}^{1,2}( \mathbf{i}\Lambda^0 T^*\hat{N}_i)
\]
is in $\ker_{ex}\hat{\mathcal{T}}^*_{\hat{C}_i}$, then $\mathbf{i}f\in T_1G_i$. Thus
\[
\mathfrak{C}_i^- =  \partial^0_\infty \ker_{ex}\hat{\mathcal{T}}^*_{\hat{C}_i}
 \cong \partial_\infty T_1G_i.
\]
%By Lagragian condition (see (4.1.22) of Section 4.1.5 of Nicolaescu's book),
%\[
%\mathfrak{C}_i^+ \oplus \mathfrak{C}_i^- = T_1G_\infty.
%\]
%Combine these two equations, we have
%\[
%\mathfrak{C}_i^+ =  \partial^0_\infty \ker_{ex}\hat{\mathcal{T}}_{\hat{C}_i} = T_1(G_\infty/\partial_\infty \hat{G}_i)
%\]

For manifolds with cylindrical end, we can choose a generic perturbation in a $b^+$-dimensional space just as in the compact case (see page 404 of Nicolaescu's book \cite{Nicolaescu2000NotesOS} for a proof). 
Since $b^+( X_0)>0$, we can choose a compactly supported $2$-form $\eta$ such that all monopoles on $\hat{N}_1 = X_0$ are irreducible. 
Since $\hat{N}_2 = \S^2\times D^2$ and $N= \S^2\times \S^1$ admit PSC metric, all monopoles on $\hat{N}_2 = \S^2\times D^2$ and $N$ are reducible. So $\mathfrak{C}_1^- = 0$ and $ \mathfrak{C}_2^- \cong \R$. So $\Delta_-^0$ is an isomorphism in the obstruction diagram.  
Since each row of the diagram is asymptotically exact, any unit vector of $S_r (\ker\Delta_-^0) $ approaches $0$ as $r\to \infty$. So $S_r (\ker\Delta_-^0)=0$ and thus $\ker\Delta_-^0$ must be trivial when $r$ is large enough. Since each column of the diagram \ref{equ:3O} is exact, ${\H}_r^-  \cong \ker\Delta_-^c$.

Next we identify $L_i^-$. We have assumed $\M(X_0)$ is $1$-dimensional, and since $D^2\times  \S^2$ has a PSC metric and $H^1(D^2\times  \S^2) = 0$, $\M(D^2\times  \S^2)$ is only one reducible point. $\S^1\times \S^2$ also has a PSC metric and $H^1(\S^1\times \S^2) = 0$, so $\M(\S^1\times \S^2)$ is a circle of reducible solutions. So 
\begin{align*}
\dim_\R H^1_{\hat{\mathsf{C}}_1} &= 1, \\
\dim_\R H^1_{\hat{\mathsf{C}}_2} &= 0, \\
\dim_\R T_{C_\infty}\M_\sigma &= 1. 
\end{align*}
In the first row of diagram \ref{equ:3T}, $L^+_2 = \Delta_+^c (H^1_{\hat{\mathsf{C}}_2})$. Hence $L^+_2$ is certainly $0$. By complementarity equations from the Lagrangian condition (see (4.1.22) of Section 4.1.5 of Nicolaescu's book), we have 
\[
L_i^+\oplus L_i^- =  T_{C_\infty}\M_\sigma.
\]
So $L_2^-$ is $\R$. Thus in the first row of obstruction diagram \ref{equ:3O}, $L_1^- + L_2^- = \R$. Since $H^2(F(\hat{\mathsf{C}}_1))\oplus H^2(F(\hat{\mathsf{C}}_2)) = \R$, $\Delta_-^c $ is an isomorphism and ${\H}_r^-  \cong \ker\Delta_-^c=0$.
\end{proof}

\subsection{Global gluing theorem}\label{section:global-gluing}
We already have local gluing results. Now we can combine them to prove that, the moduli space of solutions of the new manifold is the fiber product of two old moduli spaces. %Our cases are different from the examples in Section 4.5.4 of Nicolaescu's book, but we can mimic the method there.

We assume the following:

$\mathbf{A_1}$  $(N,g)$ is $\S^3$ or $\S^1\times \S^2$ with a positive scalar metric.

$\mathbf{A_2}$  $b_+(\hat{N}_1) > 0$, $b_+(\hat{N}_2)=0$.

$\mathbf{A_3}$  All the finite energy monopoles on $\hat{N}_1$ are irreducible and strongly regular.

$\mathbf{A_4}$  Any finite energy ${\hat{\sigma}_2}$-monople $\hat{\mathsf{C}}_2$ is reducible and $\dim_{\R} H^1_{\hat{\mathsf{C}}_2}$ is $0$ or $1$.

$\mathbf{A_5}$ The obstruction space $\mathcal{H}^-_r$ is $0$ when $r$ is large enough.

Recall that
\[
\hat{\mathcal{Z}}_{\Delta} := \{ (\hat{\mathsf{C}}_1,\hat{\mathsf{C}}_2)\in \hat{\mathcal{Z}}_1\times\hat{\mathcal{Z}}_2; \partial_\infty \hat{\mathsf{C}}_1 =  \partial_\infty \hat{\mathsf{C}}_2\}
\]
is the space of compatible monopoles, and $\widehat{\mathcal{G}}_i$ is the gauge group on $\hat{N}_i$. Define
\[
\widehat{\mathcal{G}}_\Delta := \{(\hat{\gamma}_1,\hat{\gamma}_2) \in \widehat{\mathcal{G}}_1\times\widehat{\mathcal{G}}_2 ; \partial_\infty\hat{\gamma}_1 =  \partial_\infty\hat{\gamma}_2\}.
\]
%where the base point $*$ is a fixed point on $N$.

Let
\[
\hat{\mathfrak{N}}:= \hat{\mathcal{Z}}_{\Delta}  /\widehat{\mathcal{G}}_\Delta .
\]
The cutoff trick described before (see \ref{equ:cut} and \ref{equ:paste}) gives gluing maps
\[
\#_r: \widehat{\mathcal{G}}_\Delta\to \widehat{\mathcal{G}}_{\hat{N}_r}
\]
and
\[
\#_r: \hat{\mathcal{Z}}_{\Delta} \to \mathcal{C}_{\hat{N}_r}.
\]
The second one is  $(\widehat{\mathcal{G}}_\Delta , \widehat{\mathcal{G}}_{\hat{N}_r})$-equivariant, since these gluing maps share the same parameter $r$. So we can mod out by the $(\widehat{\mathcal{G}}_\Delta , \widehat{\mathcal{G}}_{\hat{N}_r})$-action, and get
\[
\hat{\#}_r: \hat{\mathfrak{N}} \to \hat{\mathcal{B}}_{\hat{N}_r}
\]
We also denote the image of this map by $\hat{\mathfrak{N}}$.

\begin{theorem}\label{thm:ggt}
Under assumptions ($\mathbf{A_1}$) - ($\mathbf{A_5}$), for large enough $r$, $\hat{\#}_r\hat{\mathfrak{N}}$ is isotopic to the moduli space of genuine monopoles ${\M}(\hat{N}_r)$ as submanifolds of $\hat{\mathcal{B}}_{\hat{N}_r}$.

\end{theorem}
\begin{proof}
For any point $ (\hat{\mathsf{C}}_1,\hat{\mathsf{C}}_2)$ in $\hat{\mathcal{Z}}_{\Delta}$, let 
\[
\mathsf{C}_r = \#_r  (\hat{\mathsf{C}}_1,\hat{\mathsf{C}}_2) = \hat{\mathsf{C}}_1 \#_r \hat{\mathsf{C}}_2.
\]
By assumption $\mathbf{A_1}$, all monopoles on $N$ are reducible. Thus $T_1G_\infty = \R$. By assumption $\mathbf{A_4}$, $\mathfrak{C}_2^- = \R$, so that $\mathfrak{C}_2^+ = 0$. Hence $\Delta^0_+$ must be an isomorphism in the last row of diagram \ref{equ:3T}. So 
\begin{equation}\label{equ:H+=kerDelta}
\H_r^+ \cong^a \ker\Delta_+^c,
\end{equation}
where $ \cong^a$ means that the isomorphism is given by an asymptotic map in the sense of \cite{Nicolaescu2000NotesOS} page 301.

Now we want to show 
\begin{equation}\label{equ:kerDelta=Tangent}
\ker\Delta_+^c \cong T_{[\mathsf{C}_r]}\hat{\mathfrak{N}}.
\end{equation}
By the definition of $H^1_{\hat{\mathsf{C}}_i}$ and boundary difference map $\Delta_+^c $, a point in $\ker\Delta_+^c $ is a pair $(\underline{\hat{\mathsf C}}_1, \underline{\hat{\mathsf C }}_2)\in \mathcal{S}_{\hat{\mathsf{C}}_1} \times \mathcal{S}_{\hat{\mathsf{C}}_2}$ in the local slice of monopoles, such that $\partial_\infty \underline {\hat{\mathsf{C}}}_1 =  \partial_\infty \underline{\hat{\mathsf{C}}}_2$. 
On the other hand, any point of $T_{[\mathsf{C}_r]}\hat{\mathfrak{N}}$ can be represented by $(\hat{\gamma}_1\underline{\hat{\mathsf C}}_1, \hat{\gamma}_2\underline{\hat{\mathsf C }}_2)\in T\hat{\mathcal{Z}}_{\Delta}$ for  $(\underline{\hat{\mathsf C}}_1, \underline{\hat{\mathsf C }}_2)\in\ker\Delta_+^c$ and $(\hat{\gamma}_1,\hat{\gamma}_2) \in \widehat{\mathcal{G}}_1 \times\widehat{\mathcal{G}}_2$, by the definition of slice. Since $\underline{\hat{\mathsf C}}_1$ and $\underline{\hat{\mathsf C}}_2$ have the same boundary value, and $(\hat{\gamma}_1\underline{\hat{\mathsf C}}_1, \hat{\gamma}_2\underline{\hat{\mathsf C }}_2) \in T\hat{\mathcal{Z}}_{\Delta}$, $\hat{\gamma}_1$ and $\hat{\gamma}_2$ must coincide on the boudary. Thus $(\hat{\gamma}_1,\hat{\gamma}_2) \in T\widehat{\mathcal{G}}_\Delta$. Therefore, $\ker\Delta_+^c \cong T_{[\mathsf{C}_r]}\hat{\mathfrak{N}}$.

By (\ref{equ:H+=kerDelta}) and (\ref{equ:H+=kerDelta}), the family of $\H_r^+$ indexed by $\mathsf{C}_r$ forms the tangent bundle of $\hat{\mathfrak{N}}$ when $r$ is sufficiently large. We again denote it by $\H_r^+$. By the definition of $\mathcal{Y}_r^+$, it's the normal bundle of $\hat{\mathfrak{N}}$ in $\hat{\mathcal{B}}_{\hat{N}_r}$. By condition $\mathbf{A_5}$, the map $\kappa_r$ in theorem \ref{thm:lgc} must be zero. We conclude that ${\M}(\hat{N}_r) $ is a section of the normal bundle of $ \hat{\mathfrak{N}}$ locally. Thus for each $\mathsf{C}_r$, there exists an open neighborhood $U_r$, such that ${\M}(\hat{N}_r) \cap U_r \cong  \hat{\mathfrak{N}}\cap U_r $. By theorem \ref{thm:ls}, this fact is globally true.
\end{proof}

Now we can show that $\hat{\mathfrak{N}}$ above is desired fiber product of moduli space.
\begin{lemma}\label{lem:fibProd}
Let $\mathcal{Z}$ be monopoles on $N$. Define
\[
\mathcal{G}^{\partial_\infty}:= \partial_\infty\widehat{\mathcal{G}}_1 \cdot  \partial_\infty\widehat{\mathcal{G}}_2,
\]
\[
\M^{\partial_\infty} := \mathcal{Z}/\mathcal{G}^{\partial_\infty},
\]
\[
\hat{\mathcal{Z}} := \{ (\hat{\mathsf{C}}_1,\hat{\mathsf{C}}_2)\in \hat{\mathcal{Z}}_1\times\hat{\mathcal{Z}}_2; \partial_\infty \hat{\mathsf{C}}_1 \equiv  \partial_\infty \hat{\mathsf{C}}_2 \mod \mathcal{G}^{\partial_\infty} \}.
\]
Then we have
\[
\hat{\mathcal Z}/\widehat{\mathcal{G}}_1\times\widehat{\mathcal{G}}_2 = 
 \{ ([\hat{\mathsf{C}}_1],[\hat{\mathsf{C}}_2])\in \hat{\M}_1\times\hat{\M}_2; \partial_\infty [\hat{\mathsf{C}}_1] =  \partial_\infty [\hat{\mathsf{C}}_2] \in \M^{\partial_\infty}\}
\]
and
\[
\hat{\mathcal Z}/\widehat{\mathcal{G}}_1\times\widehat{\mathcal{G}}_2 \cong \hat{\mathfrak{N}}.
\]
\end{lemma}
\begin{proof}
The first equility is just by definition. We prove the second one:

$\hat{\mathfrak{N}}$ is certainly a subset of $\hat{\mathcal Z}/\widehat{\mathcal{G}}_1\times\widehat{\mathcal{G}}_2$. For any $([\hat{\mathsf{C}}_1],[\hat{\mathsf{C}}_2])$ in $\hat{\mathcal Z}/\widehat{\mathcal{G}}_1\times\widehat{\mathcal{G}}_2$, suppose it's represented by $(\hat{\mathsf{C}}_1,\hat{\mathsf{C}}_2)\in \hat{\mathcal{Z}}$. Then there exists $g\in \mathcal{G}^{\partial_\infty}$ such that $g\cdot \partial_\infty \hat{\mathsf{C}}_1 =  \partial_\infty \hat{\mathsf{C}}_2$. Suppose $g=\partial_\infty g_1\cdot \partial_\infty g_2$, where $g_i\in \widehat{\mathcal{G}}_i$. Now $([\hat{\mathsf{C}}_1],[\hat{\mathsf{C}}_2])=([g_1\cdot\hat{\mathsf{C}}_1],[g_2^{-1}\cdot\hat{\mathsf{C}}_2])\in \hat{\mathcal Z}/\widehat{\mathcal{G}}_1\times\widehat{\mathcal{G}}_2$ and $(g_1\cdot\hat{\mathsf{C}}_1,g_2^{-1}\cdot\hat{\mathsf{C}}_2)\in \hat{\mathcal{Z}}_\Delta$. So $\hat{\mathcal Z}/\widehat{\mathcal{G}}_1\times\widehat{\mathcal{G}}_2 \subset \hat{\mathfrak{N}}$.
\end{proof}

\begin{corollary}\label{fiberProduct}
\begin{align}
\M(X)&\cong\M(X_0) \times_{\M(\S^1\times \S^2)} \M(\S^1\times D^3)\\
\M(X')&\cong\M(X_0) \times_{\M(\S^1\times \S^2)} \M(D^2\times \S^2)
\end{align}
\end{corollary}
\begin{proof}
By Proposition \ref{prop:triObs} and \ref{prop:trivialObs}, all assumptions of Theorem \ref{thm:ggt} are satisfied. Thus $\M(X)\cong \hat{\mathfrak{N}}$. By Lemma \ref{lem:fibProd},
\[
\M(X)\cong\M(X_0) \times_{\M^{\partial_\infty}(\S^1\times \S^2)} \M(\S^1\times D^3).
\]
But in our case, $H^1(X_0)\to H^1(\S^1\times \S^2)$ is surjective. Thus $\partial_\infty\widehat{\mathcal{G}}_1 =  \mathcal{G}$. Therefore $\M^{\partial_\infty}(\S^1\times \S^2) = \M(\S^1\times \S^2)$. 

The proof of the second equation is similar.
\end{proof}

\subsection{The proof of $1$-surgery formula}\label{subsection:proofUnparemetrized}
Now we can investigate Seiberg-Witten invariants of $X$ and $X'$. According to section 2.2 of \cite{LL01}, for higher dimensional moduli space ${\M}(\hat{N}_r)$, given an integral cohomology class $\Theta$ of moduli space $\hat{\mathcal{B}}_{\hat{N}_r}$, the Seiberg-Witten invariant associate to this class is 
\[
SW^{ \Theta}(\hat{N}_r, \ss) := \langle\Theta,[ {\M}(\hat{N}_r,\ss) ] \rangle
\]

Since $H^1(X)= H^1(X_0) = \R$, $\hat{\mathcal{B}}_X\cong \hat{\mathcal{B}}_{X_0}\cong  \C P^\infty_+\times \S^1$. We choose $\Theta$ to be a generator of $H^1( \C P^\infty_+\times \S^1,\Z)$. 

We first show that the invariant $SW^{ \Theta}$ is well defined:
\begin{lemma}\label{lem:detect-exotic}
Suppose that $b^+(X)>1$ and that $f:X\to X$ is a diffeomorphism. Let $h$ and $k$ be generic paramters. Then $SW^\Theta(E_X, \ss, h) = SW^\Theta(E_X, \ss, k)$.
\end{lemma}
\begin{proof}
Since $b^+(X)>1$, by a generic argement (similar to the one in the proof of \ref{prop:noncompactTransversality-general}), there exists a generic path $K$ from $h$ to $k$. Hence there exists a cobordism from $\M(E_X,\ss,h)$ to $\M(E_X,\ss,k)$. This cobordism is a $2$-dimensional manifold with $1$-dimensional boundary, so after cutting it by the class $\Theta$, we obtain a $1$-dimensional cobordism which gives $SW^\Theta(E_X, \ss, h) = SW^\Theta(E_X, \ss, k)$ (see Figure \ref{fig:infinite-cobordism}).
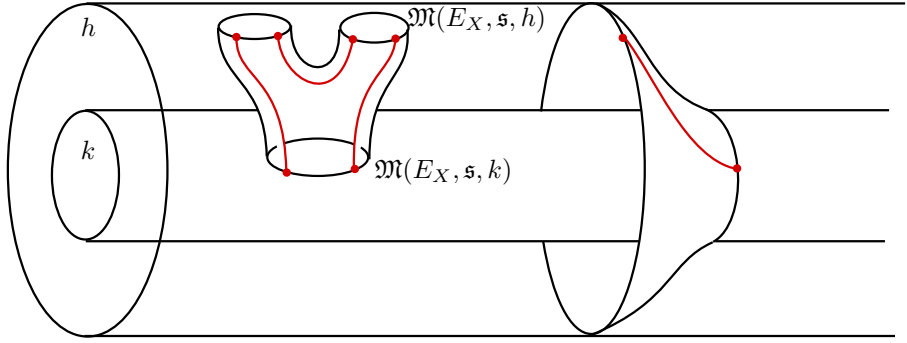
\begin{figure}[ht!]
    \begin{Overpic}{\includegraphics[scale=0.6]{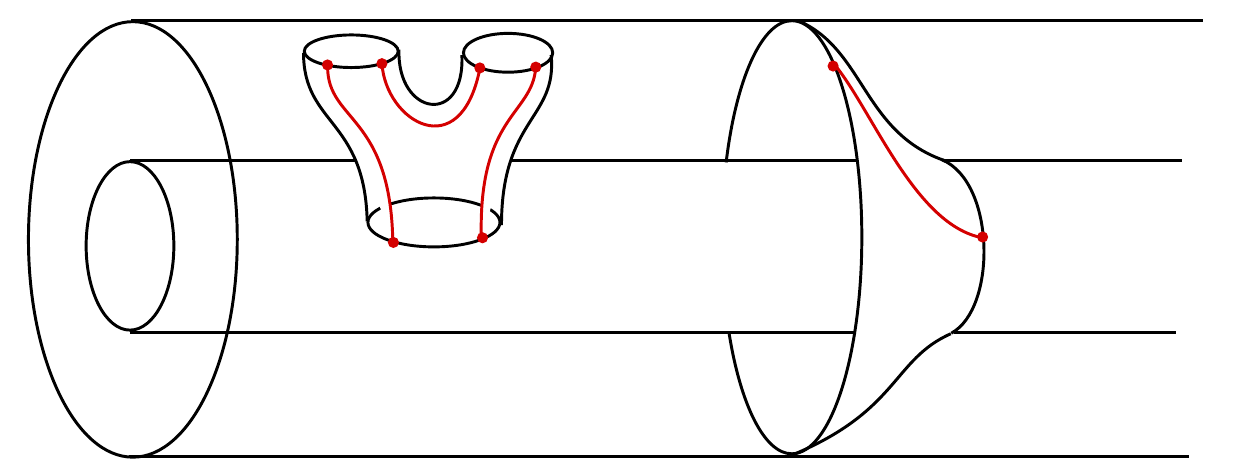}}
     \put(10,33){$h$}
    \put(10,20){$k$}
         \put(44.5,34){$\M(E_X,\ss,h)$}
        \put(41,18){$\M(E_X,\ss,k)$}
        
    \end{Overpic}
    \caption{The cobordism in $\CP^\infty \times \S^1 \times I$.}
    \label{fig:infinite-cobordism}
\end{figure}
\end{proof}

\begin{theorem}\label{main-thm}
$SW^\Theta(X,\ss)=SW(X',\ss')$.
\end{theorem}

\begin{proof}
Since each of $\M(\S^1\times \S^2)$ and $\M(\S^1\times D^3)$ is a circle of reducibles, and these circles are given by the monodromy of connections around their $\S^1$ factor, it's clear that
\[
\partial_\infty: \M(\S^1\times D^3) \to \M(\S^1\times \S^2)
\]
is identity. By Corollary \ref{fiberProduct}, $\M(X)\cong\M(X_0)$.

For $\M(X_0)$, $\partial_\infty:\M(X_0) \to \M(\S^1\times \S^2)$ is not necessarily a homeomorphism, but we can prove that this map is a submersion. Recall that we have choosen a generic perturbation $\eta$ such that $\M(X_0)=\M(X_0,\eta)$ contains only strongly regular points. By the long exact sequence \ref{equ:longExact}:
\[
\cdots
\stackrel{\phi}{\to} H^1_{\hat{\mathsf{C}}_1} = T_{\hat{\mathsf{C}}_1} \M(X_0)
\to H^1(B)=T_{\partial_\infty\hat{\mathsf{C}}_1} \M^{\partial_\infty X_0}(\S^1\times \S^2)
\to H^2(F) =0
\to \cdots
\]
where 
\[
\M^{\partial_\infty X_0}(\S^1\times \S^2) = \mathcal{Z}(\S^1\times \S^2)/\partial_\infty\widehat{\mathcal{G}}_1= \mathcal{Z}(\S^1\times \S^2)/{\mathcal{G}} = \M(\S^1\times \S^2)
\]
(since $H^1(X_0) \stackrel{i^*}{\to} H^1(\S^1\times \S^2) $ is surjective), 
%we do surgery around the nontrivial loop of $X$. So by page 237 and Lemma 8.1.5 of Salamon's book, the map
\[
\partial_\infty:\M(X_0) \to \M(\S^1\times \S^2)
\]
is a submersion.
%Let $\Theta_0$ be the generator of $H^1(\hat{\mathcal{B}}_{X_0}, \Z) \cong H^1( \C P^\infty_+\times \S^1,\Z)$.

By compactness result, $\M(X_0)$ is a disjoint union of finite many circles, say
$\amalg_{i\in\Gamma} \S^1_i $. Let $d_i$ be the mapping degree of $\partial_\infty|_ {\S^1_i} :  \S^1_i  \to  \M(\S^1\times \S^2) = \S^1$. We claim that 
\[
SW(X,\Theta) = \sum_{i\in \Gamma} d_i.
\]

Let
\[
\hat{\mathfrak{N}}_i :=  \S^1_i \times_{\M(\S^1\times \S^2)} \M(\S^1\times D^3)  \subset \hat{\mathfrak{N}}
\]
be the space of configurations obtained by gluing $\S^1_i$ and $\M(\S^1\times D^3)$. Consider the pullback diagram of moduli spaces:
\begin{equation}
% \ar@{}[r]|-{\rotatebox[origin=c]{0}{$=$}}
%\ar@{}[d]|-{\rotatebox[origin=c]{90}{$\subset$}}
\xymatrix@C=4ex{
\M(X) \subset & \mathcal{B}_X =   \C P^\infty_+\times \S^1 \ar@<-7ex>[d]^{p_1}  \ar[r]^{p_2} & \mathcal{B}_{\S^1\times D^3} = \C P^\infty_+\times \S^1 \ar@<-7ex>[d]^{\partial_\infty^2}  &\supset \M(\S^1\times D^3)=  \{0\} \times\S^1 \ar@<8.3ex>[d]^\cong \\
 \S^1_i \subset & \mathcal{B}_{X_0} =  \C P^\infty_+\times \S^1 \ar[r]^{\partial_\infty^1}   & \mathcal{B}_{\S^1\times \S^2}= \C P^\infty_+\times \S^1  & \supset \M(\S^1\times \S^2) = \{0\} \times\S^1
}
\end{equation}
When restricted to $\S^1$-factors, $\partial_\infty^1$ and $\partial_\infty^2$ are identity maps of $\S^1$, so $p_1$ and $p_2$ are identity maps of $\S^1$. Therefore, $\hat{\mathfrak{N}}_i $ winds around the $\S^1$-factor of $\mathcal{B}_X$ by $d_i$ times. So
\[
\langle [\hat{\mathfrak{N}}_i], \Theta\rangle = d_i.
\]
By Theorem \ref{thm:ggt}, $\M(X)$ is isotopic to $\hat{\mathfrak{N}}$ in $ \mathcal{B}_X $, so 
\[
 \langle[ {\M}(X) ],\Theta \rangle = \sum_{i\in \Gamma} d_i.
 \]

On the other hand, 
\[
\partial_\infty: \M(D^2\times S^2) \to \M(\S^1\times \S^2)
\]
is the inclusion of one point. Thus we have
\begin{equation}
% \ar@{}[r]|-{\rotatebox[origin=c]{0}{$=$}}
%\ar@{}[d]|-{\rotatebox[origin=c]{90}{$\subset$}}
\xymatrix@C=4ex{
\M(X') \subset & \mathcal{B}_{X'}=   \C P^\infty_+\times \S^1 \ar@<-7ex>[d]^{p_1}  \ar[r]^{p_2} & \mathcal{B}_{D^2\times S^2} = \C P^\infty_+ \ar@<-7ex>[d]^{\partial_\infty^2}  &\supset \M(D^2\times S^2)=  \{0\}  \ar@<9ex>[d] \\
 \S^1_i \subset & \mathcal{B}_{X_0} =  \C P^\infty_+\times \S^1 \ar[r]^{\partial_\infty^1}   & \mathcal{B}_{\S^1\times \S^2}= \C P^\infty_+\times \S^1  & \supset \M(\S^1\times \S^2) = \{0\} \times\S^1
}
\end{equation}
Since $\partial_\infty^1|_{ \S^1_i}$ is a submersion, $\hat{\mathfrak{N}}_i :=  \S^1_i \times_{\M(\S^1\times \S^2)} \M(D^2\times S^2)$ contains $d_i$ points. Again by Theorem \ref{thm:ggt}, $\M(X')$ is isotopic to $\hat{\mathfrak{N}}$ in $ \mathcal{B}_{X'} $. So 
\[
SW(X') =  \sum_{i\in \Gamma} d_i = SW(X,\Theta).
\]
\end{proof}

\begin{remark}\label{rem:high-dim-moduli}
Theorem \ref{main-thm} works for $\dim \M(X) >1$ as long as it is odd. In that case we define $SW^\Theta(X,\ss)$ by 
\[
SW^\Theta(X,\ss) := \langle[ {\M}(X) ],\Theta \cup c_1(\C P^\infty)^n \rangle
\]
for $\dim \M(X) =2n+1$. Note that in this case $\dim \M(X') =2n$ and the ordinary invariant is 
\[
SW(X',\ss') := \langle[ {\M}(X') ], c_1(\C P^\infty)^n \rangle.
\]
Hence for $\dim \M(X) >1$, the argument of Theorem \ref{main-thm} follows from a similar proof.
\end{remark}

\subsection{Exotic smooth structures on nonsimply connected manifolds}
First observe that by definition and Lemma \ref{lem:detect-exotic}, the cut-down invariant also detects exotic smooth structures. As lots of exotic smooth structures are detected by $SW$, we can now generalize those results to nonsimply connected manifolds by the surgery formula:
\begin{theorem}\label{thm:not-diffeomorphic}
Suppose $X_1$, $X_2$ are two simply connected smooth $4$-manifolds with $b^+_2(X_i)>1$. Suppose $\ss _1$ is a $\text{spin}^c$-structure on $X_1$, such that for any $\text{spin}^c$-structure $\ss_2$ of $X_2$, 
\[
SW(X_1, \ss_1) \neq SW(X_2, \ss_2).
\]
Then $X_1\# (\S^1\times \S^3)$ is not diffeomorphic to $X_2\# (\S^1\times \S^3)$.
\end{theorem}
\begin{proof}
Let $\ss_i'$ be the $\text{spin}^c$-structure of $X_i\# (\S^1\times \S^3)$ such that $\ss_i'$ coincides with $\ss_i$ on the common part. Then by Remark \ref{rem:high-dim-moduli}, 
\[
SW^\Theta(X_1\# (\S^1\times \S^3), \ss_1') \neq SW^\Theta(X_2\# (\S^1\times \S^3), \ss_2').
\]
If there exists a diffeomorphism $f:X_1\to X_2$, by Lemma \ref{lem:detect-exotic}, we have
\[
SW^\Theta(X_1\# (\S^1\times \S^3), \ss_1') = SW^\Theta(X_2\# (\S^1\times \S^3), f(\ss_1')).
\]
Since $H^2(X_2;\Z)\cong H^2(X_2\# (\S^1\times \S^3);\Z)$, there exists a $\text{spin}^c$-structure $\ss_2$ on $X_2$ such that $f(\ss_1') = \ss_2'$. This contradicts the inequality.
\end{proof}

Therefore, we have a lot of exotic nonsimply connected manifolds, for example:
\begin{corollary}
Suppose that $b^+(X) > 1$ and $\pi_1(X) = 1 =\pi_1(X - T )$ where $T$ is a homologically nontrivial torus of self-intersection $0$. Suppose that there exists a $\text{spin}^c$-structure $\ss$ on $X$ such that $SW(X,\ss) \neq 0$. Then $X\# S^1 \times S^3$ admits infinitely many exotic smooth structures. In particular, for the elliptic surface $E(n)$ with $n> 1$, the nonsimply connected manifold $E(n) \# S^1 \times S^3$ admits infinitely many exotic smooth structures.
\end{corollary}
\begin{proof}
For such $X$, Fintushel-Stern knot surgery theorem (see \cite{fintushel1997knotslinks4manifolds}, as well as their lecture notes \cite{fintushel2007lectures4manifolds} Lecture 3) says for any knot $K\subset \S^3$, there exists a manifold $X_K$ homeomorphic to $X$ and 
\[
\max_{n\in \Z}\{SW(X_K,\ss+n[T])\}
\]
depends on the largest coefficient of the Alexander polynomial of $K$. Any symmetric Laurent polynomial whose coefficient sum is $\pm 1$ is the Alexander polynomial of some knot. Hence the set 
\[
\{\max_{n\in \Z}\{SW(X_K,\ss+n[T])\}, K \text{ is a knot in }\S^3\}
\]
is infinite, and therefore we have an infinite family of manifolds that are homeomorphic to $X$ and satisfy the conditions of Theorem \ref{thm:not-diffeomorphic}.
%For elliptic surfaces we have another way to produceFor a fixed $n\ge 1$, denote the logarithmic transform of order $r\ge1$ performed on $E(n)$ by $E(n,r)$. Each $E(n,r)$ is homeomorphic to $E(n)$, and the maximal Seiberg-Witten invariant is diff
\end{proof}

\section{Setup for the family $1$-surgery formula}
Let $X$ be a compact, smooth, oriented $4$-manifold. Let's consider a family of SW equations for $X$ and the resulting moduli space. We want to vary all stuff that SW equations depends on.  Recall that the parameters of SW equations include the metric of $X$, the $\text{spin}^c$ structure of $X$, and the perturbations of the equation.

Assume we have a compact topological space $B$ of parameters for SW equations. Since isomorphism classes of $\text{spin}^c$ structures are discrete, the definition of a continuous parameter family should be a continuous map
\begin{equation}\label{equ:parameterfamily}
B \to \Pi (X):=\bigsqcup_{g\in Met(X)} L_{k-1}^2(\Lambda_g^+(X)).
\end{equation}

Imagine a classical example: $B = \S^1$. Now we have a $1$-parameter family of $X$, but in order to obtain an invariant of diffeomorphisms, we should be able to glue two ends of $X\times [0,1]$ nontrivially. So instead of $X\times B$,  the object we are considering would be a fiber bundle $E_X$ over $B$ with each fiber $F=X$. The next step is to find out a suitable structure group. To define SW equations on each fiber, one should fix an orientation and an isomorphism class of the $\text{spin}^c$ structure for each fiber. Typically one choose such data on each local trivialization of $E_X$, and glue them compatibly. So the structure group should preserve the orientation and the isomorphism class of the $\text{spin}^c$ structure. Optionally, one can also require the structure group preserve the homology orientation. Let $\mathfrak s$ be an isomorphism class of $\text{spin}^\C$ structure of $X$. 
Let $\mathcal{O}$ be a homology orientation of $X$, which is an orientation of the vector space $H^1(X;\R) \oplus H^{2,+}(X;\R)$. 
In \cite{Konno2018CharacteristicCV},
\[
\text{Diff}(X,\mathfrak s):= \{f\in \text{Diff}^+(X)|f^*\mathfrak s = \mathfrak s\}
\]
is the group of orientation-preserving and $\text{spin}^\C$-structure-preserving diffeomorphisms, and 
\[
    \text{Diff}(X,\mathfrak s, \mathcal{O}):= \{f\in \text{Diff}(X,\mathfrak s)|f^*\mathcal{O} = \mathcal{O}\}
\]
is the group of diffeomorphisms that preserve the homology orientation of $X$ in addition to $\text{Diff}(X,\mathfrak s)$. Now the object in our consideration is the following bundle:
\[
\xymatrix{
X \ar[r] & E \ar[d]\\
&B 
}
\]
with structure group either $G = \text{Diff}(X,\mathfrak s)$ or $G = \text{Diff}(X,\mathfrak s, \mathcal{O})$. 

Now since the family of $X$ is a nontrivial fiber bundle $E$, the definition of the parameter family (\ref{equ:parameterfamily}) should be updated to 
\[
B \to \Pi(E) := \bigsqcup_{b\in B} \Pi(E_b).
\]
Under this setting, one has to take care of $\text{spin}^c$ structures and the gauge group to define a parametrized moduli space (see subsection \ref{subsection:SpinGL} and subsection \ref{subsection:Parametrized}). 

In the subsection \ref{subsection:proofUnparemetrized}, we saw that the moduli space of monopoles of $X$ is a circle that winds around the moduli space of irreducible configurations $\mathcal{B}^*_X = \C P^\infty\times \S^1$, and to measure how many turns it has wrapped, we used $\C P^\infty\times \{0\}$ to cut it. In family case, the parametrized moduli space of irreducible configrations of is a fiber bundle $  \mathcal{F}\mathcal{B}^* $: 
\[
\xymatrix{
\C P^\infty\times \S^1 \ar[r] &  \mathcal{F}\mathcal{B}^* \ar[d]\\
&B 
}
\]
It turns that the ``winding number'' of the $1$-dimensional moduli space is still a useful invariant to characterize it. To cut down this moduli space, we may still use a codimension $1$ submanifold, or equivalently, a $1$-dimensional cohomology class $\Theta$ of $\mathcal{F}\mathcal{B}^*$. But the existence of such class is not guranteed. We will discuss this issue in the second part of the subsection \ref{subsection:Parametrized}.

We want to do surgery on $E_X$ fiberwise and get a new fiber bundle $X'\to E_{X'} \to B$ with each fiber $F=X'$. So we have to assume a subbundle $E_{\S^1}$ of $E_X$, and some infomation of the family of framings to perform the surgery. We investigate this infomation in subsection \ref{subsection:familysurgery}.

By the same reason as section 1, we can extend the $\text{spin}^\C$ structure $\mathfrak s$ to $X'$ fiberwise. Suppose $E_{X'}$ is a fiber bundle over the base space $B\subset \Pi (X')$. Let $E_{X_0}$, $E_{}$

We assume $B$ is connected to avoid different framings since the framing of surgery on each fiber is an element in $\pi_1SO(3)=\Z_2$.

%Under such assumption, the parametrized moduli space of irreducible configrations of is a fiber bundle $ \mathcal{B}$: 
%\[
%\xymatrix{
%\C P^\infty\times \S^1 \ar[r] & \mathcal{B} \ar[d]\\
%&B 
%}
%\]

%To cut down the dimension of moduli space, we may still use a $1$-dimensional cohomology class $\Theta$ of $ \mathcal{B}$. But now $ \mathcal{B}$ may have trivial $1$-dimensional cohomology group, even if we use $\Z_2$ coefficient.

\subsection{$\text{Spin}_{\text{GL}}^c$ structure}\label{subsection:SpinGL}
For generality, it's better to assume that the metric of $X$ is unfixed over $B$. But recall that in \autoref{subsection:Spin}, the definition of a $\text{Spin}^\C$ structure depends on the choice of a metric on $X$, since 
\[
 \text{Spin}^\C(4)  := S^1 \times  \text{Spin}(4) / \{\pm (1,I)\},
 \]
where $ \text{Spin}(4)$ is a double cover of $SO(4)$. Konno\cite{konno2019positive} developed an approach to avoid the dependence on metrics, which is $\text{spin}_{\GL}^c$ structure. Choose a nontrivial (in the sense of Remark \ref{remark:spin}) double cover \[\widetilde{\GL_4^+(\R)}\] of $\GL_4^+(\R)$, and define
\[
 \text{Spin}^c_{\GL}(4)  = S^1 \times \widetilde{\GL_4^+(\R)} / \{\pm (1,I)\},
 \]
 then there is a covering map $\rho:\text{Spin}^c_{\GL}(4) \to \GL_4^+(\R)$, just as (\ref{equ:rhoc}).
 
As the definition of $\text{spin}^c$ structure (Definition \ref{def:spinc}), we can define a $\text{spin}^c_\GL$ structure $s$ on $X$ to be a principal $\text{Spin}^c_{\GL}(4)$-bundle 
$P_{\text{Spin}^c_{\GL}(4)} \to X$, with a bundle map $\psi$ from $P_{\text{Spin}^c_{\GL}(4)}$ to the frame bundle $P_{\GL_4^+(\R)}$ of $X$, which restricts
to the obvious covering map $\rho$ on each fiber.

As before, $\rho$ is an $\S^1$-fibration. Thus $P_{\text{Spin}^c_{\GL}(4)}$ admits some freedom over $P_{\GL_4^+(\R)}$. Let $\ss$ be the isomorphism class of $s$. Denote the automorphism group of the principal $\text{Spin}^c_{\GL}(4)$-bundle 
$P_{\text{Spin}^c_{\GL}(4)} \to X$ by
\begin{equation}\label{equ:defAut}
\Aut(X, s) = \{ (f,\tilde{f}); f\in \text{Diff}(X,\mathfrak s), \tilde{f}:  P_{\text{Spin}^c_{\GL}(4)} \stackrel{\cong}{\to} P_{\text{Spin}^c_{\GL}(4)} ,\psi \circ  \tilde{f}= df \circ \psi \},
\end{equation}
where
\[
df : P_{\GL_4^+(\R)}   \stackrel{\cong}{\to}  P_{\GL_4^+(\R)}
\]
is the isomorphism of the frame bundle induced by $f$. Then we have an exact sequence:
\begin{equation}\label{equ:gaugeAuto}
0\to \mathscr{G}_s \to \Aut(X, s) \to  \text{Diff}(X,\mathfrak s),
\end{equation}
where $\mathscr{G}_s$ is the gauge group of $s$, which is isomorphic to $\Map(X,\S^1)$. 

We also define 
\[
\Aut(X, s,\mathcal{O}) := \{ (f,\tilde{f}) \in \Aut(X, s); f\in \text{Diff}(X,\mathfrak s, \mathcal{O})\}.
\]
When the structure group of $E$ is $G = \text{Diff}(X,\mathfrak s)$ or $ \text{Diff}(X,\mathfrak s, \mathcal{O})$, we define $\tilde{G}$ to be $\Aut(X, s) $ or $\Aut(X, s,\mathcal{O})$, respectively.

Note that, given a $\text{spin}_{\text{GL}}^c$ structure $s$, each metric $g$ induces a $\text{spin}^c$ structure $s_g$, and for two metrics $g_1$ and $g_2$, $s_{g_1}$ and $s_{g_2}$ belong to the same isomorphism class of $\text{spin}^c$ structures. Moreover, it's easy to see from the definition that, the determinant line bundle of $s$ is the same as the determinant line bundle of $s_g$ for any metric $g$.

\subsection{Parametrized moduli space}\label{subsection:Parametrized}
In this subsection, we consider the parametrized moduli space of monopoles and the parametrized moduli space of configurations. 
To define the parametrized moduli space of monopoles, we just review \cite{Konno2018CharacteristicCV} section 4, which is based on Ruberman's observations \cite{Ruberman1998AnOT} and Nakamura's ideas \cite{NAKAMURA10}. The second task of this subsection is about configurations. Other than \cite{Konno2018CharacteristicCV} and \cite{BK2020}, where they consider $0$-dimensional moduli space of monopoles, we have to consider $1$-dimensional moduli space of monopoles, so we need to investigate the topological infomation of the parametrized moduli space of configurations to find out a cut-down cohomology class.

For a $4$-manifold $X$, we fix a $\text{spin}^c_{GL}$ structure $s$ through out this sunsection. Define
\[
\Pi(X) := \bigsqcup_{g\in \text{Met}(X)} L^2_{k-1}(\Lambda_g^+(X))
\]
to be the space of perturbations. For each $(g,\eta) \in \Pi(X)$, define
\begin{align*}
    \mathcal{C}{(s, g,\eta)} &:= L^{k,2}(\mathscr{A}(s)) \times L^{k,2}(S_g^+)\\
    \mathcal{D}{(s,g,\eta)} &:= L^{k-1,2}(i\Lambda^+(X)) \times L^{k-1,2}(S_g^-),
\end{align*}
where $ \mathscr{A}(s)$ is the space of $U(1)$-connections of the determinant line bundle of the $\text{spin}^c_{GL}$ structure $s$, and $S_g^\pm$ is the positive or negative spinor bundle of the $\text{spin}^c$ structure $s_g$ induced from $s$. Recall that the definition of the spinor bundle (\cite{salamon2000spin} page 154): given a $\text{spin}^c$ structure $s$,
\[
S:= P_{\text{Spin}^c }\times_{\Gamma_0} \C^{2^n},
\]
where $ P_{\text{Spin}^c }$ is the principle $\text{Spin}^c$-bundle of $s$, $\Gamma_0: \text{Spin}^c(\R^{2n}) \to \End(\C^{2^n})$ is a representation, and $n=2$. 

Define
\begin{align}\label{equ:defsw}
    sw_{(s,g,\eta)} : \mathcal{C}{(s,g,\eta)} &\to \mathcal{D}{(s,g,\eta)} \\
    (A,\phi) &\mapsto (F^{+_g}_A + i\eta - \rho^{-1}(\sigma(\Phi, \Phi)), \D_A \Phi)
\end{align}
where $\rho: \Lambda^+_g (X) \to \mathfrak{su}(S^+_g)$ is the map defined by the Clifford multiplication, $\sigma$
is the quadratic form given by $\sigma(\Phi, \Phi) = \Phi \otimes \Phi^* - \frac{1}{2} |\Phi|^2 id$.

For a fiber bundle $X \to E \to B$, define
\[
\Pi(E) := \bigsqcup_{b\in B} \Pi(E_b)
\]
be the family of parameters. Given a section $\sigma$ of the bundle $\Pi(E) \to B$, we can define the parametrized moduli space of monoples as follows:

We can choose an open cover $\{U_\alpha\}_\alpha$ of $B$ such that $U_{\alpha\beta} := U_\alpha \cap U_\beta$ is contractible and $E_X|_{U_\alpha}$ is trivial for each $\alpha$ and $\beta$. Note that $\sigma$ is a system of maps $\{\sigma_\alpha\}_\alpha$ where $\sigma_\alpha: U_\alpha \to \Pi (X)$ (it looks like (\ref{equ:parameterfamily}) since they are the cases where the bundle is trivial). For each pair of parameter, the moduli space of monopoles is
\[
\M(s,g,\eta) := sw_{(s,g,\eta)}^{-1}(0)/\mathscr{G}.
\]
Then for each $U_\alpha$ of $B$, the parametrized moduli space of monopoles is
\[
\FM_{U_\alpha}(s,\sigma_\alpha) :=\bigsqcup_{b\in U_\alpha} \M(s,\sigma_\alpha(b)).
\]
Next we show one can glue them compatibly. Suppose $\{g_{\alpha\beta}: U_{\alpha\beta} \to G\}$ is a family of transition functions for $E_X$ corresponding to the open cover $\{U_\alpha\}_\alpha$. Note that for each $b\in U_{\alpha\beta}$, $g_{\alpha\beta}(b)$ is a diffeomorphism of $X$ and induces a map $g_{\alpha\beta}(b)^*$ from $\Pi (X)$ to $\Pi (X)$. Then there is a relation
\[
\sigma_\alpha(b) = g_{\alpha\beta}(b)^* \circ \sigma_\beta(b).
\]
On the other hand, since each $U_{\alpha\beta}$ is contractible, one can choose a lift $\tilde{g}_{\alpha\beta} : U_{\alpha\beta} \to \tilde{G}$ of ${g}_{\alpha\beta}$, and Ruberman  \cite{Ruberman1998AnOT}\cite{ruberman1999polynomial}\cite{ruberman2002positive} observes that $\tilde{g}_{\alpha\beta} (b)$ induces an invertible map
\[
\tilde{g}_{\alpha\beta} (b)^*: \M(\sigma_\beta(b)) \to \M(g_{\alpha\beta}(b)^* \circ \sigma_\beta(b)) = \M(\sigma_\alpha(b)).
\]
Thus we have an invertible map 
\begin{equation}\label{equ:cocycleModGauge}
\tilde{g}_{\alpha\beta}^*: \FM_{U_\beta}(\sigma_\beta)|U_{\alpha\beta} \to \FM_{U_\alpha}(\sigma_\alpha) |U_{\alpha\beta}.
\end{equation}
(\ref{equ:cocycleModGauge}) satisfies ``cocycle condition modulo gauge'': 
\[
\tilde{g}_{\alpha\beta}^* \circ \tilde{g}_{\beta\gamma}^* \circ \tilde{g}_{\gamma\alpha}^*
\]
is a lift of ${g}_{\alpha\beta}^* \circ {g}_{\beta\gamma}^* \circ {g}_{\gamma\alpha}^* = id$, and thus it is in $\mathscr{G}$ by the exact sequence (\ref{equ:gaugeAuto}). Therefore we can define the parametrized moduli space of monopoles by
\begin{equation}\label{equ:parametrized}
\FM(\sigma) = \FM(s,\sigma) := \bigsqcup_{\alpha} \FM_{U_\alpha}(s,\sigma_\alpha) /\sim,
\end{equation}
where $\sim$ is the relation given by the system of maps $\{\tilde{g}_{\alpha\beta}^*\}_{\alpha\beta}$.

Now we consider the parametrized moduli space of configurations. Let's reconsider above cocycle conditions. $\tilde{g}_{\alpha\beta} (b)^*$ mentioned above is actually induced from the invertible map
\begin{equation}\label{equ:configCocycle}
\tilde{g}_{\alpha\beta} (b)^*: \mathcal{C}(\sigma_\beta(b)) \to \mathcal{C}(g_{\alpha\beta}(b)^* \circ \sigma_\beta(b)) = \mathcal{C}(\sigma_\alpha(b)).
\end{equation}
Here $\tilde{g}_{\alpha\beta}(b)$ is a lift of ${g}_{\alpha\beta}(b)\in {G}$ in $\tilde{G} \subset \Aut(X, s)$. The lift is unique up to the action of an element in $\mathscr{G}$ by the exact sequence (\ref{equ:gaugeAuto}). Thus the map (\ref{equ:configCocycle}) induces a well defined map 
\begin{equation}\label{equ:moduliCocycle}
\tilde{g}_{\alpha\beta} (b)^*: \mathcal{B}^*(\sigma_\beta(b)) \to \mathcal{B}^*(g_{\alpha\beta}(b)^* \circ \sigma_\beta(b))=\mathcal{B}^*(\sigma_\alpha(b))
\end{equation}
between the moduli spaces of irreducible configurations. Recall that when $b_1(X) = 1$, the moduli space of irreducible configurations is $\mathcal{B}^*_{X} \simeq \C P^\infty\times \S^1$. The $\S^1$-factor is $H^1(X;\R)/H^1(X;\Z)$, the space of $U(1)$-connections modulo the action of the gauge group. Note that ${g}_{\alpha\beta} (b)$ is an orientation preserving diffeomorphism, so ${g}_{\alpha\beta}(b)^*$ acts on $H^1(X;\R)=\R$ identically. The gauge group acts on $H^1(X;\R)=\R$ by translation of integers. So $\tilde{g}_{\alpha\beta}(b)^*$ acts on $H^1(X;\R)/H^1(X;\Z) = \R /\Z$ identically. 
(Geometrically, $H^1(X;\R)$ corresponds to $U(1)$-connections on the nontrivial loop, which is the rotation on the fiber $U(1)$ around the loop. A diffeomorphism of $X$ should preserve the rotation, and an automorphism of the principal $\text{Spin}^c_{\GL}(4)$-bundle would accelerate the rotation by some integer.)

Now we know the transition map $\tilde{g}_{\alpha\beta} (b)^*$ of the $\S^1$-factor is always identity, and the transition map of the $\C P^\infty$-factor is independent of the $\S^1$-factor, in other word
\[
\tilde{g}_{\alpha\beta} (b)^* = id_{\S^1} \times \tilde{g}_{\alpha\beta} (b)^*|_{\C P^\infty}.
\] 
Hence the parametrized moduli space of irreducible configurations defined by similar formula as (\ref{equ:parametrized}) is
\begin{equation}\label{equ:FB}
\mathcal{F}\mathcal{B}^*_{X}  \simeq \S^1 \times E_{\C P^\infty}
\end{equation}
for some $\C P^\infty$-bundle $E_{\C P^\infty} \to X$. Let 
\[
\Theta = PD([E_{\C P^\infty}]) \in H^1(\mathcal{F}\mathcal{B}^*_{X}),
\]
then on each unparametrized moduli space of irreducible configurations, $\Theta$ restricts to the cohomology class in $H^1(\mathcal{B}^*_{X})$ we choose in the subsection \ref{subsection:proofUnparemetrized}. We will use $\Theta $ to cutdown the $1$-dimensional parametrized moduli space of monopoles later.

For $b_1(X)>1$, in general one cannot expext a trivial $\S^1 $-bundle as (\ref{equ:FB}), and a suitable cutdown class might not exist. Here is an example:

\begin{example}
Let $X = T^2\times \S^2$. Let $\phi$ be a diffeomorphism on $T^2$ such that
\begin{align*}
\phi^*: H^1(T^2) &\to H^1(T^2) \\
(a,b) &\mapsto (a,a+b).
\end{align*}
Let $E_X$ be the mapping torus of $\phi\times id_{\S^2}$. It's easy to check that the diffeomorphism of 
\[
H^1(X;\R)/H^1(X;\Z)=T^2\]
induced by $\phi$ is again $\phi$. Hence the parametrized moduli space of irreducible configurations $\mathcal{F}\mathcal{B}^*_{X}$ for $E_X$ is the inner product of two bundles over $\S^1$:
\[
\mathcal{F}\mathcal{B}^*_{X} = E_{T^2} \otimes E_{\C P^\infty},
\]
where $\pi : E_{T^2} \to \S^1$ is the mapping torus of $\phi$. Let $m$ be a loop in $T^2 = \pi^{-1}(1)$ such that 
\begin{align*}
\langle (1,0) , [m]\rangle &= 1\\
\langle (0,1) , [m]\rangle &= 0.
\end{align*}
If $m\times \{1\} \subset E_{T^2}$ is the $1$-dimensional parametrized moduli space of monopoles, then it is homologically trivial (let $[l]$ be another generator of $H_1(T^2)$, then $[l\times \{1\}] = [l\times \{1\}] + [m\times \{1\}] \in H_1( E_{T^2} )$ since $E_{T^2} \to \S^1$ is the mapping torus of $\phi$), so we cannot find any cohomology class of $\mathcal{F}\mathcal{B}^*_{X}$ to cut it down.
\end{example}

\subsection{Construction of family $1$-surgery}\label{subsection:familysurgery}
%%%%%%%%%%%%%%%%%%
%TODO

To define the family of surgery, one have to define the family of $4$-manifolds with a specified circle in each of them. Namely, 
one have to define the bundle with fiber a $4$-manifold and some nice conditions, such that it's possible to do family surgery. 

To define the family of $1$-surgery, we first examine the case of $0$-surgery, which is the connected sum. 
The definition of the family of connected sum in \cite{BK2020} Section 4 is based on the following setting: 
Let $M$ and $N$ be two compact, smooth, oriented $4$-manifolds. Let the parameter space $B$ be a compact smooth manifold of dimension $d$.
Consider the bundles $\pi_M: E_M \to B$ and $\pi_N: E_N \to B$ with fiber $M$ and $N$ respectively. Let $D\pi_M: TE_M \to TB$ and $D\pi_N: TE_N \to TB$ be the 
differentials of $\pi_M$ and $\pi_N$ respectively. Let $T(E_M/B)=\ker D\pi_M$ and $T(E_N/B)=\ker D\pi_N$ be the vertical tangent bundles of $E_M$ and $E_N$ respectively. 
To define a connected sum family, \cite{BK2020} assumes the following data:
\begin{enumerate}
    \item Two sections $s_M: B \to E_M$ and $s_N: B \to E_N$. (Hence $s_M^*(T(E_M/B))$ and $s_N^*(T(E_N/B))$ are rank $4$ vector bundles over $B$.)
    \item An orientation reversing diffeomorphism of bundles $\phi: s_M^*(T(E_M/B)) \to s_N^*(T(E_N/B))$.
    \label{item:orientation}
\end{enumerate}
Then the connected sum family can be obtained by the connected sum around $s_M(b)$ and $s_N(b)$ for each point $b \in B$. 
The requirement for the structure group $G$ to be the group of orientation preserving diffeomorphisms, 
is to ensure that the $s_M^*(T(E_M/B))$ and $s_N^*(T(E_N/B))$ are orientable, 
and this works with data \ref{item:orientation} to ensure that the family of framing is well-defined. 
So the requirement for a family of framing for $0$-surgery is actually
\begin{enumerate}
    \item[(b')] $s_M^*(T(E_M/B))$ and $s_N^*(T(E_N/B))$ are orientable, and thay are diffeomorphic as oriented vector bundles over $B$.
    \label{item:0framing}
\end{enumerate}
Let's consider the following non example:

\begin{example}
Let $B = \S^1$ and $M$ be any compact smooth oriented $4$-manifold that admits an orientation-reversing diffeomorphism $f$. Let $E_M$ be the mapping torus of $f$. 
Assume in addition that $f(m) = m$ for a point $m\in M$. Let $s_M$ be the section of $E_M$ with fixed value $m$.
Let $N$ be any compact smooth oriented $4$-manifold and $E_N$ be a trivial bundle over $B$. Let $s_N$ be the zero section of $E_N$. 
One can certainly form the connected sum $M \# N$ around $s_M(b)$ and $s_N(b)$ fiberwise. However, it's impossible to choose a continuous family of framing 
for the image of $s_M$ in their virtical tangent space (A disk around $m$ is removed and the boundary of the rest is stretched to be glued with the punctured $N$. 
When $M$ goes around $B = \S^1$, this boundary has to be stretched to both directions of $M$ at some point of $B$). 
The requirement (b') prevents this situation from happening.
\end{example}

Recall that, the framing of a surgery is an identification between the trivial bundle and the normal bundle over the attaching sphere.
For $0$-surgery, the framing is determined by a choice of orientations of $M$ and $N$, 
so the family of framing is given by a family of orientations of $M\sqcup N$, 
which is equivalent to orientations of $s_M^*(T(E_M/B))$ and $s_N^*(T(E_N/B))$. 

But to construct a family of $1$-surgery, a requirement similar to (b') is not enough. Consider the following example:

\begin{example}
Let $B = \S^1$ and $X =  \S^1 \times \S^3$. Let $f$ be the Dehn twits around the $\S^1$ factor of $X$:
\begin{align*}
    f: X &\to X \\
    (s , (z, w)) &\mapsto (s , (z, sw)),
\end{align*}
where $s \in \C$, $|s|^2 = 1$, $(z, w) \in \C^2$, and $|z|^2+ |w|^2 = 1$. Let $E_X$ be the mapping torus of $f$. 
Let $E_{\S^1}$ be the subbundle of $E_X$ with fiber $\S^1 \times \{(0,0)\}$. The normal bundle of $\S^1 \times \{(0,0)\}$ in each fiber of $E_X$ is orientable, 
and the vertical normal bundle of $E_{\S^1}$ is orientable, 
but the framing of $\S^1 \times \{(0,0)\}$ is changed when it goes around $B$. 
So it's impossible to choose a continuous family of framing for $1$-surgery around $E_{\S^1}$.
\end{example}

Thus we need more infomation to specify the family of framing. 
We are trying to find a family of identifications between trivial bundles and the normal bundles over the attaching sphere $\S^1$. 
Actullay we can do this piecewisely and then glue them together:

Suppose $\{U_\alpha\}_\alpha$ is an open cover of $B$ such that $E_X|_{U_\alpha}$ and $E_{\S^1}|_{U_\alpha}$ are trivial for each $\alpha$. 
Suppose $\{g_{\alpha\beta}: U_{\alpha\beta} \to G\}$ is a family of transition functions for $E_X$ corresponding to the open cover $\{U_\alpha\}_\alpha$.
\begin{enumerate}[(i)]
    \item An $\S^1$-bundle $E_{\S^1}$ as a subbundle of $E_X$, with an embedding $i: E_{\S^1} \to E_X$. 
    (Hence the virtical tangent bundle $T(E_{\S^1}/B)$ is a rank $1$ vector bundle over $E_{\S^1}$, 
    and $i^*T(E_X/B)$ is a rank $4$ vector bundle over $E_{\S^1}$.)
    \item \label{item:1framing} an identification $f_\alpha : U_\alpha \times \S^1 \times \R^3 \to (i^*T(E_X/B)) / T(E_{\S^1}/B) |_{(E_{\S^1}|_{U_\alpha})}$ for each $\alpha$, 
    such that for any $b\in U_{\alpha\beta}$, 
    the following diagram commutes up to an isotopy of bundle isometries:
    \[
        \xymatrix{
            (i^*T(E_X/B)) / T(E_{\S^1}/B) |_{\{b\}\times\S^1}  & \\
            \{b\}\times\S^1\times \R^3 \ar[u]^{f_\alpha(b,-)} \ar[r]_{f_\beta(b,-)} & (i^*T(E_X/B)) / T(E_{\S^1}/B) |_{\{b\}\times\S^1 \ar[ul]^{g_{\alpha\beta}(b)^*}}
        }
    \]
    (This means that the difference between $f_\alpha$ and $g_{\alpha\beta}(b)^* \cdot f_\beta$, 
    regarded as a map from $\S^1$ to $GL(3,\R)$, is $1 \in \pi_1(GL(3,\R))$.)
    % An identification of $(\sigma^*(T(E_X/B))) / T(E_{\S^1}/B) \to E_{\S^1}$ with the trivial bundle $E_{\S^1} \times \R^3$.
    % Two linear independent normal vector fields $v_1, v_2:E_{\S^1} \to  \sigma^*(T(E_X/B)) $, i.e, Two linear independent sections of 
    % the vector bundle $(\sigma^*(T(E_X/B))) / T(E_{\S^1}/B) \to E_{\S^1}$.
\end{enumerate}
\begin{remark}
When $1$-surgery is repalced by $0$-surgery, $\S^1$ and $\R^3$ in data \ref{item:1framing} are replaced by $\S^0$ and $\R^4$, and data \ref{item:1framing} degenerates to data (b'), since orientation is obtained by gluing such data piecewisely.
\end{remark}

With these data, we can construct the family of $1$-surgery around $E_{\S^1}$ as follows:

Let $\nu = (i^*T(E_X/B)) / T(E_{\S^1}/B)$ be the normal bundle of $E_{\S^1}$ in $E_X$. Then $\nu$ is a real rank $3$ vector bundle over $E_{\S^1}$. 
Fix a family of metric on each fiber $X$ of $E_X$, then it induces a metric on the bundle $\nu$. 
Let $D(\nu) \to E_{\S^1}$ and $S(\nu) \to E_{\S^1}$ be the unit open disc bundle and the unit sphere bundle of $\nu$ respectively. 
Let $N = E_{\S^1 \times D^3}$ be the tubular neighborhood of $E_{\S^1}$ in $E_X$, equipped with a diffeomorphism
\[
    e: D(\nu) \to N.
\]
Then $E_{X_0} = E_X \setminus N$ is a bundle over $B$ with fiber $X_0 = X \setminus (\S^1 \times D^3)$ and boundary $S(\nu)$. 
By attaching the family of cylinders $S(\nu) \times [0, \infty)$ to $E_{X_0}$ along $ S(\nu) \times \{0\}$, 
we obtain a bundle $E_{\hat{X}}$ over $B$ with fiber the cylindrical end $4$-manifold $\hat{X} = X_0 \cup_{\S^1 \times\S^2}  \S^1 \times \S^2  \times [0, \infty)$. 

Now we prepare the other side, which is a bundle $E_{D^2 \times \S^2}$ with base space $B$ and fiber $D^2 \times \S^2$. The transition map of this bundle is the extension of
\[
    f_\alpha ^{-1} \cdot g_{\alpha\beta}(b)^* \cdot f_\beta : \S^1\times \S^2 \to \S^1\times \S^2
\]
to $D^2 \times \S^2$. This extension is possible because the condition \ref{item:1framing} ensures that $ f_\alpha ^{-1} \cdot g_{\alpha\beta}(b)^* \cdot f_\beta$
is smoothly isotopic to the identity map of $\S^2$-bundle. Similarly, we regard the fiber of $E_{D^2 \times \S^2}$ as a manifold with cylindrical end 
$\S^1 \times \S^2 \times (-\infty, 0]$.

Recall that in unfamily case, we glue $\hat{X} = X_0 \cup_{\S^1 \times\S^2}  \S^1 \times \S^2  \times [0, \infty)$ and 
$D^2 \times \S^2 \cup_{\S^1 \times\S^2}\S^1 \times \S^2 \times (-\infty, 0]$ along their neck to produce a closed $4$-manifold $X'(r)$ with a length $r$ neck. 
We can carry out the same procedure to the family case, and obtain a bundle $E_{X'(r)}$ over $B$ with fiber $X'(r)$. Similarly, $E_{X(r)}$ is topologically the bundle $E_{X}$ but now each fiber $X = X_0 \cup_{\S^1 \times\S^2}  \S^1 \times D^3$ has a length $r$ neck. 

Now we consider the metric. On each fiber $X_0$ of $E_{\hat{X_0}}$, the metric is the product metric of the metric on $X_0$ and the standard metric on $[0, \infty)$...
%TODO

%%%%%%%%%%%%%%%%%%%%%%%%

%%%%%%%%%%%%%%%%%%%%%%%%%%%%%%%%%%%%%%%%%%%%%%%%%%%%%%%%%%%%%%%%%%%%%%%%%begin
\subsection{Parameter family}

Let
\begin{align}
    \label{equ:configOnG}\mathcal{C}{(g,\eta)} &= L^{k,2}(\mathscr{A}(s)) \oplus L^{k,2}(S_g^+)\\
    \mathcal{D}{(g,\eta)} &= L^{k-1,2}(i\Lambda^0(X)) \oplus L^{k-1,2}(i\Lambda^+(X)) \oplus L^{k-1,2}(S_g^-)\\
     \mathcal{F}_{(g,\eta)}: \mathcal{C}{(g,\eta)} &\to \mathcal{D}{(g,\eta)}\\
    \label{def:F}\mathcal{F}_{(g,\eta)}\begin{pmatrix}
         A\\
         \Phi\\
    \end{pmatrix}
    &= \begin{pmatrix}
         d^*(A-A_0)\\
         F^{+_g}_A + i\eta - \rho^{-1}(\sigma(\Phi, \Phi))\\
          \D_A \Phi\\
    \end{pmatrix}
\end{align}
This operator integrates a part of the action by the gauge group (see (\ref{equ:L*})) and the Seiberg-Witten map. The reason to do this is to make the kernel have finite dimension and compute the index. The differential of $\mathcal{F}_{(g,\eta)}$ at $(A,\Phi)$ is a linear operator with some zeroth order perturbations:
\begin{equation}\label{equ:Dlinearization}
d_{(A,\Phi)}\mathcal{F}_{(g,\eta)}\begin{pmatrix}
         \alpha\\
         \phi\\
    \end{pmatrix}
    =\begin{pmatrix}
         d^*\alpha\\
         d^+ \alpha\\
         \D_A \phi\\
    \end{pmatrix} 
    + \begin{pmatrix}
         0\\
          -\rho^{-1}(\sigma(\Phi, \phi)+\sigma(\phi, \Phi))\\
         \Gamma(\alpha)\Phi\\
    \end{pmatrix} 
\end{equation}
and it is a Fredholm operator with index
\[
\ind \D_A + b_1 - 1 - b^+.
\]
Here $b_1$ comes from the kernel of the operator $\mathcal{D}^+ := d^*\oplus d^+$, which is $H^1(X,i\R)$ (see \cite{salamon2000spin} section 8.4 for a proof). 
$- 1 - b^+$ comes from $\text{coker }d^*=H^0(X,i\R)$, i.e. the space of constant functions in $L^{k-1,2}(X,i\R)$, and $\text{coker }d^+=H^+(X,i\R)$ (see \cite{salamon2000spin} Propotion 7.10). 

To apply the implicit function theorem, our goal is to minimize the cokernel of $d_{(A,\Phi)}\mathcal{F}_{(g,\eta)}$. Notice that in (\ref{equ:Dlinearization}) the zeroth order perturbation of $d_{(A,\Phi)}\mathcal{F}_{(g,\eta)}$ in $L^{k-1,2}(i\Lambda^0(X))$ is zero, so the cokernel of $d_{(A,\Phi)}\mathcal{F}_{(g,\eta)}$ must contain $H^0(X,i\R)$ and is at least $1$-dimensional. 
To fix this issue, one can narrow down the first summand of $\mathcal{D}(g,\eta)$ to 
\[
L^{k-1,2}_0(i\Lambda^0_0(X;g)),
\]
the space of functions in $L^{k-1,2}(X,i\R)$ which has mean value zero with respect to the metric $g$. We still denote the operator defined in (\ref{def:F}) with new target by 
\begin{equation}\label{def:F0}
\mathcal{F}_{(g,\eta)}:\mathcal{C}{(g,\eta)} \to L^{k-1,2}_0(i\Lambda^0_0(X;g))\oplus L^{k-1,2}(i\Lambda^+(X)) \oplus L^{k-1,2}(S_g^-). 
\end{equation}

Now, if the cokernel of $d_{(A,\Phi)}\mathcal{F}_{(g,\eta)}$ is $0$-dimensional, we call the perturbation $\eta$ \textbf{regular}. If it is the case, by (\ref{equ:Dlinearization}), $\Phi$ must be nontrivial to cut down the cokernel of $d^+$ and $\D_A$, and by the implicit function theorem, the zero set of  $\mathcal{F}_{(g,\eta)}$ is a smooth manifold of irreducibles
\[
\widetilde{\M}^*(g,\eta)
\] 
with dimension $\ind \D_A + b_1 - b^+$. Since each point on it is irreducible, it admits a free $\S^1$-action by the gauge group:
\begin{equation}\label{equ:Saction}
(A,\Phi)\mapsto (A,e^{i\theta}\Phi).
\end{equation}
The quotient is the moduli space of irreducible monopoles ${\M}^*(g,\eta)$ and it has dimension $\ind \D_A + b_1 - 1 - b^+$.

The general argument for unfamily case is that the set of regular perturbations $\mathcal{Z}_{reg}$ is of the second category in the sense of Baire (a countable intersection of open and dense sets) when $b^+>0$.
It goes as follows:

$D_{(A,\Phi)}\mathcal{F}_{(g,\eta)}$ can be decomposed as $(\mathcal{F}_0, \mathcal{F}_1)$ where  
\begin{align}
\mathcal{F}_0\begin{pmatrix}
        A\\
          \Phi\\
    \end{pmatrix}
    &=
    \begin{pmatrix}
         d^* A\\
          \D_A \Phi \\
    \end{pmatrix}, \\
\label{equ:defF1}\mathcal{F}_1\begin{pmatrix}
        A\\
          \Phi\\
    \end{pmatrix}
    &=     F^{+_g}_A + i\eta - \rho^{-1}(\sigma(\Phi, \Phi)).
\end{align}
$\mathbb{M} := \mathcal{F}_0^{-1}(0)$ is called a \textbf{universal moduli space} and $\mathcal{F}_1$ sends it to the space of perturbations. 
\[
\mathcal{F}_{(g,\eta)}^{-1}(0)=
\mathcal{F}_1^{-1}(\eta) \cap \mathbb{M} = (\mathcal{F}_1 |_{\mathbb{M}})^{-1}(\eta)
\]
is a slice of solutions to the Seiberg-Witten
equation with the perturbation $\eta$. 
Now we restrict the domain of these operators to the submanifold $\mathcal{C}{(g,\eta)}\cap \{\Phi \neq 0\}$. 
Since $d_{(A,\Phi)}(\mathcal{F}_0|_{\{\Phi \neq 0\}})$ is surjective by some analytical computations (see \cite{salamon2000spin} Lemma 8.17), $\mathbb{M}^* :=(\mathcal{F}_0|_{\{\Phi \neq 0\}})^{-1}(0)$ is a smooth manifold. 
Regular values of $\mathcal{F}_1 |_{\mathbb{M}^*}$ are of the second category in the sense of Baire by the Sard-Smale theorem. 
For each regular value $\eta$, $(\mathcal{F}_1 |_{\mathbb{M}^*})^{-1}(\eta)$ is a smooth manifold. Moreover, the ``wall''
\begin{equation}\label{equ:wall}
\mathcal{W}^{k-1}_{g,s} := \{\eta\in L^{k-1,2}(i\Lambda^+(X));  \exists A\in \mathscr{A}(s),  F^{+_g}_A + i\eta = 0\}
\end{equation}
is an affine space of codimension $b^+$ (see \cite{salamon2000spin} Proposition 7.10). For each $\eta$ outside $\mathcal{W}^{k-1}_{g,s}$, all $\eta$-monopoles are irreducible, i.e. 
\[
\mathcal{F}_{(g,\eta)}^{-1}(0) = (\mathcal{F}_1 |_{\mathbb{M}^*})^{-1}(\eta).
\]
Hence $\mathcal{Z}_{reg}$ contains all regular values of $\mathcal{F}_1 |_{\mathbb{M}^*}$ outside $\mathcal{W}^{k-1}_{g,s}$, which is of the second category in the sense of Baire when $b^+>0$.

%the index of $(\mathcal{F}_0, \mathcal{F}_1)$ is the index of $\mathcal{F}_1 |_{\mathcal{F}_0^{-1}(0)}$, whose Thus the set of generic perturbations is of the second category.

Now we want to choose a section $\eta: B \to \Pi(E)$ such that it is generic in the family sense. 
This means that, the image of this section intersects with the image of the universal moduli space transversally. 
The set of such sections is dense: %(\cite{BK2020} used this result. 
%It seems that this is proved by deciding the value of each cell of $B$ inductively in \cite{Konno2018CharacteristicCV}. 
%But I don't know how to prove it by sard-smale. \cite{Nicolaescu2000NotesOS} used sard-smale to prove this when $B = [0,1]$.)

\begin{theorem}\label{thm:regular}
Let $b_1 = \dim H^1(X)$, $b^+ =  \dim H^{2,+}(X)$. Assume 
\[
b^+ \geq \dim B + 1\]
and 
\begin{equation}\label{equ:conditionDim}
\ind \D_A + b_1 - 1 - b^+ + \dim B = 1.
\end{equation}
Fix an arbitrary matrics family $g:B\to \text{Met}(X)$ first. Let $\mathcal{Z}=\mathcal{Z}^{k-1}$ be the space of smooth sections of the bundle
\[
 \bigsqcup_{b\in B} L^{k-1,2}(\Lambda_{g(b)}^+(X)) \to B.
 \]
%Denote by $\mathcal{Z}$ the space of all smooth sections of $ \Pi(E_{X}) \to B$.
Then there exists a set $\mathcal{Z}_{reg}  \subset \mathcal{Z}$ of the second category in the sense of Baire such that for every $\eta \in \mathcal{Z}_{reg}$, all $\eta$-monopoles are irreducible, and the space $\FM(X, g,\eta)$ defined in (\ref{equ:parametrized}) is a smooth manifold of dimension $1$.
\end{theorem}

\begin{proof}
In the ordinary theory, in order to prove that the Seiberg-Witten moduli space is well defined up to a cobordism, one has to show the existence of a regular path in the perturbation space (see for example, \cite{salamon2000spin} Theorem 7.21, whose ``skeleton'' is \cite{salamon2000spin} Proposition B.17). We just replace the segment $[0,1]$ by $B$. %Since $B$ is compact, all arguments of transversality theory goes through as \cite{salamon2000spin} Proposition B.17. 
However, some modification is needed, since here the $X$-bundle $E_X$ over $B$ is nontrivial and we allow the metric to vary.

%In the definition (\ref{equ:configOnG}), the spinor bundle $S_g$ is defined by a $\text{spin}^c$ structure $s_g$ induced from a $\text{Spin}_\text{GL}^c$ structure $s$ by a metric $g$:
%\[
%\mathcal{C}{(g)} = L^{k,2}(\mathscr{A}(s)) \times L^{k,2}(S_g^+)
%\]

Let
\begin{align*}
    \mathcal{C}_{g} &= L^{k,2}(\mathscr{A}(s)) \oplus L^{k,2}(S_g^+) \oplus L^{k-1,2}(i\Lambda^+(X))\\
    \mathcal{D}_{g} &= L^{k-1,2}(i\Lambda^0_0(X;g)) \oplus L^{k-1,2}(i\Lambda^+(X)) \oplus L^{k-1,2}(S_g^-)\\
     \mathcal{F}_{g}: \mathcal{C}_{g} &\to \mathcal{D}_{g}\\
    \mathcal{F}_{g}\begin{pmatrix}
         A\\
          \Phi\\ 
          \eta\\
    \end{pmatrix}
    &= \begin{pmatrix}
         d^*(A-A_0)\\
         F^{+_g}_A + i\eta - \rho^{-1}(\sigma(\Phi, \Phi))\\
          \D_A \Phi\\
    \end{pmatrix}
\end{align*}
As in the ordinary theory we consider $[0,1]\times \mathcal{C}$, here we wish to define a bundle using $g$:
\[
\mathcal{X} := \bigsqcup_{b\in B}\mathcal{C}_{g(b)}
\]
by the transition functions (\ref{equ:configCocycle}). However (\ref{equ:configCocycle}) doesn't satisfy the cocycle condition. Hence we have to consider the moduli space instead. But now above $\mathcal{F}_g$ incorporates $d^*$, not just $sw_g$ in (\ref{equ:defsw}). To make $\mathcal{F}_g$ well-defined on the parametrized moduli space, we have to reformulate its construction.

Use the notations in subsection \ref{subsection:Parametrized}, where $\{g_{\alpha\beta}: U_{\alpha\beta} \to G\}$ is a family of transition functions for $E_X$ corresponding to the open cover $\{U_\alpha\}_\alpha$. 
In subsection \ref{subsection:Parametrized} we choose a lift $\tilde{g}_{\alpha\beta} : U_{\alpha\beta} \to \tilde{G}$ of ${g}_{\alpha\beta}$. Now we want to refine our choice such that $\tilde{g}_{\alpha\beta}^* \circ \tilde{g}_{\beta\gamma}^* \circ \tilde{g}_{\gamma\alpha}^*$ preserves $(d^*)^{-1}(0)$. 
\begin{lemma}\label{lemma:cocycle-d*}
For each $U_{\alpha\beta}$, we can find a lift $\tilde{g}_{\alpha\beta}$ of ${g}_{\alpha\beta}$ such that if the cocycle $\tilde{g}_{\alpha\beta}^* \circ \tilde{g}_{\beta\gamma}^* \circ \tilde{g}_{\gamma\alpha}^*$ is the guage transformation by $u\in \mathscr{G}$, then $d^*(u^{-1}du) = 0$.
\end{lemma}
We prove Lemma \ref{lemma:cocycle-d*} later. With these lifts we define 
\[
\mathcal{X} := \bigsqcup_{b\in B}\mathcal{B}^*_{g(b)}
\]
by the transition functions (\ref{equ:moduliCocycle}). (An element of $\mathcal{X}$ can be written as $(b,A, \Phi)$.) Construct a bundle 
\[
\mathcal{Y} \to \mathcal{X} \times \mathcal{Z}%^{k-1}
\]
whose fiber over $(b,A, \Phi, \eta)$ is $ \mathcal{D}_{g(b)}$ as follows. Consider the trivial bundle 
\[
\mathcal{U}_b: \mathcal{C}^*_{g(b)}\times \mathcal{D}_{g(b)} \to \mathcal{C}^*_{g(b)}
\]
The gauge group $\mathscr{G}$ acts on $\mathcal{U}_b$ by the diagonal action and this action makes $\mathcal{U}_b$ a $\mathscr{G}$-equivariant bundle. $\mathcal{F}_{g(b)}$ is a $\mathscr{G}$-equivariant section of $\mathcal{U}_b$. The quotient of $\mathcal{U}_b$ is a bundle $[\mathcal{U}_b]$ over $\mathcal{B}^*_{g(b)}$, and $\mathcal{F}_{g(b)}$ descends to a section $[\mathcal{F}_{g(b)}]$ of $[\mathcal{U}_b]$ (note that since $\mathscr{G}$ acts on $\mathcal{C}^*_{g(b)}$ freely, the fiber of $[\mathcal{U}_b]$ is still $\mathcal{D}_{g(b)} $, and $\dim \text{coker}[\mathcal{F}_{g(b)}]$ is still $\dim\text{coker} \D_A + b^+$ since the infinitesimal action of the gauge group is in the kernel of $d\mathcal{F}_{g(b)}$). 
For each $U_\alpha \subset B$, form the family $\bigsqcup_{b\in U_\alpha}\mathcal{U}_b$. Then glue all such families by the transition functions (\ref{equ:configCocycle}) and form the quotient family $\bigsqcup_{b\in B}[\mathcal{U}_b]$, which is a bundle over $\mathcal{X}$. Let $\mathcal{Y}$ be the bundle
\[
 (\bigsqcup_{b\in B}[\mathcal{U}_b])\times \mathcal{Z} \to \mathcal{X} \times \mathcal{Z}.%^{k-1}
\]
Define a section of $\mathcal{Y}$
\[
{F}: \mathcal{X} \times \mathcal{Z} \to \mathcal{Y}
\]
by ${F}(b,A, \Phi,\eta) =[ \mathcal{F}_{g(b)}](A, \Phi,\eta(b))$. $F$ is well-defined because of Lemma (\ref{lemma:cocycle-d*}) and the fact that the Seiberg-Witten equations are $\mathscr{G}$-equivariant. We want to show that this map is transverse to the zero section. If $F(b,A, \Phi,\eta)=0$, then
\begin{align*}
D_{(b,A, \Phi,\eta)}F : T_bB \times T_c \mathcal{B}_{g(b)}\times T_\eta \mathcal{Z}^{k-1} &\to  \mathcal{D}_{g(b)}\\
(\tau, \alpha, \phi, \zeta) &\mapsto d_{(A, \Phi,\eta(b))}  \mathcal{F}_{g(b)}(\alpha, \phi,\zeta(b) + d_b\eta(\tau))
\end{align*}
Here $DF$ is the projection of the differential $dF$ to the vertical tangent space of the bundle $\mathcal{Y}$. From Equation (\ref{equ:Dlinearization}) we deduce,
\begin{equation}\label{equ:dFperturbation}
d_{(c,\eta(b))}\mathcal{F}_{g(b)}\begin{pmatrix}
         \alpha\\
          \phi\\
         \zeta(b) + d_b\eta(\tau)\\
    \end{pmatrix}
    =\begin{pmatrix}
         d^* \alpha\\
         d^+  \alpha\\
         \D_A \phi\\
    \end{pmatrix} 
    + \begin{pmatrix}
         0\\
          i(\zeta(b) + d_b\eta(\tau))-\rho^{-1}(\sigma(\Phi, \phi)+\sigma(\phi, \Phi))\\
         \Gamma( \alpha)\Phi\\
    \end{pmatrix} 
\end{equation}
The operator 
\begin{equation}
\begin{pmatrix}
         \alpha\\
          \phi\\
    \end{pmatrix}
    \mapsto
    \begin{pmatrix}
         d^* \alpha\\
          \D_A \phi + \Gamma( \alpha)\Phi\\
    \end{pmatrix} 
\end{equation}
is surjective (see \cite{salamon2000spin}) Lemma 8.17%TODO
. Note also that one can choose $\zeta$ arbitrarily. Therefore (\ref{equ:dFperturbation}) is surjective. Hence $F$ is transverse to the zero section and 
\[
\mathbb{M} := \{(b,A, \Phi, \eta)\in \mathcal{X}\times \mathcal{Z}; F(b,A, \Phi, \eta) = 0\}
\]
is therefore a Banach manifold. The projection
\[
\pi: \mathbb{M} \to \mathcal{Z}
\]
is a Fredholm map of separable Banach manifolds. By Sard-Smale theorem, the regular value of $\pi$ is of the second category in the sense of Baire.

By an argument of transversality theory (see \cite{salamon2000spin} Theorem B.16), $\eta\in \mathcal{Z}$ is a regular value of $\pi$ iff the restriction of $F$ over $\eta$: 
\begin{align}
\label{Feta} F_\eta: \mathcal{X} &\to \mathcal{Y}_\eta:= \bigsqcup_{b\in B}\mathcal{D}_{g(b)}\\
(b,A,\Phi) &\mapsto F(b,A,\Phi,\eta)=[\mathcal{F}_{(g(b),\eta(b))}](A,\Phi)
\end{align}
is transverse to the zero section. Here $\mathcal{F}_{(g(b),\eta(b))}$ is the operator defined in (\ref{def:F0}). Now choose $\eta$ to be a regular value of $\pi$, then $F_\eta$ is transverse to the zero section, and therefore by the implicit function theorem, the set
\begin{equation}\label{M(Feta)}
M(F_\eta)= \{(b,A,\Phi) \in \mathcal{X}; F_\eta(b,A,\Phi) =0\}
\end{equation}
is a submanifold of $\mathcal{X}$ of dimension 
\begin{align*}
\dim \ker [\mathcal{F}_{(g(b_0),\eta(b_0))}]+ \dim B - \dim\text{coker} [\mathcal{F}_{(g(b_0),\eta(b_0))}] &= \ind [\mathcal{F}_{(g(b_0),\eta(b_0))}] + \dim B\\
 &= \ind \D_A + b_1-1 - b^+ + \dim B\\
 &=1
\end{align*}
for any $b_0\in B$ in the projection of $M(F_\eta)$. (Figure \ref{fig:type0} and Figure \ref{fig:type1} justify this computation of the dimension. Here the index of $[\mathcal{F}_{(g(b_0),\eta(b_0))}]$ depends on some topological invariants of $X$ and determinant line bundles of the $\text{spin}_{\text{GL}}^c$ structures, and since the structure group of $E_X$ preserves the isomorphism class of the $\text{spin}_{\text{GL}}^c$ structures, $\ind [\mathcal{F}_{(g(b_0),\eta(b_0))}]$ is independent from the choice of $b_0$. 
However, $\dim \ker[ \mathcal{F}_{(g(b_0),\eta(b_0))}]$ and $\dim\text{coker} [\mathcal{F}_{(g(b_0),\eta(b_0))}]$ do depend on the choice of $b_0$. Since $F_\eta$ is transverse to the zero section, $\dim B \geq \dim \text{coker} [\mathcal{F}_{(g(b_0),\eta(b_0))} ]$. Hence $\dim \ker [\mathcal{F}_{(g(b_0),\eta(b_0))}]$ can only be $0$ or $1$.)

%Note that all configuration points $M(F_\eta)$ are irreducible, otherwise in  (\ref{equ:dFperturbation})
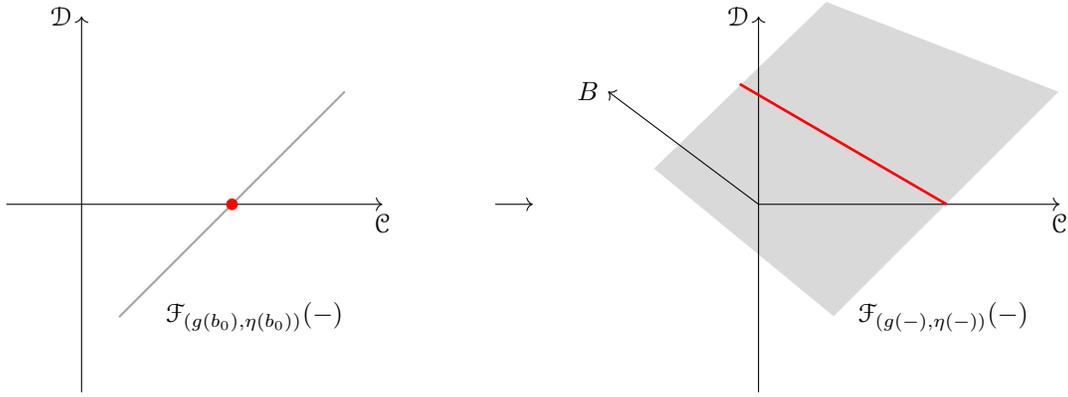
\begin{figure}
    \centering
    \begin{tikzpicture}
        % 坐标轴
        \draw[->] (-1,0) -- (4,0) node[below] {$\mathcal{C}$};
        \draw[->] (0,-2.5) -- (0,2.5) node[left] {$\mathcal{D}$};
        
        % 网格线
        %\draw[dashed] (-0.5,-0.5) grid (4.5,3.5);
        
        % 坐标轴上的标签
        %\foreach \x in {1,2,3,4}
           % \draw (\x,-0.1) -- (\x,0.1) node[below=2pt] {$\x$};
        %\foreach \y in {1,2,3}
           % \draw (-0.1,\y) -- (0.1,\y) node[left=2pt] {$\y$};
        
        % 例子：折线
        \draw[gray!70, thick] (0.5,-1.5) -- (3.5,1.5);
        \node[anchor=west] at (1,-1.5) {$ \mathcal{F}_{(g(b_0),\eta(b_0))}(-)$};
        %\draw[blue, thick] (1,1) -- (2,3) -- (3,2) -- (4,3);
        
        % 标注点
        \filldraw[red] (2,0) circle (2pt);
        %\filldraw[red] (1,1) circle (2pt) node[above right] {$(1,1)$};
        %\filldraw[red] (2,3) circle (2pt) node[above right] {$(2,3)$};
        %\filldraw[red] (3,2) circle (2pt) node[above right] {$(3,2)$};
        %\filldraw[red] (4,3) circle (2pt) node[above right] {$(4,3)$};
        
        \draw[->] (5.5,0) -- (6,0);
        
        % 复制并平移第一个图
    	\begin{scope}[xshift=9cm]
    	\filldraw[fill=gray!30,draw=white] (0.9,2.7) -- (4,1.5) -- (1,-1.5) -- (-1.4,0.475) -- cycle;
    	\node[anchor=west] at (1.2,-1.5) {$ \mathcal{F}_{(g(-),\eta(-))}(-)$};
    	% 坐标轴
        \draw[->] (0,0) -- (4,0) node[below] {$\mathcal{C}$};
        \draw[->] (0,-2.5) -- (0,2.5) node[left] {$\mathcal{D}$};
        \draw[->] (0,0) -- (-2,1.5) node[left] {$B$};
        
        \draw[red, line width=1pt] (2.5,0) -- (-0.25,1.6);

    	\end{scope}
    	
    \end{tikzpicture}
    \caption{Type 0: If the kernel of $ \mathcal{F}_{(g(b_0),\eta(b_0))}$ is $0$-dimensional, $B$ would extend the kernel.}
    \label{fig:type0}
\end{figure}
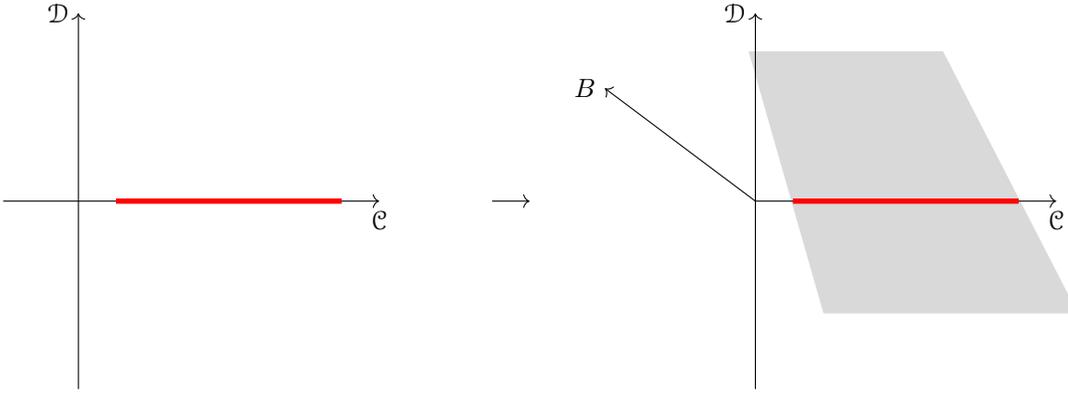
\begin{figure}
    \centering
    \begin{tikzpicture}
        % 坐标轴
        \draw[->] (-1,0) -- (4,0) node[below] {$\mathcal{C}$};
        \draw[->] (0,-2.5) -- (0,2.5) node[left] {$\mathcal{D}$};
        
        % 网格线
        %\draw[dashed] (-0.5,-0.5) grid (4.5,3.5);
        
        % 坐标轴上的标签
        %\foreach \x in {1,2,3,4}
           % \draw (\x,-0.1) -- (\x,0.1) node[below=2pt] {$\x$};
        %\foreach \y in {1,2,3}
           % \draw (-0.1,\y) -- (0.1,\y) node[left=2pt] {$\y$};
        
        % 例子：折线
        \draw[red,  line width=2pt] (0.5,0) -- (3.5,0);
        %\draw[blue, thick] (1,1) -- (2,3) -- (3,2) -- (4,3);

        \draw[->] (5.5,0) -- (6,0);
        
        % 复制并平移第一个图
    	\begin{scope}[xshift=9cm]
    	\filldraw[fill=gray!30,draw=white] (-0.1,2) -- (2.5,2) -- (4.3,-1.5) -- (0.9,-1.5) -- cycle;
    	% 坐标轴
        \draw[->] (0,0) -- (4,0) node[below] {$\mathcal{C}$};
        \draw[->] (0,-2.5) -- (0,2.5) node[left] {$\mathcal{D}$};
        \draw[->] (0,0) -- (-2,1.5) node[left] {$B$};
        
        \draw[red, line width=2pt] (0.5,0) -- (3.5,0);
        
    	\end{scope}
    	
    \end{tikzpicture}
    \caption{Type 1: If the kernel of $ \mathcal{F}_{(g(b_0),\eta(b_0))}$ is $1$-dimensional, $B$ would push it off the zero set.}
    \label{fig:type1}
\end{figure}

Notice that $M(F_\eta) = \FM^*(g,\eta)$. The family version of the wall (unfamily version is defined in (\ref{equ:wall})) is
\begin{equation}\label{equ:familywall}
\mathcal{F}\mathcal{W}^{k-1}_{g,s} := \{\eta\in \mathcal{Z}^{k-1};  \exists b\in U_\alpha \subset B, A\in \mathscr{A}(s_\alpha), \text{ such that } F^{+_{g(b)}}_A + i\eta(b) = 0\}.
\end{equation}
If $\eta$ is a regular value of $\pi$ and $\eta\notin \mathcal{F}\mathcal{W}^{k-1}_{g,s}$, then 
\[
 \FM(g,\eta)= \FM^*(g,\eta)
\]
is a smooth manifold, and such $\eta$ belongs to $\mathcal{Z}_{reg}$. On each fiber, the othorgonal complement of the wall has dimension $b^+$. 
Because we have assumed $b^+ \geq \dim B + 1$, one can perturb any $\eta\in\mathcal{Z}^{k-1}$ slightly such that it doesn't meet the wall on every point. Hence $\mathcal{Z}^{k-1} \setminus \mathcal{F}\mathcal{W}^{k-1}_{g,s}$ is an open dense set in $\mathcal{Z}^{k-1}$. Because regular values of $\pi$ is of the second category in the sense of Baire, their intersection with $\mathcal{Z}^{k-1} \setminus \mathcal{F}\mathcal{W}^{k-1}_{g,s}$, contained in $\mathcal{Z}_{reg}$, is of the second category in the sense of Baire.
\end{proof}
\begin{remark}
%The formulation of $\FM(\sigma)$ in above proof is slightly different from the construction in the subsection \ref{subsection:Parametrized}. The gauge group acts on the connections part by
%\[
%A\mapsto A + u^{-1}du
%\]
%and acts on the spinor bundle by (\ref{equ:Saction}). In the subsection \ref{subsection:Parametrized} we construct the family of configuration space and form the parametrized moduli space by these actions simultaneously (here the transition functions are given by (\ref{equ:cocycleModGauge}) that satisfies the ``cocycle condition modulo gauge'', which is enough since the moduli space is the quotient space by the gauge group). 
%In above proof we integrate the operator $d^*$ in $\mathcal{F}_{(g,\eta(b))}$ and form the parametrized space $M(F_\eta)$ first. Then form the quotient space by the action (\ref{equ:Saction}).

Konno\cite{Konno2018CharacteristicCV} describes a way to find a generic perturbations family for $0$-dimensional moduli space, which is to put $B$ cell by cell into the space of parameters family $\Pi(E)$, and then highest dimensional cells of $B$ would intersect with the projection of the universal moduli space (i.e, image of $\mathcal{F}_0^{-1}(0)$ under $\mathcal{F}_1$ defined in (\ref{equ:defF1})) in discrete points for dimension reason. 
For $1$-dimensional moduli space, the situation is subtler. $\mathcal{F}_1$ can has either $1$-dimension kernel and $(\dim B)$-dimensional cokernel, or $0$-dimension kernel and $(\dim B-1)$-dimensional cokernel (see Figure \ref{fig:genericCell}). Moreover, in the proof of the following propostion (Proposition \ref{prop:perturbationCondition}), we need the fact that the generic perturbations family is dense.
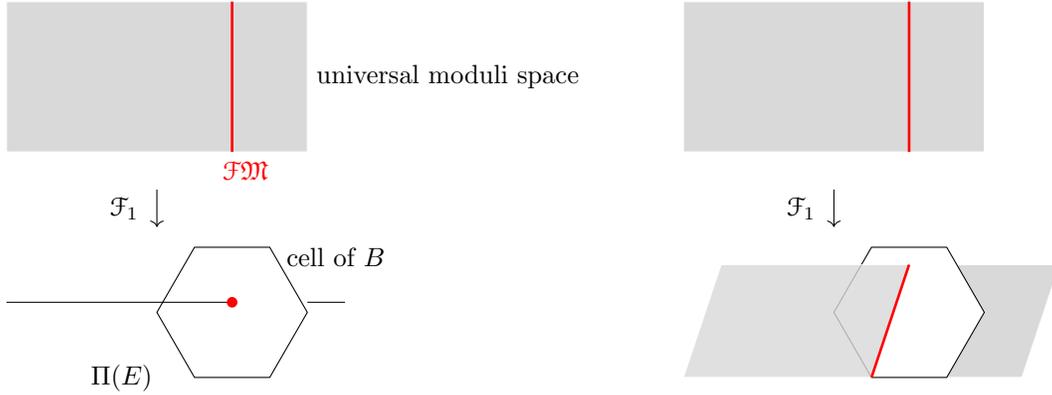
\begin{figure}
    \centering
\begin{tikzpicture}
    % 绘制第一个长方形并填充阴影
    \filldraw[fill=gray!30,draw=white] (0,0) rectangle (4,2);
    
    % 在第一个长方形旁边添加文本节点
    \node[anchor=west] at (4,1) {universal moduli space};
    
    % 在第一个长方形下方画箭头
    \draw[->] (2,-0.5) -- (2,-1);
     \node[anchor=west] at (1.25,-0.75) {$\mathcal{F}_1$};
    
    % 在箭头下面画直线
    \draw (0,-2) -- (3,-2);
    \draw (4,-2) -- (4.5,-2);
    
    % 在第一个长方形内画垂直红线
    \draw[red, line width=1pt] (3,0) -- (3,2);
    \node[red, anchor=west] at (2.75,-0.25) {$\mathcal{F}\M$};
    
    % 在直线右侧画第一个六边形
    \draw (2.5,-3) ++(0:1) -- ++(60:1) -- ++(120:1) -- ++(180:1) -- ++(240:1) -- ++(300:1) -- cycle;
    
    % 在第一个六边形旁边添加文本节点
    \node[anchor=west] at (3.6,-1.4) {cell of $B$};
    \node[anchor=west] at (1,-3) {$\Pi(E)$};
    % 在直线右端点画一个红点
    \fill[red] (3,-2) circle (2pt);
    
    % 复制并平移第一个图
    \begin{scope}[xshift=9cm]
        % 绘制第二个长方形并填充阴影
        \filldraw[fill=gray!30,draw=white] (0,0) rectangle (4,2);
        
        % 在第二个长方形旁边添加文本节点
        %\node[anchor=west] at (4,1) {universal moduli space};
        
        % 在第二个长方形下方画箭头
        \draw[->] (2,-0.5) -- (2,-1);
        \node[anchor=west] at (1.25,-0.75) {$\mathcal{F}_1$};
        
        % 在箭头下面画平行四边形
        %\filldraw[fill=gray!30,draw=black] (1,-1.5) -- (3,-1.5) -- (2.5,-3) -- (0.5,-3) -- cycle;
        %\draw (4,-2) -- (4.5,-2);
        
        % 在第二个长方形内画垂直红线
        \draw[red,line width=1pt] (3,0) -- (3,2);
        
        % 在直线右侧画第二个六边形
        \filldraw[fill=gray!30,draw=white] (3,-1.5) -- (5,-1.5) -- (4.5,-3) -- (3.5,-3) -- cycle;
        \filldraw[fill=white,draw=black] (2.5,-3) ++(0:1) -- ++(60:1) -- ++(120:1) -- ++(180:1) -- ++(240:1) -- ++(300:1) -- cycle;
        %\shade[top color=gray!30,bottom color=white] (1,-1.5) -- (3,-1.5) -- (2.5,-3) -- (0.5,-3) -- cycle;
        \filldraw[fill=gray!30,draw=white,opacity=0.8] (0.5,-1.5) -- (3,-1.5) -- (2.5,-3) -- (0,-3) -- cycle;
        \draw[red,line width=1pt] (3,-1.5) -- (2.5,-3);
        
        % 在第二个六边形旁边添加文本节点
        %\node[anchor=west] at (3.6,-1.4) {cell of $B$};
        %\node[anchor=west] at (1,-3) {$\Pi(E)$};
        % 在直线右端点画一个红点
        %\fill[red] (3,-2) circle (2pt);
    \end{scope}
\end{tikzpicture}
\caption{These situations correspond to Type 1 (Figure \ref{fig:type1}) and Type 0 (Figure \ref{fig:type0}).}
    \label{fig:genericCell}
\end{figure}

\end{remark}

\begin{proof}[Proof of Lemma \ref{lemma:cocycle-d*}]
Recall the definition of the automorphism group $\tilde{G}$ in (\ref{equ:defAut}). Fix an $f \in G$. Let
\[
\Aut(f):= \{\tilde{f};(f,\tilde{f})\in \tilde{G}\}
\]
Each $\tilde{f} \in \Aut(f)$ is an isomorphism of $P_{\text{Spin}^c_{\GL}(4)}$ that adds one more infomation to the isomorphism $df$ of the frame bundle $P_{\GL_4^+(\R)}$: the map on the $\S^1$-factor of 
\[
 \text{Spin}^c_{\GL}(4)  = S^1 \times \widetilde{\GL_4^+(\R)} / \{\pm (1,I)\}.
 \]
Hence $\Aut(f)$ is in noncanonical one-to-one correspondence with the gauge group $\mathscr{G}$: the difference of two elements in $\Aut(f)$ is an element of $\mathscr{G}$. By Hodge theory each component of $\mathscr{G}$ contains a harmonic element $u = e^{i\theta}$ (see Proposition 5.30), which means that the $i\R$-value $1$-form $u^{-1}du$ satisfies
\[
d^*(u^{-1}du)=0.
\]
Geometrically this means that the rotation of the $\S^1$-factor is at a constant speed when going around every nontrivial loop of $X$. This motivates us to fix a ``reference rotation'' as follows:

In subsection \ref{subsection:Parametrized}, for each $U_\alpha$, we have assigned a $\text{spin}_{GL}^c$-structure $s_\alpha$ on $X$. 
Now we further fix a connection $1$-form $a_\alpha$ on the determinant line bundle $L_\alpha$ of $s_\alpha$. Recall that $g_{\alpha\beta}(b)$ preserves the isomorphism class of $\text{spin}_{GL}^c$-structures, so $\tilde{g}_{\alpha\beta}(b)^*$ induces an isomorphism between determinant line bundles $L_\alpha$ and $L_\beta$. 
The pullback $\tilde{g}_{\alpha\beta}(b)^*a_\beta$ is also an $i\R$-value $1$-form. Since each component of $\mathscr{G}$ contains a harmonic element and $U_{\alpha\beta}$ is contractible, we can choose a lift $\tilde{g}_{\alpha\beta}^*$ such that for every $b\in U_{\alpha\beta}$
\[
d^*(\tilde{g}_{\alpha\beta}(b)^*a_\beta -a_\alpha)=0.
\]
Since $d$ and $d^*$ commutes with $f^*$ for any diffeomorphism $f$, we have
\begin{align*}
d^*(\tilde{g}_{\alpha\beta}^* \circ \tilde{g}_{\beta\gamma}^* \circ \tilde{g}_{\gamma\alpha}^*a_\alpha -a_\alpha )
&= \tilde{g}_{\alpha\beta}^* \circ \tilde{g}_{\beta\gamma}^*  d^*(\tilde{g}_{\gamma\alpha}^*a_\alpha)- d^* a_\alpha\\
&= \tilde{g}_{\alpha\beta}^* \circ \tilde{g}_{\beta\gamma}^*  d^*(a_\gamma)- d^* a_\alpha\\
&= \tilde{g}_{\alpha\beta}^* d^*( \tilde{g}_{\beta\gamma}^*  a_\gamma)- d^* a_\alpha\\
&= \tilde{g}_{\alpha\beta}^* d^*( a_\beta)- d^* a_\alpha\\
&= d^*(\tilde{g}_{\alpha\beta}^* a_\beta)- d^* a_\alpha\\
&= d^*(a_\alpha)- d^* a_\alpha\\
&= 0.
\end{align*}
Now for any connection $A\in L^{k,2}(\mathscr{A}(s_\alpha))$ on $U_\alpha$, we have
\[
\tilde{g}_{\alpha\beta}^* \circ \tilde{g}_{\beta\gamma}^* \circ \tilde{g}_{\gamma\alpha}^*A = A + u^{-1}du
\]
where $2u^{-1}du = \tilde{g}_{\alpha\beta}^* \circ \tilde{g}_{\beta\gamma}^* \circ \tilde{g}_{\gamma\alpha}^*a_\alpha -a_\alpha $ (the coefficient $2$ comes from the definition of the determinant bundle in (\ref{equ:determinant})). Hence $d^* \tilde{g}_{\alpha\beta}^* \circ \tilde{g}_{\beta\gamma}^* \circ \tilde{g}_{\gamma\alpha}^*A  = d^*A$.
\end{proof}

\begin{theorem}\label{thm:cylindricalTransversality}
Let $X_0$ be a cylindrical manifold such that $\partial_\infty X_0= \S^1\times \S^2$. Let $\hat{s}$ be a $\text{spin}^c$ structure of $X_0$ such that it induces a $\text{spin}^c$ structure $s$ of $ \S^1 \times \S^2 $ such that the first Chern class of the determinant line bundle is zero. Let $b_1 = \dim H^1(X_0)$, $b^+ =  \dim H^{2,+}(X_0)$. Assume 
\[
b^+ \geq \dim B + 1\]
and 
\begin{equation}%\label{equ:conditionDim}
\ind \D_A + b_1 - 1 - b^+ + \dim B = 1.
\end{equation}
Fix any matrics family $g:B\to \text{Met}(X_0)$ that restricts to the standard round metric on the boundary. Let $\mathcal{Z}=\mathcal{Z}^{k-1}$ be the space of smooth sections of the bundle
\[
 \bigsqcup_{b\in B} L^2_{k-1}(\Lambda_{g(b)}^+(X_0)) \to B.
 \]
%Denote by $\mathcal{Z}$ the space of all smooth sections of $ \Pi(E_{X}) \to B$.
Then there exists a set $\mathcal{Z}_{reg}  \subset \mathcal{Z}$ of the second category in the sense of Baire such that for every $\eta \in \mathcal{Z}_{reg}$, any $\eta$-monopole $\hat{\mathsf{C}}_0$

$\mathbf{A_1}$  is irreducible;

%$\mathbf{A_2}$  is of type $1$ (see Figure \ref{fig:type0});

$\mathbf{A_2}$  $\dim H^2(F(\hat{\mathsf{C}}_0))=\dim B$;

$\mathbf{A_3}$  the space $\FM(X_0, \hat{s}, g,\eta)$ defined in (\ref{equ:parametrized}) is a smooth manifold of dimension $1$.
\end{theorem}
\begin{proof}
For $X_0$, form the bundle $\mathcal{X}$ over $B$ as in the proof of Theorem \ref{thm:regular}:
\[
\mathcal{X} := \bigsqcup_{b\in B}\hat{\mathcal{C}}^*_{g(b),\mu,sw}/\mathscr{G}.
\]
Here $\hat{\mathcal{C}}_{\mu,sw}$ is the space of configurations on $X_0$ that restricts to monopoles on $\S^1\times \S^2$, as defined in (\ref{equ:defCsw}).
Let 
\[
\M_s = \M(\S^1\times \S^2,s, g_{round}).
\]
Exactly as in the proof of Proposition \ref{prop:dimOfReducible}, one can show that $\M_s = \S^1$. The bundle $E_{X_0}$ indeuces a bundle $ E_{\S^1\times \S^2}$, and this gives the parametrized moduli space $\FM_s$. $\FM_s$ is an $\S^1$-bundle over $B$. Define 
\[
\hat{\mathcal{Y}}_{g,\mu} := L_\mu^{1,2}(\hat{S}_{g,\hat{s}}^- \oplus \mathbf{i}\Lambda^{2,+_g}T^*\hat{N}),
\]
and $\mathcal{U}_b =\hat{\mathcal{C}}^*_{g(b),\mu,sw}\times\hat{\mathcal{Y}}_{g(b),\mu}$ and 
\[
F:\mathcal{X} \oplus \FM_s\times \mathcal Z \to  \bigsqcup_{b\in B}[\mathcal{U}_b] \oplus \FM_s \oplus \FM_s,
\]
\[
F(b, \hat{\mathsf{C}},\mathsf{C},\eta) =( \widehat{{SW}}_{\eta(b)}(\hat{\mathsf{C}}),\partial_\infty \hat{\mathsf{C}}, \mathsf{C}).
\]
Let $\Delta$ be the diagonal of $\M_s\times \M_s$. Let $\mathcal{F}\Delta \subset \FM_s \oplus \FM_s$ be an $\S^1$-bundle over $B$ with fiber $\Delta$. One can show that $F$ is transversal to $0\oplus \mathcal{F}\Delta \subset  \mathcal{Y} \oplus \FM_s \oplus \FM_s$. Then apply Sard-Smale to the projection
\[
\pi: F^{-1}(0\oplus \mathcal{F}\Delta) \to \mathcal{Z}.
\]
Let $\mathcal{Z}_{reg}$ be the set of regular values of $\pi$ which don't meet the wall. As before this set is of the second category. Given any $\eta \in \mathcal{Z}_{reg}$,
\[
F_\eta: \mathcal{X} \oplus \FM_s \to \mathcal{Y} \oplus \FM_s \oplus \FM_s
\]
is transversal to $0\oplus \mathcal{F}\Delta \subset  \mathcal{Y} \oplus \FM_s \oplus \FM_s$. Fix any $(b, \hat{\mathsf{C}}_0, \mathsf{C}_\infty)\in F^{-1}_\eta(0\oplus \mathcal{F}\Delta)$ and write $\hat{\mathsf{C}}_0 = (\hat{\Phi}_0,\hat{A}_0)$. Recall that there are two short exact sequences in the diagram \ref{equ:diagram} : 
\[
T_1 \hat{\mathcal{G}}_{\mu} 
\stackrel{i}{\hookrightarrow} T_1 \hat{\mathcal{G}}_{\mu,ex} 
\stackrel{\partial_\infty}{\twoheadrightarrow}  T_1\mathcal{G}_s
\]
and 
\[
T_{\hat{\mathsf{C}}_0} \partial_\infty^{-1}(\mathsf{C}_\infty)
\stackrel{i}{\hookrightarrow} 
T_{\hat{\mathsf{C}}_0} \hat{\mathcal{C}}_{g(b),\mu,sw} 
\stackrel{\partial_\infty}{\twoheadrightarrow}  
T_{\mathsf{C}_\infty}\mathcal{Z}_s.
\]
Note that these sequences split. Hence by chasing the left top and left bottom square of the diagram \ref{equ:diagram}, $T_{\hat{\mathsf{C}}_0}( \hat{\mathcal{C}}_{g(b),\mu,sw}/\hat{\mathcal{G}}_{\mu,ex} ) = T_{\hat{\mathsf{C}}_0} \mathcal{B}^*_{g(b),\mu,sw}$ admits a decomposition (see Figure \ref{fig:cubeOfDiagram}):
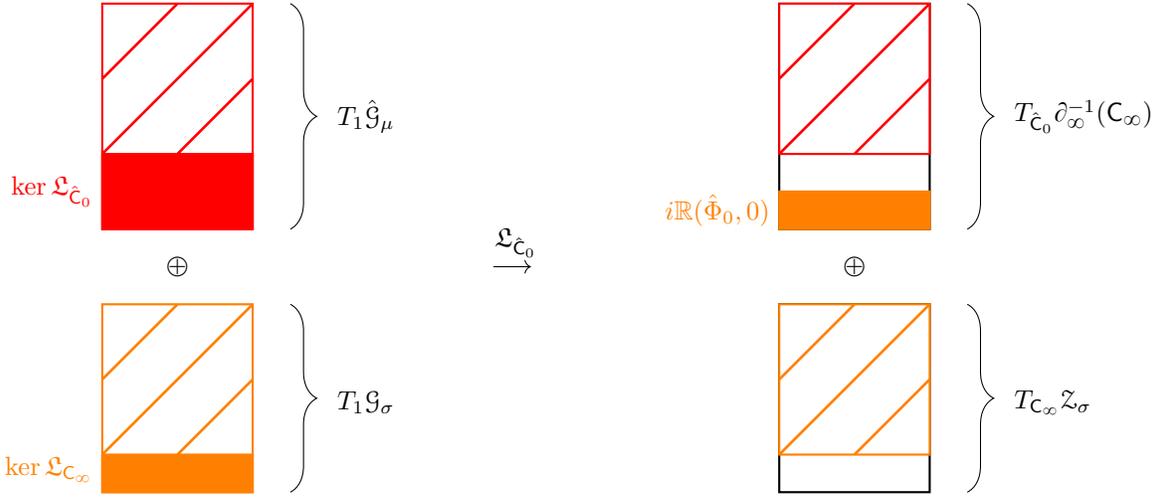
\begin{figure}
    \centering
    \begin{tikzpicture}

    \filldraw[fill=white,draw=red,thick] (0,0) -- (2,0) -- (2,3) -- (0,3) -- cycle;
    \draw[decorate, decoration={brace, amplitude=10pt}] (2.5,3) -- (2.5,0);
    \node[anchor=west] at (3,1.5) {$ T_1 \hat{\mathcal{G}}_{\mu} $};
    \draw[red,  line width=1pt] (0,2) -- (1,3);
    \draw[red,  line width=1pt] (0,1) -- (2,3);
    \draw[red,  line width=1pt] (0,0) -- (2,2);
    \filldraw[fill=red,draw=red,thick] (0,0) -- (2,0) -- (2,1) -- (0,1) -- cycle;
    \node[red, anchor=east] at (0,0.5) {$\ker \mathfrak{L}_{\hat{\mathsf{C}}_0} $};

    \node at (1,-0.5) {$ \oplus $};
    
    \filldraw[fill=white,draw=orange,thick] (0,-3.5) -- (2,-3.5) -- (2,-1) -- (0,-1) -- cycle;
    \draw[decorate, decoration={brace, amplitude=10pt}] (2.5,-1) -- (2.5,-3.5);
    \node[anchor=west] at (3,-2.3) {$ T_1\mathcal{G}_\sigma $};
    \draw[orange,  line width=1pt] (0,-2) -- (1,-1);
    \draw[orange,  line width=1pt] (0,-3) -- (2,-1);
    \draw[orange,  line width=1pt] (1,-3) -- (2,-2);
    \filldraw[fill=orange,draw=orange,thick] (0,-3.5) -- (2,-3.5) -- (2,-3) -- (0,-3) -- cycle;
        \node[orange, anchor=east] at (0,-3.2) {$\ker \mathfrak{L}_{\mathsf{C}_\infty} $};

     	\node[anchor=south] at (5.5,-0.5) {$\mathfrak{L}_{\hat{\mathsf{C}}_0}$};
        
        \draw[->] (5.2,-0.5) -- (5.7,-0.5);
        
        % 复制并平移第一个图
    	\begin{scope}[xshift=9cm]
    	 \filldraw[fill=white,draw=black,thick] (0,0) -- (2,0) -- (2,3) -- (0,3) -- cycle;
    	  \draw[decorate, decoration={brace, amplitude=10pt}] (2.5,3) -- (2.5,0);
    \node[anchor=west] at (3,1.5) {$T_{\hat{\mathsf{C}}_0} \partial_\infty^{-1}(\mathsf{C}_\infty) $};
    	 \filldraw[fill=white,draw=black,thick] (0,-3.5) -- (2,-3.5) -- (2,-1) -- (0,-1) -- cycle;
    \filldraw[fill=white,draw=red,thick] (0,3) -- (2,3) -- (2,1) -- (0,1) -- cycle;
    \draw[red,  line width=1pt] (0,2) -- (1,3);
    \draw[red,  line width=1pt] (0,1) -- (2,3);
    \draw[red,  line width=1pt] (1,1) -- (2,2);
    
    \node at (1,-0.5) {$ \oplus $};
    
    \filldraw[fill=white,draw=orange,thick] (0,-1) -- (2,-1) -- (2,-3) -- (0,-3) -- cycle;
    \draw[decorate, decoration={brace, amplitude=10pt}] (2.5,-1) -- (2.5,-3.5);
    \node[anchor=west] at (3,-2.3) {$T_{\mathsf{C}_\infty}\mathcal{Z}_\sigma$};
    \draw[orange,  line width=1pt] (0,-2) -- (1,-1);
    \draw[orange,  line width=1pt] (0,-3) -- (2,-1);
    \draw[orange,  line width=1pt] (1,-3) -- (2,-2);
    
    \filldraw[fill=orange,draw=orange,thick] (0,0) -- (2,0) -- (2,0.5) -- (0,0.5) -- cycle;
    \node[orange, anchor=east] at (0,0.25) {$i\R (\hat\Phi_0,0) $};
        %\node[anchor=west] at (2,0.8) {$T_{\hat{\mathsf{C}}_0} \hat{\mathcal{C}}_{\mu,sw} \oplus T_bB$};
        
    	\end{scope}
    	
    \end{tikzpicture}
    \caption{Decomposition of $T_1 \hat{\mathcal{G}}_{\mu,ex} $ and $T_{\hat{\mathsf{C}}_0} \hat{\mathcal{C}}_{g(b),\mu,sw}$}
    \label{fig:cubeOfDiagram}
\end{figure} 
\begin{align}
T_{\hat{\mathsf{C}}_0}( \hat{\mathcal{C}}_{g(b),\mu,sw}/\hat{\mathcal{G}}_{\mu,ex} )
&\cong T_{\hat{\mathsf{C}}_0} (\partial_\infty^{-1}(\mathsf{C}_\infty)/\hat{\mathcal{G}}_{\mu}  )/i\R (\hat\Phi_0,0) \oplus T_{\mathsf{C}_\infty}(\mathcal{Z}_s/\mathcal{G}_s)\label{equ:decomp-of-config}\\
\underline{\hat{\mathsf{C}}}_0 &\mapsto (\underline{\hat{\mathsf{C}}}_0 -  \partial_\infty^{-1}\partial_\infty \underline{\hat{\mathsf{C}}}_0)\oplus \partial_\infty \underline{\hat{\mathsf{C}}}_0. \notag
\end{align}
Here $\R (\hat\Phi_0,0) $ is the image of $\ker \mathfrak{L}_{\hat{\mathsf{C}}_0} $ under the map $\mathfrak{L}_{\hat{\mathsf{C}}_0}$, as we have discussed in the remark \ref{rem:imL(kerLinfty)}.
Notice that here $T_{\hat{\mathsf{C}}_0}( \hat{\mathcal{C}}_{g(b),\mu,sw}/\hat{\mathcal{G}}_{\mu,ex} ) = T_{\hat{\mathsf{C}}_0} \mathcal{B}^*_{g(b),\mu,sw}$ and $T_{\mathsf{C}_\infty}(\mathcal{Z}_s/\mathcal{G}_s)= T_{\mathsf{C}_\infty}\M_s$.
With above notations we deduce 
\begin{align*}
D_{(b, \hat{\mathsf{C}}_0, \mathsf{C}_\infty)}F_\eta: 
T_bB \oplus T_{\hat{\mathsf{C}}_0} \mathcal{B}^*_{g(b),\mu,sw}\oplus T_{\mathsf{C}_\infty}\M_s &\to T_0\hat{\mathcal{Y}}_{g(b),\mu}\oplus T_{\mathsf{C}_\infty}\M_s \oplus T_{\mathsf{C}_\infty}\M_s\\
(t,\underline{\hat{\mathsf{C}}}_0,\underline{{\mathsf{C}}}_\infty) &\mapsto  (\widehat{\underline{SW}}_{\eta(b)}(\underline{\hat{\mathsf{C}}}_0) + d_b\eta(t), \partial_\infty \underline{\hat{\mathsf{C}}}_0,\underline{{\mathsf{C}}}_\infty).
\end{align*}
Since $\eta$ is a regular value of $\pi$, $D_{(b, \hat{\mathsf{C}}_0, \mathsf{C}_\infty)}F_\eta$ is surjective. We deduce that 
\begin{align*}
T_bB \oplus T_{\hat{\mathsf{C}}_0} (\partial_\infty^{-1}(\mathsf{C}_\infty)/\hat{\mathcal{G}}_{\mu}  ) &\to T_0\hat{\mathcal{Y}}_{g(b),\mu}\\
(t,\underline{\hat{\mathsf{C}}}_0) &\mapsto  \widehat{\underline{SW}}_{\eta(b)}(\underline{\hat{\mathsf{C}}}_0) + d_b\eta(t)
\end{align*}
must be surjective, otherwise $T_bB \oplus T_{\hat{\mathsf{C}}_0} \mathcal{B}^*_{g(b),\mu,sw}=T_bB \oplus T_{\hat{\mathsf{C}}_0} (\partial_\infty^{-1}(\mathsf{C}_\infty)/\hat{\mathcal{G}}_{\mu}  ) \oplus T_{\mathsf{C}_\infty}(\mathcal{Z}_s/\mathcal{G}_s)$ cannot fill in $T_0\hat{\mathcal{Y}}_{g(b),\mu}\oplus T_{\mathsf{C}_\infty}\M_s$ (see Figure \ref{fig:fill}). Hence the image $d_b\eta(T_b B)$ contains $H^2(F(\hat{\mathsf{C}}_0))$, and this means that
\begin{equation}\label{equ:surjIneq}
\dim H^2(F(\hat{\mathsf{C}}_0))\leq\dim B.
\end{equation}
This result is for any specified $\mathsf{C}_\infty$, but we can use the strategy in the second part of Proposition \ref{prop:noncompactTransversality} to show that this is true for any point on $\M_s$. If the $\S^1$ in $\S^1 \times \S^2 = \partial X_0$ is a trivial loop in $X_0$, however,  that strategy cannot apply. But the result is enough.

Next we show that $\dim H^2(F(\hat{\mathsf{C}}_0))=\dim B$. Recall that We have three differential complexes:
\[\label{equ:complexF2}
0\to T_1 \hat{\mathcal{G}}_{\mu} 
\xrightarrow{\mathfrak{L}_{\hat{\mathsf{C}}_0}} T_{\hat{\mathsf{C}}_0} \partial_\infty^{-1}(\mathsf{C}_\infty) 
\xrightarrow{\widehat{\underline{SW}}_{\hat{\mathsf{C}}_0}} T_0\mathcal{Y}_\mu
\to 0
\tag{$F_{\hat{\mathsf{C}}_0}$}
\]
\[\label{equ:complexK2}
0
\to T_1 \hat{\mathcal{G}}_{\mu,ex} 
\xrightarrow{\frac{1}{2}\mathfrak{L}_{\hat{\mathsf{C}}_0}} T_{\hat{\mathsf{C}}_0} \hat{\mathcal{C}}_{\mu,sw} 
\xrightarrow{\widehat{\underline{SW}}_{\hat{\mathsf{C}}_0}} T_0\mathcal{Y}_\mu
\to 0
\tag{$\widehat{\mathcal{K}}_{\hat{\mathsf{C}}_0}$}
\]
\[\label{equ:complexB2}
0\to T_1\mathcal{G}_\sigma 
\xrightarrow{\frac{1}{2}\mathfrak{L}_{{\mathsf{C}}_\infty}} T_{\mathsf{C}_\infty}\mathcal{Z}_\sigma
\to 0
\to 0
\tag{$B_{\hat{\mathsf{C}}_0}$}
\]
From the exact sequence 
\begin{equation}
0\to \text{\ref{equ:complexF2}}%F_{\hat{\mathsf{C}}_0}
\stackrel{i}{\hookrightarrow} \text{\ref{equ:complexK2}} %\widehat{\mathcal{K}}_{\hat{\mathsf{C}}_0}
\stackrel{\partial_\infty}{\twoheadrightarrow} \text{\ref{equ:complexB2}} %B_{\hat{\mathsf{C}}_0}
\to 0
\tag{\textbf E}
\end{equation}
we deduce
\[
\chi(\widehat{\mathcal{K}}_{\hat{\mathsf{C}}_0} ) = \chi(F_{\hat{\mathsf{C}}_0}) + \chi(B_{\hat{\mathsf{C}}_0}).
\]
In our case $H^0(B_{\hat{\mathsf{C}}_0}) = 1$ since all monopoles on the boundary are reducible, and $H^1(B_{\hat{\mathsf{C}}_0}) = \dim \M_s = 1$. Hence $ \chi(B_{\hat{\mathsf{C}}_0}) = 1-1 =0$. On the other hand $H^0(\widehat{\mathcal{K}}_{\hat{\mathsf{C}}_0} ) = 0$ since all monopoles on $X_0$ are irreducible, and thus $H^0(F_{\hat{\mathsf{C}}_0} ) = 0$ by the long exact sequence (\ref{equ:longExact}). Therefore by the Proposition \ref{prop:virtualDimX0}
\begin{equation}\label{chiF}
-H^1(F_{\hat{\mathsf{C}}_0})+H^2(F_{\hat{\mathsf{C}}_0})= \chi(\widehat{\mathcal{K}}_{\hat{\mathsf{C}}_0}) = -d(\hat{\mathsf{C}}_0)= -1 +\dim B.
\end{equation}
Since $\dim H^0(\widehat{\mathcal{K}}_{\hat{\mathsf{C}}_0} ) = 0$ and $\dim H^0(B_{\hat{\mathsf{C}}_0}) = 1$, $H^1(F_{\hat{\mathsf{C}}_0})$ is at least $1$-dimensional ($i\R(\hat{\Phi}_0,0)$ is the image of $H^0(B_{\hat{\mathsf{C}}_0}) $). Hence by (\ref{equ:surjIneq}) and (\ref{chiF}), $\dim H^2(F_{\hat{\mathsf{C}}_0})= \dim B$ and $\dim H^1(F_{\hat{\mathsf{C}}_0}) = 1$.
\begin{figure}
    \centering
    \begin{tikzpicture}
    \filldraw[fill=gray!30,draw=white] (-1,0) -- (4,0) -- (6,1.2) -- (1,1.2) -- cycle;
        % 坐标轴
        \draw[->] (-1,0) -- (4,0) node[below] {$T_{\hat{\mathsf{C}}_0} (\partial_\infty^{-1}(\mathsf{C}_\infty)/\hat{\mathcal{G}}_{\mu}  ) \oplus T_bB$};
        %\draw[->] (0,0) -- (3,2) node[below] {};
        \draw[->] (0,-2.5) -- (0,2.5) node[left] {$T_{\mathsf{C}_\infty}\M_s$};
        \node[anchor=west] at (2,0.8) {$T_{\hat{\mathsf{C}}_0} \mathcal{B}^*_{g(b),\mu,sw} \oplus T_bB$};

        % 网格线
        %\draw[dashed] (-0.5,-0.5) grid (4.5,3.5);
        
        % 坐标轴上的标签
        %\foreach \x in {1,2,3,4}
           % \draw (\x,-0.1) -- (\x,0.1) node[below=2pt] {$\x$};
        %\foreach \y in {1,2,3}
           % \draw (-0.1,\y) -- (0.1,\y) node[left=2pt] {$\y$};
        
        % 例子：折线
        \draw[red,  line width=1pt] (-1,-2) -- (1,2);
        %\draw[blue, thick] (1,1) -- (2,3) -- (3,2) -- (4,3);
     	\node[anchor=west] at (6,0.3) {$DF_\eta$};
        
        \draw[->] (6.2,0) -- (6.7,0);
        
        % 复制并平移第一个图
    	\begin{scope}[xshift=9cm]
    	% 坐标轴
        \draw[->] (-1,0) -- (4,0) node[below] {$T_0\hat{\mathcal{Y}}_{g(b),\mu}$};
        \draw[->] (0,0) -- (3.5,2) node[left] {$T_{\mathsf{C}_\infty}\M_s$};
        \draw[->] (0,-2.5) -- (0,2.5) node[left] {$T_{\mathsf{C}_\infty}\M_s$};
        %\node[anchor=west] at (2,0.8) {$T_{\hat{\mathsf{C}}_0} \hat{\mathcal{C}}_{\mu,sw} \oplus T_bB$};
        
    	\end{scope}
    	
    \end{tikzpicture}
    \caption{}
    \label{fig:fill}
\end{figure} 
\end{proof}

\begin{theorem}
Under the assumptions in Theorem \ref{thm:cylindricalTransversality}, there exists a set $\mathcal{Z}_{reg}  \subset \mathcal{Z}^{k-1}$ of the second category in the sense of Baire such that in addition to the properties in Theorem \ref{thm:cylindricalTransversality}, all monopoles are of type $1$ (see Figure \ref{fig:type0}).
\end{theorem} 

In our case, we have $\Pi(E_{X_0})$ on one side, and $\Pi(E_{\S^1 \times D^3})$ or $\Pi(E_{D^2 \times \S^2 })$ on the other side. 
We hope each of $\eta_{\S^1 \times D^3}: B \to \Pi(E_{\S^1 \times D^3})$ and $\eta_{D^2 \times \S^2 }: B \to \Pi(E_{D^2 \times \S^2 })$ 
sents each point of $B$ to a fixed PSC metric and vanished perturbation, and this property is preserved after gluing. 
Thus, we have to find an $\eta_{X_0}: B \to \Pi(E_{X_0})$ such that it vanishes on $[R_{\text{vanish}},\infty)$ of the neck.

\begin{prop}\label{prop:perturbationCondition}
    It's possible to choose families perturbations $\eta_{X_0}$ satisfying the following assumptions:
\begin{enumerate}
\item[A1] $\eta_{X_0}$ doesn't meet the wall.
% \item[A2] The embedding of family moduli space of $\hat{N}_1$ in $\mathcal B$ is a submersion in the vertical tangent space, i.e, $T\M(\hat{N}_1)\to T(\mathcal B /B)$ is injective. 
\item[A3] (A5) of Proposition 7.2 in \cite{BK2020} is satisfied.
\item[A4] $\eta_{X_0}: B \to \Pi(E_{X_0})$ is generic in the family sense (in the sense of Theorem \ref{thm:regular}).%: 
%The image of this section intersects with the image of the universal moduli space transversally.
\item[A5] For any $b\in B$, $\eta_{X_0}(b)$ vanishes on $[R_{\text{vanish}},\infty)$ of the neck.
\end{enumerate}
\end{prop}
\begin{proof}
We just replace the dimension condition in the proof of Proposition 7.2 in \cite{BK2020}.
\end{proof}

\section{The proof of $1$-surgery formula for families invariant}

\subsection{Computation of obstruction bundles}\label{sec:obstruction}
Suppose $\sigma = (g,\eta)$ is a regular section in the sense of Theorem \ref{thm:regular}%. Let's reconsider the mainifold $M(F_\eta)$ defined in (\ref{M(Feta)}). Recall that $F_\eta$ defined in (\ref{Feta}) is a section of the bundle $\mathcal{Y}_\eta \to \mathcal{X}$. Let 
%\[
%p:M(F_\eta) \to B
%\] 
%be the projection. By an argument of transversality theory (see \cite{salamon2000spin} Theorem B.16), $b\in B$ is a regular value of $\pi$ iff the restriction of $F_\eta$ over $b$: 
%\begin{align*}
% (F_\eta)_b: \mathcal{X} &\to \mathcal{Y}_\eta= \bigsqcup_{b\in B}\mathcal{D}_{g(b)}/\mathscr{G}\\
%(b,A,\Phi) &\mapsto F(b,A,\Phi,\eta)=[\mathcal{F}_{(g(b),\eta(b))}](A,\Phi)
%\end{align*}
, and $b\in B$ is a point such that $sw_{\sigma(b)}^{-1}$ is nonempty. As the discussion in the proof of Theorem \ref{thm:regular}, either
\begin{align*}
\dim \ker [\underline{sw}_{\sigma(b)}] &= 0  \\
\dim \coker [\underline{sw}_{\sigma(b)}] &= 1-\dim B = -\ind \underline{sw}_{\sigma(b)}
\end{align*}
or
\begin{align*}
\dim \ker [\underline{sw}_{\sigma(b)}] &= 1  \\
\dim \coker [\underline{sw}_{\sigma(b)}] &= -\dim B= 1-\ind \underline{sw}_{\sigma(b)}.
\end{align*}

Alternatively, one can compute $H^2(F(\hat{\mathsf{C}}_0)) $ from the exact sequence \ref{equ:longExact}
\begin{equation}
\cdots
\to H^1_{\hat{\mathsf{C}}_0}
\to H^1(B_{\hat{\mathsf{C}}_0})
\to H^2(F_{\hat{\mathsf{C}}_0}) 
\to H^2_{\hat{\mathsf{C}}_0}
\to 0
\end{equation}
Not like the formula (\ref{equ:obstructionF=kerexD*}), now we have irrducible monopole. Hence 
\[
H^2(F(\hat{\mathsf{C}}_0)) = \text{ker}_{\text{ex}}\D^*_{\hat{A}_0}\oplus H^2_+(\hat{N})\oplus L^2_{top}, 
\] 
As in (\ref{equ:kerex=ker}), we have
\[
\ker_{ex}\D^*_{\hat{A}_0} = \ker_{L^2}\D^*_{\hat{A}_0}.
\]
%%%%%%%%%%%%%%%%%%%%%%%%%%%%%%%%%%%%%%%%%%%%%%%%%%%%%%%%%%%%end

%It's possible to choose family perturbations $\eta_{\hat{N}_1}$ satisfying the following assumptions:
%\begin{enumerate}
%\item[A1] $\eta_{\hat{N}_1}=\eta +L(t)$ doesn't meet the wall.
%\item[A2] The embedding of family moduli space of $\hat{N}_1$ in $\mathcal B$ is a submersion in the vertical tangent space, i.e, $T\M(\hat{N}_1)\to T(\mathcal B /B)$ is injective. 
%\item[A3] (A5) of Proposition 7.2 in \cite{BK2020} is satisfied.
%\end{enumerate}

Let $\hat{N}_1 \simeq X_0$. Let $\eta$ be the perturbation family chosen in Theorem \ref{thm:cylindricalTransversality}. Let $b\in B$ be a point and $\hat{\mathsf{C}}_1$ be an $\eta(b)$-monopole of $\hat{N}_1$. The sequence (\ref{equ:longExact}) becomes 
\begin{equation}\label{L?}
\begin{tikzcd}
& \dim H^0(F_{\hat{\mathsf{C}}_1})=0 \arrow[d]
& \dim H^1(F_{\hat{\mathsf{C}}_1})=1 \arrow[d]
& \dim H^2(F_{\hat{\mathsf{C}}_1})=\dim B \arrow[d] & \\
\arrow[r, phantom, ""{coordinate, name=Y}] & \dim H^0_{\hat{\mathsf{C}}_1}=0  \arrow[d]\arrow[r, phantom, ""{coordinate, name=Z}]
&  \dim H^1_{\hat{\mathsf{C}}_1}=?  \arrow[d]\arrow[r, phantom, ""{coordinate, name=T}]
& \dim H^2_{\hat{\mathsf{C}}_1} = ? \arrow[d] &\\
0   \arrow[ruu,
"",
rounded corners,
to path={
-| (Y) [near end]\tikztonodes
|-  (\tikztotarget)}]  
& \dim H^0(B_{\hat{\mathsf{C}}_1}) = 1  \arrow[ruu,
"",
rounded corners,
to path={
-| (Z) [near end]\tikztonodes
|-  (\tikztotarget)}]  
&\dim H^1(B_{\hat{\mathsf{C}}_1}) = 1\arrow[ruu,
"",
rounded corners,
to path={
-| (T) [near end]\tikztonodes
|-  (\tikztotarget)}]  
& 0 & 
\end{tikzcd}
\tag{\textbf{L?}}
\end{equation}
As we discussed in the proof of Theorem \ref{thm:cylindricalTransversality}, the homomorphism $ H^0(B_{\hat{\mathsf{C}}_1}) \to H^1(F_{\hat{\mathsf{C}}_1})$ is an isomorphism. Since $\chi(\widehat{\mathcal{K}}_{\hat{\mathsf{C}}_1}) = -d(\hat{\mathsf{C}}_1)= -1 +\dim B$, there are only two cases: either $\dim H^1_{\hat{\mathsf{C}}_1} =  0$ and $\dim H^2_{\hat{\mathsf{C}}_1} = \dim B -1$, or $\dim H^1_{\hat{\mathsf{C}}_1}=1$ and $\dim H^2_{\hat{\mathsf{C}}_1} = \dim B$. In fact both cases are possible.

\textbf{Case 0:} $\dim H^1_{\hat{\mathsf{C}}_1} =  0$ and $\dim H^2_{\hat{\mathsf{C}}_1} = \dim B -1$. In this case we have
\begin{equation}\label{L0}
\begin{tikzcd}
& \dim H^0(F_{\hat{\mathsf{C}}_1})=0 \arrow[d]
& \dim H^1(F_{\hat{\mathsf{C}}_1})=1 \arrow[d]
& \dim H^2(F_{\hat{\mathsf{C}}_1})=\dim B \arrow[d] & \\
\arrow[r, phantom, ""{coordinate, name=Y}] & \dim H^0_{\hat{\mathsf{C}}_1}=0  \arrow[d]\arrow[r, phantom, ""{coordinate, name=Z}]
&  \dim H^1_{\hat{\mathsf{C}}_1}=0  \arrow[d]\arrow[r, phantom, ""{coordinate, name=T}]
& \dim H^2_{\hat{\mathsf{C}}_1} = \dim B -1 \arrow[d] &\\
0   \arrow[ruu,
"",
rounded corners,
to path={
-| (Y) [near end]\tikztonodes
|-  (\tikztotarget)}]  
& \dim H^0(B_{\hat{\mathsf{C}}_1}) = 1  \arrow[ruu,
"",
rounded corners,
to path={
-| (Z) [near end]\tikztonodes
|-  (\tikztotarget)}]  
&\dim H^1(B_{\hat{\mathsf{C}}_1}) = 1\arrow[ruu,
"",
rounded corners,
to path={
-| (T) [near end]\tikztonodes
|-  (\tikztotarget)}]  
& 0 & 
\end{tikzcd}
\tag{\textbf{L0}}
\end{equation}
Since the virtual dimension of the moduli space for a fixed parameter is $\dim H^1_{\hat{\mathsf{C}}_1} =  0$, $\hat{\mathsf{C}}_1$ is of type 0 (see Figure \ref{fig:type1}).

In the first row of diagram \ref{equ:3T}, $L^+_1 = \partial_\infty^c (H^1_{\hat{\mathsf{C}}_1})$. Hence $L^+_1$ is certainly $0$. By complementarity equations from the Lagrangian condition%(see (4.1.22) of Section 4.1.5 of Nicolaescu's book)
, we have 
\[
L_1^+\oplus L_1^- =  T_{C_\infty}\M_\sigma.
\]
So $L_1^-$ is $\R$. Let $\hat{N}_2 = \S^1\times D^3$. Then the dimension of $H^2(F_{\hat{\mathsf{C}}_2})$, $L_2^-$, $\mathfrak{C}_1^-$ and $\mathfrak{C}_2^-$ are computed as in \ref{prop:trivialObs}. We deduce $\dim {\H}_r^- = \dim B -1$ by the obstruction diagram \ref{Fig:O0}. Since the obstruction space has one less dimension than the parameter space, we have type 0 configuration $\hat{\mathsf{C}}_1 \#_r \hat{\mathsf{C}}_2$ (see Figure \ref{fig:genericCell}).
\begin{figure}%[h]
\begin{center}
\begin{tikzpicture}[commutative diagrams/every diagram]
\node (P0) at (1.7cm, 0cm) {$0$};
\node (P1) at (5cm, 0cm) {$0$};
\node (P2) at (8.3cm, 0cm) {$0$};

\node (P3) at (0cm, -1.6cm) {$0$};
\node (P4) at (1.7cm, -1.6cm) {$ \ker\Delta_-^c$ };
\node (P41) at (0.8cm, -1.1cm) {$\dim B -1$ };
\node (P5) at (5cm, -1.6cm) {$H^2(F_{\hat{\mathsf{C}}_1}) \oplus H^2(F_{\hat{\mathsf{C}}_2})$};
\node (P51) at (4.3cm, -1.1cm) {$\dim B$};
\node[red] (P52) at (5.5cm, -1.1cm) {$0$};
\node (P6) at (8.3cm,-1.6cm) { $L_1^- + L_2^- $};
\node (P61) at (7.8cm,-1.1cm) { $1$};
\node[red] (P62) at (8.8cm,-1.1cm) { $0$};
\node (P7) at (10cm, -1.6cm) {$0$};

\node (P8) at (0cm, -3.2cm) {$0$};
\node (P9) at (1.7cm, -3.2cm) {${\H}_r^-$ };
\node (P91) at (0.8cm, -2.7cm) {$\dim B -1$ };
\node (P10) at (5cm, -3.2cm) {$\ker_{ex}\hat{\mathcal{T}}^*_{\hat{\mathsf{C}}_1}\oplus \ker_{ex}\hat{\mathcal{T}}^*_{\hat{\mathsf{C}}_2}$};
\node (P11) at (8.3cm,-3.2cm) { $\hat{L}_1^- + \hat{L}_2^-$};
\node (P12) at (10cm, -3.2cm) {$0$};

\node (P13) at (0cm, -4.8cm) {$0$};
\node (P14) at (1.7cm, -4.8cm) {$\ker\Delta_-^0$ };
\node (P141) at (1.2cm, -5.3cm) {$0$};
\node (P15) at (5cm, -4.8cm) {$\mathfrak{C}_1^-\oplus \mathfrak{C}_2^-$};
\node (P151) at (4.5cm, -5.3cm) {$0$};
\node[red] (P152) at (5.5cm, -5.3cm) {$1$};
\node (P16) at (8.3cm,-4.8cm) { $\mathfrak{C}_1^- + \mathfrak{C}_2^-$};
\node (P161) at (7.8cm, -5.3cm) {$0$};
\node[red] (P162) at (8.8cm, -5.3cm) {$1$};
\node (P17) at (10cm, -4.8cm) {$0$};

\node (P18) at (1.7cm,-6.4cm) {$0$};
\node (P19) at (5cm, -6.4cm) {$0$};
\node (P20) at (8.3cm, -6.4cm) {$0$};

\path[commutative diagrams/.cd, every arrow, every label]
(P0) edge node {} (P4)
(P1) edge node {} (P5)
(P2) edge node {} (P6)

(P3) edge node {} (P4)
(P4) edge node {$S_r$} (P5)
(P5) edge node {$\Delta_-^c$} (P6)
(P6) edge node {} (P7)

(P4) edge node {} (P9)
(P5) edge node {} (P10)
(P6) edge node {} (P11)

(P8) edge node {} (P9)
(P9) edge node {$S_r$} (P10)
(P10) edge node {$\Delta_+^c$} (P11)
(P11) edge node {} (P12)

(P9) edge node {} (P14)
(P10) edge node {} (P15)
(P11) edge node {} (P16)

(P13) edge node {} (P14)
(P14) edge node {$S_r$} (P15)
(P15) edge node {$\Delta_-^0$} (P16)
(P16) edge node {} (P17)

(P14) edge node {} (P18)
(P15) edge node {} (P19)
(P16) edge node {} (P20)
;
\end{tikzpicture}
\end{center}
\caption{Obstruction diagram for case 0 and $\hat{N}_2 = \S^1\times D^3$, with dimension for each term.}
    \label{Fig:O0}
\end{figure}

We have to identify $ {\H}_r^-$ explicitly. By the definition of $H^2_{\hat{\mathsf{C}}_1}$ and $H^2(F_{\hat{\mathsf{C}}_1})$ we have 
\begin{align*}
\dim(\widehat{\underline{SW}}T_{\hat{\mathsf{C}}_1}( \hat{\mathcal{C}}_{\mu,sw}/\hat{\mathcal{G}}_{\mu,ex} ))^\bot &= \dim H^2_{\hat{\mathsf{C}}_1} = \dim B -1\\
\dim(\widehat{\underline{SW}}T_{\hat{\mathsf{C}}_1} (\partial_\infty^{-1}(\mathsf{C}_\infty)/\hat{\mathcal{G}}_{\mu}  ))^\bot  &= \dim H^2(F_{\hat{\mathsf{C}}_1})=\dim B
\end{align*}
Namely, the image of 
\[
T_{\hat{\mathsf{C}}_1}( \hat{\mathcal{C}}_{\mu,sw}/\hat{\mathcal{G}}_{\mu,ex} )
\cong
 T_{\hat{\mathsf{C}}_1} (\partial_\infty^{-1}(\mathsf{C}_\infty)/\hat{\mathcal{G}}_{\mu}  )/i\R (\hat\Phi_1,0) \oplus T_{\mathsf{C}_\infty}(\mathcal{Z}_s/\mathcal{G}_s)
\]
under $\widehat{\underline{SW}}_{\hat{\mathsf{C}}_1}$ has one more dimension than the image of 
\[
T_{\hat{\mathsf{C}}_1} (\partial_\infty^{-1}(\mathsf{C}_\infty)/\hat{\mathcal{G}}_{\mu}  )
\]
under $\widehat{\underline{SW}}_{\hat{\mathsf{C}}_1}$. Since $\widehat{\underline{SW}}_{\hat{\mathsf{C}}_0}(i\R (\hat\Phi_1,0) ) =0$ and $\widehat{\underline{SW}}_{\hat{\mathsf{C}}_1}( T_{\mathsf{C}_\infty}(\mathcal{Z}_s/\mathcal{G}_s))$ is $1$-dimensional, this means that
\[
\widehat{\underline{SW}}_{\hat{\mathsf{C}}_1}( T_{\mathsf{C}_\infty}(\mathcal{Z}_s/\mathcal{G}_s)) 
\notin 
\widehat{\underline{SW}}_{\hat{\mathsf{C}}_1} (T_{\hat{\mathsf{C}}_1} (\partial_\infty^{-1}(\mathsf{C}_\infty)/\hat{\mathcal{G}}_{\mu}  )).
\]
Assume that the restriction $H^1(X_0,\R) \to H^1(\S^1\times \S^2,\R)$ is surjective, one can find a closed form $\alpha\in \Omega^1(X_0)$ such that $\partial_\infty \alpha \in T_{\mathsf{C}_\infty}(\mathcal{Z}_s/\mathcal{G}_s)$. Then $d^+\alpha = 0$ so
\[
\widehat{\underline{SW}}_{\hat{\mathsf{C}}_1}( T_{\mathsf{C}_\infty}(\mathcal{Z}_s/\mathcal{G}_s)) 
\in 
\widehat{\underline{SW}}_{\hat{\mathsf{C}}_1} (T_{\hat{\mathsf{C}}_1} (\partial_\infty^{-1}(\mathsf{C}_\infty)/\hat{\mathcal{G}}_{\mu}  )).
\]
Hence in this case (Case 0) $H^1(X_0,\R) \to H^1(\S^1\times \S^2,\R)$ is not surjective, i.e. $L^1_{top}=0$ and $(L^1_{top})^\bot =\R =  H^1(\S^1\times \S^2,\R)$. From the long exact sequence
\begin{equation}\label{equ:boundaryLongExact}
\begin{tikzcd}
H^3(N;\R) 
& H^3(\hat{N};\R)\arrow[l] \arrow[d, phantom, ""{coordinate, name=X}] 
& H^3(\hat{N},N;\R)\arrow[l]\\
H^2(N;\R) \arrow[urr,
"",
rounded corners,
to path={
|- (X) [near end]\tikztonodes
-|  (\tikztotarget)}]  
& H^2(\hat{N};\R)\arrow[l]  \arrow[d, phantom, ""{coordinate, name=Y}] 
& H^2(\hat{N},N;\R)\arrow[l,"f"']\\
H^1(N;\R) \arrow[rru,
"\delta",
rounded corners,
to path={
|- (Y) [near end]\tikztonodes
-|  (\tikztotarget)}]  
& H^1(\hat{N};\R)\arrow[l]
& H^1(\hat{N},N;\R)\arrow[l]
\end{tikzcd}
\tag{\textbf{BL}}
\end{equation}
(where $ H^*(\hat{N},N;\R)$ is the de Rham cohomology with compact support on the cylindrical manifold $\hat{N}$) and the Poincar\'e dual theorem:
\begin{align*}
H^2(N;\R) &\cong H^1(N;\R)\\
\alpha &\mapsto *\alpha
\end{align*}
\begin{equation}\label{equ:PDin2d}
\hat *:H^2(\hat{N},N;\R) \cong H^2(\hat{N};\R) 
\end{equation}
%\begin{align*}
%i:H^2(\hat{N},N;\R) &\cong H_2(\hat{N};\R)\stackrel{\text{UCT}}{\cong} H^2(\hat{N};\R) \\
%\alpha &\mapsto \langle [\hat{N}] , \alpha\rangle
%\end{align*}
\[
 H^3(\hat{N},N;\R) \cong H^1(\hat{N},N;\R)
\]
\[
H^3(\hat{N};\R) \cong  H^1(\hat{N},N;\R)
\]
we deduce that (see Figure \ref{fig:BoundaryExactStructure1})
\[
H^2(N;\R) = *L^1_{top} \oplus L^2_{top}
\]
and 
\[
H^1(N;\R) = L^1_{top} \oplus *L^2_{top},
\]
where 
\[
L^*_{top} := \text{im }H^*(\hat{N};\R) \to H^*({N};\R).
\]
Moreover $L^2_{top}  \cong *L^2_{top} = (L^1_{top})^\bot $ is isomorphic to $\delta( (L^1_{top})^\bot)$, the kernel of the natural forgetful morphism $f :  H^2(\hat{N},N;\R) \to H^2(\hat{N};\R)$ in the long exact sequence (\ref{equ:boundaryLongExact}). %Note that $f$ is not an isomorphism. Actually the isomorphism
%\[
%H^2(\hat{N},N;\R) \cong H_2(\hat{N};\R)\stackrel{\text{UCT}}{\cong} H^2(\hat{N};\R)
%\]
%is given by 
%\[
%[\alpha] \mapsto \langle [\hat{N}] , [\alpha]\rangle^*
%\]
%where $ \langle [\hat{N}] , \alpha\rangle$ is the slant product of the fundamental class of $\hat{N}$ and $\alpha$, and $\langle [\hat{N}] , \alpha\rangle^*$ is the cochain that sends the chain $\langle [\hat{N}] , \alpha\rangle$ to $1$. Then 
%\[
%f(\alpha) \cup \langle [\hat{N}] , \alpha\rangle^* = 1
%\]
%where $\cup$ is the cup product on the chain level.
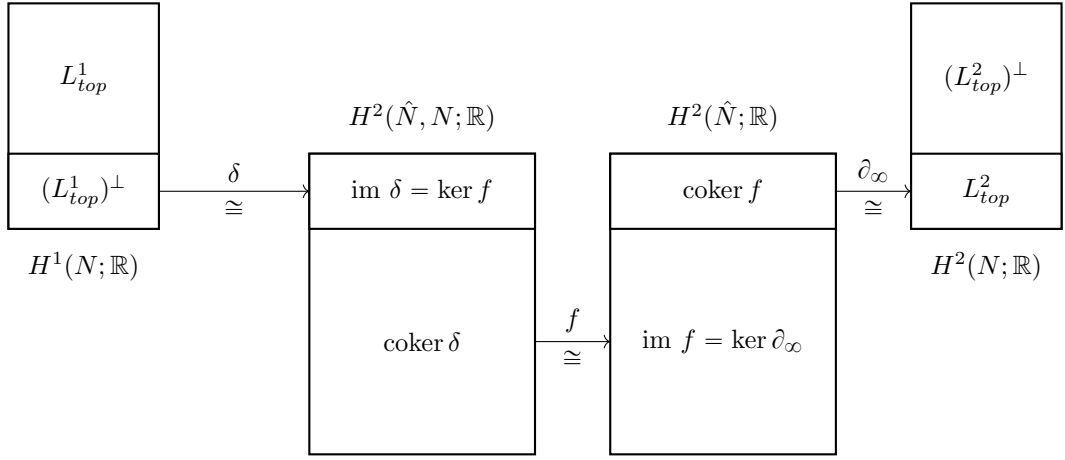
\begin{figure}
    \centering
    \begin{tikzpicture}

    \filldraw[fill=white,draw=black,thick] (0,0) -- (2,0) -- (2,3) -- (0,3) -- cycle; %\draw[decorate, decoration={brace, amplitude=10pt}] (2.5,3) -- (2.5,0);
    \node[black] at (1,2) {$L^1_{top} $};
    \filldraw[fill=white,draw=black,thick] (0,0) -- (2,0) -- (2,1) -- (0,1) -- cycle;
    \node[black] at (1,0.5) {$(L^1_{top})^\bot $};

    \node at (1,-0.5) {$ H^1(N;\R) $};
    \node[anchor=south] at (3,0.5) {$\delta$};
    \draw[->] (2,0.5) -- (4,0.5);
    \node[anchor=north] at (3,0.5) {$\cong$};
    
    \begin{scope}[xshift=4cm]
    \filldraw[fill=white,draw=black,thick] (0,-3) -- (3,-3) -- (3,1) -- (0,1) -- cycle;
    \node at (1.5,-1.5) {$\coker \delta$};
    \filldraw[fill=white,draw=black,thick] (0,0) -- (3,0) -- (3,1) -- (0,1) -- cycle;
     \node at (1.5,0.5) {$\text{im } \delta = \ker f$};
         
     \node at (1.5,1.5) {$ H^2(\hat N, N;\R) $};
     
       \node[anchor=south] at (3.5,-1.5) {$f$};
    \draw[->] (3,-1.5) -- (4,-1.5);
    \node[anchor=north] at (3.5,-1.5) {$\cong$};
        \end{scope}

        \begin{scope}[xshift=8cm]
    \filldraw[fill=white,draw=black,thick] (0,-3) -- (3,-3) -- (3,1) -- (0,1) -- cycle;
    \node at (1.5,-1.5) {$\text{im }f = \ker \partial_\infty$};
    \filldraw[fill=white,draw=black,thick] (0,0) -- (3,0) -- (3,1) -- (0,1) -- cycle;
     \node at (1.5,0.5) {$\coker f$};
         
     \node at (1.5,1.5) {$ H^2(\hat N;\R) $};
        
\node[anchor=south] at (3.5,0.5) {$\partial_\infty$};
    \draw[->] (3,0.5) -- (4,0.5);
    \node[anchor=north] at (3.5,0.5) {$\cong$};
    
        \end{scope}
    	
    	 \begin{scope}[xshift=12cm]
     \filldraw[fill=white,draw=black,thick] (0,0) -- (2,0) -- (2,3) -- (0,3) -- cycle; %\draw[decorate, decoration={brace, amplitude=10pt}] (2.5,3) -- (2.5,0);
    \node[black] at (1,2) {$(L^2_{top})^\bot $};
    \filldraw[fill=white,draw=black,thick] (0,0) -- (2,0) -- (2,1) -- (0,1) -- cycle;
    \node[black] at (1,0.5) {$L^2_{top} $};

    \node at (1,-0.5) {$ H^2(N;\R) $};
        \end{scope}
        
    \end{tikzpicture}
    \caption{Some terms in the long exact sequence (\ref{equ:boundaryLongExact})}
    \label{fig:BoundaryExactStructure1}
\end{figure} 
The isomorphism $\hat *$ in (\ref{equ:PDin2d}) is the Hodge star operator on $\hat N$ and by definition it comes from the cup product, which is a nondegenrate pairing 
\begin{align*}
\int_{\hat{N}} : H^2(\hat{N},N;\R) &\otimes H_2(\hat{N};\R) \to \R\\
[\omega] &\otimes [\tau] \phantom{===.}\mapsto \int_{\hat{N}} \omega \wedge \tau.
\end{align*}
If $f([\alpha])=0$, then there exists a $1$-form $a\in C^1(\hat{N})$ such that $da= \alpha$. Then for any cocycle $\beta \in C^2(\hat{N},N)$,
\begin{align*}
\int_{\hat{N}}\alpha \wedge \beta &= \int_{\hat{N}}da \wedge \beta\\
&= \int_{\hat{N}}d(a \wedge \beta)\pm \int_{\hat{N}}a \wedge d\beta \\
&=  \int_{\partial \hat{N}}a \wedge \beta\pm 0\\
&= 0.
\end{align*}
Hence $[\alpha]$ is in the radical of the intersection form $Q$ of $H^2(\hat{N},N;\R)$. Conversely, if $f[\alpha] \neq 0$, then there exists an element $[\omega]\in H^2(\hat{N},N;\R) $ such that 
\[
\int_{\hat{N}}\omega \wedge\alpha \neq 0.
\]
So $\text{Rad }Q$ of $H^2(\hat{N},N;\R)$ is precisely $\ker f =\delta( (L^1_{top})^\bot) \cong *L^2_{top}$.

From above discussion and based on Figure \ref{fig:BoundaryExactStructure1}, we have Figure \ref{fig:BoundaryExactStructure2}, where $\partial_\infty^0 = \partial_\infty -\partial_\infty^c $ and $\partial_\infty^c $ is the contraction by $\partial_t$ and then taking $\partial_\infty$. In particular, one has an exact sequence
\begin{equation}\label{L2exsequence}
 1\to \left.\ker_{L^2}(\hat{d}+\hat{d}^*)\right|_{\Omega^2_+} \to  \left.\ker_{ex}(\hat{d}+\hat{d}^*)\right|_{\Omega^2_+}  \stackrel{\partial_\infty^0}{\to} L^2_{top} \to 1.
 \end{equation}

 \begin{figure}
    \centering
    \begin{tikzpicture}

    \filldraw[fill=white,draw=black,thick] (0,0) -- (2,0) -- (2,3) -- (0,3) -- cycle; %\draw[decorate, decoration={brace, amplitude=10pt}] (2.5,3) -- (2.5,0);
    \node[black] at (1,2) {$L^1_{top} $};
    \filldraw[fill=white,draw=black,thick] (0,0) -- (2,0) -- (2,1) -- (0,1) -- cycle;
    \node[black] at (1,0.5) {$(L^1_{top})^\bot $};

    \node at (1,-0.5) {$ H^1(N;\R) $};
    \node[anchor=south] at (2.5,0.5) {$\delta$};
    \draw[->] (2,0.5) -- (3,0.5);
    \node[anchor=north] at (2.5,0.5) {$\cong$};
    
    \begin{scope}[xshift=3cm]
    \filldraw[fill=white,draw=black,thick] (0,-3) -- (3,-3) -- (3,1) -- (0,1) -- cycle;
    \node at (1.5,-1) {$Q_+$};
     \node[anchor=south] at (4,-1) {$\hat * = f$};
    \draw[->] (3,-1) -- (5,-1);
    
    \node at (1.5,-2) {$Q_-$};
    \node[anchor=south] at (4,-2) {$\hat * = -f$};
    \draw[->] (3,-2) -- (5,-2);
    \filldraw[fill=white,draw=black,thick] (0,0) -- (3,0) -- (3,1) -- (0,1) -- cycle;
     \node at (1.5,0.5) {$\text{Rad } Q$};
     \node[anchor=south] at (4,1.3) {$f $};
     \draw[->] (3,1) -- (4.7,1.5);
      \node[anchor=west] at (4.7,1.5) {$0 \in$};
      \node[anchor=south] at (4,0.5) {$\hat * $};
       \node[anchor=north] at (4,0.5) {$\cong$};
    \draw[->] (3,0.5) -- (5,0.5);
         
     \node at (1.5,1.5) {$ H^2(\hat N, N;\R) $};
     
       %\node[anchor=south] at (3.5,-1.5) {$f$};
   % \draw[->] (3,-1.5) -- (4,-1.5);
    %\node[anchor=north] at (3.5,-1.5) {$\cong$};
        \end{scope}

        \begin{scope}[xshift=8cm]
    \filldraw[fill=white,draw=black,thick] (0,-3) -- (3,-3) -- (3,1) -- (0,1) -- cycle;
    \node at (1.5,-1) {$\left.\ker_{L^2}(\hat{d}+\hat{d}^*)\right|_{\Omega^2_+} $};
    \node at (1.5,-2) {$\left.\ker_{L^2}(\hat{d}+\hat{d}^*)\right|_{\Omega^2_-} $};
    \filldraw[fill=white,draw=black,thick] (0,0) -- (3,0) -- (3,1) -- (0,1) -- cycle;
     \node at (1.5,0.5) {$\coker f$};
         
     \node at (1.5,1.5) {$\ker_{ex}(\hat{d}+\hat{d}^*)$};
        
\node[anchor=south] at (3.5,0.5) {$\partial_\infty^0$};
    \draw[->] (3,0.5) -- (4,0.5);
    \node[anchor=north] at (3.5,0.5) {$\cong$};
    
        \end{scope}
    	
    	 \begin{scope}[xshift=12cm]
     \filldraw[fill=white,draw=black,thick] (0,0) -- (2,0) -- (2,3) -- (0,3) -- cycle; %\draw[decorate, decoration={brace, amplitude=10pt}] (2.5,3) -- (2.5,0);
    \node[black] at (1,2) {$(L^2_{top})^\bot $};
    \filldraw[fill=white,draw=black,thick] (0,0) -- (2,0) -- (2,1) -- (0,1) -- cycle;
    \node[black] at (1,0.5) {$L^2_{top} $};

    \node at (1,-0.5) {$ H^2(N;\R) $};
        \end{scope}
        
    \end{tikzpicture}
    \caption{}
    \label{fig:BoundaryExactStructure2}
\end{figure} 

By the computation of the \textbf{ASD} operator $d^+ \oplus d^*
$ (see Example 4.1.24), one has
\[
\ker_{ex}\textbf{ASD}^* = \left.\ker_{ex}(\hat{d}+\hat{d}^*)\right|_{\Omega^2_+} \oplus H^0(\hat{N};\R).
\]
and 
\[
\partial_\infty \ker_{ex}(\textbf{ASD}^*) \cong L^2_{top} \oplus L^0_{top}.
\]

%$(L^1_{top})^\bot$ is the dual of $L^2_{top}$, i.e. $(L^1_{top})^\bot=*L^2_{top}$. 
Let $\R_+ \times \S^1\times \S^2$ be the neck of $\hat{N}_1 \simeq X_0$. Let $t$ be the coordinate of $\R_+$. Let $\theta$ be the coordinate of $\S^1$. 
%Let $(\theta,\phi)$ be the coordinate of $\S^2$. Let $\sin \theta d\theta\wedge d\phi$ be the volumn form on $\S^2$. Then $L^2_{top} = \R$ is generated by $\sin \theta d\theta\wedge d\phi$ and 
Then $H^1(\S^1\times \S^2;\R) = \R$ is generated by $d\theta$. Let 
\begin{equation}\label{equ:baseDirectionVector}
\alpha = \beta(t)d\theta \in \Omega^1(\hat N)
\end{equation}
where $\beta(t)$ is the cutoff function defined in the begining of the subsection . Then $\partial_\infty \alpha$ generates $T_{\mathsf{C}_\infty}(\mathcal{Z}_s/\mathcal{G}_s)$. Notice that $[d\alpha]$ is precisely $\delta[d\theta]$ by the definition of the connecting homomorphism $\delta$ in the sequence (\ref{equ:boundaryLongExact}).

As before we regard $T_{\mathsf{C}_\infty}(\mathcal{Z}_s/\mathcal{G}_s)$ as a subspace of $T_{\hat{\mathsf{C}}_0}( \hat{\mathcal{C}}_{g(b),\mu,sw}/\hat{\mathcal{G}}_{\mu,ex} )$. So $\alpha$ generates $T_{\mathsf{C}_\infty}(\mathcal{Z}_s/\mathcal{G}_s)$. Hence 
\[
d^+\alpha = \frac{\hat * d + d}{2}\alpha
\]
generates $\widehat{\underline{SW}}_{\hat{\mathsf{C}}_1}( T_{\mathsf{C}_\infty}(\mathcal{Z}_s/\mathcal{G}_s)) $. 
As $[d\alpha] = \delta [d\theta]$ for $[d\theta] \in (L^1_{top})^\bot $ , the projection of $\hat * d \alpha$ to $\ker_{ex}(\hat{d}+\hat{d}^*)$, denoted by $\H(\hat * d \alpha)$, is nonzero (see Figure \ref{fig:BoundaryExactStructure2}). On the other hand, $\H(d \alpha)$ is zero since $f([d\alpha]) = f\circ \delta [d\theta] = 0$. In conclusion, $d^+\alpha$ projects to a nonzero element of $\ker_{ex}(\hat{d}+\hat{d}^*)$, and in addition, $\partial_\infty^0 $ sends this element to $L^2_{top}=\R$.

In the right top of the obstruction diagram (Figure \ref{Fig:O0}) we have the map $\partial_\infty^c: H^2(F_{\hat{\mathsf{C}}_1})  \to L_1^-$. Here $H^2(F_{\hat{\mathsf{C}}_1})$ is a subspace of 
\[
\left.\ker_{ex}(\hat{d}+\hat{d}^*)\right|_{\Omega^2_+}  \subset \ker_{ex}\textbf{ASD}^*\] 
(it's a proper subspace if the kernel of the twisted Dirac operator is nontrivial), and $ L_1^- = (L^1_{top})^\bot = *L^2_{top}  \cong L^2_{top}$. For a self dual two form, $\partial_\infty^c = * \partial_\infty^0$. 
Hence 
\begin{equation}\label{equ:image-of-fiber-nonzero}
\partial_\infty^c \H (\widehat{\underline{SW}}_{\hat{\mathsf{C}}_1}( T_{\mathsf{C}_\infty}(\mathcal{Z}_s/\mathcal{G}_s))) \neq 0
\end{equation}
is nonzero. 
Note that previously in (\ref{equ:decomp-of-config}) we choose $T_{\mathsf{C}_\infty}(\mathcal{Z}_s/\mathcal{G}_s)$ to be the $L^2$-complement of 
\[
T_{\hat{\mathsf{C}}_0} (\partial_\infty^{-1}(\mathsf{C}_\infty)/\hat{\mathcal{G}}_{\mu}  )/i\R (\hat\Phi_0,0)
\] 
in $T_{\hat{\mathsf{C}}_0}( \hat{\mathcal{C}}_{g(b),\mu,sw}/\hat{\mathcal{G}}_{\mu,ex} )$. 
Now we can choose another $T_{\mathsf{C}_\infty}(\mathcal{Z}_s/\mathcal{G}_s)$ such that it still satisfies (\ref{equ:decomp-of-config}) and in addition, 
\[
\H (\widehat{\underline{SW}}_{\hat{\mathsf{C}}_1}( T_{\mathsf{C}_\infty}(\mathcal{Z}_s/\mathcal{G}_s))) \subset H^2(F_{\hat{\mathsf{C}}_1}).
\]
This is possible because $H^2(F_{\hat{\mathsf{C}}_1})$ is by defintion the $L^2$-complement of $\widehat{\underline{SW}}_{\hat{\mathsf{C}}_0} (T_{\hat{\mathsf{C}}_0} (\partial_\infty^{-1}(\mathsf{C}_\infty)/\hat{\mathcal{G}}_{\mu}  ))$.

Recall that Figure \ref{Fig:O0} shows that
\begin{equation}\label{equ:H-=kerDelta-c}
 {\H}_r^- \cong \ker\Delta_-^c.
 \end{equation}
Therefore by (\ref{equ:image-of-fiber-nonzero}) $ {\H}_r^-$ is isomorphic to the $L^2$ complement of $\H (\widehat{\underline{SW}}_{\hat{\mathsf{C}}_1}( T_{\mathsf{C}_\infty}(\mathcal{Z}_s/\mathcal{G}_s))) $ in $H^2(F_{\hat{\mathsf{C}}_1})$. Note that this isomorphism is given by gluing the obstruction spaces of $\hat{N}_1$ and $\Hat{N}_2$ (which is trivial).
%By (4.1.29), 
%\[
%\partial^c_\infty \ker_{ex}(\textbf{ASD}^*) = *L^2_{top}
%\]

Recall that we have chosen the perturbation family $\eta$ such that \[
d_b\eta(T_bB)
\oplus 
\widehat{\underline{SW}}_{\hat{\mathsf{C}}_1} (T_{\hat{\mathsf{C}}_1} (\partial_\infty^{-1}(\mathsf{C}_\infty)/\hat{\mathcal{G}}_{\mu}  )) 
=\hat{\mathcal{Y}}_{g(b),\mu}.
\]
Choose $v\in T_bB$ such that 
\[
\widehat{\underline{SW}}_{\hat{\mathsf{C}}_0}( T_{\mathsf{C}_\infty}(\mathcal{Z}_s/\mathcal{G}_s)) 
\in 
\R d_b\eta(v)
 \oplus
  \widehat{\underline{SW}}_{\hat{\mathsf{C}}_0} (T_{\hat{\mathsf{C}}_0} (\partial_\infty^{-1}(\mathsf{C}_\infty)/\hat{\mathcal{G}}_{\mu}  )).
  \]
Fix any Riemann metric on $B$. Let $V\subset T_bB$ be the othogonal complement of $\R v$, then $\dim V = \dim B -1$ (see Figure \ref{fig:Type0image}). Observe that 
\begin{equation}\label{equ:P-minus-iso-1}
d_b\eta(V) \cong \H_r^-,
\end{equation}
and this isomorphism is given by first gluing perturbations on $\hat{N}_1$ and $\hat{N}_2$ (which is $0$), then projecting it to $\H_r^-$.
 \begin{figure}
    \centering
    \begin{tikzpicture}
    % 坐标轴
        \draw[->] (0,0) -- (3,0) node[below] {$\R v$};
        \draw[->] (0,0) -- (-2.2,-1.1) node[below] {$V\subset T_bB$};%\draw[->] (0,0) -- (-2.5,-1.4) node[right] {$\text{  }V\subset d_b\eta(T_bB)$};
        \draw[->] (0,0) -- (0,3) node[left] {$T_{\hat{\mathsf{C}}_1}( \hat{\mathcal{C}}_{\mu,sw}/\hat{\mathcal{G}}_{\mu,ex} )$};
        \draw[red,  line width=1pt] (0,0) -- (2,2.8) node[right] {$T\M$};
        %\draw[red, dashed] (2,0) -- (2,2.8);
        \node[anchor=west] at (0,-1) {$T_b B$};
        
     	\node[anchor=west] at (5,0.3) {$DF_\eta$};
        
        \draw[->] (5.2,0) -- (5.7,0);
        
        % 复制并平移第一个图
    	\begin{scope}[xshift=9cm]
    	% 坐标轴
        \draw[->] (0,0) -- (3,0) node[below] {$\widehat{\underline{SW}}_{\hat{\mathsf{C}}_0}( T_{\mathsf{C}_\infty}(\mathcal{Z}_s/\mathcal{G}_s)) $};
        \draw[->] (0,0) -- (-2.2,-1.1) node[below] {$ d_b\eta(V)$};%\draw[->] (0,0) -- (-2.5,-1.4) node[right] {$\text{  }V\subset d_b\eta(T_bB)$};
        \draw[->] (0,0) -- (0,3) node[left] {$\widehat{\underline{SW}}_{\hat{\mathsf{C}}_1} (T_{\hat{\mathsf{C}}_1} (\partial_\infty^{-1}(\mathsf{C}_\infty)/\hat{\mathcal{G}}_{\mu}  ))$};
        \draw[->] (0,0) -- (2,2.8) node[right] {$\R d_b\eta(v)$};
        \draw[dashed] (2,0) -- (2,2.8);
        %\node[anchor=west] at (2,0.8) {$T_{\hat{\mathsf{C}}_0} \hat{\mathcal{C}}_{\mu,sw} \oplus T_bB$};
        
    	\end{scope}
    	
    \end{tikzpicture}
    \caption{Case 0 corresponds to Type 0 in Figure \ref{fig:type1}. Remember the othogonal complement of $\widehat{\underline{SW}}_{\hat{\mathsf{C}}_0}( T_{\mathsf{C}_\infty}(\mathcal{Z}_s/\mathcal{G}_s)) $ is $ H^2(F(\hat{\mathsf{C}}_0))$.}
    \label{fig:Type0image}
\end{figure}
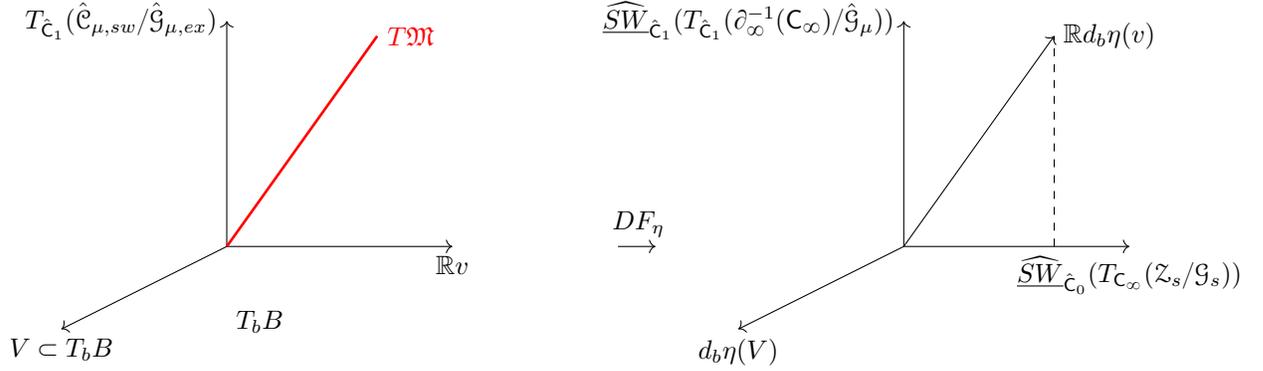 

Now let $\hat{N}_2 = D^2\times \S^2$. From Proposition \ref{prop:trivialObs}, we have Figure \ref{Fig:O0-D2}. Surprisingly, the dimesion of the obstruction space is different from the one in Figure \ref{Fig:O0}. Since the dimension of the obstruction space matches the dimension of the parameter space, we have type 1 configuration $\hat{\mathsf{C}}_1 \#_r \hat{\mathsf{C}}_2$ (see Figure \ref{fig:genericCell}).
\begin{figure}
\begin{center}
\begin{tikzpicture}[commutative diagrams/every diagram]
\node (P0) at (1.7cm, 0cm) {$0$};
\node (P1) at (5cm, 0cm) {$0$};
\node (P2) at (8.3cm, 0cm) {$0$};

\node (P3) at (0cm, -1.6cm) {$0$};
\node (P4) at (1.7cm, -1.6cm) {$ \ker\Delta_-^c$ };
\node (P41) at (1cm, -1.1cm) {$\dim B$ };
\node (P5) at (5cm, -1.6cm) {$H^2(F_{\hat{\mathsf{C}}_1}) \oplus H^2(F_{\hat{\mathsf{C}}_2})$};
\node (P51) at (4.3cm, -1.1cm) {$\dim B$};
\node[orange] (P52) at (5.5cm, -1.1cm) {$1$};
\node (P6) at (8.3cm,-1.6cm) { $L_1^- + L_2^- $};
\node (P61) at (7.8cm,-1.1cm) { $1$};
\node[orange] (P62) at (8.8cm,-1.1cm) { $1$};
\node (P7) at (10cm, -1.6cm) {$0$};

\node (P8) at (0cm, -3.2cm) {$0$};
\node (P9) at (1.7cm, -3.2cm) {${\H}_r^-$ };
\node (P91) at (1cm, -2.7cm) {$\dim B $ };
\node (P10) at (5cm, -3.2cm) {$\ker_{ex}\hat{\mathcal{T}}^*_{\hat{\mathsf{C}}_1}\oplus \ker_{ex}\hat{\mathcal{T}}^*_{\hat{\mathsf{C}}_2}$};
\node (P11) at (8.3cm,-3.2cm) { $\hat{L}_1^- + \hat{L}_2^-$};
\node (P12) at (10cm, -3.2cm) {$0$};

\node (P13) at (0cm, -4.8cm) {$0$};
\node (P14) at (1.7cm, -4.8cm) {$\ker\Delta_-^0$ };
\node (P141) at (1.2cm, -5.3cm) {$0$};
\node (P15) at (5cm, -4.8cm) {$\mathfrak{C}_1^-\oplus \mathfrak{C}_2^-$};
\node (P151) at (4.5cm, -5.3cm) {$0$};
\node[orange] (P152) at (5.5cm, -5.3cm) {$1$};
\node (P16) at (8.3cm,-4.8cm) { $\mathfrak{C}_1^- + \mathfrak{C}_2^-$};
\node (P161) at (7.8cm, -5.3cm) {$0$};
\node[orange] (P162) at (8.8cm, -5.3cm) {$1$};
\node (P17) at (10cm, -4.8cm) {$0$};

\node (P18) at (1.7cm,-6.4cm) {$0$};
\node (P19) at (5cm, -6.4cm) {$0$};
\node (P20) at (8.3cm, -6.4cm) {$0$};

\path[commutative diagrams/.cd, every arrow, every label]
(P0) edge node {} (P4)
(P1) edge node {} (P5)
(P2) edge node {} (P6)

(P3) edge node {} (P4)
(P4) edge node {$S_r$} (P5)
(P5) edge node {$\Delta_-^c$} (P6)
(P6) edge node {} (P7)

(P4) edge node {} (P9)
(P5) edge node {} (P10)
(P6) edge node {} (P11)

(P8) edge node {} (P9)
(P9) edge node {$S_r$} (P10)
(P10) edge node {$\Delta_+^c$} (P11)
(P11) edge node {} (P12)

(P9) edge node {} (P14)
(P10) edge node {} (P15)
(P11) edge node {} (P16)

(P13) edge node {} (P14)
(P14) edge node {$S_r$} (P15)
(P15) edge node {$\Delta_-^0$} (P16)
(P16) edge node {} (P17)

(P14) edge node {} (P18)
(P15) edge node {} (P19)
(P16) edge node {} (P20)
;
\end{tikzpicture}
\end{center}
\caption{Obstruction diagram for case 0 and $\hat{N}_2 = D^2\times \S^2$, with dimension for each term.}
\label{Fig:O0-D2}
\end{figure}

\begin{equation}
\xymatrix{
&
0\ar[d] &
0\ar[d] & 
0 \ar[d]& 
\\
0\ar[r] &
\ker\Delta_-^c \ar[d] \ar[r]^(0.35){S_r} &
H^2(F_{\hat{\mathsf{C}}_1}) \oplus H^2(F_{\hat{\mathsf{C}}_2}) \ar[d] \ar[r]^(0.62){\Delta_-^c} &
L_1^- + L_2^- \ar[d]\ar[r] & 
0\\
0\ar[r] &
{\H}_r^- \ar[d] \ar[r]^(0.3){S_r} &
\ker_{ex}\hat{\mathcal{T}}^*_{\hat{\mathsf{C}}_1}\oplus \ker_{ex}\hat{\mathcal{T}}^*_{\hat{\mathsf{C}}_2} \ar[d] \ar[r]^(0.64){\Delta_+^c} &
\hat{L}_1^- + \hat{L}_2^- \ar[d]\ar[r] &
 0 \\
0\ar[r] &
\ker\Delta_-^0 \ar[d] \ar[r]^{S_r} &
\mathfrak{C}_1^-\oplus \mathfrak{C}_2^- \ar[d] \ar[r]^{\Delta_-^0} &
\mathfrak{C}_1^- + \mathfrak{C}_2^- \ar[d]\ar[r] & 
0 \\
&
0 &
0& 
0& 
\\
}
\tag{\textbf{O}}
\end{equation}

\textbf{Case 1:} $\dim H^1_{\hat{\mathsf{C}}_1} = 1$ and $\dim H^2_{\hat{\mathsf{C}}_1} = \dim B$. Since the virtual dimension of the moduli space for a fixed parameter is $\dim H^1_{\hat{\mathsf{C}}_1} =  1$, $\hat{\mathsf{C}}_1$ is of type 1 (see Figure \ref{fig:type1}). Now the diagram (\ref{L?}) becomes 
\begin{equation}\label{L1}
\begin{tikzcd}
& \dim H^0(F_{\hat{\mathsf{C}}_1})=0 \arrow[d]
& \dim H^1(F_{\hat{\mathsf{C}}_1})=1 \arrow[d]
& \dim H^2(F_{\hat{\mathsf{C}}_1})=\dim B \arrow[d] & \\
\arrow[r, phantom, ""{coordinate, name=Y}] & \dim H^0_{\hat{\mathsf{C}}_1}=0  \arrow[d]\arrow[r, phantom, ""{coordinate, name=Z}]
&  \dim H^1_{\hat{\mathsf{C}}_1}=1  \arrow[d]\arrow[r, phantom, ""{coordinate, name=T}]
& \dim H^2_{\hat{\mathsf{C}}_1} = \dim B \arrow[d] &\\
0   \arrow[ruu,
"",
rounded corners,
to path={
-| (Y) [near end]\tikztonodes
|-  (\tikztotarget)}]  
& \dim H^0(B_{\hat{\mathsf{C}}_1}) = 1  \arrow[ruu,
"",
rounded corners,
to path={
-| (Z) [near end]\tikztonodes
|-  (\tikztotarget)}]  
&\dim H^1(B_{\hat{\mathsf{C}}_1}) = 1\arrow[ruu,
"",
rounded corners,
to path={
-| (T) [near end]\tikztonodes
|-  (\tikztotarget)}]  
& 0 & 
\end{tikzcd}
\tag{\textbf{L1}}
\end{equation}
We deduce that the map $H^1_{\hat{\mathsf{C}}_1} \to H^1(B_{\hat{\mathsf{C}}_1})$ is surjective. On the other hand, the projection of the paramertized moduli space to the parameter space has trivial derivative. Conbined this with the compactness result, we have
\begin{prop}\label{discreteCircle}
Under the dimension assumption (\ref{equ:dimAssumption}), for a generic perturbation family $\eta$, there exists a finite set of points $b_1,\cdots, b_n$ on $B$, such that the paramertrized moduli space for $\hat{N}_1 \simeq X_0$ are in the fiber of $b_1,\cdots, b_n$.
\end{prop}

Now we compute the obstruction space for the closed manifolds obtained by gluing. In the first row of diagram \ref{equ:3T}, 
\begin{equation}\label{equ:L+=R}
L^+_1 =\text{im } (H^1_{\hat{\mathsf{C}}_1}) = \R
\end{equation}
since the map $H^1_{\hat{\mathsf{C}}_1} \to L^+_1$ is by definition $H^1_{\hat{\mathsf{C}}_1} \to H^1(B_{\hat{\mathsf{C}}_1})$. By complementarity equations from the Lagrangian condition%(see (4.1.22) of Section 4.1.5 of Nicolaescu's book)
, we have 
\[
L_1^+\oplus L_1^- =  T_{C_\infty}\M_\sigma.
\]
So $L_1^-=0$. Let $\hat{N}_2 = \S^1\times D^3$. Then we deduce $\dim {\H}_r^- = \dim B $ by the following obstruction diagram:
\begin{center}
\begin{tikzpicture}[commutative diagrams/every diagram]\label{equ:3O-param-case1-1} 
\node (P0) at (1.7cm, 0cm) {$0$};
\node (P1) at (5cm, 0cm) {$0$};
\node (P2) at (8.3cm, 0cm) {$0$};

\node (P3) at (0cm, -1.6cm) {$0$};
\node (P4) at (1.7cm, -1.6cm) {$ \ker\Delta_-^c$ };
\node (P41) at (1cm, -1.1cm) {$\dim B$ };
\node (P5) at (5cm, -1.6cm) {$H^2(F_{\hat{\mathsf{C}}_1}) \oplus H^2(F_{\hat{\mathsf{C}}_2})$};
\node (P51) at (4.3cm, -1.1cm) {$\dim B$};
\node[red] (P52) at (5.5cm, -1.1cm) {$0$};
\node (P6) at (8.3cm,-1.6cm) { $L_1^- + L_2^- $};
\node (P61) at (7.8cm,-1.1cm) { $0$};
\node[red] (P62) at (8.8cm,-1.1cm) { $0$};
\node (P7) at (10cm, -1.6cm) {$0$};

\node (P8) at (0cm, -3.2cm) {$0$};
\node (P9) at (1.7cm, -3.2cm) {${\H}_r^-$ };
\node (P91) at (1cm, -2.7cm) {$\dim B$ };
\node (P10) at (5cm, -3.2cm) {$\ker_{ex}\hat{\mathcal{T}}^*_{\hat{\mathsf{C}}_1}\oplus \ker_{ex}\hat{\mathcal{T}}^*_{\hat{\mathsf{C}}_2}$};
\node (P11) at (8.3cm,-3.2cm) { $\hat{L}_1^- + \hat{L}_2^-$};
\node (P12) at (10cm, -3.2cm) {$0$};

\node (P13) at (0cm, -4.8cm) {$0$};
\node (P14) at (1.7cm, -4.8cm) {$\ker\Delta_-^0$ };
\node (P141) at (1.2cm, -5.3cm) {$0$};
\node (P15) at (5cm, -4.8cm) {$\mathfrak{C}_1^-\oplus \mathfrak{C}_2^-$};
\node (P151) at (4.5cm, -5.3cm) {$0$};
\node[red] (P152) at (5.5cm, -5.3cm) {$1$};
\node (P16) at (8.3cm,-4.8cm) { $\mathfrak{C}_1^- + \mathfrak{C}_2^-$};
\node (P161) at (7.8cm, -5.3cm) {$0$};
\node[red] (P162) at (8.8cm, -5.3cm) {$1$};
\node (P17) at (10cm, -4.8cm) {$0$};

\node (P18) at (1.7cm,-6.4cm) {$0$};
\node (P19) at (5cm, -6.4cm) {$0$};
\node (P20) at (8.3cm, -6.4cm) {$0$};

\path[commutative diagrams/.cd, every arrow, every label]
(P0) edge node {} (P4)
(P1) edge node {} (P5)
(P2) edge node {} (P6)

(P3) edge node {} (P4)
(P4) edge node {$S_r$} (P5)
(P5) edge node {$\Delta_-^c$} (P6)
(P6) edge node {} (P7)

(P4) edge node {} (P9)
(P5) edge node {} (P10)
(P6) edge node {} (P11)

(P8) edge node {} (P9)
(P9) edge node {$S_r$} (P10)
(P10) edge node {$\Delta_+^c$} (P11)
(P11) edge node {} (P12)

(P9) edge node {} (P14)
(P10) edge node {} (P15)
(P11) edge node {} (P16)

(P13) edge node {} (P14)
(P14) edge node {$S_r$} (P15)
(P15) edge node {$\Delta_-^0$} (P16)
(P16) edge node {} (P17)

(P14) edge node {} (P18)
(P15) edge node {} (P19)
(P16) edge node {} (P20)
;
\end{tikzpicture}
\end{center}
Since the dimension of the obstruction space matches the dimension of the parameter space, we have type 1 configuration $\hat{\mathsf{C}}_1 \#_r \hat{\mathsf{C}}_2$ (see Figure \ref{fig:genericCell}).

Now %in the sequence (\ref{L2exsequence}),
we have 
\begin{align*}
\dim(\widehat{\underline{SW}}T_{\hat{\mathsf{C}}_1}( \hat{\mathcal{C}}_{\mu,sw}/\hat{\mathcal{G}}_{\mu,ex} ))^\bot &= \dim H^2_{\hat{\mathsf{C}}_1} = \dim B, \\
\dim(\widehat{\underline{SW}}T_{\hat{\mathsf{C}}_1} (\partial_\infty^{-1}(\mathsf{C}_\infty)/\hat{\mathcal{G}}_{\mu}  ))^\bot  &= \dim H^2(F_{\hat{\mathsf{C}}_1})=\dim B.
\end{align*}
Hence
\[
\widehat{\underline{SW}}_{\hat{\mathsf{C}}_1}( T_{\mathsf{C}_\infty}(\mathcal{Z}_s/\mathcal{G}_s)) 
\in 
\widehat{\underline{SW}}_{\hat{\mathsf{C}}_1} (T_{\hat{\mathsf{C}}_1} (\partial_\infty^{-1}(\mathsf{C}_\infty)/\hat{\mathcal{G}}_{\mu}  )).
\]
and 
\begin{equation}
 {\H}_r^- \cong \ker\Delta_-^c \cong H^2(F_{\hat{\mathsf{C}}_1}) = (\widehat{\underline{SW}}T_{\hat{\mathsf{C}}_1} (\partial_\infty^{-1}(\mathsf{C}_\infty)/\hat{\mathcal{G}}_{\mu}  ))^\bot.
 \end{equation}
Agian we have chosen the perturbation family $\eta$ such that \[
d_b\eta(T_bB)
\oplus 
\widehat{\underline{SW}}_{\hat{\mathsf{C}}_1} (T_{\hat{\mathsf{C}}_1} (\partial_\infty^{-1}(\mathsf{C}_\infty)/\hat{\mathcal{G}}_{\mu}  )) 
=\hat{\mathcal{Y}}_{g(b),\mu}.
\]
Therefore, similar to (\ref{equ:P-minus-iso-1}), we have
\begin{equation}\label{equ:P-minus-iso}
d_b\eta(T_bB) \cong \H_r^-,
\end{equation}
and this isomorphism is given by first gluing perturbations on $\hat{N}_1$ and $\hat{N}_2$ (which is $0$), then projecting it to $\H_r^-$.

Now let $\hat{N}_2 = D^2\times \S^2$. As in \ref{prop:trivialObs}, we have 
\begin{center}
\begin{tikzpicture}[commutative diagrams/every diagram]
\node (P0) at (1.7cm, 0cm) {$0$};
\node (P1) at (5cm, 0cm) {$0$};
\node (P2) at (8.3cm, 0cm) {$0$};

\node (P3) at (0cm, -1.6cm) {$0$};
\node (P4) at (1.7cm, -1.6cm) {$ \ker\Delta_-^c$ };
\node (P41) at (1cm, -1.1cm) {$\dim B$ };
\node (P5) at (5cm, -1.6cm) {$H^2(F_{\hat{\mathsf{C}}_1}) \oplus H^2(F_{\hat{\mathsf{C}}_2})$};
\node (P51) at (4.3cm, -1.1cm) {$\dim B$};
\node[orange] (P52) at (5.5cm, -1.1cm) {$1$};
\node (P6) at (8.3cm,-1.6cm) { $L_1^- + L_2^- $};
\node (P61) at (7.8cm,-1.1cm) { $0$};
\node[orange] (P62) at (8.8cm,-1.1cm) { $1$};
\node (P7) at (10cm, -1.6cm) {$0$};

\node (P8) at (0cm, -3.2cm) {$0$};
\node (P9) at (1.7cm, -3.2cm) {${\H}_r^-$ };
\node (P91) at (1cm, -2.7cm) {$\dim B $ };
\node (P10) at (5cm, -3.2cm) {$\ker_{ex}\hat{\mathcal{T}}^*_{\hat{\mathsf{C}}_1}\oplus \ker_{ex}\hat{\mathcal{T}}^*_{\hat{\mathsf{C}}_2}$};
\node (P11) at (8.3cm,-3.2cm) { $\hat{L}_1^- + \hat{L}_2^-$};
\node (P12) at (10cm, -3.2cm) {$0$};

\node (P13) at (0cm, -4.8cm) {$0$};
\node (P14) at (1.7cm, -4.8cm) {$\ker\Delta_-^0$ };
\node (P141) at (1.2cm, -5.3cm) {$0$};
\node (P15) at (5cm, -4.8cm) {$\mathfrak{C}_1^-\oplus \mathfrak{C}_2^-$};
\node (P151) at (4.5cm, -5.3cm) {$0$};
\node[orange] (P152) at (5.5cm, -5.3cm) {$1$};
\node (P16) at (8.3cm,-4.8cm) { $\mathfrak{C}_1^- + \mathfrak{C}_2^-$};
\node (P161) at (7.8cm, -5.3cm) {$0$};
\node[orange] (P162) at (8.8cm, -5.3cm) {$1$};
\node (P17) at (10cm, -4.8cm) {$0$};

\node (P18) at (1.7cm,-6.4cm) {$0$};
\node (P19) at (5cm, -6.4cm) {$0$};
\node (P20) at (8.3cm, -6.4cm) {$0$};

\path[commutative diagrams/.cd, every arrow, every label]
(P0) edge node {} (P4)
(P1) edge node {} (P5)
(P2) edge node {} (P6)

(P3) edge node {} (P4)
(P4) edge node {$S_r$} (P5)
(P5) edge node {$\Delta_-^c$} (P6)
(P6) edge node {} (P7)

(P4) edge node {} (P9)
(P5) edge node {} (P10)
(P6) edge node {} (P11)

(P8) edge node {} (P9)
(P9) edge node {$S_r$} (P10)
(P10) edge node {$\Delta_+^c$} (P11)
(P11) edge node {} (P12)

(P9) edge node {} (P14)
(P10) edge node {} (P15)
(P11) edge node {} (P16)

(P13) edge node {} (P14)
(P14) edge node {$S_r$} (P15)
(P15) edge node {$\Delta_-^0$} (P16)
(P16) edge node {} (P17)

(P14) edge node {} (P18)
(P15) edge node {} (P19)
(P16) edge node {} (P20)
;
\end{tikzpicture}
\end{center}
Since the dimension of hte obstruction space matches the dimension of the parameter space, we have type 1 configuration $\hat{\mathsf{C}}_1 \#_r \hat{\mathsf{C}}_2$ (see Figure \ref{fig:genericCell}).

\subsection{Local gluing theory for type $1$ configuration}
We start from case $1$ as it is simpler (the dimension of the obstruction diagram matches the dimension of the parameter space, i.e, the configuration obtained by gluing is of type 1) by the previou subsection. 

Suppose that $X$ is the $4$-manifold satisfying the dimension assumption, and $X_0$ is obtained by cutting off a neighborhood of a cohomologically nontrivial loop. In this case, the gluing configurations are of type $1$. The main theorem of this subsection is:

\begin{theorem}\label{Thm:Local-gluing-theory-for-type1}
\begin{align*}
\FM(X) &\cong \FM(X_0) \times_{\FM(\S^1\times \S^2)} \FM(\S^1\times D^3), \\
\FM(X') &\cong \FM(X_0) \times_{\FM(\S^1\times \S^2)} \FM(D^2 \times \S^2). 
\end{align*}
Here ``$\cong$'' means the isotopy in the parameterized configuration space of $X$ or $X'$.
\end{theorem}

We will use some classical ideas found in \cite{Nicolaescu2000NotesOS} and \cite{BK2020}, but generalize them as we don't have any assupmtion on the kernel of the twisted Dirac operator. Throughout this subsection, $\hat{N}_1$ is the cylindrical manifold obtaind from $X_0$, and $X(r)\simeq X$ is the closed manifold obtained by regluing $\hat{N}_2 \simeq \S^1\times D^3$ to $\hat{N}_1$ with a length $r$ neck. As there is no restriction on the twisted Dirac operator, the perturbations provided by the parameter space does not coincide with the obstruction space in general. The neck of the glued manifold should be stretched to control the error occurring from such difference. 

\begin{prop}\label{prop:perturbationCondition}
    It's possible to choose families perturbations $\eta_{X_0}$ satisfying the following assumptions:
\begin{enumerate}[I]
\item $\eta_{X_0}$ doesn't meet the wall. \label{condition:doesn't-meet-the-wall}
% \item[A2] The embedding of family moduli space of $\hat{N}_1$ in $\mathcal B$ is a submersion in the vertical tangent space, i.e, $T\M(\hat{N}_1)\to T(\mathcal B /B)$ is injective. 
\item %(A5) of Proposition 7.2 in \cite{BK2020} 
Let $\delta$ be any fixed positive number. For any $b\in B$ and any solution $\hat{\mathsf{C}}_1$ for the perturbation $\eta(b)$, let $D_b\eta$ be the differential of $\eta$ that projects to the fiber, which is a linear operator
\[
D_b\eta : T_bB \to L_\mu^{1,2}(\mathbf{i}\Lambda_+^2T^*\hat{N_1}) 
\]
and let $\pi_{\hat{\mathsf{C}}_1}$ and $\pi_{\hat{\mathsf{C}}_1}^\bot$ be the $L^2_\mu$-orthogonal projections
\begin{align*}
 \pi_{\hat{\mathsf{C}}_1}: L_\mu^{1,2}(\mathbf{i}\Lambda_+^2T^*\hat{N_1}) &\to H^2(F_{\hat{\mathsf{C}}_1}) \\
  \pi_{\hat{\mathsf{C}}_1}^\bot: L_\mu^{1,2}(\mathbf{i}\Lambda_+^2T^*\hat{N_1}) &\to H^2(F_{\hat{\mathsf{C}}_1})^\bot, 
\end{align*}
where $H^2(F_{\hat{\mathsf{C}}_1})^\bot$ is the $L^2_\mu$-orthogonal complement of $H^2(F_{\hat{\mathsf{C}}_1})$. The following conditions on the local gradient are satisfied. \label{condition:gradient}
\begin{enumerate}[a]
\item \label{pi-bot-small} There exists some $\epsilon \in (0,\delta)$, such that the projection $\pi_+$ of $D_b\eta$ to $H^+_{L^2}(\hat{N}_1)$ satisfies
\[
\|\pi_+^\bot(D_b\eta)\|_{L^2_\mu(\hat{N}_1)} \leq \epsilon\|D_b\eta\| .
\]
\item \label{D-nontrvial} There exists some $\delta' \in (0,\delta)$ such that 
\[
\| D_b\eta(t)\|_{L^2_\mu(X_0)} < \delta'\|t\|.
\]
\item \label{piD-nontrvial} There exists $\delta'' > \frac12 \delta$ such that 
\[
\|  \pi_{\hat{\mathsf{C}}_1} D_b \eta(t)\|_{L^2_\mu(X_0)}  > \delta''\|t\|
\]
\end{enumerate}
\item $\eta_{X_0}: B \to \Pi(E_{X_0})$ is generic in the family sense (in the sense of Theorem \ref{thm:cylindricalTransversality}).%: 
%The image of this section intersects with the image of the universal moduli space transversally.
\item For any $b\in B$, $\eta_{X_0}(b)$ vanishes on $[R_{\text{vanish}},\infty)$ of the neck.
\end{enumerate}
\end{prop}

\begin{proof}
The classical strategy is to first show that the space $\mathcal{O}$ of all families perturbations satisfying \ref{condition:doesn't-meet-the-wall} and \ref{condition:gradient} is open. Next 

To show the existence of at least one element satisying these assumptions, we will construct one and then modify it by a homotopy such that it satisfies local gradient conditions. We can find a homotopy on $\S^1 \times \S^{\dim B -1} \times I$

\end{proof}

Recall that the cut-off function $\alpha_r(t) = 1$ on $[0, r)$ and  $\alpha_r(t) = 0$ on $[r+1,\infty)$. Define the gluing map with $0$-section on $\hat{N}_2$ by
\begin{align*}
\Psi_r:L_\mu^{1,2}(\mathbf{i}\Lambda_+^2T^*\hat{N_1})  &\to L^{1,2}(\mathbf{i}\Lambda_+^2T^*{X(r)})  \\
\eta(t) &\mapsto \alpha_r(t)\eta(t)
\end{align*}
on the neck, and $\Psi_r$ is identity on $\hat{N_1}$ and $\Psi_r = 0$ on $\hat{N_2}$. We glue $\eta_{\hat{N}_1}$ on $\hat{N}_1$-side and the zero perturbation on $\hat{N}_2$-side, and denote it by the same notation as \cite{BK2020}:
\[
K_r(b) :=\Psi_r \eta_{\hat{N}_1}(b) =  \eta_{\hat{N}_1}(b)\#_r 0
\]
where $\#_r $ is defined in (\ref{equ:paste}).

Let's define
\[
\mathfrak{X}^k_+=\mathfrak{X}^k_+(r) := L^{k,2}(\hat{\S}^+_{\hat \sigma} \oplus \mathbf{i} T^*\hat{X}(r)), 
\mathfrak{X}^k_-:= L^{k,2}(\hat{\S}^-_{\hat \sigma} \oplus \mathbf{i} \Lambda^2_+T^*\hat{X}(r)), 
\]
\[
\mathfrak{X}^k := \mathfrak{X}^k_+ \oplus \mathfrak{X}^k_-.
\]

Let $\mathcal{F}_r$ be the operator on $X(r)$ that combines the sw map and the gauge fixing condition, as defined in (\ref{def:F}). Let $\tilde{\mathcal{F}}_r(\cdot, b) = \mathcal{F}_r(\cdot) + iK_r(b)$ for $b\in B$ be the parameterized version. Denote by 
\[
 \hat{\mathcal T}_r: \mathfrak{X}^0_+ \to \mathfrak{X}^0_-
\]
the the differential of $\tilde{\mathcal{F}}_r(\cdot, 0) $ at $(0,0)$. Then $\mathcal{F}_r(\cdot, 0) $ has the expansion
\begin{equation*}
\tilde{\mathcal{F}}_r(\underline{\hat{\mathsf{C}}}, 0)  = \mathcal{F}_r(0,0) + \hat{\mathcal T}_r(\underline{\hat{\mathsf{C}}}) + R(\underline{\hat{\mathsf{C}}})
\end{equation*}
where $R$ is the remainder term. We will choose locally constant metric such that the parameterized operator $\tilde{\mathcal{F}}_r(\underline{\hat{\mathsf{C}}}, b)$ only has one more perturbation term, namely
\begin{equation*}
\tilde{\mathcal{F}}_r(\underline{\hat{\mathsf{C}}}, b)  = \mathcal{F}_r(0) + \hat{\mathcal T}_r(\underline{\hat{\mathsf{C}}}) + R(\underline{\hat{\mathsf{C}}}) + iK_r(b).
\end{equation*}

Define
\[
\hat{L}_r:= 
\begin{bmatrix}
0 &\hat{\mathcal T}^*_r \\
\hat{\mathcal T}_r & 0
\end{bmatrix}
:\mathfrak{X}^0 \to  \mathfrak{X}^0.
\]
We want to use the eigenspace corresponds to very small eigenvalues to approximate the kernel of this operator. Let $\H_r$ be the subspace of $\mathfrak{X}^0$ spanned by
\begin{equation}\label{equ:def-of-H-r}
\{ v; \hat{L}_r v = \lambda v, |\lambda| < r^{-2}\}.
\end{equation}
Let $\mathcal{Y}^0 = \mathcal{Y}^0(r)$ be the orthogonal complement of $\H_r$ in $\mathfrak{X}^0$. Let $\H_r^\pm$ be the orthogonal projection of $\H_r$ to $\mathfrak{X}^0_\pm$.  Let $\mathcal{Y}_{\pm}^0(r)$ be the orthogonal projection of $\mathcal{Y}^0(r)$ to $\mathfrak{X}^0_\pm$. 

Let $P_\pm = P^r_\pm$ be the $L^2$-othogonal projections $\mathfrak{X}^0_\pm \to \H_r^\pm$ and let $Q_\pm = Q^r_\pm = 1-P^r_\pm$ be the $L^2$-othogonal projections $\mathfrak{X}^0_\pm \to \mathcal{Y}_{\pm}^0$.

The idea is to decompose the equation $\tilde{\mathcal{F}}_r(\underline{\hat{\mathsf{C}}},b) = 0$ to the equations 
\begin{equation}\label{decompose-F=0}
\begin{cases}
P^r_- \tilde{\mathcal{F}}_r(\underline{\hat{\mathsf{C}}},b) = 0,\\
Q^r_-\tilde{\mathcal{F}}_r(\underline{\hat{\mathsf{C}}},b) = 0.
\end{cases}
\end{equation}
Denote $P^r_+(\underline{\hat{\mathsf{C}}})$ by $\underline{\hat{\mathsf{C}}}_0$ and $Q^r_+(\underline{\hat{\mathsf{C}}})$ by $\underline{\hat{\mathsf{C}}}^\bot$. Note that $\underline{\hat{\mathsf{C}}}= \underline{\hat{\mathsf{C}}}_0\oplus \underline{\hat{\mathsf{C}}}^\bot$ and by BK Remark 6.10
\[
P^r_- \hat{\mathcal T}_r(\underline{\hat{\mathsf{C}}}) = \hat{\mathcal T}_r(P^r_+\underline{\hat{\mathsf{C}}}) = \hat{\mathcal T}_r(\underline{\hat{\mathsf{C}}}_0), 
\hspace{0.5cm} 
Q^r_- \hat{\mathcal T}_r(\underline{\hat{\mathsf{C}}}) = \hat{\mathcal T}_r(Q^r_+\underline{\hat{\mathsf{C}}}) = \hat{\mathcal T}_r(\underline{\hat{\mathsf{C}}}^\bot).
\]
Hence (\ref{decompose-F=0}) becomes
\begin{equation}\label{expand-decompose-F=0}
\begin{cases}
P^r_-  \mathcal{F}_r(0) + \hat{\mathcal T}_r(\underline{\hat{\mathsf{C}}}_0)+ P^r_- R( \underline{\hat{\mathsf{C}}}_0\oplus \underline{\hat{\mathsf{C}}}^\bot) + P^r_- iK_r(b) = 0,\\
Q^r_- \mathcal{F}_r(0) + \hat{\mathcal T}_r(\underline{\hat{\mathsf{C}}}^\bot)+ Q^r_- R( \underline{\hat{\mathsf{C}}}_0\oplus \underline{\hat{\mathsf{C}}}^\bot) + Q^r_- iK_r(b) = 0.
\end{cases}
\end{equation}

%$\hat{\mathsf{C}}_r + \underline{\hat{\mathsf{C}}}_0\oplus \underline{\hat{\mathsf{C}}}^\bot$

%Recall that in Theorem \ref{thm:regular}, we choose a regular perturbation such that the operator 
%\[
 %\hat{\mathcal T}_r: \mathfrak{X}^0_+ \to \mathfrak{X}^0_-
%\]
First we try to solve the second equation. Define $\mathcal{Y}^k =\mathcal{Y}^0 \cap \mathcal{X}^k$. Note that 
by elementary linear algebra $\hat{\mathcal T}_r$ sends $\mathcal{Y}^{k+1}_+$ to $\mathcal{Y}^k_-$ (see Remark 6.10). Let $S$ be the inverse of $ \hat{\mathcal T}_r:\mathcal{Y}^{k+1}_+ \to \mathcal{Y}^k_-$.  Denote $-SQ_- \mathcal{F}_r(0)$ by $U^\bot$. Apply the operator $S$ to the second equation of (\ref{expand-decompose-F=0}) then we get
\[
\underline{\hat{\mathsf{C}}}^\bot =  U^\bot - SQ_-R(\underline{\hat{\mathsf{C}}}^\bot + \underline{\hat{\mathsf{C}}}_0) + SQ_-iK_r(t).
\]
To solve this equation, we define 
\[
\tilde{{F}} :\underline{\hat{\mathsf{C}}}^\bot \to U^\bot - SQ_-R(\underline{\hat{\mathsf{C}}}^\bot + \underline{\hat{\mathsf{C}}}_0) + SQ_-iK_r(t)
\]
and expect it to be a contraction map and therefore admits a unique fixed point.

\begin{lemma}\label{lemma:projectionInequality}
Let $\hat{N}_1$ be the cylindrical manifold obtained from $X_0$. For any positive number $\epsilon$, define $\mathcal{O}(\epsilon)$ to be a subset of the perturbation families, such that for each $\eta\in \mathcal{O}(\epsilon)$, %there exists some $\epsilon' < \epsilon$ and
\begin{equation}\label{equ:conditionA5c}
\| \pi_{\hat{\mathsf{C}}_1}^\bot D_b\eta\|_{L^2_\mu(X_0)} \leq \epsilon\| \pi_{\hat{\mathsf{C}}_1} D_b\eta\|_{L^2_\mu(X_0)}
\end{equation}
for any $b\in B$ and any solution $\hat{\mathsf{C}}_1$ for the perturbation $\eta(b)$. Here $D_b\eta$ is the differential of $\eta$ that projects to the fiber, which is a linear operator
\[
D_b\eta : T_bB \to L_\mu^{1,2}(\mathbf{i}\Lambda_+^2T^*\hat{N_1}) 
\]
and $\pi_{\hat{\mathsf{C}}_1}$ and $\pi_{\hat{\mathsf{C}}_1}^\bot$ are the $L^2_\mu$-orthogonal projections
\begin{align*}
 \pi_{\hat{\mathsf{C}}_1}: L_\mu^{1,2}(\mathbf{i}\Lambda_+^2T^*\hat{N_1}) &\to H^2(F_{\hat{\mathsf{C}}_1}) \\
  \pi_{\hat{\mathsf{C}}_1}^\bot: L_\mu^{1,2}(\mathbf{i}\Lambda_+^2T^*\hat{N_1}) &\to H^2(F_{\hat{\mathsf{C}}_1})^\bot = \widehat{\underline{SW}}_{\hat{\mathsf{C}}_1} (T_{\hat{\mathsf{C}}_1} (\partial_\infty^{-1}(\mathsf{C}_\infty)/\hat{\mathcal{G}}_{\mu}  )), 
\end{align*}
where $H^2(F_{\hat{\mathsf{C}}_1})^\bot$ is the $L^2_\mu$-orthogonal complement of $H^2(F_{\hat{\mathsf{C}}_1})$. Then any $\eta \in \mathcal{Z}_{reg}$ found in Theorem \ref{thm:cylindricalTransversality} is in $\mathcal{O}(\epsilon)$ for some $\epsilon > 0$.
\end{lemma}
\begin{proof}
Choose an arbitrary element in $\mathcal{O}(\epsilon)$ and denote it by $\eta$. By Theorem \ref{thm:cylindricalTransversality} and the analysis in Subsection \ref{sec:obstruction}, for each $\eta(b)$-monopole $\hat{\mathsf{C}}_1$ the projection of $T_bB$ to $H^2(F_{\hat{\mathsf{C}}_1})$ is an isomorphism, so there is a positive number $\epsilon(\hat{\mathsf{C}}_1)$ satisfying (\ref{equ:conditionA5c}). By the compactness result the function $\epsilon$ is defined on a compact paramertrized moduli space and therefore there exsits a maximum.
\end{proof}

\begin{lemma}
Let $\hat{N}$ be the cylindrical manifold obtained from $X_0$. Fix any $\eta \in \mathcal{Z}_{reg}$. For any positive number $\epsilon'<1$, there exists some large $R$ such that for any $r>R$,
\begin{equation}\label{equ:conditionA5c}
\| P^r_- \circ \Psi \circ \pi_{\hat{\mathsf{C}}_1}^\bot D_b\eta\|_{L^2_\mu(X(r))} \leq \epsilon'\| P^r_- \circ \Psi \circ\pi_{\hat{\mathsf{C}}_1} D_b\eta\|_{L^2_\mu(X(r))}
\end{equation}
for any $b\in B$ and any solution $\hat{\mathsf{C}}_1$ for the perturbation $\eta(b)$.

\end{lemma}
\begin{proof}

\end{proof}

\begin{lemma}\label{lemma:BK7.16}
Let $\hat{N}$ be the cylindrical manifold obtained from $X_0$. Fix any $\eta \in \mathcal{Z}_{reg}$. For any positive number $C_0<1$, any $b\in B$, $t\in T_b B$, and any solution $\hat{\mathsf{C}}_1$ for the perturbation $\eta(b)$, for all sufficiently large $r$,
\[
\| P^r_- \circ \Psi_r \circ D_b\eta(t)\|_{L^2_\mu(X(r))} > C_0\| \pi_{\hat{\mathsf{C}}_1} D_b\eta(t)\|_{L^2_\mu(X_0)}.
\]
\end{lemma}
\begin{proof}
By a result in the linear gluing theorem (conclued in Nic Page 305), the gluing map $\Psi_r :L_\mu^{1,2}(\mathbf{i}\Lambda_+^2T^*\hat{N_1})  \to L^{1,2}(\mathbf{i}\Lambda_+^2T^*{X(r)}) $ defines an asymptotic map $\Psi_r: H^2(F_{\hat{\mathsf{C}}_1}) \to \H^-_r$. Note that $H^2(F_{\hat{\mathsf{C}}_1}) $ consists of two summands: $H^+_{L^2}(\hat{N}_1)$ and the spinor part, which we denote by $S_2$. $H^2(F_{\hat{\mathsf{C}}_1}) $ projects to them surjectively. 

%First assume that $t\in TB$ satisfies $D_b\eta(t) \in H^2(F_{\hat{\mathsf{C}}_1}) $. Note that by the fact that  $\| \left. \Psi_r \right|_{H^+_{L^2}(\hat{N}_1)}\|_{op} \to 1$ since, 

First notice that, for any $\eta \in L_\mu^{1,2}(\mathbf{i}\Lambda_+^2T^*\hat{N_1}) $,
\begin{align*}
\|  \eta \|_{L^2_\mu(\hat{N}_1)} -  \|\Psi_r \eta\|_{L^2(X(r))} &= \|  \eta \|_{L^2_\mu(\hat{N}_1)} -  \|\Psi_r \eta\|_{L^2_\mu(\hat{N}_1)} \\
&\leq \|  \eta  - \Psi_r \eta\|_{L^2_\mu(\hat{N}_1)} \\
&= \|  (1-\alpha_r) \eta\|_{L^2_\mu(\hat{N}_1)} \\
&\leq \|   \eta\|_{L^2([r,r+1]\times \S^1 \times \S^2 )}\\
&:= \epsilon(r) \|  \eta \|_{L^2_\mu(\hat{N}_1)}
\end{align*}
and since $\eta$ decays exponentially, $\epsilon(r) \to 0$ as $r\to \infty$. Therefore we have,
\begin{align}
\label{equ-d-asymp}\|  D_b\eta(t)\|_{L^2_\mu(\hat{N}_1)} &\to \|\Psi_r D_b\eta(t)\|_{L^2_\mu(X(r))} \\
\label{equ-pi-d-asymp}\| \pi_{\hat{\mathsf{C}}_1} D_b\eta(t)\|_{L^2_\mu(\hat{N}_1)} &\to \|\Psi_r \pi_{\hat{\mathsf{C}}_1} D_b\eta(t)\|_{L^2_\mu(X(r))} \\
\label{equ-pi-bot-d-asymp}\| \pi^\bot_{\hat{\mathsf{C}}_1} D_b\eta(t)\|_{L^2_\mu(\hat{N}_1)} &\to \|\Psi_r \pi^\bot_{\hat{\mathsf{C}}_1} D_b\eta(t)\|_{L^2_\mu(X(r))}
\end{align}
for any $t\in TB$. 
From (\ref{equ-pi-d-asymp}), the norm of $\Psi_r \pi_{\hat{\mathsf{C}}_1} D_b\eta(t)$ is controlled to a positive number, so by the definition of asymptotic map, the distance from $\Psi_r \pi_{\hat{\mathsf{C}}_1} D_b\eta(t)$ to $\H^-_r$ approaches $0$. 
Hence $\Psi_r \pi^\bot_{\hat{\mathsf{C}}_1} D_b\eta(t)$ almost connects $\Psi_r  D_b\eta(t)$ to $\H^-_r$.
From (\ref{equ-d-asymp}), (\ref{equ-pi-d-asymp}), and (\ref{equ-pi-bot-d-asymp}), 
\begin{equation}\label{equ:rightTriangle}
\|\Psi_r \pi_{\hat{\mathsf{C}}_1} D_b\eta(t)\|_{L^2_\mu(X(r))} +  \|\Psi_r \pi^\bot_{\hat{\mathsf{C}}_1} D_b\eta(t)\|_{L^2_\mu(X(r))} \to  \|\Psi_r D_b\eta(t)\|_{L^2_\mu(X(r))}.
\end{equation}
Hence the triangle they form is almost a right triangle. Note that $\| Q^r_-  \Psi_r  D_b\eta(t)\|_{L^2_\mu(X(r))} $ is the distance from $\Psi_r D_b\eta(t)$ to $\H^-_r$, which by definition, is less than the length of any other vector that connects $\Psi_r  D_b\eta(t)$ to $\H^-_r$. Combine all these observations, for any $C_1< 1$, there exists a number $R(C_1)$, such that for $r>R(C_1)$,
\begin{equation}\label{equ:distanceInequality}
 C_1\| Q^r_- \circ \Psi_r \circ D_b\eta(t)\|_{L^2_\mu(X(r))} <  \|\Psi_r \pi^\bot_{\hat{\mathsf{C}}_1} D_b\eta(t)\|_{L^2_\mu(X(r))}.
\end{equation}
Therefore
\begin{align}
\| P^r_- \circ \Psi_r \circ D_b\eta(t)\|_{L^2_\mu(X(r))} &= \|  \Psi_r \circ D_b\eta(t)\|_{L^2_\mu(X(r))}  - \| Q^r_- \circ \Psi_r \circ D_b\eta(t)\|_{L^2_\mu(X(r))}\notag \\
&\stackrel{\text{by (\ref{equ:distanceInequality})}}{>} \|  \Psi_r \circ D_b\eta(t)\|_{L^2_\mu(X(r))}  -  \frac{1}{C_1}\|\Psi_r \pi^\bot_{\hat{\mathsf{C}}_1} D_b\eta(t)\|_{L^2_\mu(X(r))} \notag \\
&\stackrel{\text{by (\ref{equ:rightTriangle})}}{>}(1- \frac{1}{C_1})\|  \Psi_r \circ D_b\eta(t)\|_{L^2_\mu(X(r))}  + \frac{1}{C_1}(1-\epsilon) \|\Psi_r \pi_{\hat{\mathsf{C}}_1} D_b\eta(t)\|_{L^2_\mu(X(r))} \notag \\
&\stackrel{\text{by (\ref{equ-pi-d-asymp})}}{>} (1- \frac{1}{C_1})(1-\epsilon)\|  D_b\eta(t)\|_{L^2_\mu(\hat{N}_1)} + \frac{1}{C_1}(1-\epsilon)^2 \| \pi_{\hat{\mathsf{C}}_1} D_b\eta(t)\|_{L^2_\mu(\hat{N}_1)}\label{equ:epsilonInequality}
\end{align}
for any $\epsilon \in (0,1)$ and any $r>R(\epsilon)$. The second term of (\ref{equ:epsilonInequality}) is what we desired, so we consider the first term now. By Lemma \ref{lemma:projectionInequality}, there exists some fixed positive constant $C(\eta)$, such that
\begin{align*}
(1+C(\eta))\| \pi_{\hat{\mathsf{C}}_1} D_b\eta\|_{L^2_\mu(X_0)} &> \| \pi_{\hat{\mathsf{C}}_1} D_b\eta\|_{L^2_\mu(X_0)} + \| \pi_{\hat{\mathsf{C}}_1}^\bot D_b\eta\|_{L^2_\mu(X_0)}\\
&= \|  D_b\eta(t)\|_{L^2_\mu(X(r))}.
\end{align*}
Note that $C_1<1$ and therefore $(1-\frac{1}{C_1})<0$. Hence
\[
(1- \frac{1}{C_1})(1-\epsilon)\|  D_b\eta(t)\|_{L^2_\mu(\hat{N}_1)} >
(1- \frac{1}{C_1})(1-\epsilon)(1+C(\eta))\| \pi_{\hat{\mathsf{C}}_1} D_b\eta\|_{L^2_\mu(X_0)} .
\]

Finally, choose 
\[
\epsilon \in (0, 1-(\frac{1  + C_0}{2})^{\frac12})
\]
and
\[
C_1\in (\frac{1}{\frac{1-C_0}{2(1+C(\eta))}+1},1),
\]
then for $r$ satisfying (\ref{equ:distanceInequality}) and inequality (\ref{equ:epsilonInequality}), 
\[
\| P^r_- \circ \Psi_r \circ D_b\eta(t)\|_{L^2_\mu(X(r))} > C_0\| \pi_{\hat{\mathsf{C}}_1} D_b\eta(t)\|_{L^2_\mu(X_0)} 
\]
for any $t\in TB$.
\end{proof}

From Lemma \ref{lemma:BK7.16} and Condition (\ref{piD-nontrvial}) we have
\begin{corollary} \label{corollary:projection-iso}
$P^r_- \circ \Psi_r \circ D_b\eta$ is an isomorphism.
\end{corollary}
\begin{remark}
Lemma \ref{lemma:BK7.16} and Corollary \ref{corollary:projection-iso} recover (and stronger than) Lemma 7.16 and Lemma 7.15 of \cite{BK2020}, but the proof is different since the conditions are released.
\end{remark}

From lemma \ref{lemma:BK7.16}, use the same strategy in \cite{BK2020} Lemma 7.17, we deduce that 
\begin{corollary}\label{lemma:BK7.17}
Let $B_0(r^{-4})\subset \H_r^+$ and $ B_\bot(r^{-4})\subset  \mathcal{Y}_+ $ be two balls of radius $r^{-4}$. If $\hat{\mathsf{C}}_r + \underline{\hat{\mathsf{C}}}_0\oplus \underline{\hat{\mathsf{C}}}^\bot$ is a $K_r(b)$-monopole, and $ \underline{\hat{\mathsf{C}}}_0 \in B_0(r^{-4}) $, $ \underline{\hat{\mathsf{C}}}^\bot \in B_\bot(r^{-4})$ for all sufficiently large $r$, then 
\begin{equation}\label{equ:BK7.17}
\|b\| \leq \frac{4}{\delta C_0} r^{-6},
\end{equation}
where $\delta$ is the one in Condition \ref{condition:gradient}.
\end{corollary}
%\begin{proof}

%\end{proof}
Recall that 
\[
\tilde{{F}} :\underline{\hat{\mathsf{C}}}^\bot \to U^\bot - SQ_-R(\underline{\hat{\mathsf{C}}}^\bot + \underline{\hat{\mathsf{C}}}_0) + SQ_-iK_r(t)
\]
\begin{lemma}\label{lemma:F-contraction-1}
For sufficiently large $r$, $\tilde{{F}}$ sends the ball $B_\bot(r^{-4})$ to itself and is a contraction.
\end{lemma}
\begin{proof}
From \cite{Nicolaescu2000NotesOS} (4.5.8) the operator 
\[
{F} :\underline{\hat{\mathsf{C}}}^\bot \to U^\bot - SQ_-R(\underline{\hat{\mathsf{C}}}^\bot + \underline{\hat{\mathsf{C}}}_0) 
\]
is a contraction, which is only different from $\tilde{{F}}$ by a constant term. Hence $\tilde{{F}}$ is a contraction. It remains to show that $\tilde{{F}}$ sends the ball $B_\bot(r^{-4})$ to itself.

From the norm estimate of the operator $R$ ( see \cite{Nicolaescu2000NotesOS} Lemma 4.5.6) and the norm estimate of the operator $S$ (see \cite{Nicolaescu2000NotesOS} (4.4.5)) we deduce (see also \cite{BK2020} Remark 6.13) for $\underline{\hat{\mathsf{C}}}^\bot \in B_\bot(r^{-4})$, 
\[
\|{F} (\underline{\hat{\mathsf{C}}}^\bot)\| \leq C(r^2 e^{-\mu r} + r^{-13/2}).
\]
Hence 
\[
\|\tilde{{F}} (\underline{\hat{\mathsf{C}}}^\bot)\| \leq C(r^2 e^{-\mu r} + r^{-13/2}) + \|SQ_-iK_r(t)\|.
\]
From the norm estimate of the operator $S$ (see \cite{Nicolaescu2000NotesOS} (4.4.5)) we have 
\begin{equation}\label{S-norm}
\|SQ_-iK_r(t)\| \leq C_k r^2 \|  Q_-iK_r(t)\|.
\end{equation}
Recall that by definition $K_r(b) = \Psi_r \circ \eta(b)$. Hence we have 
\begin{align}
\|Q_- K_r(b)\| &= \|Q_- \Psi_r \eta(b)\| \notag\\
&\leq \| Q_- \Psi_r \pi_{\hat{\mathsf{C}}_1}\eta(b)\| + \| Q_- \Psi_r \pi^\bot_{\hat{\mathsf{C}}_1} \eta(b)\|\notag \\
&\stackrel{\text{by (\ref{equ-pi-bot-d-asymp})}}{\leq}  \| Q_- \Psi_r \pi_{\hat{\mathsf{C}}_1}\eta(b)\| + \|  \pi^\bot_{\hat{\mathsf{C}}_1} \eta(b)\|\label{tooLarge} \\
&\leq  \| Q_- \Psi_r \pi_{\hat{\mathsf{C}}_1}\eta(b)\| + \|  \eta(b)\|\notag\\
&\stackrel{\text{by \ref{D-nontrvial}}}{\leq}  \| Q_- \Psi_r \pi_{\hat{\mathsf{C}}_1}\eta(b)\| + \delta\|  b\|. \label{second-term-too-large}
\end{align}
Recall that the gluing map $\Psi_r :L_\mu^{1,2}(\mathbf{i}\Lambda_+^2T^*\hat{N_1})  \to L^{1,2}(\mathbf{i}\Lambda_+^2T^*{X(r)}) $ defines an asymptotic map $\Psi_r: H^2(F_{\hat{\mathsf{C}}_1}) \to \H^-_r$. Recall also that $\pi_{\hat{\mathsf{C}}_1}$ is the projection from $L_\mu^{1,2}(\mathbf{i}\Lambda_+^2T^*\hat{N_1})$ to $H^2(F_{\hat{\mathsf{C}}_1})$, and $\|Q^r_-\|$ is the projection from $L^{1,2}(\mathbf{i}\Lambda_+^2T^*{X(r)}) $ to $(\H^-_r)^\bot$. Hence as $r\to 0$
\[
\| Q^r_- \circ \Psi_r \| \to 0
\]
and therefore 
\begin{equation}\label{Q-norm}
\| Q^r_-   \Psi_r \pi_{\hat{\mathsf{C}}_1}\eta(b) \| \leq \epsilon(r) \| \pi_{\hat{\mathsf{C}}_1}\eta(b) \|,
\end{equation}
where $\epsilon(r) \to 0$ as $r\to 0$.

Combine (\ref{S-norm}),  (\ref{second-term-too-large}), and (\ref{Q-norm}) we deduce
\begin{align}
\|SQ_-iK_r(t)\|
&\leq 
C_k r^2 (\epsilon(r) \| \pi_{\hat{\mathsf{C}}_1}\eta(b) \| + \delta\|b\|) \label{cntrol-S}\\
&\leq 
C_k r^2 (\epsilon(r) \| \eta(b) \| + \delta\|b\|) \notag\\
&\stackrel{\text{by \ref{D-nontrvial}}}{\leq} 
C_k r^2 ( \epsilon(r)\delta\|b\|+ \delta\|b\|).\notag
\end{align}
We hope to apply Corollary \ref{lemma:BK7.17} to control these terms. The issue is that, the second term cannot be controlled, since $\delta$ here cancels with the one in (\ref{equ:BK7.17}) when we combine them. Hence we have to reconsider how to control the second term $\pi^\bot_{\hat{\mathsf{C}}_1} \eta(b)$ of (\ref{tooLarge}). 

The main idea is to scale the spinor part to narrow down the angle between the direction of the perturbation and the image of Seiberg-Witten map. Recall that the target of the linearization $\widehat{\underline{SW}}_{\hat{\mathsf{C}}_1}$ is
\[
\hat{\mathcal{Y}}_\mu := L_\mu^{1,2}(\hat{S}_{\hat{\sigma}}^- \oplus \mathbf{i}\Lambda_+^2T^*\hat{N}_1),
\]
in which the orthogonal complement of $\widehat{\underline{SW}}_{\hat{\mathsf{C}}_1} (T_{\hat{\mathsf{C}}_1} (\partial_\infty^{-1}(\mathsf{C}_\infty)/\hat{\mathcal{G}}_{\mu}  ))$ is $ H^2(F_{\hat{\mathsf{C}}_1})$. $\eta(b)$ lives in $L_\mu^{1,2}(\mathbf{i}\Lambda_+^2T^*\hat{N_1}) $. 
$\pi_{\hat{\mathsf{C}}_1} $ and $\pi^\bot_{\hat{\mathsf{C}}_1} $ project $\eta(b)$ to $ H^2(F_{\hat{\mathsf{C}}_1})$ and $\widehat{\underline{SW}}_{\hat{\mathsf{C}}_1} (T_{\hat{\mathsf{C}}_1} (\partial_\infty^{-1}(\mathsf{C}_\infty)/\hat{\mathcal{G}}_{\mu}  ))$. 
Rescaling the norm on the spinor part of $\hat{\mathcal{Y}}_\mu$ would change $ H^2(F_{\hat{\mathsf{C}}_1})$, and therefore changes the norm of these projections (see Figure \ref{fig:rescale_spinor}).

 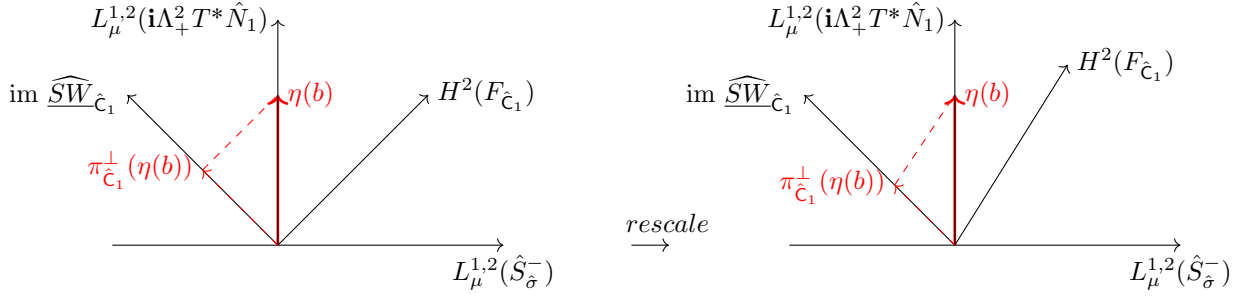
\begin{figure}
    \centering
    \begin{tikzpicture}
    % 坐标轴
        \draw[->] (-2.2,0) -- (3,0) node[below] {$ L_\mu^{1,2}(\hat{S}_{\hat{\sigma}}^- )$};
        \draw[->] (0,0) -- (-2,2) node[left] {im $\widehat{\underline{SW}}_{\hat{\mathsf{C}}_1} $};
        \draw[->,red,  line width=1pt] (0,0) -- (0,2) node[right] {$\eta(b)$};
        \draw[->] (0,0) -- (0,3) node[left] {$L_\mu^{1,2}(\mathbf{i}\Lambda_+^2T^*\hat{N}_1)$};
        \draw[->] (0,0) -- (2,2) node[right] {$ H^2(F_{\hat{\mathsf{C}}_1})$};
        \draw[red, dashed] (-1,1) -- (0,2);
        \draw[->, red, dashed] (0,0) -- (-1,1) node[left] {$\pi^\bot_{\hat{\mathsf{C}}_1}(\eta(b))$};

     	\node[anchor=west] at (4.5,0.3) {$rescale$};
        
        \draw[->] (4.7,0) -- (5.2,0);
        
        % 复制并平移第一个图
    	\begin{scope}[xshift=9cm]
    	\draw[->] (-2.2,0) -- (3,0) node[below] {$ L_\mu^{1,2}(\hat{S}_{\hat{\sigma}}^- )$};
        \draw[->] (0,0) -- (-2,2) node[left] {im $\widehat{\underline{SW}}_{\hat{\mathsf{C}}_1} $};
        \draw[->,red,  line width=1pt] (0,0) -- (0,2) node[right] {$\eta(b)$};
        \draw[->] (0,0) -- (0,3) node[left] {$L_\mu^{1,2}(\mathbf{i}\Lambda_+^2T^*\hat{N}_1)$};
        \draw[->] (0,0) -- (1.5,2.4) node[right] {$ H^2(F_{\hat{\mathsf{C}}_1})$};
        \draw[red, dashed] (-0.8,0.8) -- (0,2);
        \draw[->, red, dashed] (0,0) -- (-0.8,0.8) node[left] {$\pi^\bot_{\hat{\mathsf{C}}_1}(\eta(b))$};
        
    	\end{scope}
    	
    \end{tikzpicture}
    \caption{When the norm on $ L_\mu^{1,2}(\hat{S}_{\hat{\sigma}}^- )$ increases, the angle between $H^2(F_{\hat{\mathsf{C}}_1})$ and $ L_\mu^{1,2}(\hat{S}_{\hat{\sigma}}^- )$ would increase to make $H^2(F_{\hat{\mathsf{C}}_1})$ orthogonal to im $\widehat{\underline{SW}}_{\hat{\mathsf{C}}_1} $. }
    \label{fig:rescale_spinor}
\end{figure} 

Assume that the original norm on the space of spinor is $\|\cdot\|_s$, and the norm on the space of self dual $2$ forms is $\|\cdot\|_+$. Let $k$ be a positive number, then $k\|\cdot\|_s$ is equivalent to $\|\cdot\|_s$. We show that as $k \to \infty$, $\pi^\bot_{\hat{\mathsf{C}}_1}(\eta(b)) \to 0$, as follows.

Let 
\[
H^2_k(F_{\hat{\mathsf{C}}_1})
\]
be the orthogonal complement of $\widehat{\underline{SW}}_{\hat{\mathsf{C}}_1} (T_{\hat{\mathsf{C}}_1} (\partial_\infty^{-1}(\mathsf{C}_\infty)/\hat{\mathcal{G}}_{\mu}  ))$ under the norm $(k\|\cdot\|_s,\|\cdot\|_+)$. Let $\S^n_k$ be the unit sphere of $H^2_k(F_{\hat{\mathsf{C}}_1})$, where $n = \dim H^2_k(F_{\hat{\mathsf{C}}_1})$ is a finite number that doesn't depend on $k$. Let 
\begin{equation}\label{def:vector-e}
 e=s + \eta \in \S^n_1  \subset H^2_1(F_{\hat{\mathsf{C}}_1})
\end{equation}
where $s$ is the spinor part of $e$ and $\eta \in \Lambda_+^2 T^*\hat{N}_1$. Assume that $\epsilon(k)$ is a positive number such that 
\[
e(k) := \epsilon(k) s + \epsilon(k) k^2 \eta \in \S^n_k.
\]
To check that $e(k) \in H^2_k(F_{\hat{\mathsf{C}}_1})$, pick any $s' + \eta' \in \widehat{\underline{SW}}_{\hat{\mathsf{C}}_1} (T_{\hat{\mathsf{C}}_1} (\partial_\infty^{-1}(\mathsf{C}_\infty)/\hat{\mathcal{G}}_{\mu}  ))$, then
\begin{align*}
k^2\langle \epsilon(k)s , s'\rangle + \langle \epsilon(k)k^2\eta, \eta'\rangle &= k^2\epsilon(k) (\langle s , s'\rangle + \langle \eta, \eta'\rangle)\\
&\stackrel{\text{by (\ref{def:vector-e})}}{=} k^2\epsilon(k)\cdot 0.
\end{align*}
From the assumption that the norm of $e(k)$ is $1$, we deduce that
\[
(k\|\epsilon(k) s\|_s)^2 + \|\epsilon(k)k^2\eta\|_+^2 = 1
\]
\[
k^2 \epsilon(k)^2\langle s, s\rangle + \epsilon(k)^2k^4 \langle \eta,\eta \rangle = 1
\]
\begin{align*}
\epsilon(k)^2 &= 1/(k^2 \langle s,s\rangle + k^4 \langle \eta,\eta \rangle) \\
&\stackrel{\text{by (\ref{def:vector-e})}}{=} 1/((k^4-k^2)\|\eta\|^2 +k^2)
\end{align*}
When $k \leq 1$, there is a unique positive solution for $\epsilon(k)$. Define
\begin{align*}
f_k : \S^n_1 &\to \S^n_k, \\
e &\mapsto e(k)=  \epsilon(k) s + \epsilon(k) k^2 \eta.
\end{align*}
By finding its inverse, we can easily show that $f_k$ is actually a diffeomorphism. When $k \to \infty$,
\[
\|\epsilon(k)k^2\eta\|_+ = \sqrt{\frac{1}{(k^4-k^2)\|\eta\|_+^2 + k^2}}k^2\|\eta\|_+ \to 1.
\]
Hence the angle between $H^2(F_{\hat{\mathsf{C}}_1})$ and $H^+_{L^2}(\hat{N}_1)$ approaches $0$. From the condition (\ref{pi-bot-small}) %the property of $\eta(b)$ with respect to $H^2(F_{\hat{\mathsf{C}}_1})$
we have
\[
\|\pi^\bot_{\hat{\mathsf{C}}_1}(\eta(b))\| \to \|\pi_+^\bot\eta(b)\| \leq \epsilon\|\eta(b)\|
\]
as $k\to \infty$.

We pick (\ref{tooLarge}) up and deduce that
\begin{align}
\|Q_- K_r(b)\| &= \|Q_- \Psi_r \eta(b)\| \notag\\
&\leq \| Q_- \Psi_r \pi_{\hat{\mathsf{C}}_1}\eta(b)\| + \| Q_- \Psi_r \pi^\bot_{\hat{\mathsf{C}}_1} \eta(b)\|\notag \\
&\stackrel{\text{by (\ref{equ-pi-bot-d-asymp})}}{\leq}  \| Q_- \Psi_r \pi_{\hat{\mathsf{C}}_1}\eta(b)\| + \|  \pi^\bot_{\hat{\mathsf{C}}_1} \eta(b)\|\notag\\
&\stackrel{\text{by \ref{pi-bot-small}}}{\leq}  \| Q_- \Psi_r \pi_{\hat{\mathsf{C}}_1}\eta(b)\| + \delta\|  \eta(b)\|\notag\\
&\stackrel{\text{by \ref{D-nontrvial}}}{\leq}  \| Q_- \Psi_r \pi_{\hat{\mathsf{C}}_1}\eta(b)\| + \delta^2\|  b\|.
\end{align}
Then as what we have already seen in (\ref{cntrol-S}),
\begin{align*}
\|SQ_-iK_r(t)\|
&\leq 
C_k r^2 (\epsilon(r) \| \pi_{\hat{\mathsf{C}}_1}\eta(b) \| + \delta^2\|b\|) \\
&\leq 
C_k r^2 (\epsilon(r) \| \eta(b) \| + \delta^2\|b\|) \notag\\
&\stackrel{\text{by \ref{D-nontrvial}}}{\leq} 
C_k r^2 ( \epsilon(r)\delta\|b\|+ \delta^2\|b\|)\\
&\stackrel{\text{by (\ref{equ:BK7.17})}}{\leq} 
C_k r^2 ( \epsilon(r)+ \delta)\frac{4}{C_0}r^{-6} \\
&= C_k ( \epsilon(r)+ \delta)\frac{4}{C_0} r^{-4}.
\end{align*}
Therefore, for sufficiently large $r$ and small $\delta$, 
\[
\|SQ_-iK_r(t)\| \leq \frac{1}{2}r^{-4},
\]
which means that $\tilde{{F}}$ sends the ball $B_\bot(r^{-4})$ to itself.

%Let $k$ be a large number, then $k\|\cdot\|_s$ is equivalent to $\|\cdot\|_s$. 

%Let 
%\[
%\pi^\bot_{\hat{\mathsf{C}}_1} \eta(b) = e=s + \eta
%\]
%where $s$ is the spinor part of $e$ and $\eta \in \Lambda_+^2 T^*\hat{N}_1$. Then
%\[
%\eta(b) - e = -s + (\eta(b)-\eta)
%\]
%and by the definition of the projection, 
%\[
%\langle e, \eta(b)
%\]

%Let $\{e_i\}$ be an orthonormal basis of $\widehat{\underline{SW}}_{\hat{\mathsf{C}}_1} (T_{\hat{\mathsf{C}}_1} (\partial_\infty^{-1}(\mathsf{C}_\infty)/\hat{\mathcal{G}}_{\mu}  ))$ such that 
%\[
%\|e_i\| := \|s_i\|_s  + \|\eta_i\|_+ = 1,
%\]
%where $s_i$ is the spinor part of $e_i$ and $\eta_i \in \Lambda_+^2 T^*\hat{N}_1$. Note that by the definition of projection,
%\[
%\sum_{i} \langle \eta(b), \eta_i\rangle_+ e_i =\pi^\bot_{\hat{\mathsf{C}}_1} \eta(b).
%\]
%Let $N$ be a positive integer such that 
%\[
%\sum_{i\notin \{1,\cdots,N\}} \langle \eta(b), \eta_i\rangle_+ < \frac12 \epsilon(r).
%\]
%Let $C(b,N) = \sum_{i\in \{1,\cdots,N\}}\langle \eta(b), \eta_i\rangle_+$. This is what we want to control. Observe that if the norm of the spinor part of $e_i$ increases, to make the norm of $e_i$ equal $1$, the coefficient of $e_i$ should decrease. Hence it's able to choose a large enough $k$ such that under the new norm $k\|\cdot\|_s$, $\{e'_i\}$ is a new set of orthonormal basis and 
%\[
%e'_i = \frac{\epsilon(r)}{2C(b,N,i)}e_i,
%\]
%where $i\in \{1,\cdots,N\}$, $C(b,N,i)>C(b,N)$.
\end{proof}

We have two contraction maps:
\begin{lemma}[\cite{BK2020} Lemma 7.18]
For fixed $\hat{C}_0$ and $t\in U_i$, the map
\begin{align}\label{equ:Fcontraction}
\tilde{\mathcal{F}}:Y_+^2 &\to Y_+^2\\
\tilde{\mathcal{F}}(\underline{\hat{\mathsf{C}}}^\bot)&= U^\bot - SQ_-R(\underline{\hat{\mathsf{C}}}^\bot + \underline{\hat{\mathsf{C}}}_0) + SQ_-iK_r(b)
\end{align}
is a contraction for sufficiently large $r$.
\end{lemma}

Let
\[
\Phi(\underline{\hat{\mathsf{C}}}_0, b)
\]
be the fixed point of $\tilde{{F}}$. $\Phi$ is a smooth function with respect to $\underline{\hat{\mathsf{C}}}_0$ and $b$. Any tuple 
\[
(\underline{\hat{\mathsf{C}}}_0,b, \underline{\hat{\mathsf{C}}}^\bot = \Phi(\underline{\hat{\mathsf{C}}}_0, b) )
\]
is a solution of the second equation of (\ref{expand-decompose-F=0}). Next we consider the first equation of (\ref{expand-decompose-F=0}):
\begin{equation}\label{equ:1-equ}
P^r_- \tilde{\mathcal{F}}_r(\underline{\hat{\mathsf{C}}},b) =P^r_- {\mathcal{F}}_r(\underline{\hat{\mathsf{C}}}_0 + \underline{\hat{\mathsf{C}}}^\bot) + P^r_- iK_r(b)   = P^r_-  \mathcal{F}_r(0) + \hat{\mathcal T}_r(\underline{\hat{\mathsf{C}}}_0)+ P^r_- R( \underline{\hat{\mathsf{C}}}_0 + \underline{\hat{\mathsf{C}}}^\bot) + P^r_- iK_r(b) = 0.
\end{equation}
Recall that the isomorphism (\ref{equ:P-minus-iso}) is given by $P^r_- iK_r$. Hence it admits an inverse operator $J_r$. \cite{BK2020} uses the same strategy by applying $J_r$ to this equation and then try to solve:
\[
J_r P^r_- {\mathcal{F}}_r(\underline{\hat{\mathsf{C}}}_0 + \underline{\hat{\mathsf{C}}}^\bot) +b    = 0.
\]
They prove
\begin{lemma}[\cite{BK2020} Lemma 7.22]\label{lem:BKlem7.22}
\[
G(b) := -J_r P^r_- \mathcal{F}_r(\underline{\hat{\mathsf{C}}}_0+\Phi(\underline{\hat{\mathsf{C}}}_0,b)),
\]
is a contraction on $\{t\text{ }|\text{ } \|t\| \leq r^{-6} \}$ for sufficiently large $r$.
\end{lemma}
%%%%%%%begin

Might be deleted: In \cite{BK2020}, they consider $0$-dimensional case, so they don't need to prove that $\hat{C}_1$ depends diffentiably on $\hat{C}_0$ and $t$. But in our case, the tensor product is $1$-dimensional, so we have to prove this fact:
\begin{theorem}
$\hat{C}_1$ depends diffentiably on $\hat{C}_0$ and $t$.
\end{theorem}
 %%%%%%%end

\begin{proof}[Proof of Theorem \ref{Thm:Local-gluing-theory-for-type1}]
Let 
\[
\Psi (\underline{\hat{\mathsf{C}}}_0)
 \]
be the unique fixed point of $G$. Then for any $\underline{\hat{\mathsf{C}}}_0\in B_0(r^{-4})\subset \H_r^+$, there exist a unique tuple
\[
(\underline{\hat{\mathsf{C}}}_0,b, \underline{\hat{\mathsf{C}}}^\bot  )
\]
that can solve (\ref{expand-decompose-F=0}), where $b =\Psi (\underline{\hat{\mathsf{C}}}_0)$ and $ \underline{\hat{\mathsf{C}}}^\bot = \Phi(\underline{\hat{\mathsf{C}}}_0, b)$. 

Next for a specified $b\in B$ in Proposition \ref{discreteCircle} we identify the finite dimensional space $\H_r^+$ by the diagram (\ref{equ:3T-param}).  
\begin{equation}\label{equ:3T-param} 
\xymatrix{
&
0\ar[d] &
0\ar[d] & 
0 \ar[d]& 
\\
0\ar[r] &
\ker\Delta_+^c \ar[d] \ar[r]^{S_r} &
H^1_{\hat{\mathsf{C}}_1}\oplus H^1_{\hat{\mathsf{C}}_2}\ar[d] \ar[r]^{\Delta_+^c} &
L_1^+ + L_2^+ \ar[d]\ar[r] & 
0\\
0\ar[r] &
{\H}_r^+ \ar[d] \ar[r]^(0.3){S_r} &
\ker_{ex}\hat{\mathcal{T}}_{\hat{\mathsf{C}}_1}\oplus \ker_{ex}\hat{\mathcal{T}}_{\hat{\mathsf{C}}_2} \ar[d]^{\partial_\infty^0} \ar[r]^(0.64){\Delta_+^c} &
\hat{L}_1^+ + \hat{L}_2^+ \ar[d]\ar[r] &
 0 \\
0\ar[r] &
\ker\Delta_+^0 \ar[d] \ar[r]^{S_r} &
\mathfrak{C}_1^+\oplus \mathfrak{C}_2^+ \ar[d] \ar[r]^{\Delta_+^0} &
\mathfrak{C}_1^+ + \mathfrak{C}_2^+ \ar[d]\ar[r] & 
0 \\
&
0 &
0& 
0& 
\\
}
\tag{\textbf{T}}
\end{equation}
We always glue an irreducible to a reducible. 
Hence $ \ker\Delta_+^0 $ is always trivial. By (\ref{equ:L+=R}), $ \dim \H^+_r = \dim \ker\Delta_+^c  $ is always $\dim H^1_{\hat{\mathsf{C}}_1} + \dim H^1_{\hat{\mathsf{C}}_2} - 1$. Hence for $\hat{N}_1 = X_0$ and $\hat{N}_2 =\S^1\times D^3 $ or $D_2\times \S^2$, $ \dim \H^+_r$ is the dimension of the fiber product of the moduli spaces on $\hat{N}_1$ and $\hat{N}_2$. 
By (\ref{L1}), $\dim H^1_{\hat{\mathsf{C}}_1}=1$ and recall that for a specified $b\in B$ in Proposition \ref{discreteCircle}, the $\eta(b)$-monopole $\hat{\mathsf{C}}_2$ satisfies $\dim H^1_{\hat{\mathsf{C}}_2}=1$ for $\hat{N}_2 =\S^1\times D^3 $ and  $\dim H^1_{\hat{\mathsf{C}}_2}=0$ for $\hat{N}_2 = D_2\times \S^2$ by Corollary \ref{cor:dimOfReducible}. Hence $ \dim \H^+_r$ would be either $1$ or $0$, respectively.

Then repeat the story in section \ref{section:global-gluing} we prove that $\H^+_r$ approaches the tangent bundle of the fiber product of the moduli spaces on $\hat{N}_1$ and $\hat{N}_2$, and therefore, the genuine moduli space of $X(r)$ is isotopic to the fiber product in its configuration space.
\end{proof}

\subsection{Local gluing theory for case $0$}\label{sec:local-gluing-theory-for-case-0}
Now we try to recover Theorem \ref{Thm:Local-gluing-theory-for-type1} with the surgery on a homologically trivial loop. In this case, by the analysis in the previous section, the solution on $X_0$-side we glue is of type $0$. Hence the fiber product in the configuration space of $X(r)$ is also of type $0$. All computation carries on until (\ref{lemma:BK7.17}). Now we cannot find an inverse of $P_-^rK_r$ since we do not have the isomorphism (\ref{equ:P-minus-iso}). Instead we have the isomorphism (\ref{equ:P-minus-iso-1}) and it is given by a restriction of $P_-^riK_r$: 
\[
P_-^rK_r|_{d_b\eta(V)} : d_b\eta(V) \stackrel{\cong}{\to} \H_r^-.
\]
Hence we are only able to find an inverse operator of $P_-^rK_r|_{d_b\eta(V)} $:
\[
J_{r,b} : \H^-_r \to T_bB.
\]
To find a contraction map as in the previous subsection, we want to control the norm of $J_{r,b}$. 

The first step is to estimate the norm of $P_-^rK_r$. Lemma \ref{lemma:BK7.16} does not hold for this case, but the statement is still true if we restrict the domain of $P_-^rK_r$ to $V$:
\begin{lemma}\label{lemma:BK7.16-2}
Let $\hat{N}$ be the cylindrical manifold obtained from $X_0$. Fix any $\eta \in \mathcal{Z}_{reg}$. For any positive number $C_0<1$, any $b\in B$, $t\in V \subset T_b B$, and any solution $\hat{\mathsf{C}}_1$ for the perturbation $\eta(b)$, for all sufficiently large $r$,
\[
\| P^r_- \circ \Psi_r \circ D_b\eta(t)\|_{L^2_\mu(X(r))} > C_0\| \pi_{\hat{\mathsf{C}}_1} D_b\eta(t)\|_{L^2_\mu(X_0)}.
\]
\end{lemma}

Now we can control the location of the genuine solutions on $V$:
\begin{lemma}\label{lemma:parameter-restriction0}
If for sufficiently large $r$, $\underline{\hat{\mathsf{C}}}_0 \in B_0(r^{-4})$, $ \underline{\hat{\mathsf{C}}}^\bot \in B_\bot(r^{-4})$, $t\in V$ and
\[
t = -J_{r,b} P^r_- \mathcal{F}_r(\underline{\hat{\mathsf{C}}}_0+ \underline{\hat{\mathsf{C}}}^\bot ),
\]
then $t \in \{t\text{ }|\text{ } \|t\| \leq r^{-6} \}\cap V$ .
\end{lemma}
\begin{proof}
By Lemma \ref{lemma:BK7.16-2}, 
\begin{align*}
\|P_-^rK_r(t)\|_{L^2_\mu(X(r))} &= \| P^r_- \circ \Psi_r \circ D_b\eta(t)\|_{L^2_\mu(X(r))} \\
&> C_0\| \pi_{\hat{\mathsf{C}}_1} D_b\eta(t)\|_{L^2_\mu(X_0)}.
\end{align*}
By the assumption \ref{piD-nontrvial} on the family of perturbations, we have 
\[
\|P_-^rK_r(t)\|_{L^2_\mu(X(r))} > C_0 \frac{\delta}{2}\|t\|.
\]
Hence the operator norm $\|P_-^rK_r\|_{op} > C_0 \frac{\delta}{2}$. Therefore $\|J_{r,b}\|_{op} < \frac{2}{C_0\delta}$. We have
\begin{align*}
\|G(t)\| &\leq \|J_{r,b}\|_{op} \|P^r_- \mathcal{F}_r(\underline{\hat{\mathsf{C}}}_0+\underline{\hat{\mathsf{C}}}^\bot)\| \\
&\leq  \|J_{r,b}\|_{op}(\|P^r_-  \mathcal{F}_r(0)\| + \|\hat{\mathcal T}_r(\underline{\hat{\mathsf{C}}}_0)\|+ \|P^r_- R( \underline{\hat{\mathsf{C}}}_0 + \underline{\hat{\mathsf{C}}}^\bot) \|)
\end{align*}
As a projection, the operator norm of $P^r_-$ can not be larger than $1$. Recall that the configuration from gluing is an approximation of the genuine solution, so we have the estimate (see Nico00, Lemma 4.5.5)
\[
\|P^r_-  \mathcal{F}_r(0)\|  \leq \|  \mathcal{F}_r(0)\| \leq C_1e^{-\mu r}
\]
and (see Nico00, Lemma 4.5.6)
\[
\|R(\hat{\mathsf{C}})\|_{L^{1,2}} \leq C_2 r^{3/2} \|\hat{\mathsf{C}}\|^2_{L^{2,2}}.
\]
From the definition (\ref{equ:def-of-H-r}) that $\H_r$ is the subspace of $\mathfrak{X}^0$ spanned by
\[
\{ v; \hat{L}_r v = \lambda v, |\lambda| < r^{-2}\},
\]
we deduce that $H^+_r  \subset \mathfrak{X}^0_+ $ is the span of the eigenspaces of  $\hat{\mathcal T}^*_r\hat{\mathcal T}_r$ with the eigenvalues in the range $[0,r^{-4})$ (which is a basic observation of linear algebra, see BK Remark 6.10). Because $\underline{\hat{\mathsf{C}}}_0 \in \H^+_r$ we have
\begin{align*}
 \|\hat{\mathcal T}_r(\underline{\hat{\mathsf{C}}}_0)\|^2 &= \langle \hat{\mathcal T}_r(\underline{\hat{\mathsf{C}}}_0), \hat{\mathcal T}_r(\underline{\hat{\mathsf{C}}}_0)\rangle \\
 &=   \langle \underline{\hat{\mathsf{C}}}_0, \hat{\mathcal T}_r^*\hat{\mathcal T}_r(\underline{\hat{\mathsf{C}}}_0)\rangle\\
 &\leq r^{-4}   \|\underline{\hat{\mathsf{C}}}_0\|^2.
 \end{align*}
Combine these estimates, we deduce 
\begin{align*}
\|G(t)\| &\leq  \|J_{r,b}\|_{op}(\|P^r_-  \mathcal{F}_r(0)\| + \|\hat{\mathcal T}_r(\underline{\hat{\mathsf{C}}}_0)\|+ \|P^r_- R( \underline{\hat{\mathsf{C}}}_0 + \underline{\hat{\mathsf{C}}}^\bot) \|)\\
&\leq  \|J_{r,b}\|_{op}(\|P^r_-  \mathcal{F}_r(0)\| + \|\hat{\mathcal T}_r(\underline{\hat{\mathsf{C}}}_0)\|+ \|P^r_- R( \underline{\hat{\mathsf{C}}}_0 ) \| +  \|P^r_- R(  \underline{\hat{\mathsf{C}}}^\bot ) \|)\\
&\leq \frac{2}{C_0\delta}(C_1e^{-\mu r} + r^{-2}   \|\underline{\hat{\mathsf{C}}}_0\| + C_2 r^{3/2} \|\underline{\hat{\mathsf{C}}}_0\|^2_{L^{2,2}} + C_2 r^{3/2} \|\underline{\hat{\mathsf{C}}}^\bot\|^2_{L^{2,2}} ).
\end{align*}
When $\underline{\hat{\mathsf{C}}}_0 \in B_0(r^{-4})$ and $ \underline{\hat{\mathsf{C}}}^\bot \in B_\bot(r^{-4})$, 
\[
\|G(t)\| \leq \frac{2}{C_0\delta}(C_1e^{-\mu r} + r^{-6} + 2C_2 r^{-13/2} ).
\]
Recall that in Proposition \ref{prop:perturbationCondition} we can choose $\delta$ to be any large positive number, and in Lemma \ref{lemma:BK7.16-2} we can choose $C_0$ to be closed to $1$. Hence when $r$ is large enough, we have $\|G(t)\| \leq r^{-6}$.
\end{proof}

Now we recover Lemma \ref{lemma:F-contraction-1}. Let
\[
\tilde{{F}}_t :\underline{\hat{\mathsf{C}}}^\bot \to U^\bot - SQ_-R(\underline{\hat{\mathsf{C}}}^\bot + \underline{\hat{\mathsf{C}}}_0) + SQ_-iK_r(t).
\]
\begin{lemma}
For all sufficiently large $r$, for any $t\in V$, $\tilde{{F}}_t$ sends the ball $B_\bot(r^{-4})$ to itself and is a contraction.
\end{lemma}
\begin{proof}
The proof is almost the same as Lemma \ref{lemma:F-contraction-1}. The only change is to replace (\ref{equ:BK7.17}) by Lemma \ref{lemma:parameter-restriction0} in the last part.
\end{proof}

Let
\[
\Phi(\underline{\hat{\mathsf{C}}}_0, t)
\]
be the fixed point of $\tilde{{F}}_t$.

\begin{lemma}
When $\underline{\hat{\mathsf{C}}}_0 \in B_0(r^{-4})$,
\[
G(t) := -J_r P^r_- \mathcal{F}_r(\underline{\hat{\mathsf{C}}}_0+\Phi(\underline{\hat{\mathsf{C}}}_0,t))
\]
is a contraction on $\{t\text{ }|\text{ } \|t\| \leq r^{-6} \}\cap V$ for all sufficiently large $r$.
\end{lemma}
\begin{proof}
The proof is exactly the same as the proof of Lemma \ref{lem:BKlem7.22}.
\end{proof}

Let 
\[
\Psi (\underline{\hat{\mathsf{C}}}_0)
 \]
be the unique fixed point of $G$. Then for any $\underline{\hat{\mathsf{C}}}_0\in B_0(r^{-4})\subset \H_r^+$, there exist a unique tuple
\[
(\underline{\hat{\mathsf{C}}}_0,t, \underline{\hat{\mathsf{C}}}^\bot  )
\]
that can solve (\ref{expand-decompose-F=0}), where $t =\Psi (\underline{\hat{\mathsf{C}}}_0) \in V$ and $ \underline{\hat{\mathsf{C}}}^\bot = \Phi(\underline{\hat{\mathsf{C}}}_0, t) \in  \mathcal{Y}_{+}^0$.

Next we identify the finite dimensional space $\H_r^+$ by the diagram (\ref{equ:3T-param-2}).  
\begin{equation}\label{equ:3T-param-2} 
\xymatrix{
&
0\ar[d] &
0\ar[d] & 
0 \ar[d]& 
\\
0\ar[r] &
\ker\Delta_+^c \ar[d] \ar[r]^{S_r} &
H^1_{\hat{\mathsf{C}}_1}\oplus H^1_{\hat{\mathsf{C}}_2}\ar[d] \ar[r]^{\Delta_+^c} &
L_1^+ + L_2^+ \ar[d]\ar[r] & 
0\\
0\ar[r] &
{\H}_r^+ \ar[d] \ar[r]^(0.3){S_r} &
\ker_{ex}\hat{\mathcal{T}}_{\hat{\mathsf{C}}_1}\oplus \ker_{ex}\hat{\mathcal{T}}_{\hat{\mathsf{C}}_2} \ar[d]^{\partial_\infty^0} \ar[r]^(0.64){\Delta_+^c} &
\hat{L}_1^+ + \hat{L}_2^+ \ar[d]\ar[r] &
 0 \\
0\ar[r] &
\ker\Delta_+^0 \ar[d] \ar[r]^{S_r} &
\mathfrak{C}_1^+\oplus \mathfrak{C}_2^+ \ar[d] \ar[r]^{\Delta_+^0} &
\mathfrak{C}_1^+ + \mathfrak{C}_2^+ \ar[d]\ar[r] & 
0 \\
&
0 &
0& 
0& 
\\
}
\tag{\textbf{T}}
\end{equation}
By (\ref{L0}), $\dim H^1_{\hat{\mathsf{C}}_1}=0$, so $L^+_1 =\text{im } (H^1_{\hat{\mathsf{C}}_1}) =0$. Recall that by Corollary \ref{cor:dimOfReducible}, for a specified $b\in B$ and $\hat{N}_2 =\S^1\times D^3 $, the $\eta(b)$-monopole $\hat{\mathsf{C}}_2$ satisfies $\dim H^1_{\hat{\mathsf{C}}_2}=1$ and $\dim L_2^+ = 1$. By (\ref{equ:H-=kerDelta-c}), $ \dim \H^+_r = \dim \ker\Delta_+^c  = \dim H^1_{\hat{\mathsf{C}}_1} + \dim H^1_{\hat{\mathsf{C}}_2} - 1$. Hence $ \dim \H^+_r = 0$.

In this case, $\H^+_r$ is not longer the tangent bundle of the fiber product obtained by gluing, which has formal dimension $1$ from our dimension assumption. Instead, for each configuration $\hat{\mathsf{C}}_1 \#_r \hat{\mathsf{C}}_2$, we can assign a $1$-dimensional space $\R v$ (see Figure \ref{fig:Type0image}), which is orthogonal to $V$. Now 
\[
(\Psi (0) , \Phi(0, \Psi (0))) \in V \oplus \mathcal{Y}_{+}^0
\]
is the unique genuine solution on $V$. Such $V$ and the corresponding genuine solution form a normal bundle and a section of this bundle. Hence the genuine moduli space of $X(r)$ is isotopic to the fiber product in its configuration space. 

For $\hat{N}_2 =\S^1\times D^3 $, the configuration obtained by gluing is of type 1 (see Figure \ref{Fig:O0-D2}), so the local gluing theory is described in the previous section.

\subsection{The proof of family surgery formulas}
\begin{theorem}\label{ThmA}
Let $B$ be any compact manifold and $X$ be a smooth $4$-manifold with $b^+(X) > \dim B + 1$ and $H^1(X) = \Z$. Let $s$ be a $\text{spin}^c_{GL}$ structure on $X$ (defined in Subsection \ref{subsection:SpinGL}) such that 
\[
\frac{c_1^2(s) -(2\chi(X) + 3\sigma(X))}{4} + \dim B = 1.
\]

Suppose that $E_X$ is a bundle over $B$ with fiber $X$ and structure group $G = \Aut(X, s,\mathcal{O})$ (defined in (\ref{equ:defAut}). Let $\sigma = (g, \eta)$ be a generic parameter family (defined in Theorem \ref{thm:regular}). Assume that $\Theta = PD([E_{\C P^\infty}]) \in H^1(\mathcal{F}\mathcal{B}^*_{X})$ is a well defined cohomology class in the configuration space (see Subsection \ref{subsection:Parametrized}). Let $\FM(s, \sigma)$ be the parameterized Seiberg-Witten moduli space defined in (\ref{equ:parametrized}). Denote the integral
\[
\langle \FM(E_X, s, \sigma) , \Theta\rangle
\]
by $FSW^\Theta(E_X, s)$ since it doesn't depend on the choice of the parameter family. 

Let $\tau$ be a generator of $H^1(X)$. Suppose $E_{\S^1} $ is an $\S^1$-subbundle of $E_X$ and $\tau$ evaluates $1$ at each fiber of $E_{\S^1} $. Assume that 
family $1$-surgery for $E_X$ at $E_{\S^1}$ is well defined (see Subsection \ref{subsection:familysurgery}). Denote the resulting bundle by $E_{X'}$ and the $\text{spin}^c_{GL}$ structure by $s'$. Denote the number of signed points
\[
\#  \FM(E_{X'}, s', \sigma)
\]
by $FSW(E_{X'}, s')$. Then
\[
FSW^\Theta(E_X, s) = FSW(E_{X'}, s').
\]
\end{theorem}
\begin{proof}

The parameterized configuration space $\mathcal{F}\mathcal{B}^*_{X} $ is a bundle over $B$ with fiber $\mathcal{B}_X^*$.

We will apply the gluing results in the previous sections to 
\[
X_0 \cup_{\S^1\times \S^2} \S^1 \times D^3
\]
and 
\[
X_0 \cup_{\S^1\times \S^2} D^2 \times \S^2.
\]

First, we consider the gluing theory for $X_0 \cup_{\S^1\times \S^2} \S^1 \times D^3$. By Theorem \ref{Thm:Local-gluing-theory-for-type1}, we have the isotopy (through out the proof we will omit the input of the $\text{spin}^c_{GL}$ structure and the perturbation family, for example $\FM(X_0) := \FM(E_{X_0}, s|_{X_0}, \sigma|_{X_0}) $)
\[
\FM(X_0 \cup_{\S^1\times \S^2} \S^1 \times D^3) \cong \FM(X_0) \times_{\FM(\S^1\times \S^2)} \FM(\S^1\times D^3).
\]

Recall that by the choice of the perturbation family, the Seiberg-Witten equations on 
$\S^1\times \S^2$ and $\S^1 \times D^3$ have $0$ perturbation. So $\FM (\S^1\times \S^2) $ and $ \FM (\S^1 \times D^3)$ consist of flat connections, and both of them is a circle bundle over $B$. The map between them is the identity. 

Denote by $\mathcal{F}\mathcal{B}^{red}, \mathcal{F}\mathcal{B}^* \subset \mathcal{F}\mathcal{B}$ the reducible configuration space, the irreducible configuration space, and the configuration space. Denote by $f: \mathcal{F}\mathcal{B} \to \mathcal{F}\mathcal{B}^{red}$ the forgetting map that throws away the spinor. Since the structure map of $E_X$ preserves the homology orientation, $\mathcal{F}\mathcal{B}^{red} (X)$ is homotopy equivalent to the trivial bundle $\S^1 \times B$. Hence we have the following diagram:
\[\begin{tikzcd}
	{\FM(X)} &&&&& {\FM({\mathbb{S}^1\times D^3})} \\
	& {\mathcal{F}\mathcal{B}_X} &&& {\mathcal{F}\mathcal{B}_{\mathbb{S}^1\times D^3}} \\
	&& {\mathbb{S}^1{\times}B} & {\mathbb{S}^1{\times}B} \\
	&& {\mathbb{S}^1{\times}B} & {\mathbb{S}^1{\times}B} \\
	& {\mathcal{F}\mathcal{B}_{X_0}} &&& {\mathcal{F}\mathcal{B}_{\mathbb{S}^1\times \mathbb{S}^2}} \\
	{\FM({X_0})} &&&&& {\FM({\mathbb{S}^1\times \mathbb{S}^2})}
	\arrow[dashed, from=1-1, to=1-6]
	\arrow["inclusion"{pos=0.5}, from=1-1, to=2-2]
	\arrow[dashed, from=1-1, to=6-1]
	\arrow["inclusion"'{pos=0.5}, from=1-6, to=2-5]
	\arrow["{=}"{description, pos=0.4}, bend left=18, from=1-6, to=3-4]
	\arrow["="',from=1-6, to=6-6]
	\arrow[from=2-2, to=2-5]
	\arrow["f", from=2-2, to=3-3]
	\arrow[from=2-2, to=5-2]
	\arrow["f"', from=2-5, to=3-4]
	\arrow["{\partial_\infty}", from=2-5, to=5-5]
	\arrow["{=}", from=3-3, to=3-4]
	\arrow["{=}",from=3-3, to=4-3]
	\arrow["="',from=3-4, to=4-4]
	\arrow["{\cong}",from=4-3, to=4-4]
	\arrow["f"'{pos=0.6}, from=5-2, to=4-3]
	\arrow["{\partial_\infty}"', from=5-2, to=5-5]
	\arrow["f"{pos=0.6}, from=5-5, to=4-4]
	\arrow["inclusion"'{pos=0.5}, from=6-1, to=5-2]
	\arrow[from=6-1, to=6-6]
	\arrow["{=}"{description, pos=0.4}, bend right=18, from=6-6, to=4-4]
	\arrow["inclusion"{pos=0.5}, from=6-6, to=5-5]
\end{tikzcd}\]
where the largest square is a pullback square, and all triangles and squares commute. By the property of the pullback diagram, $\FM(X)$ is isotopic to $\FM(X_0)$ as $1$-dimensional manifolds in $\S^1 \times B$. Hence 
\begin{equation}\label{equ:Mx=Mx0}
\langle \FM(X) , \Theta\rangle = \langle \FM(X_0) ,  \Theta\rangle
\end{equation}

Second, we consider the gluing theory for $X_0 \cup_{\S^1\times \S^2} D^2 \times \S^2$. By Theorem \ref{Thm:Local-gluing-theory-for-type1}, we have the isotopy
\[
\FM(X_0 \cup_{\S^1\times \S^2} D^2 \times \S^2) \cong \FM(X_0) \times_{\FM(\S^1\times \S^2)} \FM(D^2 \times \S^2).
\]
The only difference is that, for each parameter, the configuration space $\mathcal{B}^{red}(D^2 \times \S^2)$ is contractible since $H^1(D^2 \times \S^2)$ is trivial. So $\mathcal{F}\mathcal{B}^{red}(D^2 \times \S^2) \cong B$. But we cannot recover the previous commutative diagram since the squre  
\[\begin{tikzcd}
	& {\mathcal{F}\mathcal{B}_{D^2 \times\mathbb{S}^2 }} \\
	B \\
	{\mathbb{S}^1{\times}B} \\
	& {\mathcal{F}\mathcal{B}_{\mathbb{S}^1\times \mathbb{S}^2}}
	\arrow["f"', from=1-2, to=2-1]
	\arrow["{\partial_\infty}"', from=1-2, to=4-2]
	\arrow["{\{1\}\times B}"', from=2-1, to=3-1]
	\arrow["f", from=4-2, to=3-1]
\end{tikzcd}\]
is noncommutative. However, $\FM(D^2 \times \S^2) \cong B$ maps to the zero zection of $\FM({\mathbb{S}^1\times \mathbb{S}^2}) \cong \S^1 \times B$ and thus the diagram 
\begin{equation}\label{equ:flat-connection-is-zero-section}
\begin{tikzcd}
	&& {\FM_{D^2 \times\mathbb{S}^2 }} \\
	& {\mathcal{F}\mathcal{B}_{D^2 \times\mathbb{S}^2 }} \\
	B \\
	{\mathbb{S}^1{\times}B} \\
	& {\mathcal{F}\mathcal{B}_{\mathbb{S}^1\times \mathbb{S}^2}} \\
	&& {\FM_{\mathbb{S}^1\times \mathbb{S}^2}}
	\arrow["inclusion"'{pos=0.4}, from=1-3, to=2-2]
	\arrow["{=}"{description, pos=0.4}, bend left=18, from=1-3, to=3-1]
	\arrow["{\partial_\infty}"',  from=1-3, to=6-3]
	\arrow["f"', from=2-2, to=3-1]
	\arrow["\{1\}\times B"', from=3-1, to=4-1]
	\arrow["f"{pos=0.5}, from=5-2, to=4-1]
	\arrow["{=}"{description, pos=0.4}, bend right=18, from=6-3, to=4-1]
	\arrow["inclusion"{pos=0.5}, from=6-3, to=5-2]
\end{tikzcd}\end{equation}
still commutes. Therefore, we have the commutative diagram
\[\begin{tikzcd}
	{\FM({X'})} &&&&& {\FM({D^2 \times \S^2})} \\
	&&&& {\mathcal{F}\mathcal{B}_{D^2 \times \S^2}} \\
	&&& B \\
	&& {\mathbb{S}^1{\times}B} & {\mathbb{S}^1{\times}B} \\
	& {\mathcal{F}\mathcal{B}_{X_0}} &&& {\mathcal{F}\mathcal{B}_{\mathbb{S}^1\times \mathbb{S}^2}} \\
	{\FM({X_0})} &&&&& {\FM({\mathbb{S}^1\times \mathbb{S}^2})}
	\arrow[dashed, from=1-1, to=1-6]
	\arrow[dashed, from=1-1, to=6-1]
	\arrow["inclusion"', from=1-6, to=2-5]
	\arrow["{{=}}"{description, pos=0.4}, bend left=18, from=1-6, to=3-4]
	\arrow["{\partial_\infty}"', from=1-6, to=6-6]
	\arrow["f"', from=2-5, to=3-4]
	\arrow["{\{1\}\times B}"', from=3-4, to=4-4]
	\arrow["{{\cong}}", from=4-3, to=4-4]
	\arrow["f"'{pos=0.6}, from=5-2, to=4-3]
	\arrow["{{\partial_\infty}}"', from=5-2, to=5-5]
	\arrow["f"{pos=0.6}, from=5-5, to=4-4]
	\arrow["inclusion"', from=6-1, to=5-2]
	\arrow["{{\partial_\infty}}", from=6-1, to=6-6]
	\arrow["{{=}}"{description, pos=0.4}, bend right=18, from=6-6, to=4-4]
	\arrow["inclusion", from=6-6, to=5-5]
\end{tikzcd}\]

Recall that by Propsition \ref{discreteCircle}, $\FM(X_0)$ consists of finite many circle, and they are contained in finite many fibers of $\S^1 \times B$ under the map $f \circ inclusion$. In addition, from the sequence (\ref{L1}), we deduce that the map from $H^1_{\hat{\mathsf{C}}_1}$ (the tagent space of $\FM(X_0)$) to $H^1(B_{\hat{\mathsf{C}}_1})$ (the tangent space of $\FM({\mathbb{S}^1\times \mathbb{S}^2})$) is surjective, so $\FM(X_0) \cong \partial_\infty(\FM(X_0))$ intersects the zero section of $\FM({\mathbb{S}^1\times \mathbb{S}^2}) \cong \S^1 \times B$ transversally. We have:
\begin{enumerate}
\item[\textbullet]  Their intersection is just $\FM(X')$, a set of finite many points, by the property of the fiber product.
\item[\textbullet]  the number of their intersection is $\langle \FM(X_0) , \Theta\rangle$, by the definition of $\Theta$.
\end{enumerate}
Combine these with (\ref{equ:Mx=Mx0}), the conclusion follows.
\end{proof}

\begin{remark}
It's important to note that. Thanks to Propsition \ref{discreteCircle}
\end{remark}

In the previous theorem, the structure group of $E_X$ is chosen to preserve the homology orientation. If it is not the case, the space of reducible solutions might be nonorientable $\S^1$-bundle. For example, when $B=\S^1$ and $E_X$ is the mapping torus of some reflection map that reverses the homology orientation, the space of reducible solutions is a Klein bottle. In this case $\Theta = PD([E_{\C P^\infty}]) \in H^1(\mathcal{F}\mathcal{B}^*_{X},\Z_2)$. The surgery formula still holds:
\begin{theorem}
Use the notation of Theorem \ref{ThmA}. But assume the following instead:
\begin{itemize}
\item The structure group of $E_X$ is $G = \Aut(X, s)$;
\end{itemize}
We can still define a $\Z_2$ invariant 
\[
FSW^{\Theta , \Z_2}(E_X, s) := 
\langle \FM(E_X, s, \sigma) , \Theta\rangle
\]
and a $\Z_2$ invariant
\[
FSW^{\Z_2}(E_{X'}, s') := \#  \FM(E_{X'}, s', \sigma).
\]
We still have 
\[
FSW^{\Theta , \Z_2}(E_X, s) = FSW^{\Z_2}(E_{X'}, s') .
\]
\end{theorem}
\begin{proof}
The proof is exactly the same as the previous one, with $\S^1 \times B$ in those diagrams replaced by some possible nonorientable $\S^1$-bundle.
\end{proof}

Now we consider the case where we do the surgery on a homologically trivial loop. We have the following vanishing result:
\begin{theorem}\label{thm:nullhomology-orientable}
Use the notation of Theorem \ref{ThmA}. But assume the following instead/additionally:
\begin{itemize}
\item $\dim B > 0$;
\item $E_{\S^1} $ is an orientable $\S^1$-subbundle of $E_X$;
\item $H^1(X)$ is trivial and each fiber of $E_{\S^1} $ is homologically trivial in the fiber of $E_X$.
\end{itemize}
Then
\[
FSW(E_{X'}, s')=0.
\]
\end{theorem}
\begin{proof}
In this case, we are doing surgery on a homologically trivial loop of $X$. We will apply the gluing results in Section \ref{sec:local-gluing-theory-for-case-0} to 
\[
X_0 \cup_{\S^1\times \S^2} \S^1 \times D^3
\]
and 
\[
X_0 \cup_{\S^1\times \S^2} D^2 \times \S^2.
\]

First, we consider the gluing theory for $X_0 \cup_{\S^1\times \S^2} \S^1 \times D^3$. We still have 
\[
\FM(X_0 \cup_{\S^1\times \S^2} \S^1 \times D^3) \cong \FM(X_0) \times_{\FM(\S^1\times \S^2)} \FM(\S^1\times D^3),
\]
but now $ \FM(X_0) \cong B$ and we don't have the commutative diagram 
\[\begin{tikzcd}
	& B & {\mathbb{S}^1{\times}B} \\
	{\mathcal{F}\mathcal{B}_{X_0}} &&& {\mathcal{F}\mathcal{B}_{\mathbb{S}^1\times \mathbb{S}^2}}
	\arrow[from=1-2, to=1-3]
	\arrow["f"'{pos=0.6}, from=2-1, to=1-2]
	\arrow["{{\partial_\infty}}"', from=2-1, to=2-4]
	\arrow["f"{pos=0.6}, from=2-4, to=1-3]
\end{tikzcd}\]
Unlike (\ref{equ:flat-connection-is-zero-section}), we don't know the top arrow in the diagram 
\[\begin{tikzcd}
	&& B & {\mathbb{S}^1{\times}B} \\
	& {\mathcal{F}\mathcal{B}_{X_0}} &&& {\mathcal{F}\mathcal{B}_{\mathbb{S}^1\times \mathbb{S}^2}} \\
	{\FM({X_0})} &&&&& {\FM({\mathbb{S}^1\times \mathbb{S}^2})}
	\arrow[from=1-3, to=1-4]
	\arrow["f"'{pos=0.6}, from=2-2, to=1-3]
	\arrow["f"{pos=0.6}, from=2-5, to=1-4]
	\arrow["inclusion"', from=3-1, to=2-2]
	\arrow[from=3-1, to=3-6]
	\arrow["{{=}}"{description, pos=0.4},bend right=18, from=3-6, to=1-4]
	\arrow["inclusion", from=3-6, to=2-5]
\end{tikzcd}\]
neither, since $\FM(X_0)$ can be any $1$-manifold. 

The solution is to record more infomation then just the reducible solution. We want to remenber the holonomy. Let $hol: {\mathcal{F}\mathcal{B}_{X_0}}  \to \R \times B$ be the holonomy along the loop that we do the surgery. Then the diagram
\[\begin{tikzcd}
	&& \R \times B  & {\mathbb{S}^1{\times}B} \\
	& {\mathcal{F}\mathcal{B}_{X_0}} &&& {\mathcal{F}\mathcal{B}_{\mathbb{S}^1\times \mathbb{S}^2}} \\
	{\FM({X_0})} &&&&& {\FM({\mathbb{S}^1\times \mathbb{S}^2})}
	\arrow["{{p}}", from=1-3, to=1-4]
	\arrow["hol"'{pos=0.6}, from=2-2, to=1-3]
	\arrow["{{\partial_\infty}}"', from=2-2, to=2-5]
	\arrow["f"{pos=0.6}, from=2-5, to=1-4]
	\arrow["inclusion"', from=3-1, to=2-2]
	\arrow[from=3-1, to=3-6]
	\arrow["{{=}}"{description, pos=0.4},bend right=18, from=3-6, to=1-4]
	\arrow["inclusion", from=3-6, to=2-5]
\end{tikzcd}\]
commutes, where $p$ is the covering map times the identity on $B$.
%$\R \to \S^1$ on each fiber. 
Therefore, we have the commutative diagram 
\[\begin{tikzcd}
	{\FM(X)} &&&&& {\FM({\mathbb{S}^1\times D^3})} \\
	& {\mathcal{F}\mathcal{B}_X} &&& {\mathcal{F}\mathcal{B}_{\mathbb{S}^1\times D^3}} \\
	&& {\R{\times}B} & {\mathbb{S}^1{\times}B} \\
	&& {\R{\times}B} & {\mathbb{S}^1{\times}B} \\
	& {\mathcal{F}\mathcal{B}_{X_0}} &&& {\mathcal{F}\mathcal{B}_{\mathbb{S}^1\times \mathbb{S}^2}} \\
	{\FM({X_0})} &&&&& {\FM({\mathbb{S}^1\times \mathbb{S}^2})}
	\arrow[dashed, from=1-1, to=1-6]
	\arrow["inclusion", from=1-1, to=2-2]
	\arrow[dashed, from=1-1, to=6-1]
	\arrow["inclusion"', from=1-6, to=2-5]
	\arrow["{{=}}"{description, pos=0.4}, bend left=18, from=1-6, to=3-4]
	\arrow["{=}"', from=1-6, to=6-6]
	\arrow[from=2-2, to=2-5]
	\arrow["hol", from=2-2, to=3-3]
	\arrow[from=2-2, to=5-2]
	\arrow["f"', from=2-5, to=3-4]
	\arrow["{{\partial_\infty}}", from=2-5, to=5-5]
	\arrow["p", from=3-3, to=3-4]
	\arrow["{=}", from=3-3, to=4-3]
	\arrow["{=}"', from=3-4, to=4-4]
	\arrow["p", from=4-3, to=4-4]
	\arrow["hol"'{pos=0.6}, from=5-2, to=4-3]
	\arrow["{{\partial_\infty}}"', from=5-2, to=5-5]
	\arrow["f"{pos=0.6}, from=5-5, to=4-4]
	\arrow["inclusion"', from=6-1, to=5-2]
	\arrow[from=6-1, to=6-6]
	\arrow["{{=}}"{description, pos=0.4}, bend right=18, from=6-6, to=4-4]
	\arrow["inclusion", from=6-6, to=5-5]
\end{tikzcd}\]
By the property of the pullback squre, we have
\begin{equation}\label{equ:Mx-iso-Mx0}
 \FM(X) \cong  \FM(X_0)
\end{equation}
in $\R \times B$.

Second, we consider the gluing theory for $X_0 \cup_{\S^1\times \S^2} D^2 \times \S^2$. We have the isotopy
\begin{equation}\label{nullhomology-second}
\FM(X_0 \cup_{\S^1\times \S^2} D^2 \times \S^2) \cong \FM(X_0) \times_{\FM(\S^1\times \S^2)} \FM(D^2 \times \S^2).
\end{equation}
and the commutative diagram 
\[\begin{tikzcd}
	{\FM(X')} &&&&& {\FM({D^2 \times \mathbb{S}^2})} \\
	& {\mathcal{F}\mathcal{B}_{X'}} &&& {\mathcal{F}\mathcal{B}_{{D^2 \times \mathbb{S}^2}}} \\
	&& {\R{\times}B} & B \\
	&& {\R{\times}B} & {\mathbb{S}^1{\times}B} \\
	& {\mathcal{F}\mathcal{B}_{X_0}} &&& {\mathcal{F}\mathcal{B}_{\mathbb{S}^1\times \mathbb{S}^2}} \\
	{\FM({X_0})} &&&&& {\FM({\mathbb{S}^1\times \mathbb{S}^2})}
	\arrow[dashed, from=1-1, to=1-6]
	\arrow["inclusion", from=1-1, to=2-2]
	\arrow[dashed, from=1-1, to=6-1]
	\arrow["inclusion"', from=1-6, to=2-5]
	\arrow["{{{=}}}"{description, pos=0.4}, bend left=18, from=1-6, to=3-4]
	\arrow["{\{1\} \times B}"', from=1-6, to=6-6]
	\arrow[from=2-2, to=2-5]
	\arrow["hol", from=2-2, to=3-3]
	\arrow[from=2-2, to=5-2]
	\arrow["f"', from=2-5, to=3-4]
	\arrow["pj", from=3-3, to=3-4]
	\arrow["{{=}}", from=3-3, to=4-3]
	\arrow["{\{1\} \times B}", from=3-4, to=4-4]
	\arrow["p", from=4-3, to=4-4]
	\arrow["hol"'{pos=0.6}, from=5-2, to=4-3]
	\arrow["{{{\partial_\infty}}}"', from=5-2, to=5-5]
	\arrow["f"{pos=0.6}, from=5-5, to=4-4]
	\arrow["inclusion"', from=6-1, to=5-2]
	\arrow[from=6-1, to=6-6]
	\arrow["{{{=}}}"{description, pos=0.4},bend right=18, from=6-6, to=4-4]
	\arrow["inclusion", from=6-6, to=5-5]
\end{tikzcd}\]
We see that $\FM(X_0)$ is a loop in $\R\times B$ under the map $hol\circ inclusion$, so is it in $\S^1 \times B$ under the map $p\circ hol\circ inclusion$. By the lifting property of the covering map we have a diagram
% https://q.uiver.app/#q=WzAsNixbMSwyLCJcXFJcXHRpbWVzIEIiXSxbMiwyXSxbMywyLCJcXG1hdGhiYntTfV4xIFxcdGltZXMgQiJdLFszLDAsIkIiXSxbMSwwLCJcXFJcXHRpbWVzIEIiXSxbMCwzLCJcXEZNKFhfMCkiXSxbMCwyLCJwIl0sWzMsMiwiXFx7MVxcfVxcdGltZXMgQiJdLFs0LDIsImV4cCBcXHRpbWVzIGlkIl0sWzAsNCwiaWQiLDJdLFszLDQsIlxcezBcXH1cXHRpbWVzIEIiLDJdLFs1LDBdLFs1LDRdXQ==
\[\begin{tikzcd}
	& {\R\times B} && B \\
	\\
	& {\R\times B} & {} & {\mathbb{S}^1 \times B} \\
	{\FM(X_0)}
	\arrow["{exp \times id}", from=1-2, to=3-4]
	\arrow["{\{0\}\times B}"', from=1-4, to=1-2]
	\arrow["{\{1\}\times B}", from=1-4, to=3-4]
	\arrow["id"', from=3-2, to=1-2]
	\arrow["p", from=3-2, to=3-4]
	\arrow[from=4-1, to=1-2]
	\arrow[from=4-1, to=3-2]
\end{tikzcd}\]
We see that the image of $\FM(X_0)$ and $B$ in $\mathbb{S}^1 \times B$ have algebraic intersection $0$. We still need to show that they intersect transversally.

We choose the perturbation family such that $1 \in \S^1$ is the point where the argument around (\ref{equ:surjIneq}) applies. In this case, we have Figure \ref{fig:Type0image}, and it shows that the differential of the moduli space on $X_0$ to the moduli space of the boundary (which is reducible) is surjective. This exactly means that the differential of $\FM(X_0)$ to the $\mathbb{S}^1$-factor (connection part) in $\mathbb{S}^1 \times B$ is surjective at $1\in \mathbb{S}^1$. Hence the image of $\FM(X_0)$ and $B$ intersect transversally. 

Therefore, the fiber product has $0$ signed points and by (\ref{nullhomology-second}) we have 
\[
\#\FM(X_0 \cup_{\S^1\times \S^2} D^2 \times \S^2) = 0.
\] 
\end{proof}

\begin{theorem}\label{thm:nullhomology-nonorientable}
Use the notation of Theorem \ref{ThmA}. But assume the following instead/additionally:
\begin{itemize}
\item $B$ is a circle;
\item $E_{\S^1} $ is an $\S^1$-subbundle of $E_X$, and it is a Klein bottle;
\item Each fiber of $E_{\S^1} $ is homologically trivial in the fiber of $E_X$.
\end{itemize}
Then
\[
FSW^{\Z/2}(E_{X'}, s')\equiv SW(X, s) \mod 2.
\]
(Here the family invariant is defined by counting the points.)
\end{theorem}
\begin{proof}
First consider the shape of the parameterized moduli space $\FM(X)$. Since we have chosen a regular parameter $\sigma= (g ,\eta)$, on $B $ minus any point, $\FM(X)$ is a cobordism. At each point $b$ of $B$, the parameter $\sigma(b) = (g(b) ,\eta(b))$ is regular to define 
\[
\M(X, g(b) ,\eta(b)).
\]
Hence the projection $\FM(X) \to B$ has degree $\# \M(X)$. Here $\# \M(X)$ doesn't depend on the choice of $b$, so we obmit the input.

Since $E_{\S^1}$ is a nonorientable $\S^1$-subbundle, the holonomy around each fiber of $E_{\S^1}$ would give the map 
\[
\mathcal{F}\mathcal{B}_{X} \to \R\widetilde{\times} B
\]
where $\R\widetilde{\times} B$ is the nonorientable $\R$ bundle over $B$. As in the first part in the proof of Theorem \ref{thm:nullhomology-orientable}, we have the diagram
\[\begin{tikzcd}
	{\FM(X)} &&&&& {\FM({\mathbb{S}^1\times D^3})} \\
	& {\mathcal{F}\mathcal{B}_X} &&& {\mathcal{F}\mathcal{B}_{\mathbb{S}^1\times D^3}} \\
	&& {\R\widetilde{\times} B} & {\mathbb{S}^1\widetilde{\times} B} \\
	&& {\R\widetilde{\times} B} & {\mathbb{S}^1\widetilde{\times} B} \\
	& {\mathcal{F}\mathcal{B}_{X_0}} &&& {\mathcal{F}\mathcal{B}_{\mathbb{S}^1\times \mathbb{S}^2}} \\
	{\FM({X_0})} &&&&& {\FM({\mathbb{S}^1\times \mathbb{S}^2})}
	\arrow[dashed, from=1-1, to=1-6]
	\arrow["inclusion", from=1-1, to=2-2]
	\arrow[dashed, from=1-1, to=6-1]
	\arrow["inclusion"', from=1-6, to=2-5]
	\arrow["{{=}}"{description, pos=0.4}, bend left=18, from=1-6, to=3-4]
	\arrow["{=}"', from=1-6, to=6-6]
	\arrow[from=2-2, to=2-5]
	\arrow["hol", from=2-2, to=3-3]
	\arrow[from=2-2, to=5-2]
	\arrow["f"', from=2-5, to=3-4]
	\arrow["{{\partial_\infty}}", from=2-5, to=5-5]
	\arrow["p", from=3-3, to=3-4]
	\arrow["{=}", from=3-3, to=4-3]
	\arrow["{=}"', from=3-4, to=4-4]
	\arrow["p", from=4-3, to=4-4]
	\arrow["hol"'{pos=0.6}, from=5-2, to=4-3]
	\arrow["{{\partial_\infty}}"', from=5-2, to=5-5]
	\arrow["f"{pos=0.6}, from=5-5, to=4-4]
	\arrow["inclusion"', from=6-1, to=5-2]
	\arrow[from=6-1, to=6-6]
	\arrow["{{=}}"{description, pos=0.4}, bend right=18, from=6-6, to=4-4]
	\arrow["inclusion", from=6-6, to=5-5]
\end{tikzcd}\]
from which we have
\begin{equation}
 \FM(X) \cong  \FM(X_0) 
\end{equation}
in $\R \widetilde{\times}  B$.

Now consider the gluing theory for $X_0 \cup_{\S^1\times \S^2} D^2 \times \S^2$. We have the isotopy
\begin{equation}
\FM(X_0 \cup_{\S^1\times \S^2} D^2 \times \S^2) \cong \FM(X_0) \times_{\FM(\S^1\times \S^2)} \FM(D^2 \times \S^2).
\end{equation}
Remember now $B\cong \S^1$. Let $\R\widetilde{\times} B$ be the Mobius band and $\S^1 \widetilde{\times} B$ be the Klein bottle. We have 
% https://q.uiver.app/#q=WzAsMTIsWzAsMCwiXFxGTShYJykiXSxbNSwwLCJcXEZNKHtEXjIgXFx0aW1lcyBcXG1hdGhiYntTfV4yfSkiXSxbMSwxLCJcXG1hdGhjYWx7Rn1cXG1hdGhjYWx7Qn1fe1gnfSJdLFs0LDEsIlxcbWF0aGNhbHtGfVxcbWF0aGNhbHtCfV97e0ReMiBcXHRpbWVzIFxcbWF0aGJie1N9XjJ9fSJdLFsyLDIsIlxcUntcXHRpbWVzfUIiXSxbMywyLCJCIl0sWzIsMywiXFxSXFx3aWRldGlsZGV7XFx0aW1lc30gQiJdLFszLDMsIlxcbWF0aGJie1N9XjEgXFx3aWRldGlsZGV7XFx0aW1lc30gQiJdLFsxLDQsIlxcbWF0aGNhbHtGfVxcbWF0aGNhbHtCfV97WF8wfSJdLFs0LDQsIlxcbWF0aGNhbHtGfVxcbWF0aGNhbHtCfV97XFxtYXRoYmJ7U31eMVxcdGltZXMgXFxtYXRoYmJ7U31eMn0iXSxbMCw1LCJcXEZNKHtYXzB9KSJdLFs1LDUsIlxcRk0oe1xcbWF0aGJie1N9XjFcXHRpbWVzIFxcbWF0aGJie1N9XjJ9KSJdLFswLDEsIiIsMCx7InN0eWxlIjp7ImJvZHkiOnsibmFtZSI6ImRhc2hlZCJ9fX1dLFswLDIsImluY2x1c2lvbiJdLFswLDEwLCIiLDAseyJzdHlsZSI6eyJib2R5Ijp7Im5hbWUiOiJkYXNoZWQifX19XSxbMSwzLCJpbmNsdXNpb24iLDJdLFsxLDUsInt7ez19fX0iLDEseyJsYWJlbF9wb3NpdGlvbiI6NDAsImN1cnZlIjotMX1dLFsxLDExLCJ7XFx7MVxcfSBcXHRpbWVzIEJ9IiwyXSxbMiwzXSxbMiw0LCJob2wiXSxbMiw4XSxbMyw1LCJmIiwyXSxbNCw1LCJwaiJdLFs0LDYsInt7PX19Il0sWzUsNywie1xcezFcXH0gXFx0aW1lcyBCfSJdLFs2LDcsInAiXSxbOCw2LCJob2wiLDIseyJsYWJlbF9wb3NpdGlvbiI6NjB9XSxbOCw5LCJ7e3tcXHBhcnRpYWxfXFxpbmZ0eX19fSIsMl0sWzksNywiZiIsMCx7ImxhYmVsX3Bvc2l0aW9uIjo2MH1dLFsxMCw4LCJpbmNsdXNpb24iLDJdLFsxMCwxMV0sWzExLDcsInt7ez19fX0iLDEseyJsYWJlbF9wb3NpdGlvbiI6NDAsImN1cnZlIjoxfV0sWzExLDksImluY2x1c2lvbiJdXQ==
\[\begin{tikzcd}
	{\FM(X')} &&&&& {\FM({D^2 \times \mathbb{S}^2})} \\
	& {\mathcal{F}\mathcal{B}_{X'}} &&& {\mathcal{F}\mathcal{B}_{{D^2 \times \mathbb{S}^2}}} \\
	&& {\R\widetilde{\times} B} & B \\
	&& {\R\widetilde{\times} B} & {\mathbb{S}^1 \widetilde{\times} B} \\
	& {\mathcal{F}\mathcal{B}_{X_0}} &&& {\mathcal{F}\mathcal{B}_{\mathbb{S}^1\times \mathbb{S}^2}} \\
	{\FM({X_0})} &&&&& {\FM({\mathbb{S}^1\times \mathbb{S}^2})}
	\arrow[dashed, from=1-1, to=1-6]
	\arrow["inclusion", from=1-1, to=2-2]
	\arrow[dashed, from=1-1, to=6-1]
	\arrow["inclusion"', from=1-6, to=2-5]
	\arrow["{{{{=}}}}"{description, pos=0.4},bend left=18 , from=1-6, to=3-4]
	\arrow["{{\{1\} \times B}}"', from=1-6, to=6-6]
	\arrow[from=2-2, to=2-5]
	\arrow["hol", from=2-2, to=3-3]
	\arrow[from=2-2, to=5-2]
	\arrow["f"', from=2-5, to=3-4]
	\arrow["pj", from=3-3, to=3-4]
	\arrow["{{{=}}}", from=3-3, to=4-3]
	\arrow["{{\{1\} \times B}}", from=3-4, to=4-4]
	\arrow["p", from=4-3, to=4-4]
	\arrow["hol"'{pos=0.6}, from=5-2, to=4-3]
	\arrow["{{{{\partial_\infty}}}}"', from=5-2, to=5-5]
	\arrow["f"{pos=0.6}, from=5-5, to=4-4]
	\arrow["inclusion"', from=6-1, to=5-2]
	\arrow[from=6-1, to=6-6]
	\arrow["{{{{=}}}}"{description, pos=0.4}, bend right=18, from=6-6, to=4-4]
	\arrow["inclusion", from=6-6, to=5-5]
\end{tikzcd}\]
To understand the intersection of $\FM(X_0)$ and $\FM({D^2 \times \mathbb{S}^2})$ in $\mathbb{S}^1\widetilde{\times}B$, we consider the universal cover of the Klein bottle 
\[
exp \widetilde{\times}  exp : \R \times \R \to \mathbb{S}^1 \widetilde{\times}  B
\]
and the universal cover of the M\"obius band
\[
id \widetilde{\times}  exp : \R \times \R \to \R \widetilde{\times}  B.
\]
Consider the diagram
% https://q.uiver.app/#q=WzAsNyxbMSwwLCJcXFJcXHRpbWVzIFxcUiJdLFszLDAsIkIiXSxbMSwyLCJcXFJcXHdpZGV0aWxkZXtcXHRpbWVzfSAgQiJdLFsyLDJdLFszLDIsIlxcbWF0aGJie1N9XjEgXFx3aWRldGlsZGV7XFx0aW1lc30gIEIiXSxbMCwzLCJcXEZNKFhfMCkiXSxbMSwxXSxbMCw0LCJ7ZXhwIFxcd2lkZXRpbGRle1xcdGltZXN9ICBleHB9Il0sWzEsMCwie1xcezBcXH1cXHRpbWVzIFswLDFdfSIsMl0sWzEsNCwie1xcezFcXH1cXHRpbWVzIEJ9Il0sWzIsNCwicCJdLFs1LDBdLFs1LDJdXQ==
\[\begin{tikzcd}
	& {\R\times \R} && B={\FM({D^2 \times \mathbb{S}^2})} \\
	& {} \\
	& {\R\widetilde{\times}  B} & {} & {\mathbb{S}^1 \widetilde{\times}  B} \\
	{\FM(X_0)}
	\arrow["{{exp \widetilde{\times}  exp}}", from=1-2, to=3-4]
	\arrow["{{id \widetilde{\times}  exp}}", from=1-2, to=3-2]
	\arrow["{{\{0\}\times [0,1]}}"', dashed, from=1-4, to=1-2]
	\arrow["{{\{1\}\times B}}", from=1-4, to=3-4]
	\arrow["p", from=3-2, to=3-4]
	\arrow[dashed, from=4-1, to=1-2]
	\arrow[from=4-1, to=3-2]
\end{tikzcd}\]
where dashed lines indicate lifts of paths if we regard both $B$ and $\FM(X_0)$ as a path. Cosider the lift of $\FM(X_0)$ in $\R\times \R$. Each component of the lift of $\FM(X_0)$, that winds around $B$ once, would intersect $\Z \times \R$ at odd number of points (see Figure \ref{fig:nonorientable}). Hence the algebraic intersection number between the lift of $\FM(X_0)$ and $\Z \times \R$ is exactly the degree of the projection $\FM(X_0) \to B$ (mod 2). 
\begin{figure}[ht!]
    \begin{Overpic}{\includegraphics[scale=0.6]{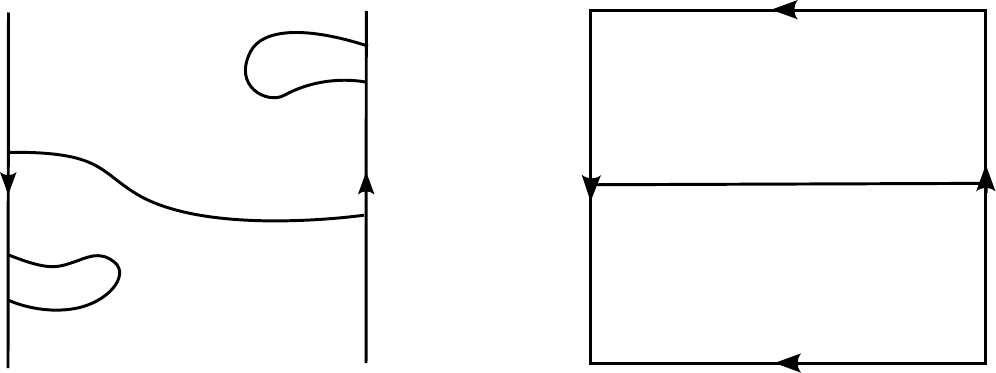}}
        \put(22,20){$\FM(X_0)$}
        \put(47,18){$\stackrel{p}{\to}$}
        \put(64,21){${\FM({D^2 \times \mathbb{S}^2})}$}
    \end{Overpic}
    \caption{}
    \label{fig:nonorientable}
\end{figure}

Note that $\Z \times \R$ is the preimage of $\FM({D^2 \times \mathbb{S}^2}) \subset \mathbb{S}^1 \widetilde{\times}  B$ under the map ${{exp \widetilde{\times}  exp}}$. So the algebraic intersection number between $\FM(X_0)$ and $\FM({D^2 \times \mathbb{S}^2})$ is exactly $\# \M(X)$ (mod 2). As in the proof of Theorem \ref{thm:nullhomology-orientable} we can assume that they intersect transversally. Hence we have
\[
\# \FM(X_0 \cup_{\S^1\times \S^2} D^2 \times \S^2)  =\# \M(X).
\]
\end{proof}

\section{Applications}
\subsection{Exotic diffeomorphisms on nonsimply connected manifolds}
Many exotic diffeomorphisms on symply connected $4$-manifolds are detected by the family Seiberg-Witten invariant. These results can be generalized to nonsymply connected manifolds by the surgery formula.

\begin{theorem}\label{thm:add-s1*s3-preserves-exotic}
Suppose $X$ is a symply connected smooth oriented compact $4$-manifold that admits an oriantation-preserving diffeomorphism $f$. Let $E_X$ be the mapping torus of $f$. Suppose the family Seiberg-Witten moduli space associated to the $\text{spin}^c_{GL}$ strcture $s$ on $E_X$ is $0$-dimensional, and 
\[
FSW(E_X , s) \neq 0.
\]
Furthermore, assume that $f$ admits an fixed point. Then the diffeomorphism $f\# id_{\S^1 \times \S^3}$ of $X\# \S^1 \times \S^3$  is not smoothly isotopic  to the identity.
\end{theorem}
\begin{proof}
Let $E_{X\# \S^1 \times \S^3}$ be the mapping torus of $f\# id_{\S^1 \times \S^3}$. Since $H^2( \S^1 \times \S^3)=0$, the $\text{spin}^c_{GL}$ strcture $s$ on $E_X$ can be extended to a $\text{spin}^c_{GL}$ strcture $s'$ on $E_{X\# \S^1 \times \S^3}$ by trivial extension. Choose the loop $\S^1 \times \{pt\}$ in the $\S^1 \times \S^3$ to be the surgery loop $\gamma$. Then all assumptions in Theorem \ref{ThmA} are satisfied and we have
\[
FSW^\Theta(E_{X\# \S^1 \times \S^3}, s') = FSW(E_X , s) \neq 0.
\]

If $f$ is replaced by the identity on $X$, we will have the trivial product $\S^1 \times (X\# \S^1 \times \S^3)$ and $\S^1 \times (X)$, and accordingly
\begin{equation}\label{equ:trivial-family-invariant}
FSW^\Theta(\S^1 \times (X\# \S^1 \times \S^3), \ss') = FSW(\S^1 \times (X), \ss) = 0
\end{equation}
for any $\text{spin}^c$ strcture $\ss$ on $X$ (the last equality comes from that the formal dimension of the Seiberg-Witten moduli space of $X$ is $-1$). 

Suppose that $f\# id_{\S^1 \times \S^3}$ is smoothly isotopic to the identity, then the isotopy $H: (X\# \S^1 \times \S^3 )\times I \to X\# \S^1 \times \S^3$ gives a diffeomorphism $F$ between $E_{X\# \S^1 \times \S^3}$ and $\S^1 \times (X\# \S^1 \times \S^3)$ that preserves the fiber by
\begin{align*}
F: \S^1 \times (X\# \S^1 \times \S^3) &\to E_{X\# \S^1 \times \S^3} \\
(t, x) &\mapsto H(x,t).
\end{align*}
But by (\ref{equ:trivial-family-invariant})
\[
FSW^\Theta(E_{X\# \S^1 \times \S^3}, s') \neq FSW^\Theta(\S^1 \times (X\# \S^1 \times \S^3), F^*s')
\]
which contradicts the hypothesis.
%Let $F^*s'$ be the pullback of the $\text{spin}^c_{GL}$ strcture $s'$.
\end{proof}

\begin{corollary}
Let $X$ be one of the following manifolds:
\begin{enumerate}
\item[\textbullet] $\mathbb{C}\mathbb{P}^2\# (\#^2\overline{\mathbb{C}\mathbb{P}}^2) \# E(n)$ for $n\geq 2$.
\item[\textbullet]  $\#^n(\mathbb{S}^2\times \mathbb{S}^2) \# (\#^nK3)$ for $n \geq 2$.
\item[\textbullet] $\#^{2n}\mathbb{C}\mathbb{P}^2  \# (\#^m\overline{\mathbb{C}\mathbb{P}}^2)$ for $n \geq 2$ and $m \geq 10n+1$.
\end{enumerate}
Then $X\# \S^1 \times \S^3$ admits an exotic diffeomorphism.
\end{corollary}
\begin{proof}
By \cite{Ruberman1998AnOT} and \cite{BK2020}, there exists a diffeomorphism $f: X \to X$ which is continuously isotopic to the identity, but the family Seiberg-Witten invariant of it is nonzero. It's possible to change $f$ by a smooth isotopy and get a diffeomorphism $f'$ that admits a fixed point. 

Let $H: X \times I \to X$ be a continuous isotopy from $id_X$ to $f'$.  Denote the fixed point of $f'$ by $x$. Let 
\begin{align*}
p: I &\to X \\
t &\mapsto H(x,t)
\end{align*}
be the trajectory of $x$ under the isotopy $H$. We want to squeeze this path to a point. By the property of the smooth manifold, there is a smooth isotopy $G:X\times I \to X$ starts from the identity, such that $G(p(t),t) = x$. Denote the end of $G$ by $g$. Then $g$ still fixes $x$. Let 
\begin{align*}
H': X\times  I &\to X \\
(x,t) &\mapsto G(H(x,t),t).
\end{align*}
Then $H'$ is a continuous isotopy starts from the identity, and preserves $x$. Denote the end of $H'$ by $f''$. $f''$ is just $g\circ f'$, so $f''$ is smoothly isotopic to $f'$. Form the connected sum $X\# \S^1 \times \S^3$ at $x\in X$ and any point in $\S^1 \times \S^3$. Then $H'\#id_{\S^1 \times \S^3} $ gives a continuous isotopy from $id_{X\# \S^1 \times \S^3}$ to $f'' \#id_{\S^1 \times \S^3}$.

$f''$ is smoothly isotopic to $f$. So the family Seiberg-Witten invariant of $f''$ is also nonzero. Therefore $f''$ satifies all assumptions in Theorem \ref{thm:add-s1*s3-preserves-exotic}. So $f''\# id_{\S^1 \times \S^3}$ is not smoothly isotopic to the identity.
\end{proof}

\subsection{Positive scalar curvature metrics on nonsimply connected $4$-manifolds}
Ruberman~\cite{ruberman2002positive} gives examples of simply connected manifolds for which the space of positive scalar curvature (psc) metrics is disconnected. This is demonstrated using family Seiberg-Witten invariant. We can generalize these results by the surgery formula.

First recall some definitions in \cite{ruberman2002positive}:
\begin{definition}
Let $X$ be a symply-connected smooth oriented compact $4$-manifold. For two generic parameters $h_0$ and $h_1$ on $X$, and $\ss$ a $\text{spin}^c$ structure such that the formal dimension of the Seiberg-Witten moduli space of $X$ is $-1$, define
\[
I(X,\ss;h_0,h_1) := \#\FM(X\times [0,1],\ss; \{h_t\})
\]
for any generic path $\{h_t\}$ connecting $h_0$ and $h_1$.
\end{definition}

\begin{definition}
For $f$ a diffeomorphism of $X$ and $h_0$ a parameter on $X$, 
\[
SW(f,\ss;h_0):= I(X,\ss;h_0,f^*h_0)
\]
Let $\mathcal O(f,\ss)$ be the orbit of $\ss$ by the action of the group $\langle f \rangle \subset\text{Diff}(X)$,
\[
SW_{tot}(f,\ss):= \sum_{\ss' \in \mathcal O(f,\ss)}SW(f,\ss'; h_0).
\]
\end{definition}

Also recall that 
\begin{enumerate}
\item[\textbullet] If $f$ preserves the orientation of $H^2_+(X)$, then
\[
SW_{tot}(f,\ss) = \sum_n I(X,\ss;f_{n}^*h_0,f_{n+1}^*h_0),
\]
where $f_{n}$ is the $n$-fold composition of $f$. If $f$ doesn't preserve the orientation of $H^2_+(X)$, then above equation is true in $\Z/2$.
\item[\textbullet] $SW_{tot}(f,\ss)$ is a finite sum since there are finite number of $\text{spin}^c$ structures on $X$ for which the parameterized moduli space is non-empty.
\item[\textbullet] $SW_{tot}(f,\ss)$ doesn't depend on the choice of the ground parameter $h_0$.
\end{enumerate}

This invariant is used to show that the space of psc metrics has infinite many components for some simply connected manifold, by showing that the total invariant is nonzero. We can generalize such result to nonsimply connected manifold. First we convert the definition of Ruberman to the family invariant setting in this paper. We can easily see the following facts from the definition of the family invariant. 
\begin{prop}
Let $h_0$ be any generic parameter of $X$. Let $E_X$ and the parameter family $h$ be the following:
\begin{enumerate}
\item[\textbullet] If $|\mathcal O(f,\ss)| = n$ is finite, let $E_X$ be the mapping torus of $f_n$, and $h$ be any generic path connecting $h_0, f^*h_0,\cdots, f_n^*h_0$ (note that $h$ is a parameter family on $E_X$ since $f_n$ sends the start of $h$ to the end of $h$).
\item[\textbullet] If $|\mathcal O(f,\ss)|$ is infinite, let $E_X$ be the the family of $X$ indexed over $\R$, and $h$ be any infinite generic path passing through $\cdots, f_{-1}^*h_0, h_0, f^*h_0,\cdots$.
\end{enumerate}
Then $SW_{tot}(f,\ss) = FSW(E_X, \ss, h)$.
\end{prop}

For a nonsimply connected manifold $X$ with $H^1(X;\Z) = \Z$ and a $\text{spin}^c$-structure $\ss$ such that the formal dimension of the parameterized moduli space is $\dim \FM(E_X,\ss)=1$, we can similarly define 
\begin{equation}\label{def:total}
SW_{tot}^\Theta(f,\ss) := FSW^\Theta(E_X, \ss, h).
\end{equation}
When $|\mathcal O(f,\ss)|$ is infinite, $\Theta$ is a noncompact element in the first cohomology group of the parameterized configuration space $\CP^\infty \times \S^1 \times \R$. In this case $FSW^\Theta(E_X, \ss, h)$ is still well defined because the parameterized moduli space is compact by an analogue of \cite{ruberman2002positive} Proposition 2.4. Also, $FSW^\Theta(E_X, \ss, h)$ doesn't depend on the choice of the ground parameter, so the definition (\ref{def:total}) makes sense:

\begin{theorem}
Suppose that $b^+(X)>2$ and that $f$ is a diffeomorphism preseving both the orientation and the homology orientation. Let $h_0$ and $k_0$ be generic paramters and $h, k$ be corresponding path. Then $FSW^\Theta(E_X, \ss, h) = FSW^\Theta(E_X, \ss, k)$.
\end{theorem}
\begin{proof}
Denote the path from $h_0$ to $f^*h_0$ by $K_{s,0}$, and the path from $k_0$ to $f^*k_0$ by $K_{s,1}$. Since $b^+(X)>1$, by Theorem \ref{thm:regular}, there exists a generic path $K_{0,t}$ from $h_0$ to $k_0$, and a generic path $K_{1,t}$ from $f^*h_0$ to $f^*k_0$. Since $b^+(X)>2$, by Theorem \ref{thm:regular} again, there exists a generic $2$-parameter family $K_{s,t}$ bounded by $K_{s,0}$, $K_{s,1}$, $K_{0,t}$, and $K_{1,t}$. Do this inductively we obtain a generic $2$-parameter family from $h$ to $k$. Hence there exists a cobordism from $\FM(E_X,\ss,h)$ to $\FM(E_X,\ss,k)$. This cobordism is a $2$-dimensional manifold with $1$-dimensional boundary, so after cutting it by the class $\Theta$, we obtain a $1$-dimensional cobordism which gives $FSW^\Theta(E_X, \ss, h) = FSW^\Theta(E_X, \ss, k)$ (see Figure \ref{fig:infinite-cobordism} for infinite $|\mathcal O(f,\ss)|$ case).
\begin{figure}[ht!]
    \begin{Overpic}{\includegraphics[scale=0.6]{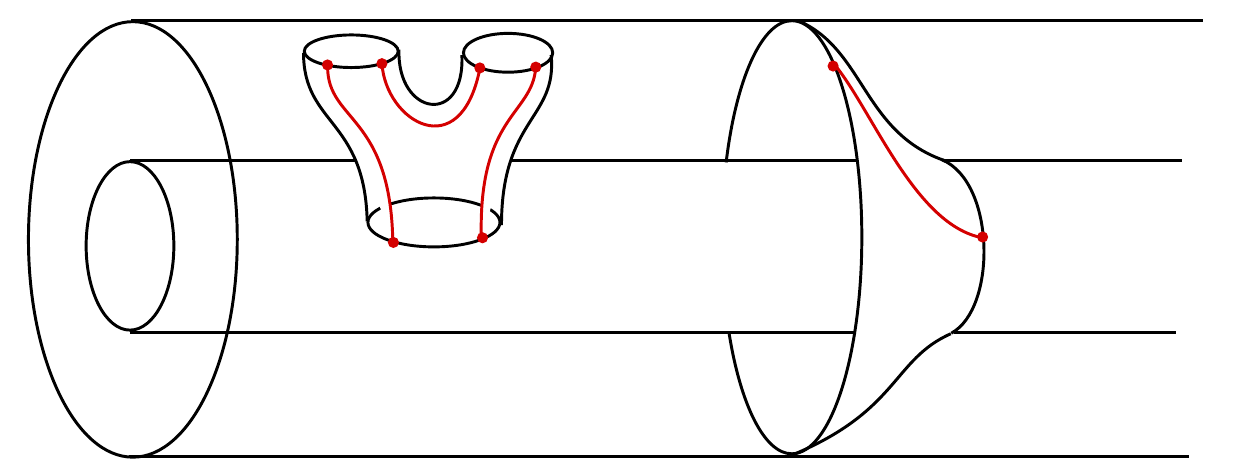}}
     \put(10,33){$h$}
    \put(10,20){$k$}
         \put(44.5,34){$\FM(E_X,\ss,h)$}
        \put(41,18){$\FM(E_X,\ss,k)$}
        
    \end{Overpic}
    \caption{The cobordism in $\CP^\infty \times \S^1 \times \R \times I$}
    \label{fig:infinite-cobordism}
\end{figure}
\end{proof}

We can show that the cut-down total invariant $SW_{tot}^\Theta(f,\ss) $ detects path components of the space of all psc metrics $\text{PSC}$:

\begin{theorem}\label{thm:component-trivial}
Suppose $b^+_2(X) > 2$ and $b^1(X)=1$. If $g_0$ is a psc metric on $Y$ such that $h_0=(g_0,0)$ is a generic parameter, and there exists a path $g_t$ in $\text{PSC}(X)$ connecting $g_0$ with $f^* g_0$, then $SW_{tot}^\Theta(f,\ss) = 0$.
\end{theorem}
\begin{proof}
In the case where we don't perturb the self dual two form, we can perturb the metric instead. The ordinary generic parameter argument (see Theorem \ref{thm:regular}) can be modified for a path of metrics, such that the term $F_A^{+_g}$ plays the role as the perturbing $2$-form in that theorem. Hence $b^+_2(X) > 2$ implies that the regular paths of metrics are generic. Being positive is an open condition, so $\text{PSC}(X)$ is open in the space of all metrics $\text{Met}(X)$. Therefore, there exists a regular path $g_t'$ in a neighborhood of $g_t$ with the same end points. Regularity means that $\FM(X\times [0,1],\ss; \{g_t'\})$ contains no reducible solutions.

By the Weitzenb{\"o}ck formula, a non-negative scalar curvature on $3$- or $4$-manifolds leads solely to reducible solutions of the Seiberg-Witten equations without the perturbing $2$-forms (see \cite{kronheimer_mrowka_2007} (4.22)). Hence $\FM(X\times [0,1],\ss; \{g_t'\})$ contains no irreducible solutions. Therefore $SW_{tot}^\Theta(f,\ss)$ is the integral on an empty space. Thus $SW_{tot}^\Theta(f,\ss) = 0$.
\end{proof}

On the other hand, the surgery formula we proved gives a relation between the total invariant and the cut-down total invariant:

\begin{theorem}\label{thm:formula-for-total-inv}
Suppose $X$ is a nonsimply connected manifold with $H^1(X;\Z) = \Z$, a diffeomorphism $f$, and a $\text{spin}^c$-structure $\ss$ such that the formal dimension of the parameterized moduli space is $\dim \FM(M(f),\ss)=1$. 

Let $\tau$ be a generator of $H^1(X;\Z)$. Suppose $\gamma$ is a loop of $X$ and $\tau$ evaluates $1$ at $\gamma$. Denote the resulting manifold by $X'$ and the $\text{spin}^c$ structure by $\ss'$. Let $f'$ be a diffeomorphism of $X'$ such that a family surgery on $(M(f),\ss)$ produces $(M(f'),\ss')$. Then 
\begin{equation}
SW_{tot}^\Theta(f,\ss) = SW_{tot}(f',\ss').
\end{equation}
\end{theorem}
\begin{proof}
By Theorem \ref{ThmA}, we have
\begin{align*}
SW_{tot}^\Theta(f,\ss) &= FSW^\Theta(M(f), \ss) \\
&= FSW^\Theta(M(f'), \ss') \\
&=SW_{tot}(f',\ss').
\end{align*}
\end{proof}

\begin{corollary}
Let $X =\#_{2n}\mathbb{C}\mathbb{P}^2\# (\#_{k}\overline{\mathbb{C}\mathbb{P}}^2) $ for any $n \le 2$ and $k> 10n$. Then the space of psc metrics on 
\[
X\# (\S^1 \times \S^3)
\] 
is nonempty and has infinitely many path components.
\end{corollary}

\begin{proof}
By \cite{ruberman2002positive} Corollary 5.2, the space of psc metrics on $X$ is nonempty and has infinitely many path components. By the results on how a surgery preserves the psc metrics (see \cite{SY} and \cite{GL}), the space of psc metrics on \[
X\# (\S^1 \times \S^3)
\]  is nonempty. By \cite{ruberman2002positive} Theorem 4.1, $X$ supports a diffeomorphism $g$ such that $SW_{tot}(g,\ss) \neq 0$. By a smooth isotopy of $f$, we can find a fixed point of $f$ and do the connected sum over that point, and get the diffeomorphism $f:= g\# (id_{\S^1 \times \S^3})$ on $X\# (\S^1 \times \S^3)$. Now we can apply Theorem \ref{thm:formula-for-total-inv} and get 
\[
SW_{tot}^\Theta(f,\ss\# \ss_0) = SW_{tot}(g,\ss) \neq 0,
\]
where $\ss_0$ is the unique $\text{spin}^c$-structure of $\S^1 \times \S^3$. 

Now we claim that $PSC(X\# (\S^1 \times \S^3))$ has infinitely many path components. Indeed, if $f_k^*g_0$ and $f_l^*g_0$ are in the same component of $PSC(X\# (\S^1 \times \S^3))$ for different integer $k$ and $l$, then there exists a path $g_t$ in $PSC(X\# (\S^1 \times \S^3))$ connecting $f_k^*g_0$ and $f_l^*g_0$, which indicates $SW_{tot}^\Theta(f_{k-l},\ss\# \ss_0) = 0$ by Theorem \ref{thm:component-trivial}. But this means that $SW_{tot}^\Theta(f,\ss\# \ss_0) = 0$ by definition (see the proof of \cite{ruberman2002positive} Theorem 3.4).
\end{proof}

\subsection{Path components of $\text{Diff}(X)$ for nonsimply connected manifold $X$}
We can generalize half-total invariant to nonsimply connected manifolds. 

\subsection{Family surgery on a homologically trivial but homotopically nontrivial loop}
To construct a nontrivial example, we use the construction of Gompf (\cite{Gompf95}, see also \cite{gompf19994} Theorem 10.2.10) that can construct a symplectic manifold with desired fundamental group. 

Use the notations in \cite{gompf19994} Theorem 10.2.10. Let $F$ be a genus $2$ surfaces with circles $\alpha_1, \alpha_2, \beta_1, \beta_2$ that represent a basis of $H^1(F;\Z)$ and $\alpha_i \alpha_j=0=\beta_i \beta_j, \alpha_i\beta_j = \delta_{ij}$. Take a $2$-torus $T^2$ with generating circles $\alpha, \beta$. In the product $F\times T^2$, take a collection of tori $T_i = \beta_i \times \alpha (i = 1,2)$ and $T_0 = \{pt\} \times T^2$. Purturb these tori and the product symplectic form $\omega$ on $F\times T^2$ such that the resulting tori $\{T_i'\}_0^2$ are disjoint symplectic submanifolds in $(F\times T^2, \omega')$. Let $X$ to be the symplectic normal connected sum of $F\times T^2$ and $3$ copies of $E(1)$ along each torus $T_i' \subset F \times T^2$ and a genetic fiber in each copy of $E(1)$. Then $X$ is symplectic $4$-manifold with a symlectic form $\omega_X$ and $\pi_1(X) = \langle \alpha_1, \alpha_2\rangle$.

Take a reflection $r$ of $F$ that fixes $\beta_i$ and reverse $\alpha_i$. This map can be extended to an involution of $X$, and we still call it $r$. The loop we choose to do the surgery is $\gamma = (\alpha_1\alpha_2\alpha_1^{-1}\alpha_2^{-1})(\alpha_2^{-1}\alpha_1^{-1}\alpha_2\alpha_1)$. $\gamma$ is a commutator so it is trivial in the homology. As the involution $r$ sends $\alpha_i$ to $\alpha_i^{-1}$, we see that $r(\gamma) = \gamma^{-1}$. 

From the construction of $X$, $\chi(X) = \sigma(X) = 0$. So the formal dimension of the moduli space with the class $c_1(X, \omega_X)$ is $0$. Moreover,
\[
SW(X, c_1(X, \omega_X)) = \pm 1.
\]
Let $E_X$ to be the mapping torus of $r$. $E_X$ satisfies the condition of Theorem \ref{thm:nullhomology-nonorientable}. Let $X'$ be the surgery manifold of $X$ along $\gamma$.

Let $E_{X'}$ be the resulting mapping torus after the family surgery. Then 
\[
FSW(E_{X'}, s')= SW(X, c_1(X, \omega_X)) = \pm 1.
\]
Therefore, there exists an exotic diffeomorphism of $X'$.

\bibliographystyle{alpha}
%\bibliographystyle{plain}
%\bibliography{./reference/braids_links}
\bibliography{./diff}
\end{document}